\documentclass[10pt]{amsart}

\usepackage{amsmath}
\usepackage{times}
\usepackage{amsfonts}
\usepackage{amsthm}
\usepackage{hyperref}
 \usepackage{amssymb}
 \usepackage{textcomp}
 \usepackage{mathabx} 
 \usepackage{color}
 \usepackage{mathrsfs}
 \usepackage{calrsfs}
  \usepackage{float}
\usepackage[all]{xy}
\usepackage{mathrsfs}
\usepackage{latexsym}
\usepackage[OT2,T1]{fontenc}
\DeclareSymbolFont{cyrletters}{OT2}{wncyr}{m}{n}
\DeclareFontFamily{OT1}{rsfs}{}

\usepackage{verbatim}
\usepackage{graphicx}
\usepackage{nomencl}
\DeclareFontEncoding{OT2}{}{} %

\numberwithin{equation}{subsection}

\newcommand{\Mot}{\mathrm{Mot}}
\newcommand{\coAd}{\mathrm{Ad}^*}

\newcommand{\Kinfty}{\mathrm{K}_{\infty}}
\newcommand{\aG}{\mathfrak{a}_G}
\newcommand{\aGs}{\mathfrak{a}_G^*}

\newcommand{\Ginfty}{G} 
\newcommand{\Hinfty}{H}
\newcommand{\ZZ}{\mathbf{Z}}

\newcommand{\Spec}{\mathrm{Spec}}
\newcommand{\Z}{\mathbf{Z}}
\newcommand{\Afinite}{\mathbf{A}_{\mathrm{f}}}

\newcommand{\diag}{\mathrm{diag}}

\newcommand{\Ext}{\operatorname{Ext}}
\newcommand{\CH}{\operatorname{CH}}

\newcommand{\coad}{\widetilde{\mathfrak{g}}}
\newcommand{\Killing}{B}

\newcommand{\omegat}{\tilde{\omega}}
\newcommand{\etat}{\tilde{\eta}}
\newcommand{\cQ}{\mathcal{Q}}
\newcommand{\cR}{\mathcal{R}}
\newcommand{\cM}{\mathcal{M}}
\newcommand{\D}{\mathcal{D}}
\newcommand{\MM}{\mathrm{MM}}

\newcommand{\Gal}{\operatorname{Gal}}

\newcommand{\p}{\mathfrak{p}}
\newcommand{\tr}{\operatorname{tr}}
\newcommand{\op}{\mathrm{op}}
\newcommand{\gK}{\mathfrak{g},\Kinfty}
\newcommand{\gKZ}{\mathfrak{g},\Kinfty^0}
\newcommand{\fg}{\mathfrak{g}}

\newcommand{\sigmabar}{\bar{\sigma}}

 \newcommand{\LW}{{}^L W}

\newcommand{\Aut}{\operatorname{Aut}}
\newcommand{\End}{\operatorname{End}}
\newcommand{\Qbar}{\overline{\Q}}
\newcommand{\bm}{\mathrm{bm}}

\newcommand{\Ad}{\operatorname{Ad}}
\newcommand{\Hom}{\operatorname{Hom}}
\newcommand{\rank}{\operatorname{rank}}
\newcommand{\Lie}{\mathrm{Lie}}

\newcommand{\disc}{\operatorname{disc}}
\newcommand{\SO}{\mathrm{SO}}
\newcommand{\triv}{\mathrm{triv}}

\newcommand{\Res}{\operatorname{Res}}

\newcommand{\Q}{\mathbf{Q}}
\newcommand{\GG}{\mathbf{G}}
\newcommand{\volume}{\mathrm{volume}}

\newcommand{\SL}{\operatorname{SL}}
\newcommand{\C}{\mathbf{C}}

\newcommand{\dR}{\mathrm{dR}}
\newcommand{\mot}{\mathrm{mot}}
\newcommand{\adele}{\mathbf{A}}
\newcommand{\Sym}{\mathrm{Sym}}
\newcommand{\OO}{\mathcal{O}}
\newcommand{\R}{\mathbf{R}}

\newcommand{\G}{\mathbf{G}}
 \newcommand{\HH}{\mathbf{H}}

\newcommand{\Qlbar}{\overline{\Q}_{\ell}}
\newcommand{\LG}{{}^{L}G}

\newcommand{\PGL}{\mathrm{PGL}}

\newcommand{\Sp}{\mathrm{Sp}}

\DeclareFontShape{OT1}{rsfs}{n}{it}{<-> rsfs10}{}
\DeclareMathAlphabet{\mathscr}{OT1}{rsfs}{n}{it}
\newcommand{\GL}{\mathrm{GL}}
\newcommand{\vol}{\operatorname{vol}}

\newtheorem{theorem}{Theorem}[subsection]
\newtheorem{definition}{Definition}[subsection]
\newtheorem*{lemma*}{Lemma}
\newtheorem{lemma}{Lemma}[subsection]
\newtheorem{prediction}{Prediction}[subsection]

\newtheorem{prop}{Proposition}[subsection]
\newtheorem{conj}{Conjecture}[subsection]

\theoremstyle{remark}
\newtheorem{rem}{Remark}

\newcommand{\m}{\mathcal{M}}

\newcommand{\var}{\mathrm{Var}}
\newcommand{\cl}{\operatorname{cl}}
\newcommand{\et}{\mathrm{et}}
\newcommand{\mm}{\! \mathcal{M}}
\newcommand{\Frob}{\operatorname{Frob}}
\newcommand{\comp}{\operatorname{comp}}
\newcommand{\B}{\mathrm{B}}
\newcommand{\DR}{\mathrm{DR}}
\newcommand{\cDR}{\mathcal{DR}}

\newcommand{\num}{\mathrm{num}}

\newcommand{\rat}{\mathrm{rat}}
\newcommand{\Gr}{\mathrm{Gr}}
\newcommand{\M}{\widetilde{M}}
\newcommand{\ve}{\varepsilon}
\newcommand{\A}{\mathbf{A}}

\newcommand{\Gh}{\widehat{G}}
\newcommand{\Gth}{\mathring{G}} %

\newcommand{\GC}{^{C}{}{G}{}{}}

\renewcommand{\Re}{\operatorname{Re}}

\newcommand{\sigmab}{{\bar{\sigma}}}

\newcommand{\bM}{\mathbf{M}}

\newcommand\mnote[1]{\marginpar{\tiny #1}}

\newcommand{\cL}{\mathcal{L}}

\begin{document} 

\setcounter{tocdepth}{1}

\title{Automorphic cohomology, motivic cohomology, and the adjoint $L$-function}
\author{Kartik Prasanna and Akshay Venkatesh}
\begin{abstract} We propose  a relationship between the cohomology of arithmetic groups, and the motivic cohomology
of certain (Langlands-)attached motives. The motivic cohomology group in question is that related, by Beilinson's conjecture,
to the adjoint $L$-function at $s=1$. 
We present evidence for the conjecture  using the theory of periods of automorphic forms, and using analytic torsion. 
\end{abstract}
\maketitle

\tableofcontents

 \section{Introduction}

A remarkable feature of the cohomology $H^*(\Gamma, \C)$ of arithmetic groups  $\Gamma$ is  
their spectral degeneracy: 
Hecke operators can act in several different degrees with exactly the same eigenvalues.
For an elementary introduction to this phenomenon, see \cite[\S 3]{Takagi}.
 In some cases, such as Shimura varieties,  it can be explained by the action of a Lefschetz $\mathrm{SL}_2$ but in general it is more mysterious. 

We shall propose  here
 that this degeneracy arises from a hidden degree-shifting action of a certain motivic cohomology group on $H^*(\Gamma, \Q)$. 
 This is interesting both as an extra structure of $H^*(\Gamma, \Q)$, and because it exhibits
a way to access the motivic cohomology group.  We do not know how to define the action directly, but we give a formula for the action tensored with $\mathbf{C}$, 
 using the archimedean regulator.      Our conjecture, then, is that this action over $\C$ respects $\Q$ structures. 
  
The conjecture has numerical consequences: it predicts what the ``matrix of periods'' for a cohomological automorphic form should look like.
 We shall verify a small number of these predictions. This is the main evidence for the conjecture at present;
 we should note that we found the verifications somewhat miraculous, as they involve a large amount of cancellation 
 in ``Hodge--linear algebra.''

It takes a little while to formulate the conjecture: in \S \ref{sec:numerical_invariants} we will set up   notation for the cohomology of arithmetic groups;
as usual it is more convenient to work with adelic quotients. 
We formulate the conjecture itself in \S \ref{sec:conj_formulation}.   \S \ref{sec:tori} discusses the case of tori -- this is  just a small reality check. In \S \ref{sec:edge} we describe how to extract numerically testable predictions from the conjecture,
some of which we have verified.

  \subsection{Cohomological representations} \label{sec:numerical_invariants} 

 Fix a reductive $\Q$-group $\mathbf{G}$, which we always suppose to have no central split torus. Let $S$ 
be the associated symmetric space; for us, this will be
$G/\Kinfty^0$, where $\Kinfty^0$
is a maximal compact {\em connected} subgroup
of $G := \mathbf{G}(\R)$; thus $S$ need not be connected, but $G$ preserves an orientation on it.

Let $\Afinite$ denote the finite adeles  of $\Q$ and let $K \subset \mathbf{G}(\Afinite)$ be a level structure, i.e., an open compact subgroup; we suppose that $K$ factorizes as $K = \prod_{v} K_v$. 
We may form the associated arithmetic manifold
$$Y(K) = \mathbf{G}(\Q) \backslash S \times \mathbf{G}(\Afinite)/K.$$  
If the level structure $K$ is fixed (as in the rest of the introduction) we allow ourselves to just write $Y$ instead of $Y(K)$. 

  The cohomology $H^*(Y, \Q)$ is naturally identified with the direct sum $\bigoplus H^*(\Gamma_i, \Q)$
  of group cohomologies of various arithmetic subgroups $\Gamma_i \leqslant \mathbf{G}$, indexed by the connected components of $Y$.
  However, it is much more convenient to work with $Y$; for example, the full Hecke algebra for $\mathbf{G}$
  acts on the cohomology of $Y$ but may permute the contributions from various components. 
 
 As we recall in \eqref{dimension cle} below, the 
 action of the Hecke algebra on $H^*(Y, \C)$ 
 often exhibits the same eigencharacter in several different cohomological degrees. 
  Our conjecture will propose the existence of extra endomorphisms of $H^*(Y, \Q)$ that commute with the Hecke algebra and explain this phenomenon.

First of all, we want to localize at a given character of the Hecke algebra.  For each $v$  
not dividing the level of $K$, i.e., at which $K_v$ is hyperspecial, let $\chi_v: \mathcal{H}(\mathbf{G}(\Q_v), K_v) \rightarrow \Q$
be a character. Consider the set of automorphic representations $\pi = \otimes \pi_v$ of $\mathbf{G}(\adele)$
such that:
\begin{itemize}
\item[-] $\pi^K \neq 0$
\item[-]   $\pi_{\infty}$ has nonvanishing  $(\gKZ)$-cohomology.
\item[-] For finite places $v$ not dividing the level of $K$ (places for which $K_v$ is hyperspecial)  the representation $\pi_v$ is spherical
and corresponds to the character $\chi_v$. 
\end{itemize} 
 This is a finite set, which we shall assume to be nonempty, say $$\Pi = \{\pi_1, \dots, \pi_h\}.$$   These automorphic representations are nearly equivalent,    and therefore belong to a single Arthur packet.   We will assume that {\em each $\pi_{i}$ is cuspidal and tempered at $\infty$.}
 Here, the cuspidality assumption is to avoid complications of non-compactness;  temperedness is important
 for the way we formulate our conjecture, but (conditionally on Arthur's conjecture) should follow
 from the character $\chi_v$ being tempered for  just one place $v$.  
  
We will be interested in the part of cohomology which transforms according to the character $\chi$, which
we will denote by a subscript $\Pi$:
\begin{align} H^*(Y, \Q)_{\Pi} = \{ h \in H^*(Y, \mathbf{Q}): T h = \chi_v (T) h \mbox{ for all $T \in  \mathcal{H}(\mathbf{G}(\Q_v), K_v)$}
\\  \nonumber \mbox{  and all places $v$ not dividing the level of $K$.}\}\end{align}
   We sometimes abridge $H^i(Y, \Q)_{\Pi}$ to $H^i_{\Pi}$. 

   In particular, under our assumptions above, 
$H^*(Y, \mathbf{C})_{\Pi}$ can be computed from the $(\gKZ)$-cohomology of the $\pi_i$.
The computation of the $(\gK)$-cohomology of tempered representations (see \cite[Theorem III.5.1]{BW} and also \cite[5.5]{Borel2} for the noncompact case) 
implies that
\begin{equation} \label{dimension cle} \dim H^j(Y, \mathbf{R})_{\Pi} = k \cdot   {\delta \choose j-q}, \end{equation}
where we understand ${\delta \choose a} =0 $ if $a \notin [0,\delta]$, 
\begin{equation} \label{dim-inv}  \delta :=\rank  \Ginfty - \rank \Kinfty,    \quad q := \frac{\dim Y-\delta}{2},\end{equation}
and $k = \dim H^q(Y, \mathbf{R})_{\Pi}$. 
For example, if $\mathbf{G} = \SL_{2m}$, then $q = m^2$ and $\delta = m-1$.  

In words, \eqref{dimension cle} asserts that the Hecke eigensystem indexed by $\Pi$
occurs in every degree between $q$ and $q+\delta$, with multiplicity proportional to ${\delta \choose j-q}$.

  \subsubsection{Galois representations and motives attached to $\Pi$}
    
In the situation just described,  $\Pi$ should conjecturally \cite{BG} have attached to it a compatible system of Galois representations 
$ \rho_\ell : \Gal(\Qbar/\Q) \rightarrow \LG(\Qlbar)$. Actually all that is important for us is the composition
 with the adjoint or the co-adjoint representation of $\LG$: 
 $$\Ad \rho_\ell: \Gal(\Qbar/\Q) \rightarrow \GL(\widehat{\mathfrak{g}} \otimes \Qlbar), \quad \Ad^*  \rho_\ell: \Gal(\Qbar/\Q) \rightarrow \GL(\widetilde{\mathfrak{g}} \otimes \Qlbar),
 $$ 
 where $\widehat{\mathfrak{g}}$ denotes the Lie algebra of $\widehat{\G}$ (considered as a split reductive $\Q$-group)  and $\widetilde{\fg} = \Hom(\widehat{\mathfrak{g}}, \Q)$ is its linear dual.
 In fact,     if $\G$ is not simply connected the representation $\rho_{\ell}$ requires, for its definition, a modification of the notion of $L$-group (see again \cite{BG});
 however, no such modification should be required for
  $\Ad \rho_{\ell}$ or $\Ad^* \rho_{\ell}$.
 
  We will assume throughout, as is predicted by the Langlands program, 
 that $\Ad \rho_\ell$ and $\Ad^* \rho_\ell$ are Galois representations underlying a Grothendieck motive; this motive (which is necessarily of weight zero) will be denoted by $\Ad \Pi$  or $\Ad^* \Pi$ respectively.
 Thus, for example, the Galois representation on the etale realisation of $\Ad \Pi$  is identified with  $\Ad \rho_{\ell}$.
    
    Before we proceed, a brief remark about ``adjoint'' versus ``coadjoint.''
The representations $\Ad \rho_\ell$ and $\Ad^* \rho_\ell = (\Ad \rho_\ell)^*$ are isomorphic
 if $\mathbf{G}$ is semisimple, because of the Killing form.  Consequently,  the associated motives $\Ad \Pi$ and $\Ad^* \Pi$ 
 should be isomorphic. However,   both to handle the reductive case and to be more canonical, we will distinguish between the two.

 \subsection{The conjecture} \label{sec:conj_formulation}

  It is expected (cf. (3.2) of \cite{LM}) that the adjoint $L$-function     $$ L(s, \Pi, \Ad^*) $$
  that is to say, the $L$-function attached to the motive $\Ad^* \Pi$, 
  is holomorphic at $s=1$ under our assumptions (in particular, 
that $\G$ has no central split torus). 
According to Beilinson's conjecture, the value of this $L$-function is related to a regulator on a certain motivic cohomology group
attached to $\Ad^* \Pi$. It is this motivic cohomology group that will play the starring role in our conjecture.
We defer to later sections more careful expositions of points of detail;
in particular,  what we need of motivic cohomology and Beilinson's conjectures is summarized in \S \ref{Beilinson}. 

 First, to  the real reductive group $G = \G_{\R}$ we shall attach in \S \ref{VoganZuckerman} a canonical $\C$-vector space $\aG$, such that $\dim(\aG) = \delta$;
it can be described in either of the following ways:
\begin{itemize}
\item[-]  $\aG$ is the  split component of a fundamental  Cartan subalgebra inside $\mathrm{Lie}(\G)_{\C}$.
\item[-] The dual $\aGs := \Hom_\C (\aG, \C)$ is the fixed points, on the Lie algebra $\mathrm{Lie}(\widehat{T})$ of the dual maximal torus, 
of $w_0 \sigma$, where $w_0$ is a long Weyl element and $\sigma$ is the (pinned) action of complex conjugation on $\widehat{G}$. \end{itemize}
We shall construct in  \S \ref{VoganZuckerman} an action of the exterior algebra $\wedge^* \aGs$   on the $\gKZ$-cohomology of a tempered representation of $\G(\R)$. 
 This gives rise to  a natural degree-shifting action of $\wedge^* \aGs$
 on $H^*(Y,  \C)_{\Pi}$,  with the property that the associated map 
 \begin{equation} \label{complexaction} H^q(Y,\C)_{\Pi} \otimes \wedge^i \aGs  \stackrel{\sim}{\longrightarrow} H^{q+i}(Y, \C)_{\Pi} \end{equation}
 is an isomorphism. 
For a more careful discussion see \S \ref{VoganZuckerman}.

Next,  standard conjectures
allow us to attach to a  Grothendieck motive $M$ over $\Q$ a motivic cohomology group
$H^i_{\mm} (M_{\Z}, \Q(j))$  (the subscript $\Z$ means that these are classes that ``extend to an integral model''; 
the group  $H^i_{\mm}$ should however be independent of integral model).   
Then $H^i_{\mm} (M_{\Z}, \Q(j))$ is a $\Q$-vector space,
conjecturally finite dimensional, and is equipped with a regulator map whose target is the Deligne cohomology 
$H^i_{\D} (M_\R, \R(j))$. 
We are interested in the case of $M= \Ad^* \Pi$, and write for brevity:
\begin{equation} \label{Ldef} L := H^1_{\m}((\coAd \Pi)_{\Z},  \Q(1)). \end{equation} 
 In this case 
(\S \ref{BRag}) the target of the archimedean regulator (tensored with $\C$)  is canonically identified with $\aG$;
 we get
therefore a map
 \begin{equation} \label{eqn:arch-reg}   
L \otimes \C \longrightarrow \mathfrak{a}_G 
\end{equation}
which is conjecturally an isomorphism.

Write $L^* = \Hom(L, \Q)$ for the  $\Q$-dual and $L^*_{\C} = \Hom(L, \C)$. {
 Dualizing \eqref{eqn:arch-reg}, the map  \begin{equation} \label{eqn:arch-reg-dual}  \aGs \longrightarrow L^*_{\C} \end{equation} 
is again conjecturally an isomorphism.  
 We are ready to formulate our central 
\begin{conj}  \label{mainconjecture} Notation as above: $H^*(Y, \C)_{\Pi}$ is the subspace of cohomology associated to the automorphic form $\Pi$,
$\aG$ is the $\C$-vector space  associated to $\GG$, and $L$ is, as in \eqref{Ldef}, 
  the motivic cohomology of the adjoint motive associated to $\Pi$.

Then the action of $\wedge^* \aGs$ on $H^*(Y, \C)_{\Pi}$ described above is compatible with rational forms, i.e., 
 if an element of $\aGs$ maps to $L^{*}$, then its action on cohomology preserves   $H^*(Y,\Q)_{\Pi} \subset H^*(Y, \C)_{\Pi}$.

\end{conj}
 
 In particular, the conjecture means that 
 \begin{quote}    {\em There is a natural, graded   action of $\wedge^* L^{*}$
on $H(Y, \mathbf{Q})_{\Pi}$, with respect to which the cohomology is freely generated in degree $q$. }
\end{quote}

As we mentioned earlier,  this is interesting because it suggests a  direct algebraic relationship between   motivic cohomology
and the cohomology of arithmetic groups. At present we cannot suggest any mechanism for this connection. The occurrence
of algebraic $K$-groups of rings of integers in the stable homology of $\GL_n$ is likely a degenerate case of it. 
	For the moment, we must settle for  trying to check certain numerical consequences.

Although it is not the concern of this paper, the conjecture  has  a $p$-adic counterpart, which
itself has a rich algebraic structure. 
As written, the conjecture postulates an action of $L^*$ on 
 $H^*(Y, \Q)_{\Pi}$; this action (assuming it exists) is pinned down
 because we explicitly construct the action of $L^*_{\C}$. 
 But the conjecture also implies that $L^*_{\Q_p} = L^* \otimes \Q_p$ acts
 on the cohomology with $p$-adic coefficients $H^*(Y, \Q_p)_{\Pi}$.   
 Conjecturally, the $p$-adic regulator gives an isomorphism
 \begin{equation}
 \label{eqn:p-adic-reg}
L \otimes \Q_{p}  \stackrel{\simeq}{\longrightarrow} H^1_f(\Q, \Ad^* \rho_{p}(1)),
\end{equation}
where the subscript $f$  denotes the ``Bloch-Kato Selmer group,''
\cite{BK}.    
This means that  there should be an action of $H^1_f(\Q, \Ad^* \rho_{p}(1))^{\vee} $  on $H^*(Y, \Q_p)_{\Pi}$ by degree $1$ graded endomorphisms.  
 The papers \cite{GV} and \cite{DHA} give two different ways of producing this action. One advantage of the $p$-adic analogue of the conjecture
 is that it is more amenable to computations, and 
   numerical evidence for  its validity will be given in  \cite{HV}. 
   
   Finally we  were informed by Michael Harris that 
   	 Goncharov has also suggested, in  private communication, the possibility of a connection between    the motivic cohomology group $L_{\Pi}$ and the cohomology of the arithmetic group.

\subsection{The case of tori} \label{sec:tori}

We briefly explicate our constructions in the case of tori. In this case the conjecture is easy,
but this case is helpful for reassurance and to pinpoint where there need to be duals in the above picture.  

 Let $\mathbf{T}$ be an anisotropic $\Q$-torus.   Let $\mathfrak{a}_T^*$ be the canonical $\C$-vector space attached 
 to $\mathbf{T}$, as in  the discussion preceding \eqref{complexaction}.  Then
 $\mathfrak{a}_T^*$ is canonically identified with the dual of
 $$\mathfrak{a}_T = \Lie(\mathbf{S}) \otimes \C,$$ where 
 $\mathbf{S}$  is the maximal $\mathbf{R}$-split subtorus of $\mathbf{T}_{\R}$.
This identification gives  a natural logarithm map
 $$ \log: \mathbf{T}(\R)  \rightarrow \mathfrak{a}_T$$
characterized by the fact that it is trivial on the maximal compact subgroup $\Kinfty$
  and coincides with the usual logarithm map on the connected component
  of $\mathbf{S}(\R)$.   
  
The associated symmetric space is
  $$Y = \mathbf{T}(\Q) \backslash  \mathbf{T}(\R) \times \mathbf{T}(\adele_f) / K\Kinfty^{\circ}$$
  Then $Y$ has the structure of a compact   abelian Lie group: each component is the quotient of $\mathbf{T}(\R)^{\circ}/\Kinfty^{\circ} \simeq  \mathfrak{a}_T$ by  the image of
  $$ \Delta  = \{t \in \mathbf{T}(\Q): t \in  \mathbf{T}(\R)^\circ \cdot  K\},$$
  which is a discrete   cocompact subgroup  of $\mathbf{T}(\R)$. 

   As in the general discussion above, there is a 
   natural action of $\wedge^* \mathfrak{a}_T^*$ on the cohomology of $Y$. In this case  the action of $\nu \in \wedge^* \mathfrak{a}_T^*  $ is by taking cup product with $\Omega(\nu)$. Here, 
      $$\Omega: \wedge^* \mathfrak{a}_T^* \longrightarrow \mbox{ invariant differential forms on $Y$}$$  comes
  from the identification of the tangent space of $\mathbf{T}(\R)/\Kinfty$ at the identity with $\mathfrak{a}_T$.
   Note that, for $\nu \in \mathfrak{a}_T^*$,  the cohomology class of $\Omega(\nu)$ is rational (i.e.,   lies in $H^1(Y, \Q)$)
if and only if $ \langle \log(\delta), \nu \rangle \in \Q$
for all $\delta \in \Delta$. 

On the other hand, as in our prior discussion, to any cohomological representation $\Pi$
is associated a motive $\Ad^* \Pi$ of dimension equal to $\dim(T)$.  
In fact, $\Ad^* \Pi$
is the Artin motive whose Galois realization is the (finite image) Galois representation on $X_*(\mathbf{T}) \otimes \Q$.  
Then $H^1_{\cM}(\Ad^* \Pi, \Q(1))  = \mathbf{T}(\Q) \otimes \Q$ and the  subspace of ``integral'' classes  is then identified with 
\begin{equation}  \label{int struct} H^1_{\cM}((\Ad^* \Pi)_{\Z}, \Q(1)) = \Delta \otimes \Q.\end{equation}
 The regulator map $H^1_{\cM}(\Ad^* \Pi_{\Z}, \Q(1)) \rightarrow \mathfrak{a}_T$  
is just the logarithm map.  

Then Conjecture \ref{mainconjecture} just says: if $\nu \in \mathfrak{a}_T^*$ takes $\Q$-values on $\log(\Delta)$, 
then cup product with $\Omega(\nu)$ preserves $H^*(Y, \Q)$.
But this is obvious, because  the assumption means that $\Omega(\nu)$ defines a class in $H^1(Y, \Q)$.

 \subsection{Numerical predictions and evidence for the conjecture}  \label{sec:edge}   We now turn to  describing  our evidence for the conjecture. 
 To do so, we must first  extract numerical consequences from the conjecture;  for this we put metrics on everything. 
It turns out there are plenty of consequences that can be examined even with minimal knowledge of motivic cohomology. 
 
 Throughout this section, we continue with the general setup of  \S \ref{sec:numerical_invariants}; in particular,
 all the cohomological automorphic representations that we consider are tempered cuspidal. 
 
  By a {\em metric} on a real vector space
 we mean a positive definite quadratic form; by a metric on a complex vector space we mean a positive definite Hermitian form. 
  If $V$ is a vector space with metric, there are induced metrics  on $\wedge^* V$ and on $V^*$; these arise
  by thinking of a metric as an isomorphism to the (conjugate) dual space, and then by transport of structure. 
A perfect pairing $V \times V' \rightarrow \mathbf{R}$ of metrized real vector spaces will be said to be a ``metric duality''
when there are  dual bases for $V, V'$ that are simultaneously orthonormal (equivalently: $V' \rightarrow V^*$ is isometric, for
 the induced metric on $V^*$). 
 
 If $V$ is a metrized real vector space and $V_{\Q} \subset V$ is a $\Q$-structure, 
 i.e., the $\Q$-span of an $\R$-basis for $V$, then we can speak of the volume of $V_{\Q}$,
\begin{equation} \label{voldef}  \vol V_{\Q} \in \R^*/\Q^*,\end{equation}
 which is, by definition,   the covolume of $\Z v_1 + \dots \Z v_n$
 for any $\Q$-basis $\{v_1, \dots, v_n\}$ for $V_{\Q}$, with respect to the volume form on $V_{\R}$ defined by the metric. 
Explicitly $$(\vol V_{\Q})^2 = \det(\langle v_i, v_j \rangle).$$
We will later allow ourselves to use the same notation even when the form $ \langle - , - \rangle$ is not positive definite;
thus $\vol V_{\Q}$ could be a purely imaginary complex number. 
   
Fix an invariant bilinear $\Q$-valued form on $\mathrm{Lie}(\mathbf{G})$, for which the Lie algebra of $\Kinfty$
is negative definite and the induced form on the quotient is positive definite. 
This gives rise to a $\Ginfty$-invariant metric on the symmetric space, and thus to a Riemannian metric on $Y$. 
Once this is  fixed,   $H^j(Y, \mathbf{R})_{\Pi}$ and $H^j(Y, \mathbf{C})_{\Pi}$ both  get metrics by means of  their identifications with harmonic forms. 
   (Scaling the metric $g \mapsto \lambda g$ leaves the notion of harmonic form unchanged; 
 but
 it scales the metric on $H^i$ by $\lambda^{d/2-i}$, where $d=\dim (Y)$.)

 The Poincar{\'e} duality pairing $H^j(Y, \mathbf{R}) \times H^{j^*} (Y, \mathbf{R}) \rightarrow \mathbf{R}$, 
 where $j +j^* = \dim(Y)$, induces a metric duality, in the sense just described.
  The same conclusions are true for 
 the induced pairing
\begin{equation} \label{PDpairing}H^j_{\Pi}(Y, \mathbf{R}) \times H^{j^*}_{\widetilde{\Pi}} (Y, \mathbf{R}) \rightarrow \mathbf{R}\end{equation} 
 between the $\Pi$ part and the $\widetilde{\Pi}$-part, where $\widetilde{\Pi}$ denotes the contragredient of $\Pi$;  since we are supposing that $\Pi$ arose from a $\Q$-valued character of the Hecke algebra,
  we have in fact $\widetilde{\Pi} \simeq \Pi$. 
 
 In \S \ref{motmetricss}, we explain how to 
 introduce on $\aGs  $ a metric  
for which the action of $\wedge^*  \aGs$ is ``isometric,''
i.e., for $\omega \in H^q(Y, \C)_{\Pi}$ and $\nu \in \wedge^t  \aGs$ we have
\begin{equation} \label{motmetric} \| \omega \cdot \nu\| = \|\omega\| \cdot \|\nu\|.\end{equation} 
This metric on $\aGs$ depends, of course, on the original choice of invariant form on $\mathrm{Lie}(\GG)$.
It also induces a metric, by duality, on $\aG$. 

Note that we also introduce an $\R$-structure on $\aGs$  -- the ``twisted real structure'', see Definition
\ref{twistedReal}-- which  is compatible with the real structure $L \otimes \R \subset L \otimes \C$,  and preserves the real structure $H^q(Y, \R) \subset H^q(Y, \C)$ --
see Lemma \ref{Beilinson regulator real structure}  and Proposition \ref{basic properties}.
Therefore, we get also corresponding statements for real cohomology.

 With these preliminaries, we now examine explicit period identities that follow from our conjecture:

\begin{prediction} \label{pred1}
Suppose that $\dim H^q(Y, \C)_{\Pi}=1$.
Let $\omega$ be a harmonic $q$-form  on $Y$ whose cohomology class generates
$H^q(Y, \Q)_{\Pi}$.  Then 
\begin{equation}\label{frodo0}
\langle \omega, \omega \rangle \sim  ( \vol \ L) \end{equation}
where the volume of $L$ is measured with respect to the metric induced
by the inclusion $L \subset \mathfrak{a}_G$; 
we have used the notation $A \sim B$ for $A/B \in \mathbf{Q}^*$. 
\end{prediction}

At first sight, \eqref{frodo0} looks like it would require a computation of the motivic cohomology group $L$ to test. However,   Beilinson's conjecture implies a formula for  $\vol(L)$ 
 in terms of the adjoint $L$-function and certain other Hodge-theoretic invariants (see \S \ref{volLcomp}).  Thus, although not formulated
 in a way that makes this evident, \eqref{frodo0} can be effectively tested without computation of motivic cohomology. 
  Note that  \eqref{frodo0} is equivalent to 
\begin{equation}\label{frodo}
 \frac{\langle \omega, \omega \rangle }{ \left| \int_{\gamma} \omega  \right|^2} \sim  ( \vol \ L) \end{equation}
where $\omega$ is now an arbitrary nonzero harmonic $q$-form belonging to $H^q(Y, \C)_{\Pi}$ and  $\gamma$ is  a generator for $H_q(Y, \Q)_{\Pi}$.

\proof   (that Conjecture \ref{mainconjecture} implies Prediction \ref{pred1}): 
 Let $\nu$ generate $\wedge^{\delta} L^*$ (the top exterior power).  The conjecture implies that $\omega' = \omega \cdot \nu$ 
gives a nonzero element of  $H^{q^*} (Y, \Q)_{\Pi}$, where $q+q^* = \dim(Y)$. 
Then  \eqref{PDpairing} and \eqref{motmetric} give
\begin{equation} \label{Judy Hopps} \|\omega\|_{L^2} \cdot \|\omega'\|_{L^2} \in \mathbf{Q}^*  \stackrel{\eqref{motmetric}}{\implies}
\langle \omega, \omega \rangle \cdot \|\nu\|  \in \mathbf{Q}^*.\end{equation}

Now $\|\nu\|$ is precisely  the volume of $L^*$ with respect to the given metric on $\aGs$; 
said differently, $\|\nu\|^{-1}$ is the volume of $L$ for the dual metric on $\mathfrak{a}_G$. \qed 

The first piece of evidence for the conjecture, informally stated, is a verification of Prediction 1, in the following sense (see Theorem \ref{MainPeriodTheorem} for precise statement):  
\begin{quote}  
Assume  Beilinson's conjecture, as formulated in \S \ref{Beilinson}.
Assume also the Ichino--Ikeda conjecture on period integrals and the ``working hypotheses'' on local period integrals,
all formulated in \S \ref{sec:periodintegrals}.\footnote{Note that  many cases of the Ichino--Ikeda conjecture are already  known:
we include in our formulation the $\GL_n \times \GL_{n+1}$ cases, which were established by Jacquet, Piatetski-Shapiro and Shalika. 
Also, the 
working hypotheses on local period integrals are primarily used to handle archimedean integrals. In view of recent work there is reason to hope
that they should be soon removed.}   

Let $(\GG,\HH)$ be as in the``cohomological GGP cases'' of \S \ref{Ggp-precise}:
either $(\PGL_{n+1} \times \PGL_n \supset \GL_n)$ over $\Q$ \footnote{In this case, we prove not \eqref{oinkers} but a slight modification thereof, since the hypothesis $\dim H^q (Y,\C)=1$ is not literally satisfied.}, or $(\PGL_{n+1} \times \PGL_n \supset \GL_n)$ over a quadratic imaginary field,
or $(\mathrm{SO}_{n+1} \times \mathrm{SO}_n \supset \mathrm{SO}_n)$ over a quadratic imaginary field. 

Then, for $\omega$ a cohomological form on $\GG$,  and
 $\gamma$ the homology class of  the cycle defined by $\HH$ we have
\begin{equation} \label{oinkers}  \frac{ \left|\int_{\gamma} \omega  \right|^2}{\langle \omega, \omega \rangle} \in \sqrt{\Q} (\vol L)^{-1},\end{equation}
In other words, \eqref{frodo} is always compatible with the period conjectures of Ichino--Ikeda, up to possibly a factor in $\sqrt{\Q}$.

  \end{quote}

 \begin{rem}
  The left-hand side of \eqref{oinkers} is nonzero if and only if 
  both:
  \begin{itemize}
  \item 
the central value of the Rankin--Selberg $L$-function for $\Pi$ is nonvanishing, where $\Pi$ is the automorphic representation underlying $\omega$.
\item  (in the $\SO$ cases only): there is abstractly a nonzero $\HH(\adele)$-invariant functional on $\Pi$  (this condition can be rephrased in terms of $\varepsilon$-factors, by \cite{Waldspurger}). 
 \end{itemize}
\end{rem}

Without getting into details let us say why we found the proof of this striking. 
The conjecture is phrased in terms of the motivic cohomology group $L$; this group  is closely related to the adjoint $L$-function $L(s, \Pi, \Ad)$ at the edge point $s=0$ or $s=1$. 
 By contrast, the Ichino--Ikeda conjecture involves  various Rankin-Selberg type $L$-functions,  
 and it is, at first, difficult to see what they have to do with $L$.
We are saved by the following surprising feature (for which we have no satisfactory explanation, beyond it coming out of the linear algebra):
Beilinson's conjecture for the central values of these Rankin-Selberg type $L$-functions (which in this case is due to Deligne) involves many of the same ``period invariants''  as Beilinson's conjecture for the adjoint $L$-functions at $s=1$,
leading to various surprising cancellations.  
   A further miracle is that all the factors of $\pi$ (the reader can  glance at the Table in \S 7 to get a sense of how many of them there are) 
all cancel with one another. Finally, there are various square classes that occur at several places in the argument, giving rise to the $\sqrt{\Q}$ factor.
To the extent that we tried to check it, these square classes indeed cancel, as we would expect; however,
we found that this added so much complexity to the calculations that we decided to omit it entirely. 

We should also like to acknowledge that there is a substantial body of work on the cohomological period in degree $q$, for example \cite{Schmidt, KazhdanMazurSchmidt, shahidiraghuram,
Raghuram, Grobner}.   
The focus of those works is the relationship between this period and Deligne's conjecture, and many of these
papers go much further than we do in verifying what we have simply called ``working hypotheses,''
and in evading the issues arising from possibly vanishing central value.  Our work adds nothing in this direction, 
but our focus is fundamentally different: it sheds light not on the interaction betwen this period and Deligne's conjecture,
but rather its interaction with the motivic cohomology group mentioned above. 
 
 \begin{rem} The fact that we obtain no information when the Rankin--Selberg $L$-function vanishes at the critical point may seem disturbing at first. However,   we do not regard it as onerous:  if one assumes standard expectations about the   the frequency of non-vanishing $L$-values, it should be possible to deduce 
\eqref{frodo} for {\em all} such $\Pi$ -- again, up to $\sqrt{\Q^*}$, and assuming Beilinson's conjectures.  

Consider, for example, the case of $\PGL_n$ over an imaginary quadratic field.
For any cohomological automorphic representation $\pi_2$ on $\PGL_2$,
the equality \eqref{frodo} can be verified using known facts about nonvanishing of $L$-functions. (Note that in this case 
the evaluation of the  left-hand side in terms of $L$-functions was already carried out by Waldspurger \cite{Waldspurger}.)  Now for a given form $\pi_3$ on $\PGL_3$ one
expects that there should be a cohomological form $\pi_2'$ on $\PGL_2$
for which $L(\frac{1}{2}, \pi_3 \times \pi_2') \neq 0$; in this case, our result \eqref{oinkers}  
above permits us to deduce the validity of \eqref{frodo} for $\pi_3 \times \pi_2'$, and thus for $\pi_3$ (and then also for $\pi_3 \times \pi_2$ for any $\pi_2$). We may then proceed inductively in this way to $\PGL_n$ for arbitrary $n$. 

Admittedly, such non-vanishing results seem to be beyond current techniques of proof;  nonetheless this reasoning suggests that the result above should be regarded as evidence in a substantial number of cases. 

  As for the ``working hypotheses'' on archimedean period integrals,  these do not seem entirely out of reach;
  a key breakthrough on nonvanishing has now been made by Sun \cite{Sun}.  We have formulated the hypotheses fairly precisely and  we hope that the  results of this paper will give   further impetus
  to studying and proving them. 
  \end{rem}

Next, suppose that $\dim H^q(Y,\C)_{\Pi} =d>1$. 
 Choose a basis $\omega_1, \dots, \omega_d$  of harmonic forms whose classes give a $\Q$-basis for $H^q(Y, \Q)_{\Pi}$.
 Then  similar reasoning to the above gives
\begin{equation} \label{period_matrix} \det \left( \langle \omega_i, \omega_j \rangle  \right)  \sim (\vol L)^d. \end{equation} 
 In particular, the determinant of the period matrix in lowest degree should be ``independent of inner form'' -- a phenomenon that   has been observed for Shimura varieties where it is closely tied to the Tate conjecture \cite{Shimura0}, \cite{Oda}, \cite{HarrisCrelle}, \cite{PrasannaThesis}.
 More precisely, if  $\mathbf{G}, \mathbf{G}'$ are inner forms of one another, 
 we may equip the associated manifolds $Y$ and $Y'$ with compatible metrics -- i.e., arising from invariant bilinear forms on $\Lie(\GG)$ and $\Lie(\GG')$
 which induce the same form on $\Lie(\GG) \otimes \overline{\Q} = \Lie(\GG') \otimes \overline{\Q}$. 
 Construct automorphic representations $\Pi$ and $\Pi'$ as in \S  \ref{sec:numerical_invariants} 
starting with (for almost all $v$) matching characters  $\chi_v, \chi_v'$ of the local Hecke algebras. We assume that all the representations in $\Pi$ and $\Pi'$ are tempered cuspidal, as before.

 \begin{prediction} 
Suppose, as discussed above,  that
 $\mathbf{G}, \mathbf{G}'$ are inner forms of one another,   and 
 $\Pi, \Pi'$ are almost-equivalent automorphic representations,  contributing to the cohomology of both $Y$ and $Y'$. 
Equip $Y, Y'$ with compatible metrics, as  explained above. 
Then $$\det \left( \langle \omega_i, \omega_j \rangle  \right)^{d'} \sim \det \left( \langle \omega'_i, \omega'_j \rangle  \right)^{d},$$
where $d = \dim H^q(Y, \Q)_{\Pi}$, $d'$ is similarly defined, and the $\omega, \omega'$ are as above a basis for harmonic forms
which give $\Q$-rational bases for cohomology. 
\end{prediction}
Again, this prediction is pleasant because it does not mention motivic cohomology.  
One can give some evidence
for it,  as well as give other conjectures of similar nature that involve comparisons between $Y$ and $Y'$: there are results of this form in \cite{CaVe}, for example. Rather we  move on to a more interesting consequence.
 
 The above examples mentioned only periods in the lowest cohomological degree ($q$) 
to which tempered representations contribute. 
The conjecture, however,  gives control on the cohomology groups $H^*(Y, \Q)_{\Pi}$
 in {\em intermediate} dimensions $q < j < q^*$.  In principle,
 it allows us to compute the entire ``period matrix'' of cohomology, i.e.,  the matrix of pairings
 $\langle \gamma_i, \omega_j \rangle$ between a $\Q$-basis $\gamma_i$ for homology and an orthogonal basis $\omega_j$ of harmonic forms, 
 {\em given} a complete knowledge of $L$. It is difficult, however, to test
 this directly, for two reasons:
 \begin{itemize}
 \item it is almost impossible to numerically compute with motivic cohomology, and
 \item it is hard
 to exhibit explicit cycles in those dimensions (at least, it is hard to exhibit cycles that are geometrically or group-theoretically natural). 
 \end{itemize}

Here is a case where we can  finesse both of these issues.
 Suppose that $L \supset F$ is a field extension. 
Start with an $F$-algebraic group $\mathbf{G}$; 
let $\mathbf{G}_F$ be the restriction of scalars of $\mathbf{G}$ from $F$ to $\Q$,
and let $\mathbf{G}_L$ be the restriction of scalars of $\mathbf{G} \times_F L$
to $\Q$. 
We write $\delta_F, q_F, \delta_L, q_L$ for the quantities
defined in  \eqref{dim-inv} but for $\GG_F$ and $\GG_L$ respectively.
A  (near-equivalence class of) cohomological automorphic representation(s) $\Pi_F$ for $\mathbf{G}_F$
  conjecturally determines
a base change lift $\Pi_L$ on $\mathbf{G}_L$. Let $L_{\Pi_F}$ and $L_{\Pi_L}$ 
be the motivic cohomology groups attached to $\Pi_F, \Pi_L$ respectively. We will assume
that the archimedean regulator is an isomorphism on these groups; in particular $\dim L_{\Pi_L} = \delta_L$
and $\dim L_{\Pi_F} = \delta_F$. 
  Now
there is a natural map (dualizing a norm map)  $L_{\Pi_F}^* \hookrightarrow L_{\Pi_L}^*$
and the induced map 
\begin{equation} \label{motivic inclusion} \wedge^{\delta_F} L_{\Pi,F}^* \subset \wedge^{\delta_F} L_{\Pi,L}^*\end{equation}
has image a $\Q$-line inside $\wedge^{\delta_F} L_{\Pi, L}^*$. 

To get a sense of what this implies, suppose that we can fix a level structure for $\GG_L$
such that the associated manifold $Y_L$ satisfies  $\dim H^{q_L}(Y_L, \C)_{\Pi_L} = 1$.   Then the $\Q$-line above should, according to the conjecture,
give rise to a ``distinguished'' $\Q$-line  $\Q \eta \subset H^{q_L+\delta_K}(Y_L, \Q)$ -- namely, we act
on the $\Q$-line $H^{q_L}(Y_L, \Q)_{\Pi_L}$ using the image of \eqref{motivic inclusion}.  
The conjecture also allows us to predict
various periods of the cohomology class $\eta$   in terms of $L$-functions. In some special cases when $L/F$ is quadratic
 (e.g., when $\mathbf{G} = \GL_n$) this is related to the theory of base change;
but when $[L:F]  > 2$
 this seems to be a new and ``exotic'' type of base change identity (indeed, in the classical theory, only quadratic base changes have a nice ``period'' interpretation).   We can generalize this in various evident ways, e.g.
 if $L/F$ is Galois we can isolate various subspaces of $L_{\Pi,L}$ indexed by representations of $\Gal(L/F)$,
 and make a corresponding story for each one. 
 
Let us turn this discussion into a more precise prediction in one case:

\begin{prediction} \label{pred3} 
Notation as above; 
suppose that $L/F$ is Galois, with Galois group $\Gal_{L/F}$,
 and split at all infinite primes. Choose a level structure   for $\mathbf{G}_F$
 and a $\Gal_{L/F}$-invariant level structure for $\mathbf{G}_L$, giving arithmetic manifolds $Y_F$ and $Y_L$ respectively.
 Fix compatible metrics on $Y_F$ and $Y_L$. 
 Suppose again that $$\dim  H^{q_F}(Y_F,\Q)_{\Pi_F} = \dim H^{q_L}(Y_L, \Q)_{\Pi_L} =1.$$
Then there exist  harmonic representatives  $\omega_F, \omega_L, \omega_L'$   for  nonzero
classes in 
\begin{equation} \label{Hq2} H^{q_F}(Y_F, \Q)_{\Pi_F}, H^{q_L}(Y_L, \Q)_{\Pi_L}, \ H^{q_L+\delta_F} (Y_L, \Q)_{\Pi_L}^{\Gal_{L/F}}.\end{equation}
such that
 \begin{equation}\label{frodo2}
\frac{\|\omega_{L}'\|  \|\omega_F\|^2}{\|\omega_L\|}  \in  \sqrt{[L:F]} \cdot \mathbf{Q}^* \end{equation}
\end{prediction}
In the case $\delta_F = 1$, the third space of \eqref{Hq2} is also one-dimensional and $\omega_F, \omega_F'$ and $\omega_L'$  are   all determined up to $\mathbf{Q}^*$;
  we can  achieve a similar situation in general by a slightly more careful discussion of $\omega_L'$.

As in \eqref{frodo}, we can translate this to a statement of periods and $L^2$ norms. 
The nice thing about \eqref{frodo2} is that, like the second prediction, it doesn't involve any motivic cohomology.

\proof
Let $\nu_F$ be a generator for $\wedge^{\delta_F} L_{\Pi,F}^*$.   
As in \eqref{Judy Hopps} we have $$ \langle \omega_F, \omega_F \rangle \cdot \|\nu_F\|_{\mathfrak{a}_F} \in \mathbf{Q}^*.$$
Let $\nu_L$ be the image of $\nu_F$ under \eqref{motivic inclusion} and set $\omega_L' =\omega_L \cdot \nu_L$.
Also $\|\nu\|_{\mathfrak{a}_L} = \sqrt{[L:F]} \times  \|\nu\|_{\mathfrak{a}_F}.$ 
Taking norms and using \eqref{motmetric} we get
the result. \qed

 The second piece of evidence for the conjecture  is a verification of Prediction \ref{pred3}, in the following setting (see \S \ref{pred3proof} and also Theorem \ref{SecondMainTheorem} for a more general statement):
\begin{quote}  
\eqref{frodo2} is valid up to $\sqrt{\Q^*}$ when $\mathbf{G}$ is an inner form of $\PGL_2$,  $F$ is a quadratic imaginary field,  
$L \supset F$ is a cyclic cubic extension,  and (for level structures precisely specified) 
$\Pi_F$ is the only non-trivial representation contributing to  $H^*(Y_F)$ and $\Pi_L$ is the only non-trivial representation contributing to $H^*(Y_L)$. 
\end{quote}
 Note that we do this without knowing how to produce any cycles on the nine-manifold $Y_L$  in dimension $q_L +\delta_F=4$!  
 Rather we proceed indirectly, using analytic torsion. 

In fact, in the text, we prove a stronger result (Theorem \ref{SecondMainTheorem}), which
relies for its phrasing on Beilinson's conjectures.

\subsection{Some problems and questions}

Here are a few problems that are suggested by the conjecture:  
\begin{itemize}
 
 \item[(i)] General local systems: it would be interesting to generalize our discussion beyond the case of the trivial local system.  While  the general picture should adapt to that setting,  the verifications described in 
\S \ref{sec:edge}  use the specific numerology of Hodge numbers associated to the trivial local system --  it is by no means apparent the same
miraculous cancellations should occur in general. 
 
 \item[(ii)] Non-tempered representations: our entire discussion in this paper concerns only tempered representations,
 but it seems very likely that the phenomenon continues in the non-tempered case. For example,
 that part of the cohomology of $Y$ associated with the {\em trivial} automorphic representation
 shows interesting connections with algebraic $K$-theory.   It seems important to formulate precisely
 the conjecture in the general case.

\item[(iii)]  Coherent cohomology:  a Hecke eigensystem  can appear in multiple cohomological degrees.
For example, this already happens for the modular curve, in the case of weight one.   It would be good to develop a version of the
theory of this paper that applies to that context. 

\item[(iv)]  We have formulated here a conjecture concerning rational cohomology; but, of course, it would be most desirable
to understand the integral story. It is plausible that this can be done by integrating the current discussion
with that of the derived deformation ring, developed in \cite{GV}.
 
\end{itemize}

\subsection{Notation} \label{sec:notn}

We gather here some notation that will be consistently used throughout the paper.

As in \S \ref{sec:edge}, we will often refer to the ``volume'' of a vector space: if $V_{\Q}$
is a rational vector space, equipped with a real-valued symmetric bilinear form $\langle -,  - \rangle$ on $V_{\R}$,  we define 
$\vol_{\Q} V \in \C^*/\Q^*$ by the rule 
  $$(\vol V_{\Q})^2 = \det(\langle v_i, v_j \rangle),$$
  for a $\Q$-basis $v_1, \dots, v_n$.  If the form $\langle -, - \rangle$ is indefinite, 
  the volume could be imaginary.

 $\GG$ will denote a reductive group over $\Q$; for all the global conjectures we will assume that $\GG$ has no central split torus.  
   $\widehat{G}$ denotes
the dual group to $\GG$,  a complex reductive Lie group.   It is equipped with a pinning, in particular a ``Borus'' $\hat{T} \subset \hat{B}$.  
We put $\LG = \widehat{G} \rtimes \Gal(\overline{\Q}/\Q)$, as usual.  Now $\widehat{G}$ and $\LG$ can be descended
to algebraic groups over $\Z$, using the Chevalley form of $\widehat{G}$;  we will, by a slight abuse of notation, refer to the $R$-points
of the resulting groups by $\widehat{G}(R)$ and $\LG(R)$. We will also write $\widehat{G}_R$ and $\LG_R$ for the corresponding algebraic groups over $\Spec \ R$.

We denote by $G=\GG(\R)$ the real points of $\GG$ and by $\Kinfty$ a maximal compact subgroup of $G$.
Set
$ \mathfrak{g}_{\Q} = \Lie(\GG)$ to be the $\Q$-Lie algebra, and set  
$$ \mathfrak{g}_{\R}=  \Lie(\Ginfty)=    \mathfrak{g}_{\Q} \otimes \R ,  \mathfrak{k}_{\R} := \mathrm{Lie}(\Kinfty),  $$
$$ \mathfrak{g} = \mathfrak{g}_{\R} \otimes \C, \mathfrak{k} = \mathfrak{k}_{\R} \otimes \C.$$
We denote by $\GG_{\R}$ the base-change of $\GG$ from $\Q$ to $\R$, and similarly define $\GG_{\C}$.

We set 
\begin{equation} \label{GGdef} [\GG]=\GG(\Q) \backslash \GG(\adele)\end{equation} to be  the associated adelic quotient.  We will usually use the letter $K$ to denote an open compact 
subgroup of $\GG(\Afinite)$.  For such a $K$, we have an attached ``arithmetic manifold,''
\begin{equation} \label{YKdef} Y(K)  = [\GG]/\Kinfty^{\circ} \cdot K,\end{equation}
which coincides with the definition given in the introduction. 

There are two numerical invariants attached to $\GG$ and $Y(K)$ which will occur often. 
Firstly, the difference $\delta = \mathrm{rank}(G) - \mathrm{rank}(\Kinfty)$ between the ranks of $G$ and its maximal compact subgroup.
Secondly, the minimal cohomological dimension $q$ in which a tempered $G$-representation has nonvanishing $(\gK)$-cohomology;
these are related via 
 $2q + \delta =  \dim Y(K)$.

 The notation $\widehat{\mathfrak{g}}$ denotes the complex Lie algebra that is the 
Lie algebra of $\widehat{G}$ and if $R$ is any ring we denote by $\widehat{\mathfrak{g}}_R$ the Lie algebra of $\widehat{G}$ as an $R$-group. 
Also, as above, $\coad$ denotes the linear dual of $\widehat{\mathfrak{g}}$, i.e.,
$$ \coad = \Hom_{\C}(\widehat{\mathfrak{g}}, \C),$$
and we similarly define $\coad_{\Q}$ to be the $\Q$-dual of $\widehat{\mathfrak{g}}_{\Q}$. 

We use the word ``cohomological'' in a slightly more narrow way than usual. A representation  of $\GG(\R)$ is cohomological, for us, if it has nontrivial $(\mathfrak{g}, \Kinfty^0)$-cohomology. 
In other words, we do not allow for the possibility of twisting by a finite dimensional representation;
any cohomological representation, in this sense, has the same infinitesimal character as the trivial representation.

$\Pi$ will  usually denote a near-equivalence class of cohomological automorphic representations on $\GG$,
or a  variant with a stronger equivalence relation;
  $\pi$  will usually be
an automorphic representation belonging to this class.

For any automorphic $L$-function $L(s)$ and any special value $s_0$, we denote by $L^*(s_0)$ the leading term of the Taylor expansion of $L(s)$ at $s=s_0$, i.e.
$L^*(s_0) = \lim_{s \rightarrow s_0} (s-s_0)^{-r} L(s)$, where $r$ is
the order of vanishing of  the meromorphic function $L(s)$ at $s=s_0$. 

We often use the notation $A \sim B$ to mean that $A = \alpha B$ for some $\alpha \in \mathbf{Q}^*$. 
We will often also encounter situations where $(A/B)^2 \in \Q^*$, in which case we write $A \sim_{\sqrt{\Q^*}} B$.

For fields $E' \supset E$, an $E$-structure on an $E'$-vector space 
 $V'$ is, by definition, an $E$-vector subspace $V \subset V'$ such that $V \otimes_{E} E' \stackrel{\sim}{\rightarrow} V'$. 
If $V$ is a complex vector space we denote by $\overline{V}$
the conjugate vector space with the same underlying space and conjugated scaling.
So there is a tautological antilinear map $V \mapsto \overline{V}$
that we denote by $v \mapsto \bar{v}$.

If $Q$ is a nondegenerate quadratic form on a finite-dimensional vector space $V$, and $Q^*$ a form on the dual space $V^*$,
we say that $Q$ and $Q^*$ are in duality if $Q$ induces $Q^*$ via the isomorphism $V \stackrel{\sim}{\rightarrow} V^*$
associated to $Q$; this is a symmetric relation. The Gram matrices of $Q$ and $Q^*$ with reference to dual bases are inverse;  $Q$ and $Q^*$ are called ``inverse'' quadratic forms by Bourbaki  \cite[Chapter 9]{BourbakiQuadratic}.}

\subsection{Acknowledgements}

We would like to thank   M. Lipnowski for interesting discussions, A. Raghuram for his helpful suggestions about ``Hodge--linear algebra,'' 
and M. Harris for interesting discussions and encouragement.

The first-named author (K.P.) was supported by NSF grants DMS 1160720, DMS 1600494 and a fellowship from the 
Simons Foundation (\# 305784) . The second-named author  (A.V.) was supported by an NSF grant and a Packard Foundation Fellowship.

 \tableofcontents

\section{Motivic cohomology and Beilinson's conjecture}
\label{Beilinson}

The first part  (\S \ref{ssB}) of this section is a recollection of Beilinson's conjecture
and the theory of motives. 
The second part (\S \ref{polarized volumes}) is less standard: we use a polarization
to put a metric on Deligne cohomology. 
The most important result is Lemma  \ref{lem:vol-indep}, which allows us to compute volumes
of certain motivic cohomology groups in terms of values of $L$-functions. 
\subsection{Beilinson's conjecture for motives} \label{ssB}

In this section we recall Beilinson's conjecture for 
motives. For simplicity, we restrict to the case of 
motives defined over $\Q$ and coefficients in $\Q$, which 
is the main case we require.  
The summary below follows for the most part \cite{jannsen-beilinson} \S 4, which 
the reader is referred to for more details. (Our notation however is 
somewhat different.)

\subsubsection{Chow motives}
For $k$ a  field, let $\var_k$ denote the category of smooth projective varieties over $k$ and $\m_{k,\rat}$ the category of Chow motives over $k$. An object in 
$\m_{k,\rat}$ consists of a triple $M=(X,p,r)$ where $X$ is a smooth projective variety over $k$ 
of dimension $d$ say, 
$p$ is an idempotent in $\CH^{d} (X \times X)_{\Q}$ and $r\in \Z$ is an integer.    
 Informally, the reader should think of $(X, p, r)$ as the first projecting $X$ according to $p$, and then Tate-twisting by $r$. 
(Here, the notation $\CH^j(Y)_{\Q}$ for any variety $Y$ denotes the $\Q$-vector space given by 
the group of algebraic cycles of codimension $j$ on $Y$ modulo 
rational equivalence, tensored with $\Q$. If we replace rational equivalence by 
homological or numerical equivalence, the corresponding $\Q$-vector spaces 
will be denoted $\CH^j_{\hom}(Y)_\Q$ and $\CH^j_{\num}(Y)_\Q$ respectively.)

The morphisms in $\m_{k,\rat}$ are descibed thus: for $N=(Y,q,s)$  another object of  $\m$ 
\begin{equation}
\label{eqn:mot-morphisms}
\Hom (M,N) = q \circ \CH^{\dim Y + r-s} (X\times Y)_\Q \circ p.
\end{equation}
 Note that this convention is opposite to Deligne \cite{Deligne},
 who uses ``cohomological'' motives; this amounts to the opposite of the above category. 

Let $\Delta_X$ be the diagonal on $X\times X$. We denote $(X,\Delta_X,r)$ by the symbol $hX(r)$, and if further $r=0$ we denote this 
simply by $hX$. The dual motive $M^\vee$ of $M$ is defined by
\[
M^\vee = (X, p^t,d-r),
\]
where $p\mapsto p^t$ is the involution induced by interchanging the two components of $X \times X$. 
(Caution: the realizations of $M^\vee$ are closely related to but not exactly the duals of the realizations of $M$. See \S \ref{sec:cohomology} below.)
The category $\m_k$ admits a symmetric monoidal tensor structure defined by 
$$
(X,p,r) \otimes (Y,q,s) = (X\times Y, p\times q, r+s).
$$
The commutativity and associativity constraints $M\otimes N \simeq N\otimes M$ and 
$(M \otimes N)\otimes P \simeq M\otimes (N\otimes P)$ are 
induced by the obvious isomorphisms $X\times Y \simeq Y \times X$ and 
$(X\times Y)\times Z \simeq X \times (Y\times Z)$. 
If $k\rightarrow k'$ is a field extension, there is a natural base-change
functor $\m_k \rightarrow \m_{k'}$, denoted either $M\mapsto M\otimes_{k} k'$ or $M\mapsto M_{k'}$.

\subsubsection{Cohomology} \label{sec:cohomology} 
 For any subring $A$ of $\R$, we use $A(j)$ to denote $(2\pi i)^j A \subset \mathbf{C}$. We will need various cohomology theories on  $\var_{\Q}$:  
Betti cohomology $H^i_\B (X_\C, \Q(j))$, algebraic de Rham cohomology $H^i_{\DR} (X,j)$, $\ell$-adic cohomology $H^i_{\et} (X_{\Qbar}, \Q_\ell(j))$, 
the Deligne cohomology $H^i_{\D} (X_\R, \R(j))$ and  
motivic cohomology $H_{\mm}^i (X, \Q(j))$. 
%

These are all {\em twisted Poincar\'{e} duality theories} in the sense of Bloch and Ogus \cite{BlochOgus}; see e.g. \cite[Examples 6.7, 6.9, 6.10]{jannsen} and 
\S 1,2 of \cite{DeningerScholl}.
Moreover, they all admit  a
cup-product in cohomology such that the cycle class map is compatible with the product   
structure. 

Any such theory $H^*$ may be extended to $\m_{\Q,\rat}$ as follows. First, for motives of the form $hX(r)$  
set
$$
H^i (hX(r),j): = H^{i+2r} (X, j+r).
$$ 
If $f\in \Hom (hX(r),hY(s))$, define
$$
f^*: H^i(hY(s),j) \rightarrow H^i(hX(r),j)
$$ by
$$ f^*(\alpha) =\pi_{X,*} ( \cl(f) \cup \pi_Y^*(\alpha) ),
$$
where $\pi_X$ and $\pi_Y$ denote the projections from $X\times Y$ onto $X$ and $Y$ respectively.  Then for $M=(X,p,r)$, define
\begin{equation}
\label{eqn:coh-of-motive}
H^i (M,j) = p^* \circ H^i(hX(r),j).
\end{equation}

If $H^i$ is a {\it geometric} cohomology theory (such as $H^i_B$, $H^i_{\DR}$ or $H^i_{\et}(M_{\Qbar})$),
then it satisfies usual Poincar\'{e} duality and we have canonical isomorphisms
\begin{equation}
\label{eqn:coh-dual-motive}
H^{-i} (M^\vee) \simeq (H^i (M))^\vee. 
\end{equation}

\subsubsection{Comparison isomorphisms and periods} \label{deltacc}
 We continue to suppose that $M$ is an object of $\m_{\Q, \rat}$. 

There are  comparison isomorphisms
\begin{equation} \label{eqn:compbdr0}  \comp_{\B,\DR}: H^i_{\B}(M_{\C}, \Q) \otimes \C \simeq H^i_{\DR}(M) \otimes \C. \end{equation}
\begin{equation}
\label{eqn:compbet}
\comp_{\B,\et}: H^i_\B (M_\C, \Q) \otimes \Q_\ell \simeq H^i_{\et} (M_{\Qbar},\Q_\ell).
\end{equation}

Let $c_\B$ and $c_{\DR}$ denote the involutions given by 
$1\otimes c$  on the left 
and right of \eqref{eqn:compbdr0} respectively, where $c$ denotes complex conjugation. Then via $\comp_{\B,\dR}$, we have (\cite[Proposition 1.4]{Deligne})
\begin{equation} \label{Fcc}
F_\infty \cdot c_\B = c_{\DR}
\end{equation}
where,  if $M = hX$ is the motive of a variety $X$,   
then  $F_\infty$  denotes the  involution on $H^i_\B (X_\C, \Q)$ induced by the action of complex conjugation on the 
topological space $X(\C)$; this definition passes to  $H^i_\B (M_\C, \Q)$ via \eqref{eqn:coh-of-motive}.  Note that $F_{\infty}$ is complex-linear, whereas $c_{\B}$ and $c_{\DR}$ are complex antilinear. 
We will often denote $c_{\B}$ by the usual complex conjugation sign, i.e.
$$\bar{v} = c_{\B}(v).$$

More generally, we can go through the same discussion with $\Q(j)$ coefficients: replacing $M$ by its Tate twist we obtain  the comparison isomorphisms
\begin{equation}
\label{eqn:compbdr}
\comp_{\B,\DR}: H^i_\B (M_\C, \Q(j)) \otimes \C \simeq H^i_{\DR} (M,j) \otimes \C
\end{equation}
  We denote by $\delta(M,i,j)$ the determinant of the comparison map 
$\comp_{\B,\DR}$ taken with respect to the natural $\Q$-structures
$H_{\B}= H^i_\B (M_\C, \Q(j))$ and $H_{\DR} = H^i_{\DR} (M,j)$. 
This may be viewed as an element in $\C^\times/\Q^\times$.  
The equation \eqref{Fcc} needs to be modified slightly; while $c_B$ and $c_{\dR}$  are still
 defined as the complex conjugations with reference to the
real structures defined by \eqref{eqn:compbdr}, one twists $F_{\infty}$ by $(-1)^j$
to take into account the complex conjugation on $\Q(j)$.

 The $\Q$-vector space $H_\B^i(M, \Q(j))$  is in a natural way the underlying vector space of a rational Hodge structure, 
pure of weight $w=i-2j$; 
as usual we denote by $F^* H_{\dR}$ the associated Hodge filtration
on $H_{\B} \otimes \C = H^i_{\dR}(M, j) \otimes \C$. 
Thus $H_{\B} \otimes \C=\oplus_{p+q=w} H^{pq}$ and 
$F_\infty : H^{pq} \simeq H^{qp}$ is a conjugate-linear isomorphism. 
We denote by $H_\B^\pm $ the
$\pm 1$ eigenspaces for the action of $F_\infty$.

We suppose now that $(M,i,j)$ satisfies the following additional condition:  
\begin{equation} \label{Deligne_period}   \mbox{ If $w$ is even, then $F_\infty$ acts on $H^{w/2,w/2}$ as a scalar $\varepsilon = \pm 1$. }
\end{equation}
Let 
$$
p^\pm = \begin{cases} 
\frac{w-1}{2}, & \text{ if $w$ is odd;} \\
\frac{w-1}{2} \mp \frac{1}{2} \varepsilon, & \text{ if $w$ is even.}
\end{cases}.
$$
Set $F^\pm = F^{p^\pm} H_\DR $ and $H_\DR^\pm = H_\DR /F^\mp$. 
Then 
$$
\dim H_\B^\pm = \dim H_{\DR}^\pm
$$
and
the {\it Deligne period} $c^\pm (M,i,j) $ is defined to be  the determinant of the composite map
$$
H_\B^\pm \otimes \C \rightarrow H_\B \otimes \C \xrightarrow{\comp_{\B,\DR}}
H_\DR \otimes \C  \rightarrow H_\DR^\pm \otimes \C
$$
with respect to the $\Q$-structures $H_\B^\pm$ and $H_\DR^\pm$, viewed as an element of $\C^*/\Q^*$.  Note that this defined
only under the assumption of \eqref{Deligne_period}.

  \subsubsection{Cohomology of $M_{\R}$} \label{MRcohomology}

Suppose $M = hX$, and let $A$ be a subring of the complex numbers, containing $\Q$ and stable by conjugation.   
 Complex conjugation induces an involution $\iota$ of $X(\C)$. 
This involution is covered by an involution of the constant sheaf $A$, which induces
complex conjugation on each fiber, and by an involution of the de Rham complex $\Omega^*_{X(\C)}$, 
sending a differential form $\omega$ to $\overline{\iota^* \omega}$.   
Accordingly we obtain  conjugate-linear involutions on de Rham cohomology tensored with $\C$, as well as on each step of the Hodge filtration; on Betti cohomology with $A$ coefficients, 
and (since the involutions are compatible under the map $A \rightarrow \Omega^0_{X(\C)}$) also  Deligne cohomology with $A$ coefficients. 

 In each case,  the fixed points  will be denoted, following Beilinson,
by the notation $H^*_{?}(X_{\R}, -)$.    Compare \cite[p. 2037]{beilinson}.
This notation extends, as before, to general motives $M$. 

Thus, for example,   
$$H^i_{\B}(M_{\R}, A) = H^i_{\B}(M_{\C}, A)^{F_{\infty} c_{\B}},$$
is the subspace fixed by $F_{\infty} c_{\B}$, where $F_{\infty}$
 is ``acting on the topological space $M_{\C}$'' (at least when $M=hX$) and $c_{B}$ is acting on the coefficients.

On the other hand,  
 $$H^i_{\dR}(M_{\R})  = H^i_{\dR}(M) \otimes \R$$
 is simply the (real) de Rham cohomology of the associated real algebraic variety (or motive). 
 Similarly, $F^n H^i_{\dR}(M_{\R})$ is the $n$th step of the Hodge filtration on the above space. 
 Observe, then, that $F^n H^i_{\dR}(M_{\R})$ has a natural $\Q$-structure.

 \subsubsection{The fundamental exact sequences and $\Q$-structures}

For $n\ge \frac{i}{2} +1$ there are canonical isomorphisms (see \cite[\S 3.2]{beilinson}) 
\begin{align*}
H^{i+1}_{\D} (M_\R, \R(n)) & \simeq H^i_{\B} (M_\R, \C) / H^i_{\B} (M_\R, \R(n)) + F^n H^i_{\DR} (M_\R) \\
& \simeq H^i_{\B} (M_\R, \R(n-1)) / F^n H^i_{\DR} (M_\R),
\end{align*} 
 In the second equation, we regard $F^n H^i_{\DR} (M_\R)$ 
as a subspace of $H^i_{\B}(M_{\R}, \R(n-1))$ via the composite
\begin{equation} \label{tildepidef} \tilde{\pi}_{n-1} : F^n H^i_{\DR} (M_\R) \hookrightarrow H^i_{\B}(M_{\R}, \C) \xrightarrow{\pi_{n-1}} H^i_{\B}(M_{\R}, \R(n-1))\end{equation}
where the latter map is the projection along $\C = \R(n) \oplus \R(n-1)$.

This gives rise to two fundamental exact sequences:
\begin{equation}
\label{eqn:fundamental1}
0 \rightarrow F^n H^i _{\dR} (M_\R) \stackrel{\tilde{\pi}_{n-1}}{ \rightarrow}  H^i_{\B} (M_\R, \R (n-1)) \rightarrow 
H^{i+1}_{\D} (M_\R, \R (n)) \rightarrow 0,
\end{equation}
and 
\begin{equation}
\label{eqn:fundamental2}
0 \rightarrow  H^i_{\B} (M_\R, \R (n)) \rightarrow  H^i_{\DR} (M_\R) / F^n H^i _{\dR} (M_\R)\rightarrow 
H^{i+1}_{\D} (M_\R, \R (n)) \rightarrow 0,
\end{equation}
These exact sequences can be used to put (different) $\Q$-structures on the $\R$-vector space 
$\det H^{i+1}_{\D} (M_\R, \R (n))$ using the canonical $\Q$-structures on the left two terms of 
each sequence. The first, denoted $\cR (M,i,n)$ will be the 
 $\Q$-structure obtained from \eqref{eqn:fundamental1}, namely 
using the $\Q$-structures $\det H^i_{\B}(M_\R, \Q(n-1))$ and 
 $\det (F^n H^i_{\DR} (M))$.
The second, denoted $\cDR (M,i,n)$ will be the $\Q$-structure obtained from \eqref{eqn:fundamental2}, namely 
using the $\Q$-structures $\det (H^i_{\DR} (M)/F^n)$ and 
$\det H^i_{\B}(M_\R, \Q(n))$. These $\Q$-structures are related by 
\begin{equation}
\label{eqn:relation-Q}
\cDR (M,i,n) = (2\pi \sqrt{-1})^{-d^- (M,i,n)} \cdot \delta (M,i,n) \cdot \cR (M,i,n),
\end{equation}
where $d^-(M,i,n) = \dim H^i_{\B} (M_\C, \Q(n))^-$.

\subsubsection{$L$-functions}
For $M$ in $\m_{\Q,\rat}$ and $i$ an integer, 
the $L$-function $L^i (M,s)$ is defined by 
\begin{equation}
\label{eqn:eulerproduct}
L^i (M,s) = \prod_{p} L_p^i (M,s),
\end{equation}
with  
$$
L_p^i (M,s) = \det (1- \Frob_p p^{-s} | H^i_{\et} (M_{\Qbar}, \Q_\ell)^{I_p} )^{-1},
$$  
where $\Frob_p$ denotes a {\it geometric} Frobenius at $p$, the superscript $I_p$ denotes taking invariants under an inertia subgroup at $p$ and $\ell$ is 
any prime not equal to $p$. Implicit 
in the definition is the conjecture that the factor $L_p^i (M,s)$ is independent of the 
choice of $\ell$. 
The Euler product \eqref{eqn:eulerproduct} converges on some right half plane in $\C$; 
conjecturally, one also expects that $L^i (M,s)$ admits 
a meromorphic continuation to all of $\C$ that is analytic as long as 
either $i$ is odd or the pair $(M,i)$ satisfies the following condition: 
 \begin{center}
$(\star) \quad \quad \quad $ $i =2j$ is even and 
$H^{2j}_{\et} (M_{\Qbar}, \Q_\ell (j))^{\Gal (\Qbar/\Q)}=0$.
 \end{center}
One also expects that $L^i (M,s)$ satisfies a functional equation of the form:   
$$
(L_\infty \cdot L)^i (M,s) = (\ve_\infty \cdot \ve)^i (M,s) \cdot (L_\infty\cdot L)^{-i} (M^\vee, 1-s).
$$ 
where $L_{\infty}$ is the archimedean $L$-factor, and $\ve_{\infty}, \ve$ are $\varepsilon$-factors;
for definitions, we refer to \cite{Tate}. 
 
\subsubsection{Regulators and Beilinson's conjecture} \label{ss:rb}
There are regulator maps
\begin{equation} \label{rDdef}
r_\D: H^i_{\mm} (M, \Q(j)) \otimes \R \rightarrow H^i_{\D} (M_\R, \R(j))
\end{equation} 
which give rise to a morphism of twisted Poincar\'{e} duality theories. 

 Scholl  has shown \cite[Theorem 1.1.6]{scholl} that there is a unique way  
to assign $\Q$-subspaces $H^{i+1}_{\m}(M_{\Z}, \Q(n) ) \subseteq H^{i+1}_{\m} (M, \Q(n)) $  to each Chow motive over $\Q$, in a
 way that respects morphisms, products,   and so that $H^{i+1}_{\m}(hX_{\Z}, \Q(n))$
 is given by the image of   motivic cohomology of a regular model $\mathcal{X}$, when one exists
    (for details, see {\em loc. cit.}). 
We now present a version of Beilinson's conjectures relating regulators to $L$-values. 
\begin{conj} \label{conj:Beilinson} (Beilinson)
Suppose that $n\ge \frac{i}{2} +1$ and that if 
$n = \frac{i}{2}+1$ then $(M,i)$ satisfies the condition  
$(\star)$ above.

(a) Then  
$$
r_{\D} : H^{i+1}_{\m} (M_\Z, \Q(n)) \otimes \R \rightarrow H^{i+1}_{\D} (M_\R, \R(n))
$$
is an isomorphism.

(b) Further, we have equivalently 
\begin{equation}
\label{eqn:bc1}
r_{\D} \left( \det H^{i+1}_{\m} (M_\Z, \Q(n)) \right) = L^{-i} (M^\vee,1-n)^* \cdot \cR(M,i,n),
\end{equation}
and 
\begin{equation}
\label{eqn:bc2}
r_{\D} \left( \det H^{i+1}_{\m} (M_\Z, \Q(n)) \right) = L^i (M,n) \cdot \cDR(M,i,n).
\end{equation}
We understand the meromorphic continuation of the $L$-function and its functional equation 
to be part of this conjecture.  

\end{conj}

\begin{rem}
The point $n=\frac{i}{2}+1$ (for $i$ even) is the near right-of-center point. At this point,
the formulation has to be typically modified to allow for Tate cycles. However this is 
unnecessary on account of assumption $(\star)$. We have omitted the description
at the central point (conjecture of Bloch and Beilinson) since that is not required in the 
rest of the paper. Indeed, the only point of interest for us is the near right-of-center point. 
\end{rem}

\begin{rem} \label{remark4}
In Beilinson's original formulation of this conjecture one postulates the existence of a Chow motive $M_0$ (Beilinson denotes this $M^0$) such that 
\begin{equation}
\label{eqn:coh-M0}
H^{-i} (M^\vee) = H^i (M_0,i)
\end{equation}
for all geometric cohomology theories $H^*$ and all $i$.
Then $$L^{-i} (M^\vee,1-s) = L^i (M_0, i+1-s),$$
so the value $L^{-i}(M^\vee, 1-n)^*$ in \eqref{eqn:bc1} can be replaced by
$L^i (M_0, i+1-n)^*$. 
See also \S \ref{sec:dualmotive} below. 
\end{rem}

\subsubsection{Pure motives}
 \label{sec:purity}
The category of Chow motives has the disadvantage that it is not Tannakian.
To construct a (conjectural) Tannakian category one needs to 
modify the morphisms and the commutativity constraint. 
For any field $k$, let $\m_{k,\hom}$ and $\m_{k,\num}$ denote the categories obtained 
from $\m_{k,\rat}$ by replacing the morphisms in \eqref{eqn:mot-morphisms} by cycles modulo 
homological and numerical equivalence respectively. Thus there 
are natural functors
$$
\m_{k,\rat} \rightarrow \m_{k,\hom} \rightarrow \m_{k,\num}.
$$
Jannsen \cite{jannsen-inven} has shown that $\m_{k,\num}$ is a 
semisimple abelian category and that numerical 
equivalence is the only adequate equivalence relation on
algebraic cycles for which this is the case. 

To outline what would be the most ideal state of affairs, we 
assume the following standard conjectures on algebraic cycles:
\begin{enumerate} 
\item (Kunneth standard conjecture) For any smooth projective variety $X$, the Kunneth components of the diagonal 
(with respect to some Weil cohomology theory) on  
$X\times X$ are algebraic.

\item Numerical and homological equivalence coincide on $\CH^*(X)_\Q$. 

\end{enumerate}
Then the second functor above is an equivalence of categories, 
so we can identify $ \m_{k,\hom} $ and $ \m_{k,\num}$; this will be the category 
of {\it pure motives} or {\it Grothendieck motives}. 

We say moreover that a motive $M$ is {\em pure of weight $w$} when the cohomology $H^j_{\B}(M, \C)$  is concentrated in degree $j=w$.  In this case, we write for short $H_{\B}(M,\C)$ for the graded vector space
$H^*_{\B}(M, \C)$, which is entirely concentrated in degree $w$. Note that in general, a pure motive is not necessarily pure of 
a fixed weight.

To make this Tannakian, one needs to 
modify the commutativity constraint as described in \cite{deligne-milne} \S 6, \cite{jannsen-inven}.  
With this new commutativity constraint, the category 
of pure motives is Tannakian ( \cite{jannsen-inven} Cor. 2); 
we denote it $\m_k$. If $k=\Q$, then $M\mapsto H^*_B(M_\C)$, 
$M\mapsto H^*_{\DR} (M)$ and  $M\mapsto H^*_{\et} (M_{\Qbar})$
are Tannakian fiber functors.

\subsubsection{Passage from Chow motives to pure motives}
 One would like to make sense of Beilinson's conjectures for  
pure motives over $\Q$. However, this requires an additional
set of conjectures. 

For any field $k$, Beilinson conjectures the existence of a descending filtration 
$F^\bullet $ on motivic cohomology $H^i_\m (X,\Q(j))$ for $X$ in $\var_k$ satisfying the 
properties $(\tilde{\mathrm{a}})$ through $(\tilde{\mathrm{e}})$ of \cite{jannsen-motives} Remark 4.5(b):
\begin{enumerate}
\item $F^0 H^i_\m (X, \Q(j))=H^i_\m (X, \Q(j))$. 
\item \label{filonchow} On $H^{2j}_{\m} (X,\Q(j)) = \CH^j(X)_\Q$, we have $F^1 = \CH^j(X)_{\hom,\Q}$.
\item \label{func} $F^\bullet$ is respected by pushforward and pullback for maps $f:X\rightarrow Y$.
\item \label{cup} $F^r H^{i_1}_{\m} (X,\Q(j_1)) \cdot F^s H^{i_2}_{\m} (X,\Q(j_2)) \subseteq F^{r+s} H^{i_1+i_2}_{\m} (X, \Q(j_1+j_2))$.
\item \label{vanishing} $F^r H^i_{\mm}  (X,\Q(j))=0$ for $r \gg 0$. For $k$ a number field, $F^2 H^i_{\mm} (X,\Q(j))=0$.
\item \label{ext-mm} There are functorial isomorphisms 
$$
\Gr^r_F  (H^i_\m (X, \Q(j)) = \Ext^r_{\MM_k} (1, h^{i-r}(X) (j)).
$$
Here $\MM_k$ is a conjectural abelian category of {\it mixed motives}
containing $\m_{k,\hom}$ as a full subcategory and $1$ denotes the trivial motive $h(\Spec \ k)$. 
\end{enumerate}

Let us note the following consequence of the above conjectures, a proof of which can be found in \cite[\S 7.3 Remark 3.bis]{murre}:
\begin{prop} (Beilinson)
\label{prop:chow-to-hom}
Assuming the conjectures above, the functor $\m_{k,\rat} \rightarrow \m_{k,\hom}$ is essentially surjective.
Further, any isomorphism $M_1 \simeq M_2$ in $\m_{k,\hom}$ can be lifted to an isomorphism $\M_1 \simeq \M_2$ in $\m_{k,\rat}$. 
\end{prop}

The proposition implies that there is a bijection of isomorphism classes of objects in $\m_{k,\rat}$ and 
$\m_{k,\hom}$. Note that the isomorphism $\M_1\simeq \M_2$ above is not canonical. 
In particular the induced map $H^{i}_{\mm} (\M_1, \Q(j)) \simeq H^{i}_{\mm}(\M_2, \Q(j))$ is not canonical.
Nevertheless, it is easy to see that the conjectures above imply that the induced maps on the graded pieces
$$
\Gr^r H^{i}_{\mm} (\M_1, \Q(j)) \simeq \Gr^r H^{i}_{\mm}(\M_2, \Q(j))
$$ 
are independent of the choice of isomorphism and are thus canonical. This allows us to (conjecturally) define 
graded pieces of motivic cohomology for motives in $\m_{k,\hom}$. 

%
%

\smallskip

Now we specialize to the setting of Beilinson's conjectures on $L$-values. 
The key point is that even though these conjectures are formulated in 
terms of motivic cohomology, in each case it is only a certain graded piece that matters, 
so the conjectures make sense for Grothendieck motives as well. Indeed, let us now specialize to 
$k=\Q$, and 
let $X$ be a variety over $\Q$. Then: 
\begin{enumerate}
\item[(i)] For $n\ge \frac{i}{2}+1$, we see from 
\eqref{ext-mm} that
$$
\Gr^0 H^{i+1}_{\mm} (X, \Q(n)) = \Hom_{\mm_{\Q}} ( 1, h^{i+1}(X)(n)) =0
$$
since $h^{i+1}(X)(n)$ is pure of weight $i+1-2n \le -1$. Thus in this range we have 
$$
H^{i+1}_{\mm} (X, \Q(n)) = F^1 H^{i+1}_{\mm} (X, \Q(n)) = \Gr^1 H^{i+1}_{\mm} (X, \Q(n)),
$$ 
since $F^2 =0 $ by \eqref{vanishing}.

\item[(ii)] If $n= \frac{i}{2} +1$, the conjecture typically also involves   $$ 
\CH^{n-1}(X)_\Q/ \CH^{n-1} (X)_{\hom,\Q} = \Gr^0 H_{\mm}^{2n-2} (X, \Q(n-1)).
$$

\item[(iii)] If $n=\frac{i+1}{2}$, we are the center and the conjecture involves
$$
\CH^n(X)_{\hom,\Q} = F^1 H_{\mm}^{2n} (X, \Q(n)) = \Gr^1 H_{\mm}^{2n} (X, \Q(n)),
$$
since $F^2=0$ by \eqref{vanishing}.
\end{enumerate}

Finally,  one needs the analogue of Scholl's result quoted before 
Conjecture \ref{conj:Beilinson}, i.e., one wants to know that the
subspace $\Gr^r H^{i}_{\mm} (\M, \Q(j))$
arising from  $H^{i}(\M_{\Z}, \Q(j))$ 
doesn't depend on choice of $\M$.  We do not know of any work on this issue, but implicitly will assume it to be true (as one must, in order to make
any reasonable sense of Beilinson's conjecture starting with a Galois representation).

\subsubsection{The dual motive}
\label{sec:dualmotive}
One sees now that there are two different notions of the ``dual motive''. 
On the one hand, if $M = (X,p,r)$ is either a Chow motive or a Grothendieck motive,
we have defined $M^\vee = (X, p^t, d-r)$ with $d=\dim(X)$. Recall that this satisfies
\begin{equation}
\label{eqn:coh-dual1}
 H^{-i} (M^\vee) = H^i(M)^\vee
 \end{equation}
for any (geometric) cohomology theory. 
On the other hand, assuming the conjectural framework described above (so that 
$\m_{\hom}$ is a Tannakian category), one can attach to any $M$ in 
$\m_{\hom}$ a motive $M^*$ in $\m_{\hom}$ such that 
\begin{equation}
\label{eqn:coh-dual2}
 H^i (M^*) = H^i(M)^\vee
 \end{equation}
for all $i$. By Prop. \ref{prop:chow-to-hom}, one can lift $M^*$ to a Chow motive,
any two such lifts being isomorphic but not canonically so.
We note the following example: if $M=(X, \pi_j, 0)$ with $\pi_j$ the 
Kunneth projector onto $h^j(X)$, then 
\begin{align*}
M^\vee &= (X, \pi_{2d-j}, d), \\
M^* &= (X, \pi_{2d-j}, d-j).
\end{align*}
and we can take $M_0=M=(X,\pi_j,0)$ (see Remark \ref{remark4}). 

\begin{rem}
The case of most interest in this paper is when $M \in \m_{k,\hom}$ is (pure) of weight zero
so that $H^i(M)$ vanishes outside of $i=0$.  It follows then from \eqref{eqn:coh-dual1} and 
\eqref{eqn:coh-dual2} that $M^*=M^\vee$. Further, from \eqref{eqn:coh-M0} we see that 
we can choose 
$$
M_0 = M^* = M^\vee,
$$
so all notions of dual agree in this case. Let us restate Beilinson's conjecture in this case 
for $n=1$. Writing simply $L$ instead of $L^0$ and $\cR(M)$, $\cDR(M)$ for $\cR(M,0,1)$, $\cDR(M,0,1)$ respectively, the conjecture predicts, equivalently:
\begin{equation}
\label{eqn:bc-wt0-1}
r_{\D} \left( \det H^{1}_{\m} (M_\Z, \Q(1)) \right) = L(M^*,0)^* \cdot \cR(M)
\end{equation}
and 
\begin{equation}
\label{eqn:bc-wt0-2}
r_{\D} \left( \det H^{1}_{\m} (M_\Z, \Q(1)) \right) = L (M,1) \cdot \cDR(M).
\end{equation}
\end{rem}

\subsection{Polarizations, weak polarizations and volumes} \label{polarized volumes}

In this section, we examine the fundamental exact sequence \eqref{eqn:fundamental1} in the presence of 
a polarization on $M$. We also introduce the notion of a weak polarization, which for us will 
have all the properties of a polarization except that we replace the usual 
definiteness assumption by a non-degeneracy requirement. 

\subsubsection{Hodge structures}
We first discuss these in the context of rational Hodge structures. 
A $\Q$-Hodge structure of weight $m$ consists of a finite dimensional $\Q$-vector space $V$ and a decomposition
\begin{equation}
\label{eqn:hodgesplitting} V_\C = \oplus_{p+q=m} V^{p,q},
\end{equation}
such that $\overline{V^{p,q}}=V^{q,p}$. 
The Hodge filtration on $V_\C$ is given by $F^i V_\C := \oplus_{\stackrel{i\ge p}{p+q=m}} V^{p,q}$. The splitting \eqref{eqn:hodgesplitting}
can be recovered from the Hodge filtration since $V^{pq} = F^p V_\C  \cap \overline{F^q V_\C}$. 

 If $M \in \cM_{\Q}$, the Betti cohomology $H^m_{\B}(M_\C)$ carries a Hodge structure  of weight $m$.

The Hodge structure $\Q(m)$ of weight $-2m$ is defined by the cohomology of the motive $\Q(m)$, explicitly: 
\begin{equation} \label{TT} V = (2\pi \sqrt{-1} )^m \Q, \quad V^{-m,-m}=V_\C.\end{equation}
 
 In particular, $\Q(1)$ should be regarded as the Hodge structure of $H_1(\mathbb{G}_m)$ (or $H_2(\mathbb{P}^1)$, if one wants to only work with projective varieties).
 Indeed, if we identify
 $$H_{1, \B}(\mathbb{G}_m, \C) \simeq \C$$
 by integrating the form $\frac{dz}{z}$, 
the resulting identification carries the Betti $\Q$-structure  to $(2 \pi \sqrt{-1})\Q \subset \C$, and the de Rham $\Q$-structure
to $\Q \subset \C$.

 If $V$ is a $\Q$-Hodge structure then there is an action of $\mathbb{C}^*$ on $V_{\C}$, which acts
by the character 
\begin{equation} \label{CWeil} z \mapsto z^p \overline{z}^q\end{equation}  on $V^{pq}$. 
  This action preserves $V \otimes \R \subset V_{\C}$.
 
For the cohomology of $\Q$-motives this action extends to a larger group: let $W_\R$ and $W_\C$ denote the Weil groups of $\R$ and $\C$ respectively.
  Thus $W_\C=\C^*$ while $W_\R$ is the non-split extension
  $$
  1 \rightarrow \C^* \rightarrow W_\R \rightarrow \langle j \rangle \rightarrow 1
  $$
  where $j^2= -1$ and $j^{-1}z j = \bar{z}$ for $z\in \C^*$.  
For $M \in \cM_{\Q}$,  we extend the action of \eqref{CWeil} to the real Weil group 
via 
 $$ j = i^{-p-q} F_{\infty}$$
 see \cite[\S 4.4]{Tate} (we have used an opposite sign convention to match with \eqref{CWeil}).

\subsubsection{Polarizations on Hodge structures} \label{Hodge polarized}

A  {\it weak polarization} on a pure $\Q$-Hodge structure $V$ of weight $m$ will be a {\it non-degenerate} bilinear form  
$$ Q: V \times V \rightarrow \Q $$
satisfying (here we continue to write $Q$ for the scalar extension to a {\it bilinear } form $V_\C \times V_\C \rightarrow \C$)
\begin{enumerate}
\item[(i)] $Q(u,v) = (-1)^m Q(v,u)$. Thus $Q$ is $(-1)^m$-symmetric. 
\item[(ii)] $Q(V^{p,q}, V^{p',q'})=0$ unless $(p,q)=(q',p')$.
\end{enumerate}

We mention various equivalent formulations of these conditions. Firstly, since $Q$ is defined over $\Q$, we have   $ \overline{Q(u,v)} = Q(\bar{u},\bar{v}).$ 
From this it is easy to see that
(ii) may be replaced by (ii$^\prime$):
\begin{enumerate}
\item[(ii$^\prime$)] $F^i V_\C$ is orthogonal to $F^{i^*} V_\C$  where $i^*:=m-i+1$. 
\end{enumerate} 
Since $Q$ is non-degenerate and since $F^i V_\C$ and $F^{i^*} V_\C$ have complementary dimensions in $V_\C$, we can also replace (ii$^\prime$) 
by (ii$^{\prime \prime}$):
\begin{enumerate}
\item[(ii$^{\prime \prime}$)] The orthogonal complement of $F^i V_\C$ is $F^{i^*} V_\C$.
\end{enumerate}

Now define 
$$ S := (2 \pi \sqrt{-1} )^{-m} Q,$$
considered as a linear function  
$$ S: V \otimes V \rightarrow \Q (-m).$$
Then condition (ii$^{\prime \prime}$) above is exactly equivalent to saying that $S$ gives a morphism of $\Q$-Hodge structures.
 Thus we can equivalently define a weak polarization on $V$ to consist of a morphism of $\Q$-Hodge structures $S$ as above satisfying $ S(u\otimes v) = (-1)^m S(v \otimes u).$

\smallskip

A {\it polarization} on a $\Q$-Hodge structure $V$ is a weak polarization $Q$ that satisfies the following additional positivity condition:
\begin{enumerate}
\item[(iii)] If $u\in V^{p,q}$, $u \neq 0$, then $ i^{p-q} Q(u, \bar{u}) >0$. (That $i^{p-q} Q(u, \bar{u})$ lies in $\R$ follows from (i) and the fact that $Q$ is defined over $\Q$.)
\end{enumerate}

Let $C$ be the operator on $V_\C$  
given by the action  of $i \in \mathbb{C}^*$  (see \eqref{CWeil}). 
  Then we can rewrite (iii) above as $ Q(Cu, \bar{u}) >0.$
This statement holds for all $u\in V_\C$ (and not just on elements of fixed type $(p,q)$) on account of (ii). Thus condition (iii) is equivalent to:
\begin{enumerate}
\item[(iii$^\prime$)] The hermitian form $(u,v) \mapsto Q (Cu,\bar{v}) $ is positive definite. 
\end{enumerate}
Now $C$ restricts to an $\R$-linear operator on $V_\R$, and the condition (iii$^{\prime}$) is equivalent to \begin{enumerate}
\item[(iii$^{\prime \prime}$)] The $\R$-bilinear form 
$$ V_\R \times V_\R \rightarrow \R, \quad (u,v) \mapsto Q (Cu,v) $$ 
 is symmetric and positive definite.   
\end{enumerate}

\subsubsection{Polarizations on motives}

A {\it weak polarization} on a pure motive $M \in \cM_{\Q}$ of weight $m$ will be a   
morphism 
$$ s: M \otimes M \rightarrow \Q(-m)
$$ 
that is $(-1)^m$-symmetric and 
such that the induced map 
$$ M \rightarrow M^* (-m)$$ 
is an isomorphism. 
In particular, writing $V = H_{\B}(M_\C, \C)$ for the associated $\Q$-Hodge structure,  $s$ induces an isomorphism  
$ V  \stackrel{\sim}{\longleftarrow} V^*(-m),$
which gives a $(-1)^m$-symmetric  bilinear  form $$H_{\B}(s): V \otimes V \rightarrow \Q(-m),$$
commuting with the action of $\mathbb{C}^*$. 

Thus $H_{\B}(s)$ is  a weak polarization of Hodge structures, in the sense 
of \S \ref{Hodge polarized}. A {\it polarization} on $M$ is a weak polarization $s$ such that $H_{\B}(s)$ is  
a polarization on $V$.  

For the next statement, recall that $V = H_{\B}(M_\C,\C)$ is equipped with 
an involution $F_\infty$.

\begin{lemma}
The (complexification of the) weak polarization $H_{\B}(s): V \otimes V \rightarrow \Q(-m)$ is  equivariant for $c_{\B},F_\infty$  and the action of the Hodge $S^1$ on $V_{\C}$. 
\end{lemma} 

\begin{proof} 
It is enough to show these assertions for the morphism $V^*(-m) \rightarrow V$. 
But given any morphism  $f: M \rightarrow M'$ of objects in $\m_{\Q}$
the induced morphism  on Betti cohomology commutes with $c_{\B}, F_{\infty}$ and $S^1$. 
\end{proof}

 In practice, instead of a weak polarization on $M$, we can work just with part of the linear algebraic data given by such a form.  

Namely, we give ourselves a nondegenerate symmetric bilinear form  \begin{equation} \label{Sdef}
S: V \times V \rightarrow \Q(-m) = (2 \pi i)^{-m} \Q
\end{equation}
on $V=H_B (M_\C, \Q)$  whose complexification $S_{\C}$ on $V_{\C}$
satisfies:
\begin{itemize}
\item[(a)] $S_{\C}$ is invariant by $F_\infty$  and $\mathbb{C}^*$, i.e., by the action of $W_{\R}$, and 
 \item[(b)]   $S_{\C}$ 
  restricts to a $\Q$-valued form on $H_{\dR}(M)$. 
  \end{itemize}
This gives a Hermitian form $\langle \cdot,\cdot \rangle$ on $V_\C$ defined by
\[
\langle x,y \rangle =S(x,\bar{y})
\]

\subsubsection{Metrics on Deligne cohomology}   \label{Deligne metrics}
We shall now explain how to use a polarization to  equip Deligne cohomology with a  quadratic form.  
In fact, we do not need a polarization, but simply the linear algebra-data associated to a weak polarization,
as in \eqref{Sdef} and discussion after it.

Recall that (for $M\in \cM_\Q)$ the Beilinson exact sequence: 
\begin{equation}
\label{eqn:beil-exact-seq-recalled}
0 \rightarrow F^{n} H^i_{\dR} (M_\R) \xrightarrow{\tilde{\pi}_{n-1}} H^i_{\B} (M_\R, \R (n-1)) \rightarrow H^{i+1}_{\D} (M_\R, \R (n) ) \rightarrow 0,
\end{equation}
where $i$ and $n$ are integers with $i\le 2n-1$; and the first map is as in \eqref{tildepidef}. 

Let $V$ be the $\Q$-Hodge structure $H^i_B (M_\C, \Q)$.
We suppose, as in the discussion above, 
we are given the linear algebraic data associated to a weak polarization, i.e.
$$S: V \times V \rightarrow \Q(-i)$$ 
and we define $Q = (2 \pi \sqrt{-1})^i S$, as before. 
The distinction between $S$ and $Q$ is that $S$ is rational valued on de Rham cohomology,  
and $Q$ is rational valued on Betti cohomology.

\begin{prop}  \label{definiteness proposition} 
Let $(\cdot, \cdot)$ denote the bilinear form $u,v \mapsto Q(u, \bar{v})$ on $H^i_B (M_\R, \R (n-1)) $. Then 
\begin{enumerate}
\item The form $(\cdot, \cdot)$ 
is  $\R$-valued.  
\item Suppose that $i$ is even. Then the form $(\cdot,\cdot)$ is symmetric and non-degenerate and so is its restriction to the subspace $\tilde{\pi}_{n-1} ( F^n H^i_{\dR} (M_\R) )$. 
In particular, it induces by orthogonal projection a non-degenerate form, also denoted $( \cdot, \cdot )$, on the quotient $H^{i+1}_{\D} (M_\R, \R (n) )$.
\item If $i=2n-2$ and $S$ arises from a polarization, the form $(\cdot, \cdot)$ on $H^{i+1}_{\D} (M_\R, \R (n) )$ is symmetric and positive-definite.
\end{enumerate}
\end{prop}

\proof 

 $Q$  is real-valued on $H^i_{B}(M_{\C}, \R)$
and so $\overline{Q(u, \bar{v}) } =Q(\bar{u}, v). $
 Let $u, v \in H^i_B(M_\R, \R(n-1))$. Then  $\bar{u} =(-1)^{n-1} u$ and same for $v$; thus 
$$
\overline{ Q(u,\bar{v})}=      Q(\bar{u}, v) =    Q((-1)^{n-1}u, (-1)^{n-1} \bar{v}) =  Q(u,\bar{v}),
$$
from which we see that $(\cdot, \cdot)$ is $\R$-valued.

 Now suppose that $i$ is even. Then $Q$ is symmetric,
 and so
 $(v, u) = Q(\bar{v}, u) = Q(u, \bar{v}) = Q(\bar{u}, v) = (u,v)$.  Thus $(\cdot, \cdot)$ is symmetric. 

The Hermitian form $(u , v) \mapsto S(u, \bar{v})$ is nondegenerate, and so 
 $(u,v) \mapsto  \mathrm{Re} \ Q(u, \bar{v})$ is a nondegenerate real-valued quadratic form on $V_{\C}$
 considered as a real vector space.     
 
 Now the inclusion $\R(n-1) \hookrightarrow \C$ induces an identification  (\S \ref{MRcohomology})
   \begin{align}
   \label{eqn:hib}
H^i_B(M_\R, \R(n-1)) &= H^i_B(M_{\C}, \C)^{c_{\B}=(-1)^{n-1}, F_{\infty} = (-1)^{n-1}} \\  \nonumber  &= V_{\C}^{c_{\B}=(-1)^{n-1}, F_{\infty} = (-1)^{n-1}}.
\end{align}
The quadratic form $ \mathrm{Re} \ Q(u, \bar{v})$  is preserved by  $c_{\B}$, 
 and $Q(F_{\infty} u,\overline{F_{\infty} v} ) = F_{\infty} Q(u, \bar{v}) =  (-1)^i Q(u, \bar{v})$;
 since $i$ is even, we see that $\mathrm{Re} \  Q(u, \bar{v})$ is preserved by $F_{\infty}$. 
 Therefore,  the restriction of $\mathrm{Re}  \ Q(u, \bar{v})$  to  $H^i_B(M_\R, \R(n-1)) $ remains nondegenerate,
 since this subspace is an eigenspace for the action of the Klein four-group generated by $F_{\infty}, c_{\B}$,
 and this group preserves  $\mathrm{Re} \ Q(u, \bar{v})$.

The same analysis holds verbatim replacing $V_{\C}$ by $V^{p,q} \oplus V^{q,p}$,
and shows that $\mathrm{Re} \ Q(u, \bar{v})$ is nondegenerate on $\left(V^{p,q} \oplus V^{q,p} \right)^{c_{\B}=(-1)^{n-1}, F_{\infty} = (-1)^{n-1}}$. 
Since 
$$
\tilde{\pi}_{n-1} ( F^n H^i_{\dR} (M_\R) ) = \bigoplus_{\stackrel{p\geq n}{p+q=i}} (V^{p,q} \oplus V^{q,p})^{c_{\B}=(-1)^{n-1}, F_{\infty} = (-1)^{n-1}}
$$ 
 the non-degeneracy of  $(\cdot, \cdot)$ restricted to $\tilde{\pi}_{n-1} ( F^n H^i_{\dR} (M_\R) )$ follows.

Finally, for (3), we note that 
when $i=2n-2$, the orthogonal complement of $\tilde{\pi}_{n-1} ( F^{n} H^j_{\dR} (M_\R) )$ is just 
$$ (V^{n-1,n-1})^{c_{\B}=(-1)^{n-1}, F_{\infty} = (-1)^{n-1}} $$
and the restriction of $(\cdot,\cdot)$ to this subspace is positive definite if $S$ is a polarization. \qed 

\subsubsection{Motives of weight zero} \label{motive0sec}
The case of most interest to us is when $M$ is of weight $0$ and $n=1$, $i=0$ and we restrict to this 
case for the rest of this section.   

The exact sequence  
\eqref{eqn:beil-exact-seq-recalled}, specialized to $n=1$ and $i=0$ is: 
   \begin{equation} \label{key0}  0\rightarrow F^1 H_{\dR}(M)  \otimes_{\Q} \R \xrightarrow{\tilde{\pi}_0} H^0_{\B}(M_\R, \R) \rightarrow \underbrace{H^1_{\D}(M_\R, \R(1))}_{\stackrel{?}{=} H^1_{\m}(M_\Z,\Q(1)) \otimes \R.}  \rightarrow 0\end{equation} The map $\tilde{\pi}_0$ here is given by:
$$
\tilde{\pi}_0 (x) = \frac{1}{2} (x + \bar{x}).
$$
Note that the Weil group $W_{\R}$ acts naturally on $H_{\B}(M_\C, \R)$ and 
the fixed set can be described in equivalent ways:
\begin{align*}
 H_{\B}(M_\C, \R)^{W_{\R}}  &= \mbox{subspace of the $(0,0)$-Hodge part of 
 $H_B (M_\C,\C)$
 fixed by $F_\infty$ and $c_{\B}$}  \\
 &=\mbox{orthogonal complement of $\tilde{\pi}_0 \left(F^1 H_{\dR}(M) \otimes_{\Q} \R\right)$ inside  $H_{\B}(M_\R, \R) $,}
 \end{align*}
 for any weak polarization $s$ on $M$. 
 Thus \eqref{key0} induces an isomorphism 
\begin{equation} \label{WRfixed} H_{\B}(M_\C, \R)^{W_{\R}} \stackrel{\sim}{\longrightarrow} H^1_{\D}(M_\R, \R(1)).\end{equation}

Proposition \ref{definiteness proposition}  implies that, if we are given a weak polarization $s$ on $M$,
then the form $S$ induces on $H_{\B}(M_\C, \R)^{W_{\R}} \stackrel{\sim}{\longrightarrow} H^1_{\D}(M_\R, \R(1))$ 
 a non-degenerate quadratic form; if $s$ is actually a polarization, this 
quadratic form is in fact positive definite.

  \subsubsection{Volumes}  \label{volLcomp} 
  
  We continue to study the setting of a weight $0$ motive $M$. In what follows, we do not need the full
  structure of a polarization $s$:   all we need is the associated
  linear-algebraic data,   i.e., $S$ as in  equation \eqref{Sdef}, and thus we will just assume $M$ to be so equipped. 
  Recall that although $S$ is nondegenerate, no definiteness properties are imposed on it. 
   
We   can compute the volumes (in the sense of \eqref{voldef}) of the three 
   $\Q$-vector spaces appearing in \eqref{key0},  using the metric arising from $S$.

The restriction of $S_{\C}$ to  $H_B (M_\R,\R) = V_\C ^{F_\infty,c_{\B}}$ is just given by $(x,y) \mapsto S_{\C} (x,y)$.  
When we pull back this form to $F^1 H_{\dR} (M) \otimes \R$ via $\tilde{\pi}_0$, the result is  \begin{align*}
 (x,y) = \langle \frac{ x+ \bar{x}}{2}, \frac{y + \bar{y}}{2} \rangle &= \frac{1}{4} ( S_\C (x,\bar{y}) + S_\C (\bar{x},y)) = \frac{1}{4}( S_\C (x,\bar{y}) + \overline{S_\C (x,\bar{y})} ) \\&= \frac{1}{2} \Re S_\C (x,\bar{y}) = \frac{1}{2} S_\C (x,\bar{y}).
 \end{align*}
 Here we have used that $S_{\C}(x, \bar{y}) \in \R$ for $x, y \in F^1 H_{\dR}(M) \otimes \R$:  this is because $c_{\B}$ preserves $H_{\dR}(M) \otimes \R$ (since $c_{\B}$ and $c_{\dR}$ commute), and so $\bar{y} \in H_{\dR} (M) \otimes \R$ also.

   The next lemma describe some basic results concerning these volumes and their relations. In particular, up to factors of $\Q^*$, the squares
   of these volumes {\em do not depend} on the choice of $S$: 
    \begin{lemma}  
    \label{lem:vol-indep}

   With notation as above, the square of 
$ \vol_S  H_{\B}(M_\R, \Q)  $ lies in $\Q^*$, and the square of $\vol_S F^1 H_{\dR} (M) $ is, at least up to $\Q^*$, independent of the choice of the form $S$ (subject to $S$ satisfying the conditions (a) and (b)
after \eqref{Sdef}).  

If we moreover assume Beilinson's conjecture, as formulated in \eqref{eqn:bc-wt0-1}, we have:     \begin{equation} \label{samwise}  
     \vol_S  \ H^1_{\m}(M_\Z, \Q(1))  \sim_{\Q^{\times}}    L^*(M^*,0)  \cdot \frac{ \vol_S  H_B(M_\R, \Q)    }{ \vol_S    F^{1} H_{\dR} (M) }, 
   \end{equation}
where $L^*$ means highest non-vanishing Taylor coefficient; and again all volumes are computed with respect to the form $S$.) 
 \end{lemma}
\proof    
The first assertion is immediate, since $S$ is rational-valued on $H_{\B}(M_{\C}, \Q)$.   We next prove the assertion concerning $\vol_S F^1 H_{\dR}(M)$.
The form $S$  descends to a perfect pairing
\[
S: F^1 H_{\dR} (M) \times H_{\dR} (M) /F^{0} H_{\dR} (M) \rightarrow \Q, 
\]
and hence a perfect pairing
\[
S: \det F^1 H_{\dR} (M) \times \det (H_{\dR} (M) /F^{0}) \rightarrow \Q.
\]
Note also that the complex conjugation $c_{\B}$ induces an isomorphism 
$$F^1 H_{\dR} (M) \otimes \C \simeq (H_{\dR} (M) /F^{0}  ) \otimes \C.$$
Choose generators $v^+$, $v^-$ for the $\Q$-vector spaces $\det F^1 H_{\dR} (M)$ and $\det (H_{\dR} (M)/F^{0})$. 
If $\dim F^1 H_{\dR} (M) =d$, the image of $\overline{v^+}$   under     the natural projection
$$\varphi: \bigwedge^d (H_{\dR} (M) \otimes \C) \rightarrow \bigwedge^d  \left(H_{\dR} (M) /F^{0}  \otimes \C \right) $$
is a generator for  the right-hand side,  so we have
\begin{equation} \label{lambdadef}
\varphi (\overline{v^+}) = \lambda \cdot v^-
\end{equation}
for some scalar $\lambda \in \C^\times$ (in fact, in $\R^\times$) which is  obviously independent of the choice of $S$.  
The volume of $F^1 H_{\dR} (M)$ is then given by
\[
2^d \cdot \left( \vol_S F^1 H_{\dR} (M)\right) ^2 =  S_\C (v^+, \overline{v^+}) = S_\C (v^+, \varphi (\overline{v^+}) ) = \lambda \cdot S(v^+, v^-).
\]
The result follows since $S(v^+, v^-) \in \Q^*$. 

We finally verify \eqref{samwise}: by \eqref{eqn:bc-wt0-1} we have: 
\begin{equation} \label{toothless} \det(H_{\B} (M_\R, \Q) ) \cdot  L^*(M^*, 0) \sim  \det  F^1 H_{\dR} (M) \cdot  \det (H^1_{\m} (M_\Z,\Q(1))),\end{equation}
   which we should regard as an equality inside
   $$\wedge^* H_{\B} (M_\R,\R) \simeq \wedge^* (F^1 H_{\dR} (M_\R)) \otimes \wedge^* H^1_{\D} (M_\R, \R (1)).$$
Computing volumes  of both sides of \eqref{toothless} with respect to the polarization we get \eqref{samwise}.
 \qed

  We remark that the Lemma allows us to define $\vol F^1 H_{\dR}(M)$ up to $\sqrt{\Q^*}$-- namely, take  $\sqrt{\lambda}$ where $\lambda$ is in \eqref{lambdadef} --
  even without a polarization.

\newcommand{\intprod}{\mathbin{\raisebox{\depth}{\scalebox{1}[-1]{$\lnot$}}}}

\section{Fundamental Cartan and tempered cohomological representations} \label{VoganZuckerman}

In this section, we will  associate a canonical $\C$-vector space 
$\aG$ to the real reductive group $\GG_{\R}$; its complex-linear dual will be denoted by $\aGs$.  These vector spaces depend on $\GG_{\R}$ only up to isogeny. 

Despite the notation, the group $\GG_{\R}$ does not need to be the extension of a reductive group over $\Q$; for this section alone,
it can be an arbitrary real reductive group.  We denote by $G$ the real points of $\GG_{\R}$. 
Similarly, in this section alone, we will allow $\LG$ to denote the dual group of the {\em real} algebraic group, rather than the $\Q$-algebraic group;  in other words,
$$ \LG = \widehat{G} \ltimes \Gal(\C/\R),$$
rather than the variant with $\Gal(\overline{\Q}/\Q)$. 

  We shall then construct an action
of $\wedge^* \aGs$ on the cohomology of any tempered, cohomological representation of $G$, over
which this cohomology is freely generated in degree $q$.     We will always have
 \begin{equation} \dim \aGs = \delta =\mathrm{rank}(G) - \mathrm{rank}(\Kinfty),\end{equation} 
 
The short version is that the vector space $\aGs$ is dual to the  Lie algebra of the split part of a fundamental Cartan algebra, but
we want to be a little more canonical (in particular, define it up to a unique isomorphism). 

We will give two definitions of $\aGs$. The first in \S \ref{FCA1} is analogous to the definition of 
 ``canonical maximal torus'' of a reductive group. The second definition in \S \ref{FVS}
uses the dual group.

There is a natural real structure on $\aG$, arising from either of the constructions. However, 
what will be more important to us is a  slightly less apparent real structure, the ``twisted real structure,''
which we define in Definition \ref{twistedReal}. 

In \S \ref{aGcohomology} we construct the action of $\wedge^* \aGs$ on the $(\gK)$-cohomology of a tempered representation;
in fact
we will work with $(\mathfrak{g}, \Kinfty^0)$-cohomology, where $\Kinfty^0$ is the  identity component of $\Kinfty$. 
The book \cite{BW} is a standard reference for $(\gK)$ cohomology.

 We follow in this section the convention of allowing $\mathfrak{g}$ etc. to denote the complexifications of the Lie algebras  
and reserving $\mathfrak{g}_{\R}$ or $\Lie(\GG_{\R})$ for the real Lie algebra. 
We write $\mathfrak{k}_{\R}$ for the Lie algebra of $\Kinfty$; let $\theta$ be the Cartan involution of $\mathfrak{g}_{\R}$ that fixes $\mathfrak{k}_{\R}$, 
and $\mathfrak{p}_{\R}$ the $-1$ eigenspace for $\theta$,  with complexification $\mathfrak{p}$.  Thus  $\mathfrak{g} = \mathfrak{k} \oplus \mathfrak{p}$. 
 Finally, let $\mathbf{Z}_G$ be the center of $\GG_{\R}$, with Lie algebra $\mathfrak{z}$ and real Lie algebra $\mathfrak{z}_{\R}$. 
 
 Moreover, let us fix 
\begin{equation} \label{Bdisgustion}  B_{\R}  = \mbox{ an  invariant, $\theta$-invariant,  $\R$-valued quadratic form on $\mathfrak{g}_{\R}$}, \end{equation} 
 with  
 the property that $B_{\R}(X, \theta(X))$ is negative definite.  (Invariant means that it is invariant by inner automorphisms, whereas $\theta$-invariant
 means $B_{\R}(\theta(X), \theta(Y)) =B_{\R}(X,Y)$.)  
 For example, if $\GG_{\R}$ is semisimple, the Killing form has these properties. 
 Note that such a form gives rise to a positive definite metric on $\mathfrak{g}_{\R}/\mathfrak{k}_{\R}$,
 and this normalizes a Riemannian metric on the locally symmetric space $Y(K)$.

\subsection{First construction of $\aGs$ via fundamental Cartan subalgebra} \label{FCA1} 

A {\em fundamental Cartan subalgebra} of $\mathfrak{g}_{\R}$ is a $\theta$-stable Cartan subalgebra whose compact part (the fixed points of $\theta$) is of maximal dimension  
among all $\theta$-stable Cartan subalgebras.
These are all conjugate, see \cite[2.3.4]{Wallach1}. Let $\delta$ be the dimension of the split part  ($-1$ eigenspace of $\theta$) of  a fundamental Cartan subalgebra. 
Then  $\delta = \mathrm{rank}(G) - \mathrm{rank}(\Kinfty)$. Informally, $\delta$ is the smallest
dimension of any family of tempered representations of $G$.  The integer $\delta$ depends only on the inner class of $\GG_{\R}$.
For almost simple groups, $\delta = 0$ unless $\GG_{\R}$ is ``a complex group'' (i.e. $\mathbf{G}_{\R} \simeq \mathrm{Res}_{\C/\R} \mathbf{G}^*$
where $\mathbf{G}^*$ is a simple complex reductive group) or  $\GG_{\R}$ is  (up to center and inner twisting) 
$ \SL_n, \mathrm{E}_6^{\mathrm{split}}$ or $ \mathrm{SO}_{p,q}$ where $p,q$ are odd.

Consider triples
$(\mathfrak{a}, \mathfrak{b}, \mathfrak{q})$ that arise thus:
Begin with a Cartan subgroup $B \subset \Kinfty^{\circ}$, with Lie algebra
$\mathfrak{b}_{\R} \subset \mathfrak{k}_{\R}$ and complexified Lie algebra 
$\mathfrak{b} \subset \mathfrak{k}$.  Form   its centralizer 
 $\mathfrak{t}_{\R} = \mathfrak{a}_{\R} \oplus \mathfrak{b}_{\R}$ inside $\mathfrak{g}_{\R}$,  where $\mathfrak{a}_{\R}$ is the $-1$ eigenspace for $\theta$;
 it is a fundamental Cartan subalgebra  with complexification $\mathfrak{t} = \mathfrak{a} \oplus \mathfrak{b}$. %
   Pick generic $x \in i \mathfrak{b}_{\R}$
and let $\mathfrak{q}$ be the sum  of all eigenspaces of $x$ on $\mathfrak{g}$ which have non-negative eigenvalue.   Thus $\mathfrak{q}$ is a Borel subalgebra
and its torus quotient is $\mathfrak{a} \oplus \mathfrak{b}$.

  \begin{prop}\label{independence}
    Suppose $(\mathfrak{a}, \mathfrak{b}, \mathfrak{q})$ and $(\mathfrak{a}', \mathfrak{b}', \mathfrak{q}')$ arise, as described above, 
    from $(\mathfrak{b}, x)$ and $(\mathfrak{b}', x')$. 
   
  Then there  then there exists  $g \in \GG_{\R}(\C)$  such that $\Ad(g)$ carries $(\mathfrak{a}, \mathfrak{b}, \mathfrak{q})$ to  $(\mathfrak{a}', \mathfrak{b}', \mathfrak{q}')$  and preserves the real structure on $\mathfrak{a}$
  (i.e., carries $\mathfrak{a}_{\R} \subset \mathfrak{a}$ to $\mathfrak{a}_{\R}' \subset \mathfrak{a}'$). 
Moreover, any two such $g, g'$ induce the {\em same} isomorphism $\mathfrak{a} \rightarrow \mathfrak{a}'$. 
 \end{prop}
 
 Note that $(\mathfrak{a}, \mathfrak{b}, \mathfrak{q})$ and $(\mathfrak{a}', \mathfrak{b}', \mathfrak{q}')$ need not be conjugate under $\mathbf{G}_{\R}(\R)$. 
 
 \proof 
  The last (uniqueness) assertion at least is obvious: $g, g'$ differ by an element of the Borel subgroup corresponding to $\mathfrak{q}$,
  which acts as the identity on its torus quotient.
  
Now, let  $\Delta(\mathfrak{k}:\mathfrak{b})$ be the root system defined by the action of $\mathfrak{b}$  on $\mathfrak{k}$. 
For existence note that different choices of $\mathfrak{b}$, together with a positive system for $\Delta(\mathfrak{k}:\mathfrak{b})$,  are conjugate within $\Kinfty^{\circ}$; so 
 we can assume that $\mathfrak{b}=\mathfrak{b}'$ and that $\mathfrak{q}, \mathfrak{q}'$
  induce the same positive system on $\mathfrak{k}$.   Then also $\mathfrak{a} = \mathfrak{a}'$: 
  they are both the $-1$ eigenspace of the Cartan involution on the centralizer of $\mathfrak{b}$.

We will prove in the remainder of this subsection (Lemma \ref{WMC} and Lemma \ref{WMGK})  that there exists 
$w \in \mathbf{G}_{\R}(\C)$ 
  such that \begin{equation} \label{w wanted} w \mathfrak{q} =\mathfrak{q}', \ \ w \mathfrak{b} = \mathfrak{b},   \ \ w|\mathfrak{a} = \mathrm{Id}.\end{equation}
  Since $w$ is the identity on $\mathfrak{a}$, it in particular preserves the real structure. 
   \qed 
  
  \medskip

Therefore,  $\mathfrak{a}$ or $\mathfrak{a}_{\R}$ as above
 is well-defined up to unique isomorphism; we denote this common space  by $\aG$. More formally, 
\begin{equation} \label{aDef1} \aG := \varprojlim_{(\mathfrak{a}, \mathfrak{b}, \mathfrak{q})} \mathfrak{a},\end{equation} 
and we define $\aGs$ to be its $\C$-linear dual. 
 Visibly $\aG$ does not depend on the isogeny class of $\GG_{\R}$ - it depends only on the Lie algebra $\Lie(\GG_{\R})$. 
It is also equipped with a canonical real structure arising from $\mathfrak{a}_{\R} \subset \mathfrak{a}$.

We now complete the missing step in the above proof.   Write $\mathbf{M}$ for the centralizer of $\mathfrak{a}$ in $\GG_{\C}$; it is a Levi subgroup. 
Write $\mathfrak{m}$ for the (complex)  Lie algebra
of $\mathbf{M}$.

  It is proved in \cite[Proposition 18.2.3]{CSFD};  
 that the  set of roots  of $\mathfrak{b}$ on $\mathfrak{g}$ form a not necessarily reduced root system  
 inside the dual of $i \mathfrak{b}_{\R}/ i (\mathfrak{z}_{\R} \cap \mathfrak{b}_{\R})$; we regard
 the latter as an inner product space by using the form $B_{\R}$. (The reference cited uses the Killing form, but $B$
 has all the necessary properties for the argument.) 
 
 We will abuse notation slightly and simply say that these roots  form a root system 
 $$\Delta(\mathfrak{g}:\mathfrak{b}) \subset i  \mathfrak{b}_{\R}^*,$$  with the understanding that  their span is only the subspace of $i \mathfrak{b}_{\R}^*$ orthogonal to the central space
 $\mathfrak{z}_{\R} \cap \mathfrak{b}_{\R} \subset \mathfrak{b}_{\R}$. 
Then the roots  $\Delta(\mathfrak{k}:\mathfrak{b})$ on $\mathfrak{k}$ or the roots $\Delta(\mathfrak{m}:\mathfrak{b})$  on $\mathfrak{m}$ form subsystems
of $\Delta(\mathfrak{g}: \mathfrak{b})$. 
The Weyl groups of these root systems will be denoted $W_G, W_K, W_M$ respectively; these
are all regarded as subgroups of $\Aut(\mathfrak{b})$.

 \begin{lemma}  \label{WMC} Each $ w\in W_M$ has a representative
inside the normalizer of $\mathfrak{b}$ inside $M_{\C}$.
\end{lemma}

\proof 
Observe that $\mathfrak{a}, \mathfrak{b}$ are orthogonal for  $B$.   
To see that, in turn, 
it is enough to see that $\mathfrak{a}_{\R}, \mathfrak{b}_{\R}$ are orthogonal to one another  under 
$B_{\R}$,  
which follows from the fact that they are in different eigenspaces for the Cartan involution.

We can assume that $G$ is semisimple,
and that $w$ is a root reflection $s_{\beta}$. Now $\beta$ is the restriction of some root $\beta^*$
of $\mathfrak{a} \oplus \mathfrak{b}$ on $\mathfrak{m}$, and so $w$ has 
  a representative  $\tilde{w}$ inside the normalizer
of $\mathfrak{a} \oplus \mathfrak{b}$ inside $M_{\C}$; now $\tilde{w}$ 
  preserves $\mathfrak{a}$, and therefore (by using $B$) it preserves  $\mathfrak{b}$ too. 
\qed

   \begin{table} 
 \begin{tabular}{|c|c|c|c|c|c|}
\hline \hline
$\mathbf{G}$ & $\Delta(\mathfrak{k}:\mathfrak{b})$ & $\Delta(\mathfrak{g}:\mathfrak{b})$   &    $[W_G:W_K]$      \\ 		%
\hline
$\SL_{2n}^{\mathrm{split}}$ & $\mathrm{D}_n$ & $\mathrm{C}_n$  & $2$   \\									%
$\SL_{2n}^{\mathrm{quat}}$ & $\mathrm{C}_n$ & $\mathrm{C}_n$    & $1$   \\									%
$\SL_{2n+1}$ & $\mathrm{B}_n$ & $\mathrm{BC}_n$  & $1$   \\											%
$\SO_{2k+1,2l+1}$  & $\mathrm{B}_k \times \mathrm{B}_{l}$ &   $\mathrm{B}_{k +l}$  & ${\ell +k \choose k} $   \\		%
$\mathrm{E}_6^{\mathrm{split}}$   &  $\mathrm{F}_4$  &   $\mathrm{C}_4$      & 3  \\ 
$\mathrm{E}_6^{IV} $ &    $\mathrm{F}_4$ & $\mathrm{F}_4$   &1    \\ 
$\mathrm{Res}_{\C/\R} \mathbf{H} $ & $ \Phi_{H}$  & $\Phi_H$   &   $1$    \\									%
 \hline \hline
\end{tabular}
   \caption{Root systems for simple real groups with $\delta  > 0$.}

  \end{table} 
 
 Recall that a ``chamber'' for a root system is a connected component of the complement of all hyperplanes orthogonal to the roots.
 We can regard the  Borel subalgebras $\mathfrak{q}$  
  discussed before Proposition \ref{independence} as being indexed by chambers for $\Delta(\mathfrak{g}: \mathfrak{b})$:
the parabolic $\mathfrak{q}$ associated to an element $x \in i \mathfrak{b}_{\R}$ 
 depends only on the chamber of $x$.

 The Weyl group acts simply transitively on chambers  and so a fixed chamber of $\Delta(\mathfrak{k}:\mathfrak{b})$
contains precisely $[W_G: W_K]$  chambers for $\Delta(\mathfrak{g}:\mathfrak{b})$;
the number of $G$-conjugacy classes on triples
$(\mathfrak{a}, \mathfrak{b}, \mathfrak{q})$ that arise as above is precisely this
index $[W_G: W_K]$.
The possibilities for
$W_K \subset W_G$ are listed in the table.

\begin{lemma} \label{WMGK}
Suppose that $\mathcal{C}, \mathcal{C}'$ are chambers for $\Delta(\mathfrak{g}:\mathfrak{b})$ 
that lie in a fixed   chamber for $\Delta(\mathfrak{k}:\mathfrak{b})$.
Then there is $w_M \in W_M$,    the Weyl group of $\Delta(\mathfrak{m}:\mathfrak{b})$, 
such that $w_M \mathcal{C} = \mathcal{C}'$. 
In particular, the map 
\begin{equation} \label{MGK} W_M \rightarrow W_G/W_K\end{equation} 
is surjective. 
\end{lemma} 

\proof It is sufficient to check that each root  line of $\Delta(\mathfrak{g}:\mathfrak{b})$
is either a root line of $\Delta(\mathfrak{k}:\mathfrak{b})$ or a root line of
$\Delta(\mathfrak{m}:\mathfrak{b})$ -- after all, the positive chamber for $\Delta(\mathfrak{k}:\mathfrak{b})$ is
subdivided by hyperplanes orthogonal to roots for $\Delta(\mathfrak{g}:\mathfrak{b})$. If each such hyperplane $\mathcal{H}$ 
is in fact orthogonal to a root $\alpha$ for $\Delta(\mathfrak{m}:\mathfrak{b})$, the corresponding reflection $s_{\alpha} \in W(\mathfrak{m}:\mathfrak{b})$
allows one to move between the two sides of this  $\mathcal{H}$.

The claim about root lines follows readily from the above table.  We need only consider the cases with $[W_G:W_K] > 1$. 
   $W_K$ preserves $\mathfrak{a}$ and so also $\mathfrak{m}$;
   and in each of these three cases, there is just one $W_K$-orbit on root lines for $\Delta(\mathfrak{g}:\mathfrak{b}) - \Delta(\mathfrak{k}:\mathfrak{b})$.
   It suffices, then, to produce 
 to produce a single root of $\Delta(\mathfrak{m}:\mathfrak{b})$ not belonging to $\Delta(\mathfrak{k}:\mathfrak{b})$.  
 For this we will use the theory of Vogan diagrams, referring to \cite[VI.8]{Knapp} for details;
the Vogan diagram attached to $\GG_{\R}$ is the Dynkin diagram for   $\Delta(\mathfrak{g}: \mathfrak{t})$, with reference to a suitable positive system, 
 with certain roots shaded.  

We recall that a simple root $\alpha_0$ is shaded in the Vogan diagram when it is noncompact imaginary   (i.e $\alpha_0$ restricts to zero on $\mathfrak{a}$, and whose root space belongs to $\mathfrak{p}$).
Because $\alpha_0$ restricts to zero on $\mathfrak{a}$, it  defines an element of $\Delta(\mathfrak{m}: \mathfrak{b})$;
we claim that $\alpha_0$ has the desired property (and a quick glance
shows that a shaded vertex exists in each case under consideration with $[W_G:W_K]>1$).

 Suppose that the restriction of $\alpha_0$ to $\mathfrak{b}$ belonged to $\Delta( \mathfrak{k}:\mathfrak{b})$;
so we may write $\alpha_0|_{\mathfrak{b}} =  \beta |_{\mathfrak{b}} $
for some root $\beta \in \Delta(\mathfrak{g}: \mathfrak{t})$ whose associated root space $\mathfrak{g}_{\beta}$
is not wholly contained in $\mathfrak{p}$. 
Expand $ \beta = \sum m_{\alpha} \alpha  +   \sum n_{\gamma} \gamma$
in the basis of simple roots, where the $\alpha$s are the $\theta$-fixed simple roots, and the $\gamma_j$s are the
remaining simple roots. 
The difference $\alpha_0 - \beta$ is zero on $\mathfrak{b}$; 
thus
$$ (m_{\alpha_0}-1) \alpha_0 + \sum_{\alpha \neq \alpha_0} m_{\alpha} \alpha + \sum n_{\gamma} \gamma$$
is negated by $\theta$. 
Thus $m_{\alpha_0} = 1$, the other $m$s are zero, and finally $n_{\gamma}=  - n_{\theta(\gamma)}$;
since the $m$s and $n$s all have the same sign, we must have $n_{\gamma}= 0$ for all $\gamma$. 
Then $\beta = \alpha_0$, but  this contradicts the supposition about the root space of $\alpha_0$.  \qed

\medskip

{\em Conclusion of the  proof of Proposition \ref{independence}:} We must show that there is an element $w \in \GG_{\R}(\C)$ with the properties
in equation \eqref{w wanted}.  Now $\mathfrak{q}$ and $\mathfrak{q}'$ correspond to  positive chambers $\mathcal{C}, \mathcal{C}'$
for the root system $\Delta(\mathfrak{g}: \mathfrak{b})$;  and  there exists $w \in W_G$
such that $w \mathcal{C} = \mathcal{C}'$. By the prior Lemmas,  $w$ can be represented
by an element $\tilde{w}$ such that $\tilde{w}$ belongs to $\mathbf{M}$ (so centralizes $\mathfrak{a})$ and such that $\tilde{w}$ normalizes $\mathfrak{b}$;
since $\tilde{w} \mathcal{C} =\mathcal{C}'$ we have also $\tilde{w} \mathfrak{q} =\mathfrak{q}'$.  \qed

\medskip

We conclude with one more Lemma that will be needed later:
  Given a chamber $\mathcal{C}$ in $i \mathfrak{b}_{\R}$ for $\Delta(\mathfrak{g}:\mathfrak{b})$ and $x \in \mathcal{C}$
we get a  Borel  subalgebra $\mathfrak{q} \supset \mathfrak{b} \oplus \mathfrak{a}$ from the positive roots of $x$.    
   The roots that are positive on $\mathcal{C}$ give  positive systems
on each of the root systems  that we just described, which we denote in all cases by $\Delta^+$. 

\begin{lemma} \label{twoWeyl}   With notation as described, let $n_K \in \Kinfty^{\circ}$ normalize $\mathfrak{b}$ and take the parabolic subalgebra $\mathfrak{q} \cap \mathfrak{k} \subset \mathfrak{k}$ to its opposite,
 with respect to the Cartan subalgebra $\mathfrak{b}$. 
Similarly, let $n_G \in \GG_{\R}(\C)$ normalize $\mathfrak{a} \oplus \mathfrak{b}$ and carry $\mathfrak{q}$ to its opposite.
Then $n_G$ and $n_K$ both preserve $\mathfrak{a}$, and coincide on it. 
\end{lemma}

\proof 
It is clear that $n_K$ preserves $\mathfrak{a}$.  Let $w_K$ be the automorphism of $\mathfrak{b}$
induced by $n_K$ (equivalently $n_K^{-1}$). 
Let $w_G$ be an element in the Weyl group of $\Delta(\mathfrak{g}:\mathfrak{b})$
such that $w_G \mathcal{C} = -\mathcal{C}$. 

Then $w_K w_G \mathcal{C}$ and $\mathcal{C}$ both lie in the same positive chamber for  $\Delta(\mathfrak{k}:\mathfrak{b})$.
By Lemma \ref{WMGK}   there is $w_1 \in W_M$ such that $w_1\mathcal{C} = w_K w_G \mathcal{C}$.
It follows that $w_K w_G \in W_M$. By Lemma \ref{WMC}
there exists a representative $n_M \in \mathbf{G}_{\R}(\C)$  for $w_K w_G$ 
which normalizes $\mathfrak{a}$ and $\mathfrak{b}$. 

It follows, then, that $n := n_K\cdot n_M \in \mathbf{G}_{\R}(\C)$
normalizes $\mathfrak{a}$ and $\mathfrak{b}$; 
this element $n$ takes  the chamber $\mathcal{C}$ to $-\mathcal{C}$, and so 
it takes  $\mathfrak{q}$ to $\mathfrak{q}^{\mathrm{op}}$.  
We may therefore suppose $n=n_G$. 
 It follows that $n_G$ preserves $\mathfrak{a}$, and its action on $\mathfrak{a}$ coincides with $n_K$. 
 \qed 
 
\begin{definition}
The {\em long Weyl element} is the involution of 
  $\aG = \varprojlim_{(\mathfrak{a},\mathfrak{b},\mathfrak{q})} \mathfrak{a}$
induced by the common action of $n_G$ or $n_K$ from the prior Lemma. 
 \end{definition}

The long Weyl element preserves $\mathfrak{a}_{G,\R}$,
 since $w_K$ can be represented by an element of $\Kinfty$. 
  We use it to define a second real structure:

\begin{definition} \label{twistedReal}
The twisted real structure  $\mathfrak{a}_{G,\R}'$ on $\mathfrak{a}_G$
is the fixed points of the involution given by
$$ (X \mapsto \bar{X}) \otimes w,$$
where $X \mapsto \bar{X}$ is the antilinear involution defined by $\mathfrak{a}_{G,\R}$,
and $w$ is the long Weyl group element for $\mathfrak{a}_G$. 
\end{definition}

\subsection{Second construction of $\aGs$ via dual group}  \label{FVS} 
Let $\widehat{T} \subset \widehat{B}$ be the standard maximal torus and Borel in $\widehat{G}$. 
Let $\LW$ denote the normalizer of $\widehat{T}$ inside $\widehat{G} \rtimes \Gal(\C/\R)$, modulo $\widehat{T}$. 
There exists a unique lift  $w_0 \in \LW$ of the nontrivial element of $\Gal(\C/\R)$ 
with the property that $w_0$ sends $\widehat{B}$ to the opposite Borel (w.r.t. $\widehat{T}$). 
Moreover, we may choose a representative of $w_0$ that lies inside $\widehat{G}(\R) \rtimes \Gal(\C/\R)$, unique up to $\widehat{T}(\R)$;  thus 
the space $ \Lie(\widehat{T})^{w_0}$  carries a real structure arising from the real structure on $\widehat{T}$.

We will show that $\aGs$ can be identified with $\Lie(\widehat{T})^{w_0}$, 
in a fashion that carries the real structure $\mathfrak{a}_{G,\R}^{*}$
to the natural real structure on the latter space.

Observe, first of all, that a choice of $(\mathfrak{a}, \mathfrak{b}, \mathfrak{q})$ as before yields   a
torus $\mathbf{T} \subset \mathbf{G}_{\R}$ with Lie algebra $\mathfrak{a} \oplus \mathfrak{b}$,
and a Borel subgroup containing $\mathbf{T}$, with Lie algebra $\mathfrak{q}$; then we get identifications 
\begin{equation} \label{lambdaGa} \Lie(\widehat{T})  \simeq  X_*(\widehat{T}) \otimes \C = X^*(\mathbf{T}) \otimes \C =  (\mathfrak{a} \oplus \mathfrak{b})^* \end{equation} 
We have used the fact that, for any complex torus $\mathbf{S}$, we may identify $\Lie(\mathbf{S}) $ with $X_*(\mathbf{S}) \otimes \Lie(\mathbb{G}_m)$
and thus with $X_*(\mathbf{S}) \otimes \C$, choosing the basis for $\Lie(\mathbb{G}_m)$ that is dual to $\frac{dz}{z}$. 
 
 If we choose a different triple $(\mathfrak{a}', \mathfrak{b}', \mathfrak{q}')$ there exists $g \in \GG_{\R}(\C)$
 conjugating $(\mathfrak{a}, \mathfrak{b},\mathfrak{q})$ to $(\mathfrak{a}', \mathfrak{b}', \mathfrak{q}')$; the 
maps \eqref{lambdaGa} differ by $\Ad(g)$.  In particular, we get by virtue of Proposition \ref{independence}, a map
\begin{equation} \label{projlimmap} \Lie(\widehat{T}) \rightarrow  \varprojlim_{(\mathfrak{a}, \mathfrak{b}, \mathfrak{q})} \mathfrak{a}^* = \mathfrak{a}_G^*.\end{equation}

\begin{lemma} \label{realstruc}
The map \eqref{projlimmap}  carries $\Lie(\widehat{T})^{w_0}$ isomorphically onto $\aGs$,
and preserves real structures.  

Moreover, the long Weyl group element $w_{\widehat{G}}$  for $\widehat{T}$, carrying $\widehat{B}$ to its opposite,
preserves $\Lie(\widehat{T})^{w_0}$, and is carried under this identification
to the long Weyl element acting on $\aGs$ (see discussion after Lemma  \ref{twoWeyl}). 
\end{lemma}

This justifies using $\Lie(\widehat{T})^{w_0}$ as an alternate definition of $\aGs$. 

In the following proof, we will refer to the ``standard'' antiholomorphic involution on $\widehat{T}$ or its Lie algebra.
The torus $\widehat{T}$ is, by definition,  a split torus; as such it has a unique split $\R$-form,
and we refer to the associated antiholomorphic involution as the ``standard'' one. 
 
\proof
Under the identification of \eqref{lambdaGa} the action of $w_0$ on $\Lie(\widehat{T})$ is carried to 
the action on $X^*(\mathbf{T}) \otimes \C = (\mathfrak{a} \oplus \mathfrak{b})^*$
of an automorphism $\gamma$ of $\mathfrak{g}$
that belongs to the same outer class as complex conjugation, and switches
$\mathfrak{q}$ and its opposite  $\mathfrak{q}^{\op}$ relative to $\mathfrak{a} \oplus \mathfrak{b}$. 
 However, by virtue of the construction of $\mathfrak{q}$ from an element $x \in i \mathfrak{b}_{\R}$, complex conjugation  switches $\mathfrak{q}$ and $\mathfrak{q}^{\op}$.
It follows that $\gamma$ corresponds precisely to the action of complex conjugation $c$ on $X^*(\mathbf{T}) \otimes \C$.
It readily follows that it acts by $-1$ on $\mathfrak{b}^*$ and $1$ on $\mathfrak{a}^*$. 
This shows that $\Lie(\widehat{T})^{w_0}$ is carried isomorphically onto $\aG$ by 
\eqref{projlimmap}.

 Now the antiholomorphic involution $\left(c \otimes (z \mapsto \bar{z}) \right)$
on $X^*(\mathbf{T}) \otimes \C  = (\mathfrak{a} \oplus \mathfrak{b})^*$ fixes precisely $\mathfrak{a}_{\R}^* \oplus \mathfrak{b}_{\R}^*$.  Transporting to $\Lie(\widehat{T})$ by means
of the above identification, we see that the real structure on $(\mathfrak{a}_{\R} \oplus \mathfrak{b}_{\R})^* \subset (\mathfrak{a}\oplus \mathfrak{b})^*$ corresponds
to the antiholomorphic involution $c'$ on $\Lie(\widehat{T})$ which is the composition of $w_0$ with the  standard antiholomorphic involution. In particular, restricted to the $w_0$-fixed part,  $c'$ reduces to the standard antiholomorphic involution. This proves 
the statement about real structures.

For the second claim, we note that $w_{\widehat{G}}$ and $w_0$ commute, so certainly $w_{\widehat{G}}$
preserves $\Lie(\widehat{T})^{w_0}$;  under the identifications of \eqref{lambdaGa} 
$w_{\widehat{G}}$ corresponds to an element of the Weyl group of $(\mathfrak{a} \oplus \mathfrak{b})$
which sends  $\mathfrak{q}$ to the opposite parabolic. This coincides with the long Weyl element for $\aG$ by Lemma \ref{twoWeyl}. 
 \qed

\subsection{The tempered cohomological parameter}  \label{TemperedCohomologicalParameter}
We will next  construct a canonical identification  \begin{equation} \label{aGtemp}  \aGs \simeq \mbox{Lie algebra of the centralizer of $\rho:W_{\R} \rightarrow \LG$}\end{equation} 
where $\rho$ is the parameter of any tempered cohomological representation for $G$;  correspondingly we get
\begin{equation} \label{frofro2} \mathfrak{a}_G \simeq \mbox{fixed points of $\Ad^* \rho: W_{\R} \rightarrow \GL(\widetilde{\mathfrak{g}})$ on $\widetilde{\mathfrak{g}}$.}\end{equation}
where $\Ad^*: \LG \rightarrow \GL(\widetilde{\mathfrak{g}})$ is the co-adjoint representation. 

To see this we must  discuss  the $L$-parameter of tempered cohomological representations:

Write as usual $W_{\R} = \mathbf{C}^* \rtimes \langle j \rangle$, where $j^2 =-1$, for the real Weil group. 
Let $\rho: W_{\R} \rightarrow \LG$
be a tempered Langlands parameter  
whose associated $L$-packet contains  a representation 
with nonvanishing $(\gKZ)$ cohomology. In particular, the infinitesimal character of this representation
coincides with that of the trivial representation. 
By examining infinitesimal characters,  we can conjugate $\rho$ in $\widehat{G}$
to a representation $\rho_0$ such that 
\begin{equation} \label{rho0def} \rho_0|_{\mathbf{C}^*} : \mathbf{C}^* \rightarrow \LG\end{equation}
is given by $\Sigma_G(\sqrt{z/\bar{z}})$, where $\Sigma_G$ is the canonical cocharacter $\mathbb{G}_m \rightarrow \widehat{G}$
given by the sum of all positive coroots.  
 The connected centralizer of $\rho|_{\mathbf{C}^*}$ 
is then $\widehat{T}$, so  
the image of $j$ in $\LG$ must normalize $\widehat{T}$  and sends $\widehat{B}$ to $\widehat{B}^{\mathrm{op}}$.
Therefore, $\rho_0(j)$ 
 defines the same class as $w_0$ inside $\LW$ (notation of \S \ref{FVS}) and therefore
\begin{equation} \label{ag3} \aGs = \mbox{ Lie algebra of the centralizer of $\rho_0$.} \end{equation}

Now $\rho = \Ad(g) \rho_0$ for some $g \in \widehat{G}$; since the centralizer of $\rho_0$ is contained in $\widehat{T}$,
$g$ is specified up to right translation by $\widehat{T}$, and consequently the induced map
$$ \mbox{Lie algebra of the centralizer of $\rho_0$} \stackrel{\sim}{\rightarrow} \mbox{Lie algebra of the centralizer of $\rho$}$$
is independent of the choice of $g$. Composing with \eqref{ag3}, we arrive at the desired
identification \eqref{aGtemp}.

\begin{rem}
 In general,  there are multiple possibilities for the conjugacy class of $\rho$, i.e.
 multiple $L$-packets of tempered cohomological representations; 
however, if $\GG_{\R}$ is simply connected or adjoint, $\rho$ is unique up to conjugacy: any two choices of $w_0$ differ by an element $t \in \widehat{T}$, which 
lies in the fixed space for $\tau: z \mapsto 1/z^{w_0}$ on $\widehat{T}$.   Thus we must verify that 
every element of the $\tau$-fixed space $\widehat{T}^{\tau}$ is of the form $x \cdot \tau(x)$ for some $x \in \widehat{T}$;  equivalently that $\widehat{T}^{\tau}$ is connected.  If $\widehat{G}$ is simply connected, coroots give an isomorphism $\mathbb{G}_m^r \simeq \widehat{T}$,
and the map $\alpha \mapsto - w_0 \alpha$ permutes the coroots;
we are reduced to  verifying  connectivity of fixed points in  the case of $\tau$ the swap on $\mathbb{G}_m^2$
or $\tau$ trivial on $\mathbb{G}_m$, which are obvious. 
\end{rem}

\subsection{The action of the exterior algebra $\wedge^* \aGs$
on the cohomology of a tempered representation}  \label{aGcohomology}

 In this section,
we will construct an action of $\wedge^* \aGs$ on  $H^*(\gKZ; \Pi)$, for any finite length, tempered,  cohomological representation $\Pi$ of $G$. 
In this situation, by ``cohomological,'' we mean that every constituent of $\Pi$ is cohomological -- note that $\Pi$ is tempered, and thus semisimple.  

This action will have the property that the induced map
\begin{equation} \label{qdef} H^q(\gKZ; \Pi) \otimes \wedge^j \aGs \longrightarrow H^{q+j}(\gKZ; \Pi)\end{equation}
is an isomorphism.   Here $q$ is the minimal dimension in which the $\gKZ$-cohomology is nonvanishing; explicitly, we have
$2q + \mathrm{dim}_{\C} \aG = \dim Y(K)$.  
The action of $\wedge^* \aGs$ 
  will commute with the natural action of $\Kinfty/\Kinfty^0$ on $H^*(\gKZ; \Pi)$.

We construct the action first in the simply connected case, and then reduce
the general case to that one.    %

 \subsubsection{The action for $\GG_{\R}$ simply connected} \label{Gsc types}
 Here $G$ is connected, as is its maximal compact; 
and the cohomological, tempered representations are indexed (with notation as in \S \ref{FCA1}) by choices of a positive chamber $\mathcal{C}$ for the root system $\Delta(\mathfrak{g}:\mathfrak{b})$: 

We have already explained that such a chamber $\mathcal{C}$ gives rise to a Borel subgroup $\mathfrak{q}$
and a notion of positive root for $\Delta(\mathfrak{g}: \mathfrak{b})$. 
   Vogan and Zuckerman \cite{VZ} attach to $\mathcal{C}$ a tempered cohomological representation $\pi(\mathcal{C})$
characterized  by the additional fact that it contains with multiplicity one the irreducible representation $V_{\mathcal{C}}$ of $\Kinfty=\Kinfty^0$
with highest weight 
$$ \mu_{\mathcal{C}} = \mbox{ the sum of roots associated to root spaces in  $\mathfrak{u} \cap \mathfrak{p}$},$$
where  $\mathfrak{u}$ is the unipotent radical of $\mathfrak{q}$.  (See \cite[Theorem 2.5]{VZ}).  Moreover, 
it is known that the $V_{\mathcal{C}}$ is the only irreducible representation of $\Kinfty$
that occurs both in $\pi(\mathcal{C})$ and in $\wedge^* \mathfrak{p}$ (proof and discussion around \cite[Corollary 3.7]{VZ}).

We write $V_{-\mathcal{C}}$ for the dual representation to $V_{\mathcal{C}}$;  
its lowest weight  is then equal to $-\mu_{\mathcal{C}}$. Let us fix a highest $v^+ \in V_{\mathcal{C}}$  and a lowest weight vector $v^-$ in $V_{-\mathcal{C}}$, with weights $\mu_{\mathcal{C}}$ and $-\mu_{\mathcal{C}}$. 
In what follows,  a vector of  ``weight $\mu$'' means that it transforms under the character $\mu$ of $\mathfrak{q} \cap \mathfrak{k}$:
and a vector ``of  weight $-\mu$'' transforms under that character of $\mathfrak{q}^{\mathrm{op}} \cap \mathfrak{k}$,
where $\mathfrak{q}^{\mathrm{op}}$ is the parabolic subgroup associated to $-\mathcal{C}$.  In other words, ``weight $\mu$'' is a requirement
on how the vector transforms by a Borel subalgebra, not merely a toral subalgebra.

Write $W[\mathcal{C}]$ for the $V_{\mathcal{C}}$-isotypical subspace of 
 an arbitrary $\Kinfty$-representation $W$,
 and $W[-\mathcal{C}]$ for the $\widetilde{V_{\mathcal{C}}}$-isotypical subspace. 
 Thus $W[\mathcal{C}] = V_{\mathcal{C}} \otimes \Hom(V_{\mathcal{C}}, W)$  and  the rules $f \mapsto f(v^{\pm})$ gives an isomorphism
\begin{equation} \label{highestweight}  \Hom(V_{\mathcal{C}}, W)= \mbox{ vectors in $W$ of weight $\mu_{\mathcal{C}}$ under $\mathfrak{q} \cap \mathfrak{k}$}.\end{equation}
 \begin{equation} \label{lowestweight}  \Hom(V_{-\mathcal{C}}, W)= \mbox{ vectors in $W$ of weight $-\mu_{\mathcal{C}}$ under $\mathfrak{q}^{\mathrm{op}} \cap \mathfrak{k}$}.\end{equation}
Let $\overline{\mathfrak{u}}$ be the unipotent radical of $\mathfrak{q}^{\mathrm{op}}$. 
From the  splitting 
\begin{equation} \label{splitting}  \mathfrak{a}	 \oplus  (\mathfrak{u} \cap \mathfrak{p} )\oplus  (\overline{\mathfrak{u}} \cap \mathfrak{p}) = \mathfrak{p}, \end{equation}
we get a tensor decomposition of 
 $ \wedge^* \mathfrak{p}$ and of $\wedge^* \mathfrak{p}^*$. 
 For the spaces of vectors of weights $\mu$ and $-\mu$ we get 
\begin{equation} \label{igly}   \left( \wedge^* \mathfrak{p} \right)^{\mu} =  \wedge^* \mathfrak{a} \otimes  \det  (\mathfrak{u}   \cap \mathfrak{p}), \ \ \ 
 \left(\wedge^* \mathfrak{p}^*\right)^{-\mu}= 
 \wedge^* \mathfrak{a}^* \otimes \det(\overline{\mathfrak{u}} \cap \mathfrak{p})^*.\end{equation}

 In particular, there is a natural inclusion $\mathfrak{a}^* \hookrightarrow \mathfrak{p}^*$ (from \eqref{splitting}), and then  the natural action of $\wedge^* \mathfrak{a}^*$
 on $\wedge^* \mathfrak{p}^*$ makes the space of weight $-\mu$ vectors in the latter
 a free, rank one module. Note that we may regard $\wedge^* \mathfrak{p}^*$ either as a left- or a right- module for $\wedge^* \mathfrak{a}^*$;
 the two actions differ by a sign $(-1)^{\deg}$ on $\mathfrak{a}^*$.  We will
 use either version of the action according to what is convenient. 
 
 Thus we have an action of $\wedge^* \mathfrak{a}^*$ on 
\begin{equation} \label{eqqq} \Hom(V_{-\mathcal{C}}, \wedge^* \mathfrak{p}^*)  \stackrel{\sim}{\longrightarrow} (\wedge^* \mathfrak{p}^*)^{-\mu}  \ \  (\mbox{via }f \mapsto f(v^-)), \end{equation}
  given (in the left-hand space) by the rule $Xf(v^-) =  X \wedge f(v^-)$

 There is also a contraction action of  $\wedge^* \mathfrak{a}^*$ on   $\wedge^* \mathfrak{p}$: 
 for $X \in \mathfrak{a}^*$, the rule
$Y \mapsto X \intprod Y$ is a derivation  of $\wedge^* \mathfrak{p}$ with degree $-1$, which in degree $1$ realizes the pairing
$\mathfrak{a}^* \times \mathfrak{p} \rightarrow \mathbb{C}$.   As a reference for contractions, see \cite[Chapter 3]{BourbakiExterior}. 
This action again 
 makes the space of weight $\mu$ vectors a free, rank one module. 

 The two actions are adjoint:
 \begin{equation} \label{adjugate2} \langle  X \wedge A, B \rangle = \langle A, X  \intprod B \rangle,  \ \ X \in  \mathfrak{a}^*, A \in \wedge^* \mathfrak{p}^*, B \in \wedge^* \mathfrak{p}, \end{equation}
where the pairing between $\wedge^* \mathfrak{p}$ and $\wedge^* \mathfrak{p}^*$  is the usual one (the above equation looks a bit peculiar -- it might seem preferable to replace $X \wedge A$ by $A \wedge X$ on the left -- 
but in order to do that we would have to use a different pairing, which we prefer not to do). 
 
From   \begin{equation} \label{basiciso0} H^*(\gKZ; \pi(\mathcal{C})) =  \underbrace{ \left( \wedge^*   \mathfrak{p}^* \otimes   \pi(\mathcal{C}) \right)^{\Kinfty}}_{= \Hom_{\Kinfty}(\wedge^* \mathfrak{p}, \pi(\mathcal{C}))}
\simeq \underbrace{ \Hom(V_{-\mathcal{C}}, \wedge^*   \mathfrak{p}^*  )}_{\simeq \left( \wedge^* \mathfrak{p}^* \right)^{-\mu}} \otimes \Hom(V_{\mathcal{C}}, \pi(\mathcal{C})),\end{equation}
we have also constructed an action of $\wedge^* \mathfrak{a}^*$ on the $(\gKZ)$ cohomology of $\pi(\mathcal{C})$. 
Again, it can be considered either as a left action or a right action, the two being related by means of a sign; we
will usually prefer to consider it as a right action. 

This action is characterized in the following way:
 for any $f \in \Hom_{\Kinfty}(\wedge^* \mathfrak{p}, \pi(\mathcal{C}))$, 
and any vector $v$ of weight $\mu_{\mathcal{C}}$ in $\wedge^* \mathfrak{p}$,  and for $X \in \wedge^* \mathfrak{a}^*$, we have  \begin{equation} \label{char}   Xf:  v \mapsto  f(X \intprod v).\end{equation}
The left action is related to this via $f \cdot X = (-1)^{\deg(f)} (X \cdot f)$.  

   To verify  \eqref{char}, note that the map $f$ factors through $V_{\mathcal{C}} \subset \pi(\mathcal{C})$. We may replace $\pi(\mathcal{C})$ by $V_{\mathcal{C}}$,
 regarding $f$ as a $\Kinfty$-map $\wedge^* \mathfrak{p} \rightarrow V_{\mathcal{C}}$, 
and write $f^t: V_{-\mathcal{C}} \rightarrow \wedge^* \mathfrak{p}^*$ for the transpose of $f$. Now, for $v^+ \in V_{\mathcal{C}}$ a vector of weight $\mu_{\mathcal{C}}$, the evaluation 
 $f(v^+)$ is determined by its pairing with a lowest weight vector $v^{-} \in V_{-\mathcal{C}}$. We have 
$$ \langle v^{-},  Xf(v^+) \rangle = \langle (Xf)^t v^{-}, v^+ \rangle = \langle  X \wedge f^t(v^-) , v^+ \rangle \stackrel{\eqref{adjugate2}}{=}   \langle f^t(v^-),  X \intprod v^+  \rangle =  \langle v^-, f(X \intprod v^+)  \rangle.$$

In summary, we have a well-defined action of $\wedge^* \aGs$ on the $(\gKZ)$ cohomology  
of any  tempered irreducible cohomological representation. 
(Strictly speaking,  we should verify that our definitions did not depend on the
choice of $(\mathfrak{b},\mathcal{C})$. If $k \in \Kinfty$ conjugates
$(\mathfrak{b}, \mathcal{C})$ to $(\mathfrak{b}', \mathcal{C}')$,
then it carries $(\mathfrak{q}, \mu_{\mathcal{C}})$ to $(\mathfrak{q}', \mu_{\mathcal{C}'})$;
there is   an isomorphism $\iota: \pi(\mathcal{C}) \rightarrow \pi(\mathcal{C}')$,
and the actions of $\Ad(k): \wedge^* \mathfrak{a} \simeq \wedge^* \mathfrak{a}'^*$
are compatible with the map on $(\gKZ)$ cohomology induced by $\iota$; thus we get an action of $\wedge^* \aGs$ as claimed.)

 Finally, it is convenient to extend the action to representations that are not irreducible, in the obvious fashion:  If $\Pi$ is any tempered representation of finite length, 
 we have
 $$ H^*(\gKZ; \Pi) = \bigoplus_{\alpha} \Hom(\pi_{\alpha},  \Pi) \otimes H^*(\gKZ;  \pi_{\alpha}),$$
 the sum being taken over (isomorphism classes of) tempered cohomological representations $\pi_{\alpha}$;
 we define $\wedge^* \aGs$ to act term-wise.

\begin{rem} It is also possible to construct this action using the realization of tempered cohomological representations
as parabolic induction from  a discrete series on $\mathbf{M}$. We omit the details. 
\end{rem}

 \subsubsection{Interaction with automorphisms}  \label{autersec} We continue to suppose that $\GG_{\R}$ is semisimple and simply connected. 
   Suppose that $\alpha$ is an automorphism of $\GG_{\R}$ that arises from the conjugation action
 of the adjoint form $\G^{\mathrm{ad}}$, preserving $\Kinfty$.   If $\Pi$ is a  tempered representation of finite length, then so is  
 its $\alpha$-twist ${}^{\alpha} \Pi$, defined by ${}^{\alpha}\Pi(\alpha(g)) = \Pi(g)$.

Also $\alpha$ induces an automorphism $Y \mapsto \alpha(Y)$ of $\mathfrak{p}$;  
the $\Kinfty$ representations $\mathfrak{p}$ and ${}^{\alpha} \mathfrak{p}$ 
are intertwined via the inverse map $Y \mapsto \alpha^{-1}(Y)$.%

  \begin{lemma}  Let $\Pi$ be tempered cohomological of finite length. The natural map 
\begin{equation} \label{above} \Hom_{\Kinfty}(\wedge^* \mathfrak{p}, \Pi)  \rightarrow \ \Hom_{\Kinfty}(\wedge^* \mathfrak{p},   {}^{\alpha} \Pi), \end{equation}
which sends $f$ to  the composite
$ \wedge^* \mathfrak{p} \simeq \wedge^*  \left( {}^{\alpha}\mathfrak{p} \right) \stackrel{{}^{\alpha} f}{\rightarrow}{}^{\alpha} \Pi$,
commutes   the $\wedge^* \aGs$ actions on both spaces. 
\end{lemma} 

\proof
This reduces to the irreducible case. So suppose that $\pi = \pi(\mathcal{C})$, where $\mathcal{C}$
is a chamber  $\mathcal{C} \subset i\mathfrak{b}_{\R}^*$, giving rise to data $(\mathfrak{a}, \mathfrak{b},\mathfrak{q})$. 
Adjusting $\alpha$ by an element of $\Kinfty$, we may suppose that $\alpha$ preserves $\Kinfty$ and  $\mathfrak{b}$
and the Borel subalgebra $\mathfrak{q} \cap \mathfrak{k} \subset \mathfrak{k}$.

  If  $W$ is an irreducible $\Kinfty$-representation of highest weight $\mu$, then ${}^{\alpha} W$ has highest weight $\mu \circ \alpha^{-1}$. 
Therefore the representation ${}^{\alpha} \pi(\mathcal{C})$ contains the $\Kinfty$-representation with highest weight $\mu_{\mathcal{C}} \circ \alpha^{-1}$,
which is associated to the chamber $\alpha(\mathcal{C})$ and the parabolic $\alpha(\mathfrak{q})$: 
$$ {}^{\alpha} \pi(\mathcal{C}) = \pi(\alpha(\mathcal{C})).$$
 
Now $\alpha$ sends $(\mathfrak{a}, \mathfrak{b} ,\mathfrak{q})$ to $(\mathfrak{a}, \mathfrak{b}, \alpha(\mathfrak{q}))$.
Although it belongs only to $\mathbf{G}^{\mathrm{ad}}(\R)$ it can be lifted to $\mathbf{G}_{\R}(\C)$, 
 and so the following diagram commutes:
 \begin{equation} \label{twoway} 
  \xymatrix{
\aG \ar[r]^{ \mathcal{C} } \ar[d]^{\alpha}  & \mathfrak{a}  \ar[d]^=\\
 \aG \ar[r]^{ \alpha(\mathcal{C})} & \mathfrak{a}. 
  }
\end{equation}
where the vertical arrows refer to the identification of $\mathfrak{a}$ with $\aG$
induced by the triples $(\mathfrak{a},\mathfrak{b},\mathfrak{q})$ (at top)
and $(\mathfrak{a}, \mathfrak{b}, \alpha(\mathfrak{q}))$ (at bottom).

Note that the map $Y \mapsto \alpha^{-1}(Y)$  takes
$(\wedge^* \mathfrak{p})^{\mu \alpha^{-1}} \rightarrow (\wedge^* \mathfrak{p})^{\mu}$
(where the weight spaces are computed for the usual actions, not the twisted ones).
The map  \eqref{above}  explicitly sends $f$ 
   to  $f': Y \in \wedge^* \mathfrak{p} \mapsto f(\alpha^{-1}(Y))$;
if $f$ on the left factors through highest weight $\mu$, then $f'$ on the right factors through highest weight $\mu \alpha^{-1}$.

For $v \in (\wedge^* \mathfrak{p})^{\mu}$ and $X \in \wedge^* \aGs$ we have  $\alpha(v) \in (\wedge^* \mathfrak{p})^{\mu \alpha^{-1}}$ and, for $f$ as above,  
$$ (Xf)':  \alpha v \mapsto (Xf)(v) =  f(X \intprod v), $$
$$ (\alpha(X) f'): \alpha v \mapsto  f'(\alpha(X) \intprod \alpha(v))   = f(X \intprod v)$$ 
 In view of diagram \eqref{twoway} this proves the statement. 
 
  \qed

\subsubsection{Interaction with duality and complex conjugation}

Suppose that $j+j' = d = \dim(Y(K))$.  Let $\Pi $ be a tempered  cohomological representation of finite length.  There is a natural pairing
\begin{equation} \label{Poincare} H^j(\gK, \Pi) \times H^{j'}(\gK, \widetilde{\Pi}) \longrightarrow \det\mathfrak{p}^*\end{equation}
corresponding to
$$( \wedge^j \mathfrak{p}^* \otimes \Pi)^{\Kinfty} \otimes 
( \wedge^{j'} \mathfrak{p}^* \otimes  \widetilde{\Pi})^{\Kinfty}  \longrightarrow  \det \mathfrak{p}^*,$$
amounting to cup product on the first factors and the duality pairing on the second factors. 

\begin{lemma} \label{adjointness}
The pairing  \eqref{Poincare} has the following adjointness:
$$ \langle  f_1 \cdot X, f_2 \rangle =  \langle f_1, (w X) \cdot  f_2 \rangle $$
for $X \in \wedge^* \aGs$, and $w$ the long Weyl group element (Lemma \ref{twoWeyl}).

\end{lemma}
\proof 
This reduces to the irreducible case $\Pi = \pi(\mathcal{C})$; its contragredient  is $\pi(-\mathcal{C})$,
parameterized by the chamber $-\mathcal{C}$ associated to $(\mathfrak{a}, \mathfrak{b}, \mathfrak{q}^{\op})$. 
  We must verify that  $\aGs$ acts (up to sign) self-adjointly for the the cup product
$$ \wedge^j \mathfrak{p}^*[-\mathcal{C}] \otimes  \wedge^{d-j} \mathfrak{p}^*[\mathcal{C}]  \rightarrow \det \ \mathfrak{p}^*$$ 
or, what is the same, the map
$$  \Hom(V_{-\mathcal{C}}, \wedge^j \mathfrak{p}^*)  \otimes  \Hom(V_{\mathcal{C}}, \wedge^{d-j} \mathfrak{p}^*) \rightarrow 
\Hom(V_{-\mathcal{C}} \otimes V_{\mathcal{C}}, \det \ \mathfrak{p}^*) \rightarrow \det \ \mathfrak{p}^*$$ 
Suppose $f_1 \in \Hom(V_{-\mathcal{C}}, \wedge^j \mathfrak{p}^*)$ and $f_2 \in  \Hom(V_{\mathcal{C}}, \wedge^{d-j} \mathfrak{p}^*) $;
their image under the first map is given by $v_1 \otimes v_2 \mapsto f_1(v_1) \wedge f_2(v_2)$. 
 This map factors through the one-dimensional subspace of invariants on $V_{-\mathcal{C}} \otimes V_{\mathcal{C}}$;
 to evaluate it on a generator for that space, we may as well evaluate it on $v^- \otimes v^+$, which has nonzero projection to that space. 
In other words, we must prove the adjointness statement for $(f_1, f_2) \mapsto f_1(v^-) \wedge f_2(v^+)$. 
For $X \in \wedge^* \mathfrak{a}^*$ we have
$$  (f_1 \cdot X) (v^-) \wedge f_2(v^+) =     f_1(v^-) \wedge X \wedge  f_2(v^+) =     f_1(v^-) \wedge  (X f_2)(v^+),$$
where the sign is as in the statement of the Lemma.
However, the identifications of $\mathfrak{a}$ with $\aG$ arising from $(\mathfrak{a}, \mathfrak{b}, \mathfrak{q})$ and
$(\mathfrak{a}, \mathfrak{b}, \mathfrak{q}^{\op})$ differ by a long Weyl group element, as in Lemma \ref{twoWeyl}. 
\qed

\begin{lemma} \label{conjugation}
Let $\Pi$ be a tempered, finite length, cohomological representation, and observe that 
the natural real structure on $\mathfrak{p}$
induces a ``complex conjugation'' antilinear map $H^*(\gK, \Pi) \rightarrow H^*(\gK, \overline{\Pi})$,
where, as usual, $\overline{\Pi}$ denotes the representation with the same underlying vectors
but the scalar action modified by complex conjugation.

Then  the following diagram commutes:
  \begin{equation} \label{conj diagram}
 \xymatrix{
  H^*(\gK, \Pi) \otimes \wedge^* \aGs    \ar[r]   \ar[d]  &    H^*(\gK, \Pi)  \ar[d]  \\
H^*(\gK, \overline{\Pi}) \otimes \wedge^* \aGs  \ar[r]     &   H^*(\gK, \overline{\Pi})
 }
\end{equation}
where all vertical maps are complex conjugation; the complex conjugation on $\aGs$ is that corresponding to the twisted real structure. 
\end{lemma}

 \proof Again, this reduces to the irreducible case $\Pi = \pi(\mathcal{C})$. 
Fixing an invariant Hermitian form on $V_{\mathcal{C}}$, we may identify $V_{-\mathcal{C}}$ with $\overline{V_{\mathcal{C}}}$, in such a way
 that $\overline{v^+} = v^-$. 
 
 The following diagram commutes: 
   \begin{equation} \label{copacabana}
 \xymatrix{
  \Hom_K(V_{-\mathcal{C}},  \wedge^* \mathfrak{p})   \ar[r]^{\qquad S \mapsto S v^-}  \ar[d]^{S \mapsto \bar{S}} &   \left( \wedge^* \mathfrak{p}\right)^{-\mu}  \ar[d]^{\mbox{conjugation}} \\
\Hom_{K}( V_{\mathcal{C}},  \wedge^* \mathfrak{p})   \ar[r]^{\qquad R \mapsto R   \overline{v^-}}    &  \left( \wedge^* \mathfrak{p}\right)^{\mu} 
 }
\end{equation}
where we define $\bar{S}$ by $\bar{S}(\bar{v}) = \overline{S(v)}$.
There is an induced complex conjugation 
$$\underbrace{ \wedge^* \mathfrak{p}[-\mathcal{C}]}_{  \Hom_K(V_{-\mathcal{C}},  \wedge^* \mathfrak{p})  \otimes V_{-\mathcal{C}} } \rightarrow
\underbrace{ \wedge^* \mathfrak{p}[\mathcal{C}]}_{  \Hom_K(V_{\mathcal{C}},  \wedge^* \mathfrak{p})  \otimes V_{\mathcal{C}}}$$
where we tensor $S \mapsto \bar{S}$  with the conjugation on $V_{-\mathcal{C}}$, 
and then the  following diagram is also commutative:  
   \begin{equation} 
 \xymatrix{
 \wedge^* \mathfrak{p} [-\mathcal{C}]   \otimes \wedge^* \mathfrak{a}    \ar[r]   \ar[d]^{(S \mapsto \bar{S}) \otimes \mathrm{conj.}} &    \wedge^* \mathfrak{p} [-\mathcal{C}]  \ar[d]^{\mbox{conjugation}}  \\
\wedge^* \mathfrak{p}[\mathcal{C}] \otimes \wedge^* \mathfrak{a}  \ar[r]     &   \wedge^* \mathfrak{p}[\mathcal{C}] }
\end{equation}
where the conjugation on $\mathfrak{a}^*$ is that which fixes $\mathfrak{a}_{\R}^*$. 

This gives rise to \eqref{conj diagram} -- however, just as in the previous Lemma, 
 the identifications of $\mathfrak{a}$ with $\aG$ induced by $(\mathfrak{a}, \mathfrak{b}, \mathfrak{q})$
and $(\mathfrak{a}, \mathfrak{b}, \mathfrak{q}^{\op})$ again differ by the long Weyl element,
and so in \eqref{conj diagram} we take the conjugation on $\aG$ as being with reference to the {\em twisted} real structure. \qed

\subsubsection{Construction for general $\GG_{\R}$}
 
Let $\GG_{\R}$ now be an arbitrary reductive group over $\R$.

 Let $\GG'$ be the simply connected cover of the derived group of $\GG_{\R}$,
and let $\ZZ_G$ be the center of $\GG_{\R}$.   Thus there is an isogeny $\GG' \times \ZZ_G \rightarrow \GG_{\R}$. Let $\mathfrak{g}', \mathfrak{k}', \mathfrak{a}_{G'}$
be the various Lie algebras for $\GG'$.  Let $\mathfrak{a}_{Z}$
be the $\mathfrak{a}$-space for $\ZZ_G$; it is naturally identified
with the Lie algebra of a maximal split subtorus. We have 
$$ \mathfrak{a}_{G} = \mathfrak{a}_{G'} \oplus \mathfrak{a}_Z.$$

 For any representation
$\Pi$ of $G$  let $\Pi'$ be its pullback to $\GG'$; this is a finite length tempered representation.   There is a natural identification
$$H^*(\gKZ; \Pi) = \wedge^* \mathfrak{a}_Z \otimes   H^*(\mathfrak{g}',  \Kinfty';  \Pi')$$
Our  foregoing discussion 
has given  an action of $\wedge^* \mathfrak{a}_{G'}$ on the second factor;
and so we get an action of 
 $$\wedge^* \mathfrak{a}_{G'} \otimes \wedge^* \mathfrak{a}_Z = \wedge^* (\mathfrak{a}_{G'} \oplus \mathfrak{a}_{Z})^* = \wedge^* \aGs$$
on $H^*(\gKZ; \Pi)$.   Lemma \ref{adjointness} and Lemma \ref{conjugation} continue to hold in this setting.

Observe that the group $ \Kinfty/\Kinfty^0 \stackrel{\sim}{\rightarrow} \pi_0 \mathbf{G}(\R)$
acts naturally on   $H^*(\gKZ; \Pi)$.    By the discussion of \S \ref{autersec}, this action 
 of $\aGs$ will commute with the action of $\Kinfty/\Kinfty^0$. 

  \subsection{Metrization}\label{motmetricss} \  As remarked near \eqref{motmetric} it is very convenient to put a Euclidean metric on $\aGs$ in such a way
  that the induced action on cohomology is isometric.

 Let the bilinear form $B_{\R}$ be as in \eqref{Bdisgustion}. 
With notation as in  \S \ref{FCA1}, 
  $B_{\R}$ induces a invariant quadratic form on $\mathfrak{a}_{\R} \oplus \mathfrak{b}_{\R}$,
   so also on $\mathfrak{a}_{\R}$ and $\mathfrak{a}_{\R}^*$. In particular, we get a $\C$-valued positive definite hermitian form on $\mathfrak{a}_G^*$.
Then:

   \begin{lemma} \label{isometry property}  
 Let $X \in \wedge^* \aGs$.  Let $\Pi$ be a finite length   cohomological tempered representation. Let $T \in H^q(\gKZ, \Pi)$, where $q$ is the minimal cohomological degree as in \eqref{qdef}; 
 equip $H^*(\gKZ, \Pi)$ with the natural hermitian metric (arising from a fixed inner product on $\Pi$, and the bilinear form $B_{\R}$). 
 Then 
 $$ \| T \cdot X \|  =\|T\|  \|X\|.$$
 \end{lemma}
 \proof

   This reduces to the  case where $\GG_{\R}$ is simply connected, and then again to the case when $\Pi = \Pi(\mathcal{C})$ is irreducible. 
  There 
 it reduces to a similar claim about the weight space $(\wedge^* \mathfrak{p}^*)^{-\mu}$, 
 since (with notations as previous) the map
$\Hom(V_{-\mathcal{C}}, \wedge^* \mathfrak{p}^*) \stackrel{\sim}{\rightarrow}  \left( \wedge^* \mathfrak{p}^* \right)^{-\mu}   $  of \eqref{eqqq}
is isometric (up to a constant scalar, which depends on the choice of highest weight vector)  for the natural Hermitian forms on both sides. 
But the  corresponding claim about  $(\wedge^* \mathfrak{p}^*)^{-\mu}$ is clear from  \eqref{igly}, noting that the factors $\mathfrak{a}$ and
 $ (\mathfrak{u}   \oplus \bar{\mathfrak{u}}) \cap \mathfrak{p}$ are orthogonal to one another under $B$.   \qed 
    
    The following explicit computation will be useful later: 
    \begin{lemma} \label{trace-doo-doo}
    Suppose $\G_{\R}$ is one of $ \GL_n,   \Res_{\C/\R} \GL_n$,
    and endow $\mathfrak{g}$ with the invariant quadratic form $B=\tr(X^2) \mbox{ or }\tr_{\C/\R} \tr(X^2)$,
    where $\tr$ is taken with reference to the standard representation. 
    Then, with reference to the identification \eqref{projlimmap},
    the form on $\mathfrak{a}_{G,\R}^*$  induced by the dual of  $B$ is the restriction of the  trace form on $\widehat{\mathfrak{g}}$ (by which we mean
    the sum of the trace forms on the two factors, in the case of $\Res_{\C/\R} \GL_n$).
    A similar result holds when $\GG_{\R}$ is one of $\mathrm{SO}_n$ and $\Res_{\C/\R} \SO_n$, 
    except that the form on $\mathfrak{a}_{G,\R}^*$ is the restriction of $\frac{1}{4} \cdot \mbox{(trace form)}$.     
    \end{lemma}
    \proof 
    Write $\widehat{\tr}$ for the trace form on $\widehat{\mathfrak{g}}$, in each case.  
   As explained in  \eqref{lambdaGa} 
 the choice of $(\mathfrak{a}, \mathfrak{b}, \mathfrak{q})$ induces a natural perfect pairing of $\C$-vector spaces
    $$   (\mathfrak{a} \oplus \mathfrak{b})     \otimes \underbrace{  \Lie(\widehat{T})}_{\supset \Lie(\widehat{T})^{w_0}}
     \rightarrow \C,$$
     wherein $\Lie(\widehat{T}^{w_0})$ is identified with the dual of $\mathfrak{a}$. 
     We want to show that, under this pairing, the form $\tr|_{\mathfrak{a}}$
   is in duality with the form $\widehat{\tr}|_{\Lie(\widehat{T})^{w_0}}$. Since $\mathfrak{a}$ and $\mathfrak{b}$
   are orthogonal with respect to $\mathrm{tr}$, it is enough to check that 
   the form $\tr$ on $\mathfrak{a} \oplus \mathfrak{b}$ 
    and $\widehat{\tr}$ on $\Lie(\widehat{T})$ are in duality.

It is convenient to discuss this in slightly more generality: Note that, if $\HH$ is a reductive group over $\C$, the choice of a nondegenerate invariant quadratic form $Q$
on $\mathfrak{h} = \mathrm{Lie}(\HH)$ induces a nondegenerate invariant quadratic form $\widehat{Q}$  on the dual Lie algebra $\widehat{\mathfrak{h}}$.
Indeed, choose a torus and Borel   $(\mathbf{T}_H \subset \mathbf{B}_H)$ in $\HH$;
then $Q$ restricts to a Weyl-invariant form on the Lie algebra of $\mathbf{T}_H$, and the identification 
 $$ \Lie(\mathbf{T}_H) \simeq \Lie(\widehat{T})^*,$$
induced by $(\mathbf{T}_H \subset \mathbf{B}_H)$  allows us to transport $Q$ to a Weyl-invariant form on $\Lie(\widehat{T})$. 
 This does not depend on the choice of pair $(\mathbf{T}_H \subset \mathbf{B}_H)$, 
 because of invariance of $Q$.  Finally the resulting Weyl-invariant form on $\Lie(\widehat{T})$ extends uniquely to an invariant form on $\widehat{\mathfrak{h}}$. 

In this language, the question is precisely to compute $\widehat{Q}$,
where $\HH=\GG_{\C}$ and $Q$   is the complexification of $\tr$, i.e. a form on the Lie algebra of $\G_{\C}$.

    \begin{itemize}
    \item[(i)] $\GG_{\R}=\GL_n$.   Here it is clear that $\widehat{\tr} = \tr$. 
    
    \item[(ii)]  $\GG_{\R}=\Res_{\C/\R} \GL_n$.  Here again $\widehat{\tr}=\tr$:
    
     The associated complex group is
$\GL_n \times \GL_n$, and the form there is $\tr(X_1)^2+ \tr(X_2)^2$. 
The dual form on $\mathfrak{gl}_n \times \mathfrak{gl}_n$
 is thus, again, the trace form on $\GL_n \times \GL_n$.

    \item[(iii)] $\GG_{\R} = \SO_{n}$:
    In this case we have
\begin{equation} \label{oink0} \widehat{\tr} = \frac{1}{4} \left( \tr \mbox{ on the dual group.} \right). \end{equation}
    We will analyze the cases of $\mathrm{SO}(2)$ and $\mathrm{SO}(3)$, with the general cases being similar: 
    
    \begin{itemize}
    \item[(a)] Consider $\mathrm{SO}(2)$, which we realize as the stabilizer
    of the quadratic form $q(x,y) = xy$.  
    The maximal torus is the image of the generating co-character $\chi:  t \mapsto \left( \begin{array}{cc} t & 0 \\0 & t^{-1} \end{array}\right)$,
and (with the standard identifications)  $\langle \chi, \chi' \rangle = 1$ where
$\chi'$ is the co-character $\chi:  t \mapsto \left( \begin{array}{cc} t & 0 \\0 & t^{-1} \end{array}\right)$
of the dual $\mathrm{SO}(2)$.   Denote simply by $d\chi$ the image
of the standard generator of $\Lie(\mathbb{G}_m)$ under $\chi$. Then $\langle d\chi, d\chi\rangle_{\tr} =2$, 
and so $\langle d\chi', d\chi'  \rangle_{\widehat{\tr}}= \frac{1}{2}$. 
  
\item[(b)]  Consider $\mathrm{SO}(3)$, which we realize as the stabilizer of  the quadratic form $q(x,y) = xy+z^2$,
and a maximal torus is  the image of the co-character $\chi: \diag(t, t^{-1}, 1)$. 
This is dual to the same character $\chi'$ as above (now considered
as a character of $\SL_2$). 
 We reason just as in (a). 
\end{itemize}

\item[(iv)] $\Res_{\C/\R} \SO_n$.  Here  again \eqref{oink0} holds.
 To see this, note that the associated complex group is $\SO_n \times \SO_n$, and the form
there is given by $\tr(X_1^2) + \tr(X_2^2)$; then the result follows from (iii). 
\end{itemize}

 \qed

\newcommand{\cG}{{}^c G} 
\section{The motive of a cohomological automorphic representation: conjectures and descent of the coefficient field} \label{Pimotive}

 We briefly formulate a version of the standard conjectures relating cohomological automorphic forms and motives, taking some care about coefficient fields. 
 A more systematic discussion of the general conjectures is presented in Appendix \ref{Pimotive}; for the moment we present only what is needed for  
 the main text. 

 \subsection{The example of a fake elliptic curve} \label{sec: fake fake}
To recall why some care is necessary, let us consider the example of  
 a {\it fake elliptic curve} over a number field $F$: this is, by definition,
 an abelian surface $A$  over 
 $F$ which admits
 an action of an (indefinite) quaternion algebra $D \hookrightarrow \End_F(A) \otimes \Q$. 
 
 In any realization $H^1 (A)$ admits a natural right $D$-action, and thus,  for any rational prime $\ell$, 
 one gets a Galois representation
 $$
\rho_{A,\ell}: \Gal(\Qbar/F) \rightarrow \GL_{D} (H^1(A_{\Qbar}, \Q_\ell)) \simeq (D \otimes \Q_\ell)^\times,
 $$ 
 where the latter identification depends on a choice of a basis for $H^1(A_{\Qbar}, \Q_\ell)$ over 
 $D \otimes \Q_\ell$. 
 If $\ell$ is not ramified in $D$, a choice of splitting $D \otimes \Q_{\ell} \simeq M_2(\Q_{\ell})$
converts this to a genuine two-dimensional representation
$$ \rho_{A, \ell}:  \Gal(\Qbar/F) \longrightarrow \GL_2(\Q_{\ell})$$

This is expected to correspond to 
 an automorphic form $\pi$ on $\PGL_2(\A_F)$ with Hecke eigenvalues in $\Q$, characterized by the fact that we have an equality
 $$
\tr (\rho_{A,\ell} (\Frob_v) )= a_v (\pi)
$$
for all but finitely many $v$; here $\tr$ denotes the   trace and $a_v (\pi)$ is the Hecke eigenvalue of $\pi$ at $v$.

The correspondence between $\pi$ and $A$, in this case,  has two deficiencies. The first is that the dual group of 
$\PGL_2$ is $\SL_2$ but the target of the Galois representation is $(D \otimes \Q_\ell)^\times$. 
The second is that the automorphic form $\pi$ has $\Q$ coefficients; but
there is no natural way, in general,  to squeeze a  
   motive of rank two with $\Q$-coefficients out of $A$. One could get a rank two motive after 
  extending coefficients to some splitting field of $D$, but this is somewhat unsatisfactory. 
  
  However, although one cannot directly construct a rank two motive attached to $\rho_{A, \ell}$,
  it is possible to construct a rank three motive that is attached to the composition $\Ad \rho_{A, \ell}$
  with the adjoint representation $\PGL_2 \rightarrow  \GL_3$.  Namely, construct the motive 
  \begin{equation}
M = \End^0_{D} (h^1 (A)),
\end{equation}
where $\End_{D}$ denotes endomorphisms that commute with the natural (right) $D$-action on $h^1(A)$ and 
the superscript $0$ denotes endomorphisms with trace zero. 
This is a motive over $F$ of rank three with $\Q$-coefficients, which can be explicitly realized 
  as a sub-motive of $h^1(A) \otimes h^1(A)^*$.

Write $\fg$ for the Lie algebra of $\SL_1(D)$; this is a three-dimensional Lie algebra over $\Q$, and is  an inner form of $\mathfrak{sl}_2$.
We have natural conjugacy classes of identifications  
$$H^1_{\et}(M, \Q_{\ell}) \simeq \fg \otimes \Q_{\ell},$$
$$ H^1_{\B}(M_{v,\C},\Q) \simeq \fg ,$$
 for any infinite place $v$ of $F$. 

We expect that this phenomenon is quite general. 
Below we formulate, in a general setting, the properties that such an ``adjoint motive'' $M$
attached to a cohomological automorphic form should have.

\subsection{The  conjectures} \label{conj-sec}
It will be useful to formulate our conjectures over a general number field; thus let $F$ be   a number field, let $\GG_F$ be a reductive group over $F$, and 
let $\pi$ be  an automorphic cohomological tempered representation for $\GG_F$. 
(Recall from \S \ref{sec:notn} that ``cohomological,'' for this paper, means cohomological with reference to the trivial local system.) 
The definitions that follow will depend only on the near-equivalence class of $\pi$.

We suppose that $\pi$ has coefficient field equal to $\Q$, i.e.,  the representation $\pi_v$ has a $\Q$-structure for almost all $v$. 
One can attach to $\pi$ the associated archimedean parameter
$$W_{F_v} \longrightarrow \LG,$$
for any archimedean place $v$. 
The Langlands program also  predicts that $\pi$ should   give rise to   a Galois representation valued in a slight modification of $\LG$ (see \cite{BG}). 
Composing these  representation with the adjoint representation of the dual group on its  Lie algebra $\widehat{\mathfrak{g}}$, we arrive at
representations
\begin{equation} \label{rhol}  \Ad \rho_{\ell}: G_F \longrightarrow \Aut(\widehat{\mathfrak{g}}_{\overline{\Q_{\ell}}}).\end{equation}
\begin{equation} \label{rhov}  \Ad \rho_v: W_{F_v} \longrightarrow \Aut(\widehat{\mathfrak{g}}_{\C}).\end{equation}
of the Galois group and each archimedean Weil group. 
 With these representations in hand,  we can formulate the appropriate notion of ``adjoint motive attached to (the near equivalence class of) $\pi$,'' namely,

 \begin{definition} \label{AdjointMotiveDefinition}
An {\em adjoint motive associated to $\pi$} is a weight zero Grothendieck motive $M$ over $F$ with $\Q$ coefficients,  equipped with an  injection of $\Q$-vector spaces
$$\iota_v:  H_B(M_{v,\C},\Q) \longrightarrow  \widehat{\mathfrak{g}}_{\overline{\Q}}$$
for every infinite place $v$, 
such that:
\begin{quote}The image of $H_B(M_{v,\C},\Q)$ is the fixed set of an inner twisting of the standard Galois action on $\widehat{\mathfrak{g}}_{\overline{\Q}}$.  Said differently,
  $\iota_v$ identifies $H_{\B}(M_{v,\C})$ with an inner form  $\widehat{\mathfrak{g}}_{\Q,*}$ of $\mathfrak{g}$:
\begin{equation} \label{iotav} \iota_v: H_{\B}(M_{v,\C},\Q) \stackrel{\sim}{\longrightarrow} \widehat{\mathfrak{g}}_{\Q, *} \subset \widehat{\mathfrak{g}}_{\overline{\Q}}.\end{equation}
 \end{quote}
(This inner form may depend on $v$.)
Moreover, for any such $v$, and for any rational prime $\ell$, we require:
   \begin{itemize}

\item[1.]  The isomorphism   \begin{equation} \label{etmap} H_{\et}(M_{\bar{F}}, \overline{\Q}_{\ell}) \simeq H_B(M_{v,\C},\Q) \otimes_{\Q}  \overline{\Q_{\ell}} \stackrel{\iota_v}{\rightarrow} \widehat{\mathfrak{g}}_{\Q,*} \otimes \overline{\Q_{\ell}}
\simeq  \widehat{\mathfrak{g}}_{\overline{\Q_{\ell}}} 
\end{equation}
 identifies the Galois action on the {\'e}tale cohomology of $M$ with  a representation  in the conjugacy class of $\Ad \rho_{\ell}$ (see \eqref{rhol}).  
 \item[2.]  The isomorphism
  \begin{equation} \label{drmap} H_{\dR} (M) \otimes_\Q   \C   \simeq  H_B(M_{v,\C},\Q) \otimes_{\Q} \C \stackrel{\iota_v}{\rightarrow}
 \widehat{\mathfrak{g}}_{\Q,*} \otimes \C \simeq  \widehat{\mathfrak{g}}_{\C}  \end{equation}
identifies the  action of the Weil group $W_{F_v}$ on the de Rham cohomology of $M$ with a representation 
in the conjugacy class of $\Ad \rho_v$ (see \eqref{rhov}). 
   
 \item[3.]   Each  $\Q$-valued  bilinear form on $\widehat{\mathfrak{g}}_{\Q}$, invariant by the action of $\LG_{\Q}$\footnote{Explicitly, this means it is invariant both by inner automorphisms of $\GG$
 and by the pinned outer automorphisms arising from the Galois action on the root datum.},  induces a  weak polarization $M \times M \rightarrow \Q$ with the property that, for each $v$,    its Betti realization $H_B(M_{v,\C}) \times H_B(M_{v,\C}) \rightarrow \Q$ is identified, under $\iota_v$, with the given bilinear form.\footnote{Observe that a $\Q$-valued invariant bilinear form on $\mathfrak{g}_{\Q}$
 induces also a $\Q$-valued bilinear form on $\mathfrak{g}_{\Q, *}$, characterized by the fact that their linear extensions to $\mathfrak{g}_{\overline{\Q}}$ agree.}
  \end{itemize} 
\end{definition}

We are not entirely sure if {\em every} cohomological $\pi$ should have an attached adjoint motive, because of some slight subtleties about
descent of the coefficient field from $\overline{\Q}$ to $\Q$. However, it seems very likely that the overwhelming majority
should admit such attached adjoint motives, and we will analyze our conjectures carefully
only in this case.  (One can handle the general case at the cost of a slight loss of precision, simply extending coefficients from $\Q$ to a large enough number field.)  

 In Appendix \ref{Pimotive Appendix} we explain more carefully what the correct conjectures
for motives attached to automorphic representations should look like and why, if we suppose that the Galois representation has centralizer that is as small as possible, 
 these conjectures imply the existence of an adjoint motive associated to $\pi$.  

 \section{Formulation of the main conjecture}
Here  we combine the ideas of the prior two sections to precisely formulate the main conjecture. 
 (We have already formulated it in the introduction, but we take the opportunity to write out a version with the assumptions
 and conjectures  identified as clearly as possible.)

We briefly summarize  our setup.   We return to the setting of \S \ref{sec:numerical_invariants} 
so that $\GG$ is a reductive $\Q$-group without central split torus.
 Now let $\mathcal{H}$ be the Hecke algebra for $K$ at good places, i.e.,
the tensor product of local Hecke algebras at places $v$ at which $K$ is hyperspecial.  We fix a character $\chi: \mathcal{H} \rightarrow \Q$, and 
 let $$\Pi = \{\pi_1, \dots, \pi_r\}$$   be the associated set of cohomological automorphic representations which contribute to cohomology at level $K$,  defined more precisely as in \S  \ref{sec:numerical_invariants}. 
The set  $\Pi$ determines $\chi$
and  we suppress  mention of $\chi$ from our notation. 

Just as in our introductory discussion in
 \S \ref{sec:numerical_invariants}  
we make the following 
 \begin{quote}{\em Assumption:} 
 Every $\pi_i$
 is cuspidal and  tempered at $\infty$.  \end{quote}
and define 
\begin{equation} \label{Pipart} H^*(Y(K), \Q)_{\Pi} = \{ \alpha \in H^*(Y(K), \Q): T \alpha = \chi(T) \alpha  \mbox{ for all $T \in \mathcal{H}$}\} 
\end{equation}
and similarly $H^*(Y(K), \C)_{\Pi}$, etc. 

  Let $\Ad \Pi$  
 be an adjoint motive associated to $\Pi$, in the sense of Definition \ref{AdjointMotiveDefinition}.
We have attached to  $\GG$ a canonical $\C$-vector space $\aGs$ in \S \ref{VoganZuckerman}. 
Also $\aGs$  comes with a real structure, the ``twisted real structure'' of 
Definition \ref{twistedReal}.

We shall first explain (\S \ref{BRag})
why the Beilinson regulator on the motivic cohomology of $\Ad \Pi$, with $\Q(1)$ coefficients, 
takes values in (a space canonically identified with) $\aGs$, and indeed in the twisted real structure on this space. 
Then,  after a brief review of cohomological automorphic representations (\S \ref{CAFreview}) 
we will be able to define an action of $\aGs$ on the cohomology  $H^*(Y(K), \C)_{\Pi}$  and 
then we formulate precisely our conjecture in \S \ref{mainconjsec}.
  Finally, Proposition \ref{basic properties} verifies various basic properties about the action of $\aGs$ (e.g., it is self-adjoint relative to Poincar{\'e} duality and it preserves real structures). 
 
 \subsection{The Beilinson regulator.} \label{BRag}   
 The motive $\Ad \Pi$ has weight zero.  The Beilinson regulator gives 
\begin{equation} \label{BRegaG} H^1_{\m}(\Ad \Pi, \Q(1)) \stackrel{\eqref{WRfixed}}{\rightarrow} H_{\B}((\Ad \Pi)_\C, \R)^{W_{\R}} \hookrightarrow H_{\B}((\Ad \Pi)_\C, \C)^{W_{\R}}\stackrel{ \eqref{drmap}}{\longrightarrow}  \widehat{\mathfrak{g}}^{W_{\R}} \stackrel{\eqref{aGtemp}}{\longrightarrow} \aGs, \end{equation}
where the last two arrows are isomorphisms of complex vector spaces.
  Proceeding similarly for the dual motive, we get a   map 
 \begin{equation} \label{BRegaGdual} H^1_{\m}(\Ad^* \Pi, \Q(1)) \rightarrow \mathfrak{a}_G, \end{equation}  and, just
 as in the introduction, we call $L$ the image of \eqref{BRegaGdual}; thus if
 we accept Beilinson's conjecture, 
  $L$ is a $\Q$-structure on $\mathfrak{a}_G$.

We want to understand how \eqref{BRegaG} interacts with the real structure on $\aGs$. Recall 
that we have defined a second ``twisted'' real structure on $\aGs$, in Definition \ref{twistedReal}.

 \begin{lemma} \label{Beilinson regulator real structure}
 The map 
  $H_{\B}(\Ad \Pi)_{\C},  \R)^{W_{\R}} \rightarrow \aGs$  has image equal to the twisted real structure on $\aGs$.
 In particular, the Beilinson regulator  carries $H^1_{\m}(\Ad \Pi, \Q(1))$ into the twisted real structure on $\aGs$.  
 \end{lemma}
 
\proof 
We may as well suppose that  \eqref{drmap}
 identifies the $W_{\R}$-action with the action  $\rho_0: W_{\R} \rightarrow \Aut(\widehat{\mathfrak{g}})$
arising from $\rho_0$ normalized as in \eqref{rho0def}.  
Also, \eqref{drmap}  allows us to think of the ``Betti'' conjugation    $c_{\B}$ on $H_{\dR}(\Ad \Pi) \otimes \C = H_{\B}((\Ad \Pi)_{\C}) \otimes \C$ as acting 
 on $\widehat{\mathfrak{g}}$.  
 From \eqref{iotav} the fixed points of $c_{\B}$ are given
by   $\widehat{\mathfrak{g}}_{\Q, *} \otimes \R$ and so $c_{\B}$ is an inner twist of the standard antiholomorphic involution. (By ``standard antiholomorphic involution'' we mean the involution
of $\widehat{\mathfrak{g}}$ with respect to the Chevalley real form.)
 Since $\rho_0(S^1)$ preserves real Betti cohomology, 
$c_{\B}$ commutes with $\rho_0(S^1)$.

 Define an antilinear self-map $\iota$ on $\widehat{\mathfrak{g}}$ via
 $$ \iota(X) = \Ad(w_{\widehat{G}}) \overline{X},$$
 where $\overline{X}$ refers to the standard antilinear conjugation, and $w_{\widehat{G}}$
is an element of $\widehat{G}(\R)$ that normalizes $\widehat{T}$ and takes $\widehat{B}$ to $\widehat{B}^{\mathrm{op}}$. 
  Then 
$\iota$  also commutes with the action of $\rho_0(S^1)$. 

The composition $\iota c_{\B}$ is now an inner automorphism of $\widehat{\mathfrak{g}}$
which commutes with $\rho_0(S^1)$ and thus is given by conjugation by an element of $\widehat{T}$. Thus $\iota$ and $c_{\B}$
act in the same way on the Lie algebra $\widehat{\mathfrak{t}}$ of $\widehat{T}$. 

The image of $H_{\B}((\Ad \Pi)_{\C}, \R)^{W_{\R}} \stackrel{\eqref{drmap}}{\longrightarrow}  \widehat{\mathfrak{g}}^{W_{\R}} $
is just the fixed points of $c_{\B}$.    However, we have just seen that $c_{\B}$ and $\iota$ act the same way on $\widehat{\mathfrak{g}}^{W_{\R}} \subset \widehat{\mathfrak{t}}$. 
The fixed points of $\iota$ on $\widehat{\mathfrak{g}}^{W_{\R}} \simeq \aGs$  give (by Lemma \ref{realstruc} and Definition \ref{twistedReal})
the twisted real structure.  \qed

 \subsection{Trace forms}    \label{traceforms} 
  Endow $\widehat{\mathfrak{g}}_{\Q}$ with any nondegenerate $\LG_{\Q}$-invariant $\Q$-valued quadratic form  $\widehat{B}$;   it gives by scalar extension a complex valued  quadratic form on $\widehat{\mathfrak{g}}$.
  The pullback of this form under
 $$H_{\B}((\Ad \Pi)_{\C}, \Q) \simeq \widehat{\mathfrak{g}}_*  $$
 defines (part (3) of Definition  \ref{AdjointMotiveDefinition}) a weak polarization $Q$ on $\Ad \Pi$: since $\widehat{\mathfrak{g}}_*$
 is an inner form, the restriction of $\widehat{B}$ is actually $\Q$-valued on it.

We may form the corresponding Hermitian form $Q(x, c_{\B} y)$ on   
 $H_{\B}((\Ad \Pi)_{\C}, \C)$; when restricted to 
 $H_{\B}((\Ad \Pi)_{\C}, \C)^{W_{\R}} \simeq \aGs$, 
this is given by
\begin{equation} \label{Form0} (X, Y) \in \aGs \times \aGs \mapsto   \widehat{B}(X, \overline{\Ad(w_{\widehat{G}}) Y}),\end{equation} 
where the conjugation is that with reference to $\mathfrak{g}_{\R}$,
and $w_{\widehat{G}}$ is as in Lemma \ref{Beilinson regulator real structure}.

This form is real-valued  when restricted to the twisted real structure, since (writing just $wX$ for $\Ad(w) X$, etc.): 
\begin{equation} \label{realvalued}  \overline{ \widehat{B}(X,  \overline{w_{\widehat{G}}Y}) } = \widehat{B}(\overline{X},  w _{\widehat{G}} Y)  =  \widehat{B}(\overline{w_{\widehat{G}}^{-1} X}, Y) =  \widehat{B}(\overline{w_{\widehat{G}} X},  Y).\end{equation}
and $\overline{w_{\widehat{G}}X} = X, \overline{w_{\widehat{G}} Y} =Y$ on the twisted real structure. 

We warn the reader that, although real-valued,  the form \eqref{Form0} {\em need not} be positive definite on the twisted real structure. This corresponds to the 
fact that the form $\widehat{B}$ gives a weak polarization on $\Ad \Pi$ but not necessarily a polarization.

\subsection{Review of cohomological automorphic forms} \label{CAFreview}

   For any cohomological automorphic representation $\pi$ for $\mathbf{G}$,  denote by $\Omega$ the natural map   
   \begin{equation} \label{omegadef} \Omega: \Hom_{\Kinfty^{\circ}}(\wedge^p \mathfrak{g}/\mathfrak{k}, \pi^K) \rightarrow   \underbrace{\mbox{ $p$-forms on $Y(K)$}}_{\Omega^p(Y(K))}\end{equation}

 Indeed, $\Omega^p(Y(K))$ can be considered as functions on
$\G(F) \backslash \left(  \G(\adele) \times \wedge^p \mathfrak{g}/\mathfrak{k} \right) / \Kinfty^{\circ}  K$ that
are linear on each $\wedge^p \mathfrak{g}/\mathfrak{k}$-fiber. 
Explicitly, for $X \in \mathfrak{g}/\mathfrak{k}$ and $g \in \G(\adele)$, 
we can produce  a tangent vector $[g,X]$ to $\G(F) g \Kinfty^\circ K \in Y(K)$ --
namely, the derivative of the curve $\G(F) g e^{tX} \Kinfty^\circ K$ at $t=0$.
  This construction extends
 to $X_1 \wedge \dots \wedge X_p \in \wedge^p \mathfrak{g}/\mathfrak{k}$ by setting
 \[
 [g,X_1 \wedge \cdots \wedge X_p]=[g,X_1] \wedge \cdots \wedge [g,X_p],
 \]
 which belongs to the $p$th exterior power of the tangent space at the point $gK$. 
The map $\Omega$ is normalized by the requirement
that, for $f \in  \Hom_{\Kinfty^\circ}(\wedge^p \mathfrak{g}/\mathfrak{k}, \pi^K) $ and $X_i \in   \mathfrak{g}/\mathfrak{k}$, we have
\begin{equation} \label{OmegaVeryExplicit}
\Omega(f) ( [g, X_1 \wedge \cdots \wedge X_p])  = f(X_1 \wedge \cdots \wedge X_p) (g).
\end{equation}

The map $\Omega$ defines a map on cohomology 
\begin{equation} \label{Omegabal}   \Omega: H^p(\gKZ;  \pi^K) \rightarrow H^*(Y(K), \C) \end{equation}
This map is injective if $Y(K)$ is compact, or if $\pi$ is cuspidal, by \cite[5.5]{Borel2}; in particular, 
if we have fixed a Hermitian metric on $\mathfrak{g}/\mathfrak{k}$ we also get a Hermitian metric on the image,
by taking $L^2$-norms of differential forms. 

We will work only with unitary cohomological $\pi$; in that case, 
all the differentials in the complex computing $H^*(\gKZ;  \pi)$  vanish (Kuga) and we have
$H^*(\gKZ;   \pi) =  \Hom_{\Kinfty^0} (\wedge^* \mathfrak{g}/ \mathfrak{k}, \pi)$. 
We will freely make use of this identification.  It allows us to put a metric
on $H^*(\gKZ; \pi)$ for which \eqref{Omegabal} is isometric. 

Moreover, this story is compatible, in the natural way, with complex conjugation: if $T \in H^p(\gKZ;  \pi^K)$, we have $\Omega(\bar{T}) = \overline{\Omega(T)}$, 
where $\bar{T} \in H^*(\gKZ;  \overline{\pi})$ is defined  so that $\bar{T}(\bar{v}) = \overline{T(v)}$ and the embedding
$\overline{\pi} \hookrightarrow \left(\mbox{ functions on $[\G]$}\right)$ is obtained by conjugating the corresponding embedding for $\pi$.
 If $\pi$ and $\overline{\pi}$ 
are the same (i.e., they coincide as subrepresentations of functions on $[\G]$, and so we have an identification $\pi \simeq \overline{\pi}$)
 we shall say that $T$ is {\em real} if $T = \bar{T}$; 
in that case $\Omega(T)$  is a real differential form and defines a class in $H^*(Y(K), \R)$.

  \subsection{Formulation of main conjecture} \label{mainconjsec}

\medskip

In the setting at hand, the  map 
  $\Omega$ induces  (see \cite{Borel2}) an isomorphism
\begin{equation} \label{OmegaIso} \bigoplus_{i=1}^r  H^*(\gKZ;  \pi_i^K)  \stackrel{\Omega}{\longrightarrow} H^*(Y(K), \C)_{\Pi}.\end{equation}
 We have previously defined (\S \ref{aGcohomology}) an action of $\wedge^* \aGs$ on each $H^*(\gKZ; \pi_{i}^K)$, and we may transfer
 this action via $\Omega$ to get an action of  $\wedge^* \aGs$ on $H^*(Y(K), \C)_{\Pi}$. 
Then:

 \begin{quote}
{\bf Main conjecture: } {\em (motivic classes preserve rational automorphic cohomology).}  Assume part (a) of Beilinson's conjecture (Conjecture \ref{conj:Beilinson}), the assumptions on $\Pi$ formulated above, and the existence
 of an adjoint motive attached to $\Pi$, in the sense of Definition 
 \ref{AdjointMotiveDefinition}; observe that then the image of $H^1_{\m}(\Q, \Ad^* \Pi(1))$
 under \eqref{BRegaGdual}  gives a $\Q$-structure on $\mathfrak{a}_G$. 
 {\em Then the induced $\Q$-structure on $\wedge^* \aGs$ preserves $$H^*(Y(K), \Q)_{\Pi}  \subset H^*(Y(K), \C)_{\Pi}$$ for the action just defined. }
 \end{quote}

 \subsection{Properties of the $\aGs$ action} \label{section:basic properties}

\begin{prop} \label{basic properties}
The action of $\wedge^* \aGs$  on $H^*(Y(K), \C)_{\Pi}$ just defined has the following properties:
\begin{itemize}
\item[(i)]  Fix a bilinear form $B_{\R}$ on $G_{\R}$, as in \S \ref{motmetricss}; it gives rise to a hermitian metric on $\aGs$
and a Riemannian metric on $Y(K)$ by that discussion.   Then if $T \in H^q(Y(K), \C)_{\Pi}$ is in minimal cohomological degree
we have $\|X T \| = \|X \| \|T\|$ for $X \in \wedge^* \aGs$; the  hermitian metric on $H^*(Y(K), \C)_{\Pi}$
is that obtained  by its identification with harmonic forms. 
\item[(ii)]   The action of $\wedge^* \aGs$ on $H^*(Y(K), \C)_{\Pi}$ satisfies the same adjointness property
as that formulated in Lemma \ref{adjointness}, with  respect to the Poincar{\'e} duality pairing. 
\item[(iii)]  Suppose that the character  $\chi$  of the Hecke algebra is real-valued. 
Then the  twisted real structure on $\aGs$ preserves  real
cohomology $H^*(Y(K), \R)_{\Pi}$.
\end{itemize}
 \end{prop}
 
 \proof 
 
The map \eqref{OmegaIso} is isometric, so property (i) is now immediate from Lemma \ref{isometry property}. 

 It will be convenient, just for the remainder of the proof, to  abuse notation and write $\Pi$ for the direct sum $\bigoplus_{i=1}^{r} \pi_i$.

 For property (ii): Regard $\overline{\Pi}$ as embedded in functions on $\GG(\Q) \backslash \GG(\adele)$, by conjugating the elements of $\Pi$.   We  note, first of all, that for $T \in  H^*(\gKZ;  \Pi^K)$ and $T' \in H^*(\gKZ, \overline{\Pi}^K)$
with  $\deg(T) + \deg(T') = \dim(Y(K))$ the pairing
$ \int_{Y(K)}  \Omega(T) \wedge \Omega(T')$  is proportional to the natural pairing
 $ H^*(\gKZ;  \Pi^K) \otimes  H^*(\gKZ, \overline{\Pi}^K) \rightarrow (\det \mathfrak{p})^*$, 
 where we integrate $\Pi$ against $\overline{\Pi}$. (The coefficient of proportionality has to do with choices of measure, and will not matter for us.)
This integration pairing identifies  $\widetilde{\Pi}$ with $\overline{\Pi}$,  thus giving $\widetilde{\Pi}$ an embedding into the space of functions on $[\G]$; 
and so the pairing 
$$ \int_{Y(K)}  \Omega(T) \wedge \Omega(T'), \ T  \in H^*(\gKZ;  \Pi^K), T' \in H^*(\gK, \widetilde{\Pi}^K)$$ 
is proportional to the natural pairing on $H^*(\gKZ;  \Pi^K) \times H^*(\gKZ;  \widetilde{\Pi}^K)$.  Then the conclusion follows from Lemma \ref{adjointness}.

   For (iii)
note that, by the discussion at the end of \S \ref{CAFreview},  the  following diagram commutes
    \begin{equation}
 \xymatrix{
 H^*(\gK, \Pi^K)   \ar[r] \ar[d]^{\mathrm{conjugation}} &  H^*(Y(K) , \C)_{\Pi}  \ar[d]^{\mathrm{conjugation}} \\
H^*(\gK, \overline{\Pi^K})    \ar[r]    & H^*(Y(K), \C)_{\overline{\Pi}}
 }
\end{equation}
  Our claim now follows from  Lemma \ref{conjugation}.  \qed
  
  To conclude, we discuss adjointness a little more. The Langlands parameter of $\widetilde{\Pi}$ is obtained from $\Pi$  
  by composition with the Chevalley involution, which we shall denote by $\mathrm{C}_0$:
  this is a pinned involution  of $\widehat{G}$ that acts, on $\widehat{T}$, as the composition of inversion
  and the long Weyl group element.  
The general conjectures (see Appendix \ref{Pimotive Appendix}) predict that there exists an  identification of motives $\mathfrak{d}: \Ad \Pi \simeq \Ad \widetilde{\Pi}$ 
which fits into a commutative diagram         \begin{equation} \label{FOXP1}
 \xymatrix{
H_{\B}(\Ad \Pi,\C)  \ar[r]^{\qquad \eqref{iotav}}   \ar[d] &     \widehat{\mathfrak{g}} \ar[d]^{\mathrm{C}} \\
H_{\B}(\Ad \widetilde{\Pi},\C) \ar[r]^{\qquad \eqref{iotav}}     &   \widehat{\mathfrak{g}}.
}\end{equation}
where $\mathrm{C}$ is the composite of $\mathrm{C}_0$ with an inner automorphism. 

\begin{lemma} \label{adjointness2} 
Assume part (a) of Beilinson's  Conjecture \ref{conj:Beilinson}, 
 and moreover that there is an identification $\mathfrak{d}: \Ad \Pi \simeq \Ad \widetilde{\Pi}$ of motives 
which fits into the diagram \eqref{FOXP1}. 

The actions of $H^1_{\m}(\Ad^* \Pi, \Q(1))^{\vee}$ on 
$H^*(Y(K),\C)_{\Pi}$, induced by  \eqref{BRegaGdual},
and the similar action of 
 $H^1_{\m}(\Ad^* \widetilde{\Pi}, \Q(1))^{\vee}$
on $H^*(Y(K),\C)_{\widetilde{\Pi}}$, are adjoint to one another, up to sign, with respect to the Poincar{\'e} duality pairing
and the identification   of motivic cohomologies induced by $\mathfrak{d}$:

$$ \langle f_1 \cdot X, f_2 \rangle = - \langle f_1,  \mathfrak{d}(X) f_2 \rangle, \ \ X \in H^1_{\m}(\Ad^* \Pi, \Q(1))^{\vee}. $$

\end{lemma}
\proof  Conjugating the horizontal arrows  in \eqref{FOXP1}  we may suppose that the induced actions of $W_{\R}$ on $\widehat{\mathfrak{g}}$, top and bottom, 
both arise from the maps $\rho_0$ normalized as in \eqref{rho0def}; since $\mathrm{C}$ intertwines
these, it must be a conjugate of $\mathrm{C}_0$ by $\widehat{T}$. 

Thus we get:  
         \begin{equation}
 \xymatrix{
H^1_{\m}(\Ad^* \Pi, \Q(1)) \ar[r]    \ar[d]^{\mathfrak{d}} &     \widetilde{\mathfrak{g}}^{W_{\R}} \ar[d]^{\mathrm{C}} \ar[r]^{\eqref{frofro2}} &  \aG \ar[d]^{-w}  \\
H^1_{\m}(\Ad^* \widetilde{\Pi}, \Q(1))) \ar[r]     &   \widetilde{\mathfrak{g}}^{W_{\R}} \ar[r]^{\eqref{frofro2}}  &  \aG.
}\end{equation}
where  $w$ is the long Weyl element on $\aGs$, and we used Lemma \ref{realstruc} (or the same statements transposed to the dual Lie algebra).  Our conclusion now follows from the prior adjointness results (part (ii) of the Proposition). 
 \qed

This discussion has also shown:
\begin{lemma}
If  $\Pi \simeq \tilde{\Pi}$, 
then the image of $H^1_{\m}(\Ad^* \Pi, \Q(1))$ inside $\aG$
is stable by $w$. 
\end{lemma}

\newcommand{\LieGZ}{\mathfrak{g}_{\Z}}
\newcommand{\LieGQ}{\mathfrak{g}_{\Q}}
 \newcommand{\cmpct}{\mathrm{cmpt}}
 \newcommand{\Uinfty}{\mathrm{U}_{\infty}}
 \newcommand{\cpct}{\mathrm{cmpt}}
 \newcommand{\che}{\mathrm{Chev}}
 \newcommand{\proj}{\mathrm{proj}}

\section{Period integrals}  \label{mainsec:periodintegrals} 

In this section, we briefly recall the relation between cohomological period integrals  
and the type of period integrals that are more standard in the theory of automorphic forms. 
More precisely, we will study the integrals of cohomology classes on  $Y(K)$
over the sub-locally symmetric space $Z(U)$   defined by a $\Q$-subgroup $\HH \subset \GG$.

The mathematical content of the section  simply amounts to being careful about factors of $\pi$, normalizations of metrics, volumes, and so forth.  Similar results have been derived by several other authors
 in related contexts; for example, see \cite[\S 3]{Schmidt}.

 The pairs $(\GG, \HH)$  that we study are  a subset of those arising from the Gan-Gross-Prasad conjecture; we specify them in
\S \ref{Ggp-precise}. 
 There is no reason not to consider other examples of periods, but these are convenient for several reasons:
  \begin{itemize}
 \item It is an easily accessible source of examples, but sufficiently broad to involve various classical groups; 
\item There are uniform conjectural statements (after Ichino--Ikeda);
 \item Although we invoke simply the uniform conjectural statements, there are in fact many partial results towards them known. \footnote{For example, in the $\PGL$ cases it seems that all the hypotheses of \S \ref{sec:periodintegrals} are known except (iv), the exact
 evaluation of archimedean integrals on the cohomological vector.}
   \end{itemize}

Recall our notation $A \sim B$
whenever $A/B \in \mathbf{Q}^*$.

 \subsection{}  \label{sec:GH} 
Let $\HH \subset \GG$ be a reductive $\Q$-subgroup, and let $\Pi$
be as in \S \ref{sec:numerical_invariants}: a (near-equivalence class of) cohomological automorphic representation(s) for $\GG$ at level $K$,  satisfying 
the assumptions formulated there.  In particular, we may define,
as in  \eqref{Pipart},  the $\Pi$-subspace $H^*(Y(K), \Q)_{\Pi} \subset H^*(Y(K), \Q)$ of rational cohomology. 

 We write $\Hinfty, \Ginfty$ for the $\R$-points, $\Kinfty$ for a maximal compact subgroup of $\Ginfty$ 
 and $\Uinfty$ for a maximal compact subgroup of $ \Hinfty$. 
We write (e.g.) $d_H$ for the dimension of $\Hinfty$ and $r_H$ for its rank
 (for us this means always the geometric rank, i.e.  the rank of the $\C$-algebraic group $\HH_{\C}$). 
  We introduce notation for the various Lie algebras: 
 $$\mathfrak{g} = \mathrm{Lie}(\mathbf{G}_{\C}), \  \mathfrak{k} = \mathrm{Lie}(\Kinfty)_{\C},  \  \mathfrak{p} = \mathfrak{g}/\mathfrak{k},  \ p_G = \dim(\mathfrak{p}), $$
 $$  \mathfrak{h} = \mathrm{Lie}(\mathbf{H}_{\C}), \  \mathfrak{u} = \mathrm{Lie}(\Uinfty)_{\C}, \  \mathfrak{p}_H = \mathfrak{h}/\mathfrak{u}, \  p_H = \dim(\mathfrak{p}_H).$$ 
These are complex vector spaces, but they are all endowed with natural real forms; as before we denote (e.g) by $\mathfrak{h}_{\R}$ the natural real form of $\mathfrak{h}$, and so forth. 
 
Let $U \subset \HH(\Afinite)$ be a compact open subgroup,  and define the analog of $Y(K)$ (see \eqref{YKdef}) but with $\GG$ replaced by $\HH$ and $K$ replaced by $U$:
 $$ Z(U) = \HH(F) \backslash \HH(\adele) / \Uinfty^\circ U.$$
 
 Fixing an $\Hinfty$-invariant orientation on $\Hinfty/\Uinfty^{\circ}$,
 we get an $\HH(\adele)$-invariant orientation of $\HH(\adele)/\Uinfty^{\circ}U$ and thus
 an orientation of $Z(U)$. (If $Z(U)$ is an orbifold, choose a deeper level structure $U' \subset U$
 such that $Z(U')$ is a manifold; then $Z(U')$  admits a $U/U'$-invariant orientation.)
This discussion gives a fundamental class 
 $$[Z(U)] \in H_{p_H}^{\bm}(Z(U), \Q)$$
 where we work with $\Q$ coefficients, rather than $\Z$ coefficients, to take into account the possibility of orbifold structure.  

  Let $g = (g_{\infty}, g_f) \in \G(\adele) = \G(\R) \times \G(\adele_f)$ be such that   
    $\Ad(g^{-1}) \Uinfty^\circ U \subset    \Kinfty^\circ K $.      Then the map induced by right multiplication by $g$, call it 
\begin{equation} \label{iota_defn}\iota: Z(U)  \stackrel{\times g}{\longrightarrow} Y(K),\end{equation}
is a proper map;  the image of $Z(U)$ is a  $p_H$-dimensional cycle on $Y(K)$
and defines a Borel--Moore homology class 
 $$ \iota_* [Z(U)] \in   H_{p_H}^{\bm}(Y(K), \Q).$$ 
Our goal will be to compute the pairing of this with
classes in $H^*(Y(K), \Q)_{\Pi}$, and interpret the result in terms of ``automorphic periods.'' 

 \begin{rem}  \label{Cusp non cusp} Now the class $\iota_* [Z(U)]$ can only be paired with compactly supported classes,   
but since the automorphic representations we deal with are all {\em cuspidal}, the associated cohomology classes  
lift, in a canonical way, to  compactly supported cohomology, by \cite[Theorem 5.2]{Borel2}; 
if $\omega$ is a cuspidal harmonic form, the integral of $\iota^* \omega$ over $Z(U)$
coincides with the pairing of this compactly supported class with $\iota_* [Z(U)]$.
In other words,  in the setting of \S \ref{sec:numerical_invariants}, 
the map $$H^*_c(Y(K), \C) \rightarrow H^*(Y(K), \C)$$
induces an isomorphism when localized at the ideal of the Hecke algebra corresponding to $\Pi$. 
   In what follows we will then pair $\iota_* [Z(U)]$ with cuspidal cohomology classes without further comment. \end{rem} 

 \subsection{} \label{Ggp-precise}

We will study the  following cases:  \begin{itemize}
 
 \item[1.]   Let $E = \Q(\sqrt{-D_E})$  be an imaginary quadratic field.    For $(V,q)$ a quadratic space over $E$, with $\dim(V) \geq 2$,  set $(V',q') = (V,q) \oplus (E, x^2)$,
and put
$$ \HH_E= \mathrm{SO}(V) \subset \G_E = \mathrm{SO}(V') \times \mathrm{SO}(V),$$ 
 with respect to the diagonal embedding.  Put $\HH=\Res_{E/\Q} \HH_E, \GG =\Res_{E/\Q}  \G_E$.
  
 \item[2.]   Let $E = \Q(\sqrt{-D_E})$  be an imaginary quadratic field.  
 For $V$ a finite-dimensional $E$-vector space, set $V' =  V \oplus E$
 and put 
 $$ \HH_E= \mathrm{GL}(V) \subset \G_E = \mathrm{PGL}(V') \times \mathrm{PGL}(V),$$ 
Define $\HH, \G$ by restriction of scalars, as before. 

\item[3.] Let $E =\Q$. 
For $V$ a finite-dimensional $\Q$-vector space, set $V' =  V \oplus \Q$
 and put 
 $$ \HH= \mathrm{GL}(V) \subset \G = \mathrm{PGL}(V') \times \mathrm{PGL}(V),$$ 
 (in this case, we set $\HH_E=\HH, \G_E=\G$). 
 \end{itemize}

 These cases correspond to cases of the Gross--Prasad conjecture where 
 the cycle $Z(U)$ has dimension $p_H$ equal to the  {\em minimal  tempered cohomological degree for $Y(K)$, i.e.} \begin{equation} \label{roobar}  p_H =  \frac{1}{2} \left( d_G - d_K - (r_G - r_K) \right) \iff p_G - 2 p_H= r_G-r_K.\end{equation} 
 For example, this dimensional condition rules out the pairs 
 $  U_{p,q} \times U_{p+1, q} \supset U_{p,q}$ except in the case $q=1$, and similarly rules out   $\SO_{p,q} \times \SO_{p+1, q} \supset \SO_{p,q}$
 unless $q=1$; that is why we did not discuss these cases. 
 
   The numerical data in the cases we will consider is presented in Table \ref{inner5}. We shall also need
   the following lemma, which assures us that the archimedean component of $g$  (as defined before \eqref{iota_defn}) is almost determined:
\begin{lemma} \label{BoatyMcBoatFace} In all examples of \S \ref{Ggp-precise}, 
the fixed point set of (the left action of) $\Uinfty$ on $\Ginfty/\Kinfty$ is a single orbit of the centralizer of $\Hinfty$ in $\Ginfty$; in particular, the condition 
  $\Ad(g_{\infty}^{-1}) \Uinfty \subset \Kinfty$ determines $g_{\infty}$ up to right translation by $\Kinfty$
  and left translation by this centralizer.

   \end{lemma}
   
Since right translation of $g_{\infty}$ by $\Kinfty$ does not affect the embedding $Z(U) \rightarrow Y(K)$, this lemma means that we may suppose that  $\Ad(g_{\infty}^{-1}) \Hinfty \subset \Ginfty$
arises from the ``standard'' inclusion of the real group of type $\mathbf{H}$ into the group of type $\mathbf{G}$.
By explicit computations with the standard realizations, we see this inclusion is compatible
with Cartan involutions. In other words, if $\theta$ is the Cartan involution of $\Ginfty$ that fixes $\Kinfty$,
then $\Ad(g_{\infty}^{-1}) \Hinfty$ is stable by $\theta$ and  $\theta$  
induces a Cartan involution of $\Ad(g_{\infty}^{-1}) \Hinfty$, fixing $\Uinfty$.
  \proof   
 In what follows, $\mathrm{O}_n$ and  $\mathrm{U}_n$ 
  mean these compact groups in their standard realizations as stabilizers of the forms $\sum x_i^2$ on $\R^n$
  and $\sum |z_i|^2$ on $\C^n$. The embeddings $\mathrm{O}_n \hookrightarrow \mathrm{O}_{n+1}$ etc. are the standard ones also.
  
  Consider, first, case (3) in the numbering at the start of 
  \S  \ref{Ggp-precise}:
  We must compute the fixed points of $\mathrm{O}_{n-1}$ acting on
  pairs of a scaling class of a positive definite quadratic form on $\R^{n-1}$,
  and a scaling class of a positive definite quadratic form on $\R^n$. There is a unique
  fixed point on scaling classes of positive definite forms on $\R^{n-1}$. Thus, we are left to compute 
the fixed points of $\mathrm{O}_{n-1}$  acting on scaling classes of quadratic forms on $\R^n$:
 A positive definite quadratic form $q$ on $\R^n$ whose scaling class is fixed by $\mathrm{O}_{n-1}$
is actually fixed by $\mathrm{O}_{n-1}$ (it is clearly fixed up to sign, and then definiteness makes it fixed). 
By considering  the action of $-\mathrm{Id} \in \mathrm{O}_{n-1}$ we see that   $q = \sum_{i=1}^{n-1} x_i^2 + (a x_n)^2$.
Such forms constitute a single orbit of the centralizer of $\GL_{n-1}(\R)$ within $\PGL_n(\R)$,  which implies the claimed result.
 
 The remaining cases follow similarly from the computation of the following sets:
 \begin{enumerate}
 \item[Case 2:]
The fixed points $\mathrm{U}_{n-1}$ acting on scaling classes of positive definite Hermitian forms on $\C^n$:

As above, any such form is $\sum_{i=1}^{n-1} |z_i|^2 + a |z_{n}|^2$;
again, these form a single orbit of the centralizer of $\GL_{n-1}(\C)$ within $\PGL_n(\C)$, as desired.

 \item[Case 1:]The fixed points of $\mathrm{SO}_n$ acting on $\mathrm{SO}_{n+1}(\C)/\mathrm{SO}_{n+1}(\R)$.

Suppose $\mathrm{SO}_n \subset g \mathrm{SO}_{n+1} g^{-1}$ for $g \in \mathrm{SO}_{n+1}(\C)$;
then $\mathrm{SO}_n$ fixes the subspace $g \R^{n+1} \subset \C^{n+1}$;
this subspace gives a real structure on $\C^{n+1}$ 
and of course $\sum_{i=1}^{n+1}  x_i^2$ will be positive definite on this subspace. 

For $n \geq 3$, the only $\R$-structures of $\C^{n+1}$
that are fixed by $\mathrm{SO}_n$ are of the form  $ \alpha . \R^n \oplus \beta . \R, \ \ (\alpha, \beta \in  \C^*),$ and moreover if $\sum x_i^2$ is real and positive definite on this space,
this means it is simply the standard structure $\R^{n+1}$.
It follows that $g \R^{n+1} = \R^{n+1}$, and so $g \in \mathrm{SO}_{n+1}$ as desired.   Thus the fixed set mentioned above reduces to a single point.

For $n =2$, there are other real structures fixed by $\mathrm{SO}_n$, namely
$$ \{ x +i \varphi(x): x \in \R^2 \} \oplus \R,$$
where $\varphi \in M_2(\R)$ commutes with $\mathrm{SO}_2$.  However, for
$\sum x_i^2$ to be real-valued on this space we should have $\varphi+\varphi^T= 0$; 
the real structure is therefore  of the form 
$$ \{ (x+i A y, y-i A x): (x,y) \in \R^2 \} \oplus  \R,$$ 
for some $A \in \R$; definiteness of $\sum x_i^2$ means that $A^2 < 1$. This is the image of the standard real structure by the matrix $\frac{1}{\sqrt{1-A^2}}  {\small  \left(\begin{array}{cc} 1 & i A \\ -iA & 1 \end{array}\right)}$,
which lies in $\SO_2(\C) \subset \SO_3(\C)$ and (obviously) centralizes the commutative group $\SO_2(\C)$.

\end{enumerate}

\qed

   \begin{table} 
 \begin{tabular}{|c|c|c|c|c|c|}
\hline \hline
$\Ginfty$ & $ \Hinfty$ & $d_{G/K}$ &  $ d_{H/U}$ &  $d_{G/K}-2  d_{H/U}$  \\ 
\hline
$\PGL_n(\C) \times \PGL_{n+1}(\C)$ & $\GL_n(\C)$ & $ 2n^2+2n-1$  & $n^2$ &  $2n-1$   \\
$\PGL_n(\R) \times \PGL_{n+1}(\R)$ & $\GL_n(\R)$ &   $n^2+2n-1$  & $\frac{n^2+n}{2}$ &  $n-1$    \\
$\SO_n(\C) \times \SO_{n+1}(\C)$ & $\mathrm{SO}_n(\C)$ &  $n^2$ & $\frac{n^2-n}{2}$  & $n$   \\
\hline \hline
\end{tabular}
\caption{The cases of the Gross-Prasad family that we will study}  \label{inner5}
 \end{table}

 \subsection{Tamagawa measure versus Riemannian measure.}  \label{measurenormalizations}
 
On $[\G]$ there are two measures, one arising from Riemannian structure
and one the Tamagawa measure.     Our eventual goal is to compare them. For the moment,
we explain carefully how to construct both of them:

For  Riemannian measure,   we first fix once and for all 
the ``standard'' representation of  
$\mathbf{G}$, or rather of an isogenous group $\mathbf{G}'$. 
Let $\eta: \mathbf{G}' \rightarrow \GL(W)$
be the following $\Q$-rational faithful representation:
in all cases, we take $W$ to be $\Res_{E/\Q} (V' \oplus V)$,
and we take $\mathbf{G}'$ to be the restriction of scalars
of $\SL(V') \times \SL(V)$ in cases (2) and (3), and $\mathbf{G}'  =\mathbf{G}$
in case (1).

Define the form $B$ 
on $\mathfrak{g}_{\Q}$ via
\begin{equation} \label{Bdef} B(X, Y) = \mathrm{trace}( d\eta(X) . d\eta(Y) ).\end{equation}
 This defines a 
 $\GG$-invariant  $\Q$-valued quadratic form  $B$ on the Lie algebra.
 Note that (the real-linear extension of) $B$ is invariant by the Cartan involution $\theta$ on $\mathfrak{g}_{\R}$, 
 by explicit computation.
   Moreover $B$ is nondegenerate and negative definite
 on the associated splitting $\mathfrak{k}_{\R} + i \mathfrak{p}_{\R}$, because the standard representation $\eta$ just introduced carries
 the associated maximal compact of $\mathbf{G}'(\C)$ into a unitary group.  It follows that
$B$ is    negative definite on $\mathfrak{k}_{\R}$ and positive definite
on   $\mathfrak{p}_{\R}$. 
In particular, $B$ defines a Riemannian structure on $Y(K)$.

  We shall also equip $\mathfrak{h}_{\Q} \subset \mathfrak{g}_{\Q}$ with the restriction of the form $B$,
  i.e., with the form arising similarly with the representation $\eta|\mathfrak{h}$.  
  When extended to $\mathfrak{h}_{\R}$ this coincides with the pullback of $B$ under
  $\Ad(g_{\infty}^{-1}): \mathfrak{h}_{\R} \rightarrow \mathfrak{g}_{\R}$, since the form $B$ was invariant; 
therefore the restricted form is preserved by a Cartan involution  fixing $\Uinfty$ (see remark after Lemma \ref{BoatyMcBoatFace}), and similarly defines a Riemannian structure on $Z(U)$.

For Tamagawa measure, what one actually needs is a  
measure on $\mathfrak{g}_{\adele_{\Q}}$, where $\adele_{\Q}$ is the adele ring of $\Q$.  
Choose a volume form on $\LieGQ$: 
\begin{equation}\label{omegaGdef} 
\omega_G  \in \det (\LieGQ^*). 
\end{equation} 
Let $\boldsymbol{\psi}$ be the standard additive character of $\adele_{\Q}/\Q$,
whose restriction to $\R$ is given by $x \mapsto e^{2\pi i x}$. 
We choose the $\boldsymbol{\psi}_v$-autodual measure on $\Q_v$ for every place $v$;
from that and   $\omega_G$   we obtain a measure   on $\mathfrak{g}_v = \mathfrak{g} \otimes \Q_v$ for every place $v$,
and so also a measure $\mu_v$ on $\GG(\Q_v)$. 

 By abuse of notation we refer to all the measures $\mu_v$ as ``local Tamagawa measures."
 They depend on $\omega_G$, but only up to $\Q^*$, and 
  their product $\prod_{v} \mu_v$ is independent of $\omega_G$.  

We proceed similarly for $\HH$, fixing a volume form $\omega_H \in \det(\mathfrak{h}^*)$, 
which gives rise to local Tamagawa measures on $\HH(\Q_v)$
and a global Tamagawa measure on $\HH(\adele)$. 

 \subsection{Lattices inside Lie algebras}
 We choose an integral lattice inside $\mathfrak{g}$ and $\mathfrak{k}$:
 
For $\mathfrak{g}$, we simply choose
 a lattice $\mathfrak{g}_{\Z} \subset \mathfrak{g}_{\Q}$ of volume $1$ for $\omega_G$, i.e.
 $\langle \omega_G, \det \ \mathfrak{g}_{\Z} \rangle = 1$.
 
 For $\mathfrak{k}$, Macdonald \cite{Macdonald} has specified a class of lattices $\mathfrak{k}_{\Z}^{\cpct} \subset \mathfrak{k}$
deriving from a Chevalley basis of the complexified group $\mathfrak{k}_{\C}$; one takes 
the toral generators together with $X_{\alpha} +X_{-\alpha}$ and $i(X_{\alpha} - X_{-\alpha})$,
where $\alpha$ varies over all roots.

With these definitions, we define discriminants of $\mathfrak{g}, \mathfrak{k}, \mathfrak{p}$ thus: 
   \begin{align} \label{coo}  \disc \ \mathfrak{g} &:=  \langle \det \mathfrak{g}_{\Z}, \det \ \mathfrak{g}_{\Z} \rangle_{\Killing} 
\\  \disc \medskip \mathfrak{k}  &:=  \langle \det  \mathfrak{k}^{\cpct}_{\Z},  \det \mathfrak{k}^{\cpct}_{\Z} \rangle_{-\Killing}. 
 \\ \disc \mathfrak{p} &:= \left| \frac{\disc \mathfrak{g}}{\disc \mathfrak{k}}\right| .\end{align}

  Note that
\begin{equation} \label{coo2} \disc \ \mathfrak{g}=  \langle \omega_G, \omega_G \rangle_{\Killing}^{-1},\end{equation}
and that the signs of the discriminants of $\mathfrak{g}, \mathfrak{k}, \mathfrak{p}$ are given by $(-1)^{d_K}, 1,1$ respectively.
Also
all these definitions carry over to $\HH$: in particular, we define $\disc \ \mathfrak{p}_H$ in a similar way. 

We need: 
\begin{lemma} The discriminants of $\mathfrak{g}, \mathfrak{k}, \mathfrak{p}$ all belong to $\Q^*$.
\end{lemma}
\proof
For $\disc(\mathfrak{g})$ this follows from the fact that $B$ is $\Q$-valued. It is enough to prove the result for $\mathfrak{k}$. There we observe that 
$$  \det \mathfrak{k}_{\Z}\in \Q^* \cdot  i^{\frac{d_K-r_K}{2}} \det \mathfrak{k}^{\che}_{\Z},$$
where $\mathfrak{k}^{\che}_{\Z}$ is a Chevalley  lattice for the complexified group $\mathfrak{k}_{\C}$.  
The representation $\eta$ defining the bilinear form $B$ gives a representation $\eta_{\C}$ of the Chevalley group
underlying $\mathfrak{k}_{\C}$;  this representation, like all representations of the complexified Chevalley group, can be defined over $\Q$
and so the trace form takes  rational values on $\mathfrak{k}^{\che}_{\Z}$, as desired.  \qed 
 
 Note that the same reasoning applies to $\HH$; thus the discriminants of $\mathfrak{h}, \mathfrak{u}, \mathfrak{p}_H$
 all lie in $\Q^*$ too.

 \subsection{Factorization of measures on $\Ginfty$.}
 
  First let us compute the Riemannian volume of $\Kinfty$. Macdonald \cite{Macdonald} shows that, for any invariant differential form  $\nu$ on $\Kinfty^{\circ}$,  regarded also
  as a volume form on the Lie algebra in the obvious way, we have
\begin{equation} \label{nyK} \mbox{ $\nu$-volume of $\Kinfty^{\circ}$} = \prod_{i} \frac{2 \pi^{m_i+1}}{m_i!} \nu(\det \mathfrak{k}_{\Z}^{\cmpct}) \sim  \Delta_K  \cdot \nu(\det \mathfrak{k}_{\Z}^{\cmpct})  \end{equation}
 where   $\Delta_K = \pi^{(d_K+r_K)/2}$; here the $m_i$ are the exponents of the compact Lie group $\Kinfty^{\circ}$,
so that  $\sum m_i = (d_K-r_K)/2$. 
 Therefore, 
 $$\vol(\Kinfty^{\circ})  :=  \mbox{Riemannian volume of $\Kinfty^{\circ}$ w.r.t. $-B|_{\mathfrak{k}}$ }   \sim  \Delta_K \cdot \sqrt{ \disc(\mathfrak{k})}.$$

  We can factor
  $ \det(\mathfrak{g}_{\R}^*) \simeq \det(\mathfrak{k}_{\R}^*) \otimes \det(\mathfrak{p}_{\R}^*)$, and with reference to such a factorization, 
  $$\omega_G  \stackrel{\eqref{coo2}}{=}  \frac{1}{\sqrt{  |\disc  \ \mathfrak{g}| }} \cdot \omega_K \otimes \omega_P,$$
  where $\omega_K \in \det \ \mathfrak{k}_{\R}^*$ is determined  (up to sign) by the requirement that
  $\langle\omega_K, \omega_K \rangle_{-B} =1$, and similarly $\omega_P \in \det \ \mathfrak{p}_{\R}^*$
  is determined by requiring that $\langle \omega_P, \omega_P \rangle_B=1$. 
  We can regard $\omega_K$ and $\omega_P$ as differential forms on $\Kinfty$ and $\Ginfty/\Kinfty$,
  extending them from the identity  tangent space by invariance; 
the measures on $\Kinfty$ and $\Ginfty/\Kinfty$
  defined by the differential forms $\omega_K$ and $\omega_P$ 
  coincide with the Riemannian measures (associated to $-B|_{\mathfrak{k}}$ and $B|_{\mathfrak{p}}$ respectively). 
 
   This implies that 
\begin{eqnarray}  \label{browncow}  {\small \frac{  \mbox{local Tamagawa measure  on $\Ginfty$ pushed down to $\Ginfty/\Kinfty^{\circ}$} }{\mbox{Riemannian measure on $\Ginfty/\Kinfty^{\circ}$ w.r.t. $\Killing|_{\mathfrak{p}}$} } }
 &=&    \frac{ \vol(\Kinfty^{\circ})}{ \sqrt{ |\disc \mathfrak{g}|  }} \\  \nonumber & \sim&  \Delta_K \cdot  \sqrt{\disc \mathfrak{p}}  \end{eqnarray} 

 \subsection{Tamagawa factors} \label{ss:Tamagawa}

Let $\mu_f$ denote the volume of $K \subset \G(\Afinite)$ with respect to Tamagawa measure (more precisely: the product of local Tamagawa measures as in  \S \ref{measurenormalizations}, over finite places). 
Evaluating $\mu_f$ is a standard computation, and is particularly straightforward in the split cases where we use it: There is an $L$-function $\Delta_G$ attached to $\GG$, with the property that its local factor at almost all places is given by $\frac{p^{\dim \GG}}{\# \GG(\mathbf{F}_p)}$; for example, if $\GG = \mathrm{SL}_n$, 
then $\Delta_G= \zeta(2) \dots \zeta(n)$.   Then $\mu_f \sim \Delta_G^{-1}$.  
We shall later use  the notation
$$ \Delta_{G,v} = \mbox{ local factor of $\Delta_G$ at the place $v$.}$$

 Let us introduce
\begin{equation} \label{DeltaGKdef}  \Delta_{G/K} = \Delta_G/\Delta_K,\end{equation}
where $\Delta_K = \pi^{(d_K+r_K)/2}$ as before. We can define similarly $\Delta_{H/U}$. 
 
 Now examine the Riemannian measure on $Y(K)$. We write
\begin{equation} \label{disconn} Y(K) = \coprod_{i} \Gamma_i \backslash \Ginfty/\Kinfty^{\circ},\end{equation}
where $I = \G(\Q) \backslash  \G(\Afinite)/K $ and, for $ i \in I$ with representative $g_i$, we have $\Gamma_i = \G(\Q) \cap g_i K g_i^{-1}$. 
  If we choose a fundamental domain $F_i \subset \G(\R)$, right invariant by $\Kinfty^{\circ}$,  for the action of $\Gamma_i$, then 
$\coprod_i F_i g_i K$ is a fundamental domain in $\G(\adele)$ for the action of $\G(\Q)$,
and $F_i g_i K$ maps  bijectively to  $Y(K)_i$, the $i$th component of $Y(K)$ under \eqref{disconn}.  The global Tamagawa measure of $F_i g_i K$ equals
$\mu_f$ multiplied by the local Tamagawa measure of $F_i$; on the other hand, the Riemannian measure
of $Y(K)_i$ is the Riemannian measure of $F_i/\Kinfty^{\circ}$, and so  by \eqref{browncow} we have

\begin{equation} \label{duck0} \frac{ \mbox{projection of Tamagawa measure to $Y(K)$} }{\mbox{Riemannian measure on $Y(K)$}} \sim {\Delta_{G/K}^{-1}}{\sqrt{\disc \ \mathfrak{p}}}. \end{equation}
 Similarly,
\begin{equation} \label{duck} \frac{ \mbox{projection of Tamagawa measure to $Z(U)$} }{\mbox{Riemannian measure on $Z(U)$}} \sim {\Delta_{H/U}^{-1}}{\sqrt{\disc \ \mathfrak{p}_H}} \end{equation}%

 \begin{prop}  \label{proptodo}  
 Fix  $\nu_H \in \det (\p_{H,\R})$  with $\langle \nu_{H}, \nu_H \rangle_{\Killing} = 1$;
 consider $\nu_H$ as an element of $\wedge^{p_H} \p$ via 
 $$ \det(\p_H) = \wedge^{p_H} \p_H \stackrel{\Ad(g_{\infty}^{-1})}{\longrightarrow} \wedge^{p_H} \p.$$
If $T  \in \Hom_{\Kinfty^{\circ}}(\wedge^{p_H} \mathfrak{g}/\mathfrak{k}, \pi^K) $
 induces the differential form $\Omega(T)$ on $Y(K)$, as in \eqref{omegadef}, then \begin{equation}  \frac{  \left|  \int_{Z(U)} \iota^* \Omega(T)   \right|^2 }{\langle \Omega(T) , \Omega(T) \rangle_{\mathrm{R}}} \sim   \left(\disc \ \mathfrak{p}\right)^{1/2}  \frac{\Delta_{H/U}^2}{\Delta_{G/K}}
  \cdot \frac{ \left| \int_{[\HH]} g_f T(\nu_H) dh \right|^2}{  \langle T(\nu_H), T(\nu_H) \rangle } \label{RT} \end{equation}
(where we regard the statement as vacuous if $T(\nu_H) = 0$).  Here $g_f T(\nu_H)$ is the 
translate of $T(\nu_H) \in \pi$ by $g_f$, and $\langle T(\nu_H), T(\nu_H)\rangle$ is the $L^2$-norm $\int_{[\G]} |T(\nu_H)|^2$ with respect to Tamagawa measure. 

On the left-hand side the $L^2$-norm  of $\Omega(T)$ is  taken with respect to Riemannian measure\footnote{The reason we use Riemannian measure at all is that it interfaces
well with the action of $\aG$ (e.g. Proposition \ref{basic properties}
part (i)). }  on $Y(K)$ 
induced by $\Killing$ (thus the subscript $R$), whereas on the right-hand side everything is computed with respect to Tamagawa measure. 
   \end{prop}

    \proof 
      
We follow the convention that a subscript $R$ will denote a computation with respect
to the Riemannian measure induced by $B$.  
Although this measure is defined on  the locally symmetric space $Y(K)$,  we will also refer to a ``Riemannian'' measure on $[\GG]$;
 this is simply a Haar measure that is normalized to project to the Riemannian measure  under $[\GG] \rightarrow Y(K)$.

When we integrate $\Omega(T)$ over the cycle representing $\iota_* [Z(U)]$ we get
\begin{eqnarray*}       \int_{Z(U)} \iota^* \Omega(T)  =  \int_{Z(U)} T(\nu_H)  d_{\mathrm{R}}h   \stackrel{\eqref{duck}}{ \sim }
 \Delta_{H/U} \sqrt{ \disc  \ \mathfrak{p}_H}   \int_{[\mathbf{H}] g_f}  T(\nu_H)  dh. \end{eqnarray*}
here $d_{\mathrm{R}}$ is Riemannian  measure on $Z(U)$ and $dh$ is Tamagawa measure.
 
 Next we compute the norm of $\Omega(T)$ with respect to Riemannian volume
 and compare it to the Tamagawa-normalized $L^2$ norm of $T(\nu_H)$. 
Let $\mathcal{B}$ be a $\Killing$-orthogonal basis for $\wedge^{p_H} \mathfrak{p}_{\R}$.
For each $x \in \mathcal{B}$,  if we evaluate $\Omega(T)$ at $x$ (considered
at the tangent space of the identity coset) we get, by definition, $T(x)$ evaluated at the identity. 
More generally the sum
$$ \sum_{x \in \mathcal{B}} |T(x)|^2$$
defines a function on $[\G]/K$ that is $\Kinfty^{\circ}$-invariant, and therefore descends to $Y(K)$;
its value at a point of $Y(K)$ is the norm of $\Omega(T)$ at that point. 
Integrating over $Y(K)$ with respect to Riemannian norm, we see 
\begin{eqnarray} \langle \Omega(T), \Omega(T) \rangle_R &= & \int_{g \in Y(K)} \left( \sum_{x \in \mathcal{B}}  |T(x)|^2 \right) d_R g \\  \nonumber
&=&  \sum_{x \in \mathcal{B}} \langle T(x), T(x) \rangle_R  %
\\ &=& \frac{ \sum_{x \in \mathcal{B}} \langle T(x), T(x) \rangle_R}{ \langle T(\nu_H), T(\nu_H) \rangle_R}  \langle T(\nu_H), T(\nu_H) \rangle_R
\\ &= &  \frac{ \|T\|^2}{ \langle T(\nu_H), T(\nu_H) \rangle}  \langle T(\nu_H), T(\nu_H) \rangle_R.  \end{eqnarray}
 
Here we define \begin{equation} \label{Tnorm} \|T\|^2 = \sum_{x \in \mathcal{B}} \langle T(x), T(x) \rangle_{L^2([\G])},\end{equation} and the $L^2$-norm  is now  computed with respect to {\em Tamagawa} measure on $[\G]$. 
 After translating \eqref{duck0} between Riemannian and Tamagawa measure for $\langle T(\nu_H), T(\nu_H) \rangle_R$, 
the result follows from Lemma \ref{oink} below. 
\qed

  \begin{lemma}  \label{oink}
 Suppose that $(\GG, \HH)$ are as in \S \ref{Ggp-precise}
 and $B$ is the trace form defined in \eqref{Bdef}.   Let
 $T  \in \Hom_{\Kinfty^{\circ}}(\wedge^{p_H} \mathfrak{g}/\mathfrak{k}, \pi^K) $
 lie in a $\Kinfty/\Kinfty^{\circ}$ eigenspace, 
 and let $\nu_H$ be as in Proposition \ref{proptodo}, and the norm $\|T\|$ be as in \eqref{Tnorm}. Then 
 $$   \frac{ \langle T(\nu_H), T(\nu_H) \rangle}{ \|T\|^2 \langle \nu_H, \nu_H \rangle_{\Killing}}  \in  \Q,$$
 \end{lemma}
 
Note that $\langle \nu_H, \nu_H \rangle_{\Killing} = 1$, by the way it was defined in the statement of Proposition
\ref{proptodo}, 
 but we prefer to write the expression above because
it is scaling invariant.

\proof 

Observe the ratio under consideration is invariant under rescaling the norm either on source or target of $T$, or rescaling $T$,
or rescaling $\nu_H$. 
The  validity of the statement  depends only on the data
$$ (\GG(\R) \supset \Kinfty, \HH(\R) \supset \Uinfty,  g_{\infty},  \pi_{\infty}, T)$$
together with the scaling class of the form induced by $B$ on $\mathfrak{p}_{\R}$
and $\mathfrak{p}_H$.  By Lemma \ref{BoatyMcBoatFace}, 
  it suffices to treat the case when $g_{\infty} =e$, the identity element, and  
$\GG(\R) \supset \Kinfty, \HH(\R) \supset \Uinfty$ is one of the following: 
\begin{align}
 \label{first} \PGL_n(\R) \times \PGL_{n+1}(\R)  &\supset  \mathrm{PO}_n \times \mathrm{PO}_{n+1} ,    &\GL_n(\R) &\supset \mathrm{O}_n . \\ 
 \PGL_n(\C) \times \PGL_{n+1}(\C)  &\supset \mathrm{PU}_n  ,  &\GL_n(\C) &\supset \mathrm{U}_n. \\
  \SO_{n+1} (\C) \times \SO_n(\C) &\supset \SO_{n+1} \times \SO_n , &\SO_n(\C) &\supset \SO_n. \end{align}
 In all cases, $\mathrm{O}$ and $\mathrm{U}$ refer to the standard orthogonal form $\sum x_i^2$ and the standard Hermitian form $\sum |z_i|^2$.

 Note, in particular, there is no dependence on the $\Q$-form of the groups $\GG, \HH$, and we 
 may freely assume that $\GG,\HH$ are the $\Q$-split forms in the first case, and (the restriction of scalars of the) $\Q(i)$-split forms in the second and third cases. 
With these $\Q$-structures,  the inclusion of $\HH$ into $\GG$ is $\Q$-rational, the form $B$ remains $\Q$-rational on the $\Q$-Lie algebra, and moreover the maximal compacts $\Uinfty,\Kinfty$ described above are actually defined over $\Q$. 
 Therefore, $\mathfrak{p}_{\R}$ and also  $\wedge^{p_H} \mathfrak{p}_{\R}$  inherits a $\Q$-structure, and the line $\R \nu_H \subset \wedge^{p_H} \mathfrak{p}_{\R} $ is thus defined over $\Q$.
 We may freely replace $\nu_H$, then, by a $\Q$-rational element $\nu_H' \in \R.\nu_H$. 

First let us consider the latter two cases:   $\GG_{\R}$ is a ``complex group'' and  so $\Kinfty = \Kinfty^0$. 
 In this case (see \S \ref{Gsc types} or the original paper \cite{VZ}) $T$ factors through  a certain $\Kinfty$-type $\delta \subset \pi_{\infty}$, which occurs with multiplicity one inside
 $\wedge^{p_H} \mathfrak{p}$.   In particular, $\langle T(v), T(v) \rangle$
 is proportional simply to $\langle \proj_{\delta} (v), \proj_{\delta}(v) \rangle_B$. The ratio in question is therefore
 simply 
 $$(\dim \delta)^{-1}   \frac{\langle \proj_{\delta} \nu_H' , \proj_{\delta} \nu_H' \rangle_B}{\langle \nu_H', \nu_H' \rangle_{\Killing}}$$
 It suffices to see that $\proj_{\delta} \nu_H'$ is $\Q$-rational. 
However, the isomorphism class of $\delta$ is fixed by outer automorphisms of $\Kinfty$:
the highest weight of $\delta$ is the sum of positive roots, and the character of $\delta$ on the center of $\Kinfty$ is trivial. 
It follows that $\proj_{\delta}$, as a self-map of $\wedge^{p_H} \mathfrak{p}$, is actually defined over  $\Q$. 

In the remaining case \eqref{first},   fix a character $\chi: \Kinfty \rightarrow \{\pm 1\}$. 
 The subspace
 $$\Hom(\wedge^{p_H} \mathfrak{p}, \pi_{\infty})^{(\Kinfty,\chi)} \subset \Hom(\wedge^{p_H} \mathfrak{p}, \pi_{\infty})$$
transforming under $(\Kinfty, \chi)$, remains $1$-dimensional (if nonzero). 
  This space consists precisely of the $\Kinfty$-homomorphisms
 $$ \wedge^{p_H} \mathfrak{p} \longrightarrow \pi_{\infty}|_{\Kinfty} \otimes \chi.$$
In this case there is a unique irreducible $\Kinfty$-representation $\delta' \subset \wedge^{p_H} \mathfrak{p}$
 which is common to $\wedge^{p_H} \mathfrak{p}$ and $\pi_{\infty}|_{\Kinfty} \otimes \chi$.  This $\delta'$ 
splits into two irreducibles when restricted to $\Kinfty^0 =  \mathrm{PSO}_{n+1} \times \mathrm{PSO}_n(\R)$;
 these two irreducibles are switched by $\Kinfty/\Kinfty^0$, which is just the outer automorphism group of $\Kinfty^0$,
 and each irreducible occurs with multiplicity one inside $\pi_{\infty}$ (one in each irreducible factor of $\pi_{\infty}|_{\SL_{n+1} \times \SL_n}$). 
 It follows that the projection from $\wedge^{p_H} \mathfrak{p}$ to the $\delta'$-isotypical component
 is actually defined over $\Q$, and we can proceed just as before.   
\qed 

\subsection{Working hypotheses on period integrals} \label{sec:periodintegrals}
We now simplify \eqref{RT} a little bit further using the Ichino-Ikeda conjecture \cite{II}. 
Note that the original conjectures of Ichino and Ikeda were formulated only for orthogonal groups, but in fact
the analogue of their conjecture is known to be valid in the $\GL$ case (see, e.g. \cite[Theorem 18.4.1]{SV}, although the result is well-known to experts). 

At this point it is convenient, in cases (1) and (2) from  \S \ref{Ggp-precise}, 
to work with the $E$-groups $\HH_E, \G_E$ instead
of their restriction of scalars to $\Q$.  Recall that we regard $E=\Q$ in the remaining case. 

To normalize Tamagawa measures, 
we must choose a measure on $E_v$ for each place; we choose these measures so that the volume of $\adele_E/E$
is $1$ and so that the measure of the integer ring of $E_v$ is $\Q$-rational for every finite place $v$, and $1$ for almost every place $v$.  Note that this implies
that, for $v$ the archimedean place of $E$,
\begin{equation} \label{Lebesgue} \mbox{ measure on $E_v$} \sim D_E^{1/2}  \cdot \mbox{Lebesgue.}\end{equation}

Fix now $E$-rational invariant differential forms of top degree
on $\HH_E$ and $\G_E$ and use this to define Tamagawa measures $dh$ and $dg$ on $\HH_E(\adele_E) = \HH(\adele_{\Q})$ and $\G_E(\adele_E)=\G(\adele_{\Q})$,
thus on $[\HH] = [\HH_E]$ and $[\G_E]=[\GG]$; these global Tamagawa measures coincide with the ones made using $\Q$-rational differential forms.

We factorize $dh$ and $dg$ as $\prod dh_v$  and $\prod dg_v$ where $dh_v,dg_v $s are local Tamagawa measures,
and the factorization is over {\em places of $E$} rather than places of $\Q$.  As before, the $dh_v, dg_v$ depend on the choices  of differential form, but they only depend up to $\Q^*$,
since $|e|_v \in \Q^*$ for each $e \in E$ and each place $v$.  

We make the following assumptions.  (i) is the Ichino--Ikeda conjecture. As will be clear from the comments,  it should be possible to establish (i)-(iii) at least in many of the cases under consideration. However,  for lack of appropriate references (at least known to us),  we  formulate them as assumptions.

\begin{itemize}
\item[(i)] (Global integral): 
Suppose that, with reference to a factorization $\pi= \otimes \pi_v$, we factorize $T(\nu_H)$ as $\otimes_{v} \varphi_v$
and factorize also the inner product. Then 
$$ \frac{\left| \int_{[\HH]} T(\nu_H) dh \right|^2}{ \int_{[\GG]}  |T(\nu_H)|^2 dg} \sim   \prod_{v} 
 \frac{\int_{\HH(F_v)} \langle h_v \varphi_v, \varphi_v \rangle dh_v}{\langle \varphi_v, \varphi_v \rangle}.
$$
where the right-hand side is regularized as a global $L$-value according to (ii) below.

This is the conjecture of Ichino--Ikeda \cite{II}. Its validity in the $\PGL$ case is folklore, see e.g. \cite[Theorem 18.4.1]{SV}. 

\item[(ii)]  (Local integrals at almost all nonarchimedean places):  
At almost all nonarchimedean places $v$, with local Tamagawa measures $dh_v$, we have 
\begin{equation} \label{locint}   \underbrace{\frac{\int_{\HH(F_v)} \langle h_v \varphi_v, \varphi_v \rangle dh_v}{\langle \varphi_v, \varphi_v \rangle} }_{:= \mathrm{LHS}_v} = \underbrace{ \frac{\Delta_{G,v}}{\Delta_{H,v}^2} \frac{L(\frac{1}{2}, \pi_v; \rho)}{L(1, \Ad, \pi_v)} }_{:= \mathrm{RHS}_v}
\cdot   \end{equation} 
where the representation $\rho$ of the $L$-group of $\GG$ is that corresponding to the Rankin-Selberg $L$-function
in the $\mathrm{SO}$ cases, and that corresponding to the square of the Rankin-Selberg $L$-function in the $\PGL$ cases. 
Also $\Delta_{H,v}$ and $\Delta_{G,v}$
are the local factors described in \S \ref{ss:Tamagawa}. 

This is known in the $\SO$ cases by \cite[Theorem 1.2]{II} (note that our measure normalization differs from theirs),
and it is likely that one can verify it in the $\PGL$ case also but we do not know a reference.

   \item[(iii)]  (Local integrals at the remaining nonarchimedean places)
At {\em all} non archimedean  places,  we have  
\begin{equation} \label{iiicond} \frac{\mathrm{LHS}_v}{\mathrm{RHS}_v} \ \in \mathbf{Q},\end{equation}
where $\mathrm{LHS}_v, \mathrm{RHS}_v$ are the left- and right- hand sides of \eqref{locint}. \footnote{
Since here we are working at possibly ramified places, we should be clear on the definition of local $L$-factor that is used.  What is needed, for our argument,is the local $L$-factor associated to the global Galois parameter    for $\pi$,
whose existence is being assumed in our context, and  the $L$-factors occuring in \eqref{iiicond} should be understood in this sense.   Thus the desired result follows  if one knows at $v$
that the characteristic polynomial of Frobenius on inertial invariants has rational coefficients.   In the $\GL$ case it should likely also follow
from establishing a form of local-global compatibility for the global Galois parameter, together  with the fact that  the local Langlands correspondence preserves
Rankin-Selberg $L$-factors.}

We believe this is not too difficult to show:  the left-hand  side and right-hand side
should {\em individually} belong to $\Q$.
 However, we do not know a reference, and to make the argument  carefully would take us too far afield.

\item[(iv)]  (Rationality,  archimedean places)
For $v$ the unique archimedean place of $E$,  let 
  $T$ and $\nu_H$ be as in Lemma \ref{oink}.  
    If we are in the $\PGL_n \times \PGL_{n+1}$ over $\Q$ case, assume that $T$ transforms under the character of $\Kinfty/\Kinfty^{\circ} \simeq \{\pm 1\}$
    given by $x \mapsto x^{n+1}$.  
   
   Put $\varphi_v = T(\nu_H)  \in \pi_{\infty}$. 
With this choice of test vector, we have
\begin{equation} \label{arch eval} \frac{\mathrm{LHS}_v}{\mathrm{RHS}_v} \sim D_E^{d_H/2}\end{equation}
where $D_E$ was the  absolute value of the discriminant of $E$.

  Aside from the factor $D_E^{d_H/2}$, this simply states the belief that ``in good situations, the archimedean integral
 behaves like the nonarchimedean integrals.''  
 This   belief must be applied with  caution, see e.g \cite{Stade}, but seems reasonable in the instances at hand. 
 The factor   $D_E^{d_H/2}$ is necessary to  make the conjecture independent of $E$, because of \eqref{Lebesgue}. Note the very fact that $\mathrm{LHS}_v$ is {\em nonzero} is not known in all cases;
it has been proven by B. Sun by a remarkable positivity argument in the $\GL_n$ cases \cite{Sun}.

%
%
%
%

 \end{itemize}

\subsection{Summary}
 
Combining Proposition \ref{proptodo}
with the working hypotheses of  \S \ref{sec:periodintegrals}, we have proved: 
  
\begin{theorem} \label{period_theorem}
Let $\iota: Z(U) \rightarrow Y(K)$ be, as in \S \ref{sec:GH}, a map of arithmetic manifolds associated
to the inclusion $\HH \subset \GG$ and the element $g =(g_{\infty}, g_f) \in \G(\adele)$, as in \S \ref{Ggp-precise}.   Let $\pi$ be a cohomological automorphic representation for $\GG$
with $\pi = \overline{\pi}$. 
Let $T  \in \Hom_{\Kinfty^{\circ}}(\wedge^q \mathfrak{g}/\mathfrak{k}, \pi^K) $ be real and a $\Kinfty$-eigenfunction; in the $\PGL_n \times \PGL_{n+1}/_{\Q}$ 
case we assume its eigencharacter is as  noted above \eqref{arch eval}.  Let $\Omega(T)$ be the associated
differential form on $Y(K)$ (as in \eqref{omegadef}).  

Assuming the working hypothesis on period integrals (\S \ref{sec:periodintegrals}). 
Then
\begin{equation} \label{bigmonster} \frac{ \left( \int_{Z(U)} \iota^* \Omega(T)  \right)^2  }{\langle\Omega(T), \Omega(T)\rangle_{R}} \in \Q  c_f c_{\infty} 
     \left(  \frac{L(\frac{1}{2}, \pi; \rho)}{L(1, \Ad, \pi)} \right) 
 \end{equation}
  where $\rho$ is the representation of  the $L$-group occurring in \eqref{locint},  $c_f^2 \in \Q^*$, $c_{\infty}$ is a  half-integral power of $\pi$,   and the subscript $R$ means that we compute the $L^2$-norm with respect to a Riemannian  measure normalized as in \S \ref{measurenormalizations}. Explicitly: 
\begin{eqnarray*} c_f  & = &
   \left(    \disc \mathfrak{p} \cdot  D_E^{d_H}  \right)^{1/2}  \left( \in \sqrt{\Q^*} \right),  \\
   c_{\infty} & =& 
 \left(  \frac{\Delta_K}{\Delta_U^2} \right)   \left( \frac{ \Delta_{G,\infty}}{\Delta_{H,\infty}^2} \right)  \left(  \frac{L_{\infty}(\frac{1}{2}, \pi; \rho)}{L_{\infty}(1, \Ad, \pi)} \right)  
  \end{eqnarray*}

  Moreover, if $L(\frac{1}{2}, \pi; \rho) \neq 0$ and there exists a nonzero $\HH(\adele)$-invariant functional on the space of $\pi$, 
  it is possible to choose the data $(g_{\infty}, g_f)$ and level structure $U$ 
 in such a way that the left-hand side of \eqref{bigmonster}
  is also nonzero. 
 \end{theorem}
 
 \proof This follows by putting together  Proposition \ref{proptodo}  with the  statements of \S  \ref{sec:periodintegrals}.
 The assumption that $T$ was real  means that   $\Omega(T)$ is a real differential form, and that $T(\nu_H)$ is a real-valued function on $[\G]$; this allows us to drop absolute value signs. 
The last sentence of the Theorem follows because for each finite place $v$ of $E$ and any nonzero $\varphi_v \in \pi_v$, 
it is (under the assumption quoted) possible to choose $g_v \in \mathbf{G}_E( E_v)$ with the property that $\int_{\mathbf{H}_E(E_v)} \langle h_v  g_v \varphi_v, g_v \varphi_v \rangle \neq 0$
(see \cite{Waldspurger} or \cite{SV}).

 \section{Compatibility with the   Ichino--Ikeda conjecture}

We now study more carefully the compatibility of our conjecture with the Ichino--Ikeda conjectures on periods.  We work in the following situation:

Let $\HH \subset \GG$ be as discussed in \S \ref{Ggp-precise}.  
Each case involves a field $E$, which is either imaginary quadratic $\Q(\sqrt{-D_E})$ or $E=\Q$.
We will generically use $\iota$ to denote an embedding $E \hookrightarrow\C$.  

As in \S \ref{Ggp-precise}, we use (e.g.)   $\GG_E$ for $\GG$ regarded as an $E$-group
and $\GG$ for it as a $\Q$-group and similar notations for Lie algebras: in particular $\mathfrak{g}_E$
is the Lie algebra of $\GG_E$, an $E$-vector space, whereas
$\mathfrak{g}_{\Q} = \mathrm{Res}_{E/\Q} \mathfrak{g}_E$ is the $\Q$-Lie algebra that is the Lie algebra of $\GG$.  

We use other notation as in section    \S \ref{sec:GH} and   \ref{Ggp-precise} ;
in particular we have a map of arithmetic manifolds $$\iota: Z(U) \rightarrow Y(K)$$ associated to $\HH, \GG$
and the element $g=(g_{\infty}, g_f) \in \G(\adele)$. 
The Borel--Moore cycle $\iota_* [Z(U)]$ defined by $\HH$ lies in  the minimal
cohomological dimension for tempered representations for $G$ (see \eqref{roobar} and Table \ref{inner5}), which we shall now call $q$:
$$ q =p_H = \mbox{minimal cohomological dimension for tempered representations.}$$

\subsection{Setup} \label{7setup}
We suppose that $\Pi$ is, as in \S \ref{sec:numerical_invariants},  the collection of cohomological representations for $\GG$ at level $K$
arising from a fixed character of the Hecke algebra with values in $\Q$.  These representations are all nearly equivalent, so it makes sense to talk about the $L$-function of $\Pi$ and so on. 
Just as in  \S \ref{sec:numerical_invariants} we want every representation in  $\Pi$ to be cuspidal and tempered at $\infty$, but we will impose
a more precise condition below.

\subsubsection{Condition on $\Pi$}
 In the cases {\em besides} $\PGL_n \times \PGL_{n+1}$ over $\Q$, 
we  assume that the level structure $K$ has multiplicity one for $\Pi$, in the sense that  \begin{equation} \label{multone}\dim H^q(Y(K), \Q)_{\Pi} = 1.\end{equation}
In particular, in this case,  there is just one automorphic representation in $\Pi$ contributing to  this cohomology, $\Pi = \{\pi\}$ say;
  in particular $\pi = \overline{\pi}$. We ask that $\pi$ be  tempered cuspidal (just as in our prior discussion in \S  \ref{sec:numerical_invariants}).  
 
 In this case, we obtain from $\pi$ a harmonic differential form
  $$\omega \in \Omega^q(Y(K))$$ whose cohomology class  generates $H^q(Y(K), \R)_{\Pi}$.
 
In the remaining case $\GG = \PGL_n \times \PGL_{n+1}/_{\Q}$, it impossible to satisy \eqref{multone} because of disconnectedness issues.  We ask instead that\footnote{  For example, for the group $\PGL_2$, a tempered cohomological dimension contributes two dimensions to cohomology -- an antiholomorphic form and a holomorphic form;
these are interchanged by the action of $\mathrm{O}_2$, and so \eqref{multone} holds.}
\begin{equation} \label{multone2} H^q(Y(K), \Q)_{\Pi}^{\pm } = 1, \end{equation}
 where $\pm$ denotes eigenspaces under
  $\Kinfty/\Kinfty^{\circ} \simeq \{\pm 1\}$. This again means
  there is just one automorphic representation $\Pi=\{\pi\}$ contributing to this cohomology (see discussion of cohomological representations for $\PGL_n(\R)$
  in \cite[\S 3]{Mahnkopf} or \cite[\S 5]{shahidiraghuram}); 
 we again require that $\pi=\overline{\pi}$ is tempered cuspidal.

 In this case, we obtain from $\pi$  harmonic differential forms
  $$\omega ^{\pm} \in \Omega^q(Y(K))$$ whose cohomology classes generate $H^q(Y(K), \R)_{\Pi}^{\pm}$.

\subsubsection{The cycle $Z(U)$ and its twisted versions} \label{psidef}
 We have available in all cases the class $\iota_* [Z(U)]$ in Borel--Moore homology. However in the case $\GG = \PGL_n \times \PGL_{n+1}$
 we want to twist it:  Fix an auxiliary quadratic character $\psi$ of   $\mathbf{A}_{\Q}^*/\Q^*$ which, at $\infty$, gives the sign character of
$\mathbf{R}^*$.

  The function $\psi \circ \det$ now gives rise to a locally constant function on $[\HH]$,
and thus a Borel--Moore cycle 
$$[Z(U)]_{\psi} \in H^{\bm}_{p_H}(Z(U), \Q)$$
of top dimension on $Z(U)$. To be precise first choose $U' \subset U$
on which $\psi$ is constant, so that $\psi$ gives a locally constant function on $Z(U')$, then push forward  the resulting cycle and multiply by
$\frac{1}{[U:U']}$; however, this will equal zero unless $\psi$ was trivial on $U$ to start with.   It will be convenient to write for $\varepsilon \in \pm 1$
\begin{equation} \label{ZUpmdef}  [Z(U)]^{\varepsilon} = \begin{cases} [Z(U)],  &  \varepsilon = (-1)^{n+1};  \\ [Z(U)]_{\psi}, & \varepsilon = (-1)^{n}. \end{cases}\end{equation}
 
\subsection{Motivic cohomology; traces and metrics}

We assume that there exists an adjoint motive $\Ad \Pi$ attached to $\Pi$, in the sense of Definition \ref{AdjointMotiveDefinition}.
By its very definition, it is equipped with an isomorphism
\begin{equation} \label{basic} H_B(\Ad \Pi, \C) \simeq \widehat{\mathfrak{g}}_{\Q,*} \otimes \C =  \widehat{\mathfrak{g}}. \end{equation}
where  $\widehat{\mathfrak{g}}_{\Q,*}$ is as in Definition \ref{AdjointMotiveDefinition}. 
Now we may define the motivic cohomology group 
\begin{equation} \label{Ldefagain} L_{\Pi} = H^1_{\m}(\coAd \Pi,  \Q(1)),\end{equation}
as in \eqref{Ldef}.
 As described in \S \ref{BRag}, 
 the regulator on $L_{\Pi}$ takes the shape 
\begin{equation} \label{Lreg} L_{\Pi}   \longrightarrow  \aG \end{equation}
and indeed $L_{\Pi}$ lands inside the twisted real structure on $\aG$.

\subsection{The volume of $L_{\Pi}$} 
There are two natural metrics which can be used to compute  
the volume of $L_{\Pi}$. One of these metrics arises from a bilinear form on the Lie algebra of $\mathbf{G}$,
and the other one arises from a bilinear form on the Lie algebra of the dual group. We will 
need to pass between the volumes with respect to these metrics in our later computations, and
so we explain now why they  both give the same volume, up to ignorable factors. 
 
As  per  \S \ref{traceforms} we can equip $\Ad \Pi$ with a weak polarization
whose Betti incarnation is the standard trace form 
 on $\widehat{\mathfrak{g}}$ itself. Note that $\widehat{\mathfrak{g}}$ is a sum of classical Lie algebras; by ``standard trace form,'' we mean that we take
 the form $\mathrm{tr}(X^2)$ on each factor, where  we use the standard representation of that factor. This is visibly $\Q$-valued on $\widehat{\mathfrak{g}}_{\Q}$. 
 We refer to this as the ``trace weak polarization'' and denote it by $\widehat{\tr}$.    
 
 This induces a quadratic form (denoted  $\widehat{\tr}^*$)  on $\widetilde{\mathfrak{g}}$, by duality, 
which corresponds to a weak polarization on $\Ad^* \Pi$.  
As in   \S \ref{Deligne metrics} we may use this to induce
a quadratic real-valued form  on $H^1_{\mathcal{D}}( (\Ad^* \Pi)_{\R}, \R(1))$, which we
extend to a Hermitian form on  $$H^1_{\mathcal{D}}( (\Ad^* \Pi)_{\R}, \R(1)) \otimes \C.$$  As in  \S \ref{motive0sec} this $\C$-vector space is identified
with $\widetilde{\mathfrak{g}}^{W_{\R}}$, and thus with $\aG$.  (Here, and in the remainder
of this proof, we understand $W_{\R}$ to act on $\mathfrak{g}$ by means of 
the tempered cohomological parameter, normalized  as in  \S \ref{TemperedCohomologicalParameter}.)

Explicitly, this induced  Hermitian form on $\aG$ is given by  
\begin{equation} \label{inducedmetric}  (X, Y) \in \aG \times \aG \mapsto \widehat{\tr}^*(X \cdot \overline{wY})\end{equation}
where $w$ is the long Weyl group element; we used the computation of the Betti conjugation
in the proof of   Lemma \ref{Beilinson regulator real structure}.

By its construction the Hermitian form \eqref{inducedmetric} is a real-valued quadratic form when it is restricted
to the twisted real structure $\mathfrak{a}_{G,\R}'$.   {\em This quadratic form need not be positive definite,}  since we started only with a ``weak'' polarization, but this makes little difference to us. 
The volume of $L_{\Pi}$
 with reference to $\widehat{\tr}^*$ may be analyzed by means  of Lemma    \ref{lem:vol-indep} (the failure
 of positive definiteness means that the volume may be purely imaginary: the square of the volume is, by definition, the determinant of the Gram matrix). 
 We denote this volume by $\vol_{\widehat{\tr}}(L_{\Pi})$.

On the other hand there is a different Hermitian form on $\aGs$, which {\em is} positive definite,
and whose interaction with the norm on harmonic forms is easy to understand. 
Namely, we have equipped (\S \ref{measurenormalizations}) $\mathfrak{g}_{\Q}$
 with a $\Q$-rational bilinear form, the trace form for a standard representation; 
this form endows $Y(K)$ with a Riemannian metric.  Then, by (i) of Proposition  \ref{basic properties},  $\aGs$  acts ``isometrically'' (in the sense specified there) for the dual  of the form given by
  \begin{equation} \label{Form2} (X, Y) \in  \aG \times \aG   \mapsto B(X, \overline{Y})\end{equation}
This form is  also real-valued on the twisted real structure $(\mathfrak{a}_{G,\R}^*)'$, and moreover it defines a positive definite quadratic form there.   It is positive definite
because $B(X, \bar{X}) > 0$ for  $X \in \mathfrak{p}$, and this 
  contains (a representative for) $\aG$.  To see that it is real-valued,   observe that   \begin{equation} \label{realvalued2}  \overline{ B(X,  \overline{w Y}) }=  B(\overline{X},  w  Y)  = B(\overline{w^{-1} X}, Y) =  B(\overline{w  X},  Y).\end{equation}
so $B(X,  Y) \in \R$ if $X, Y$ belong to the twisted real structure; 
but if $Y$ belongs to the twisted real structure, so does $\bar{Y}$.

By Lemma \ref{trace-doo-doo}, the quadratic forms given by restriction of $\widehat{\tr}$ to $\mathfrak{a}_{G,\R}^* \subset \widehat{\mathfrak{g}}$,
and the restriction of $B$ to $\mathfrak{a}_{G,\R}$, are in duality with one another (after possibly multiplying $\widehat{\tr}$ by $\frac{1}{4}$); thus also their complex-linear extensions 
 $\widehat{\tr}$ on $\aGs \subset \widehat{\mathfrak{g}}$
and $B$ on $\mathfrak{a}_G$ are dual to one another (up to the same possible rescaling). 
Noting that  $\widehat{\tr}^*$ on $\widetilde{\mathfrak{g}}^{W_{\R}}$ and
$\widehat{\tr}$ on $\mathfrak{g}^{W_{\R}}$ are also dual quadratic forms,
it follows that (as quadratic forms on $\aG$) we have an equality $B = \widehat{\tr}^*$ (up to the same possible rescaling).

We will be interested in  $$\vol_{\tr}(L_{\Pi}) := \mbox{ volume of $L_{\Pi}$ with respect to  \eqref{Form2}.}$$
 Choosing a $\Q$-basis $x_i$ for $L_{\Pi}$ and with notation as above,   we compute: 

\begin{align} \label{tr-tr} \vol_{\tr}(L_{\Pi})^2 &\stackrel{\eqref{Form2}}{=} \det B(x_i, \overline{x_j})  \stackrel{\mathrm{Lem.} \ \ref{trace-doo-doo}}{=}  4^k \widehat{\tr}^*(x_i, \overline{x_j})  =  4^k \det(w) \widehat{\tr}^*(x_i, \overline{w x_j})  \\ &=
(4^k \det(w)) \vol_{\widehat{\tr}}(L_{\Pi})^2 \end{align}
 for some $k \in \Z$.   Clearly $\det(w) = \pm 1$; it is possible that $\det(w) = -1$, but in any case our final results
will have factors of $\sqrt{\Q^*}$ which allow us to neglect this factor.

 \medskip

\subsection{}

We may  state our theorem: 

\begin{theorem} \label{MainPeriodTheorem} 
Notation as before, so that $(\HH, \GG)$ is as in  \S \ref{Ggp-precise}, 
the embedding $Z(U) \rightarrow Y(K)$ is set up as in \S \ref{sec:GH}, 
and the cuspidal cohomological automorphic representation $\Pi$ is as in \S \ref{7setup}.

 Make the following assumptions:
 
 \begin{itemize}
  \item[(a)] Beilinson's conjectures on special values of $L$-functions (both parts (a) and (b) of Conjecture \ref{conj:Beilinson}). 
 \item[(b)] Existence of an adjoint motive attached to $\Pi$  (as in Definition \ref{AdjointMotiveDefinition}), arising from a    $\widehat{G}$-motive attached to $\Pi$ (Conjecture \ref{Conj Pimotive} in Appendix
 \ref{Pimotive Appendix}).\footnote{The latter conjecture is, roughly speaking,  a generalization of requiring the existence of an adjoint motive,   but replacing the adjoint representation of the dual group 
 by all representations at once. However Conjecture \ref{Conj Pimotive} is a little less precise about coefficient fields than the existence of an adjoint motive as in Definition \ref{AdjointMotiveDefinition}.}
    \item[(c)] Working hypotheses on period integrals (\S \ref{sec:periodintegrals}). 
 \end{itemize}
 
Then, with $\omega, \omega^{\pm}$ as in \S \ref{7setup},  
and cycles
$[Z(U)]^{\pm}$ as in \S \ref{psidef}, 
we have
   \begin{equation} \label{Frodo2}  \frac{  \langle \omega, \iota_* [Z(U)] \rangle^2 }{\langle \omega, \omega \rangle}  = \sqrt{q} \left(  \vol_{\tr} \ L_{\Pi}\right)^{-1}, \ \  q \in\Q. \end{equation} 
  \begin{equation} \label{Frodo3}\frac{ \langle \omega^{+}, \iota_* [Z(U)]^+ \rangle^2 } {\langle \omega^+,  \omega^+ \rangle}    \frac{ \langle \omega^{-}, \iota_* [Z(U)]^- \rangle^2 } {\langle \omega^-, \omega^- \rangle}    = \sqrt{q}    \left(  \vol_{\tr} \ L_{\Pi}\right)^{-2}, \ \  q \in\Q. \end{equation} 
where $[Z(U)]^{\pm}$  is as in \eqref{ZUpmdef}; the pairing $\langle \omega, \iota_* [Z(U)] \rangle$ is to be interpreted 
as in Remark \ref{Cusp non cusp}.

  Moreover, in case \eqref{Frodo2}:  if the central value of the Rankin-Selberg $L$-function attached to $\Pi$ is nonvanishing 
  and there exists a nonzero $\HH(\adele)$-invariant functional on the space of $\Pi$, 
  it is possible to choose the data $(g_{\infty}, g_f)$ and level structure $U$ 
in such a way that  $q \neq 0$.  A similar assertion holds for \eqref{Frodo3}, where we require the same conditions both for $\Pi$ and its twist $\Pi^{\psi}$ 
(see \eqref{Pipsidef}).

 \end{theorem}

    Note that \eqref{Frodo2} and \eqref{Frodo3} conform exactly to the prediction of the conjecture
 -- see \eqref{frodo} and \eqref{period_matrix}.     
   In an early draft of this paper, we attempted to eliminate the factor of $\sqrt{\Q^*}$ as far as possible,
   and indeed found that (to the extent we computed) the square classes all appear to cancel -- often in a rather interesting way.
   However, this makes the computation exceedingly  wearisome, and to spare both ourselves and our readers
   such pain,  we have omitted it from the present version of the paper.

   \proof 
   We will now give the proof of the Theorem, relying however on several computations that will be carried out in the next section.
   To treat the two cases uniformally, it will be convenient to use the following shorthand {\em for this proof only}:
      
   \begin{itemize}
   \item[-] For all cases  {\em except} $\PGL_n \times \PGL_{n+1}$ over $\Q$, we put $\pi = \Pi$.  The reader
   is advised to concentrate on this case, the modifications for the other case being straightforward but notationally complicated. 
   
\item[-] In the remaining case of $\PGL_n \times \PGL_{n+1}$ over $\Q$, we ``double'' everything. 
First of all, factor $\Pi = \Sigma_{\PGL_n} \boxtimes \Sigma_{\PGL_{n+1}}$ as an external tensor product
of an automorphic representation on $\PGL_n$ and an automorphic representation on $\PGL_{n+1}$, as we may. 

Now  define a new automorphic representation on $\PGL_n \times \PGL_{n+1}$ 
\begin{equation} \label{Pipsidef}  \Pi^{\psi} = \begin{cases}  (\Sigma_{\PGL_n}  \cdot \psi   ) \boxtimes \Sigma_{\PGL_{n+1}},   & n \in 2\Z. \\  \Sigma_{\PGL_n}  \boxtimes (\Sigma_{\PGL_{n+1}} \cdot \psi),  & \mbox{else}. \end{cases}   \end{equation}
  be the twist of $\Pi$  by the quadratic character $\psi$, i.e., we twist by $\psi \circ \det$ only on the {\em even-dimensional} factor
  so that the resulting automorphic representation remains on $\PGL$.

Now, put 
$$ \pi = \Pi \boxtimes \Pi^{\psi}$$  considered as an automorphic representation of $(\PGL_n \times \PGL_{n+1})^2$.
Observe that the adjoint motive attached to $\Pi^{\psi}$ is identified with the adjoint motive attached to $\Pi$; 
thus  $L_{\pi} = L_{\Pi} \oplus L_{\Pi}$. 

Finally  replace all the groups $\Ginfty, \Kinfty, H, \Uinfty$ by a product of two copies: thus  $\Ginfty =  (\PGL_n(\R) \times \PGL_{n+1}(\R))^2, H_{\infty} = \GL_n(\R) \times \GL_n(\R)$ and so on. 
   \end{itemize}

 We have proved in Theorem \ref{period_theorem} 
 that 
\begin{equation} \label{recall}  \mbox{left-hand side of \eqref{Frodo2}} \sim  \sqrt{q} c_f c_{\infty} 
 \frac{L (\frac{1}{2}, \pi; \rho)}{L (1, \Ad, \pi)},\end{equation} where $\rho$ 
 is the representation of the dual group of $\GG$ described in that theorem.  Note in particular that
 $c_f \in \sqrt{\Q^*}.$
 In the $(\PGL_n \times \PGL_{n+1})^2$ case,  the same result holds, replacing \eqref{Frodo2} by \eqref{Frodo3}, and  now taking $\rho$ to be the sum of  two copies of the  tensor product representations of the two factors.

So let us look at the right-hand side of \eqref{Frodo2} or \eqref{Frodo3}, according to which case we are in. 
 Lemma \ref{lem:vol-indep}, applied with $\tr$ the trace weak polarization  
  and $p$ an {\em arbitrary} weak polarization on $\Ad \pi$, implies      \begin{eqnarray*}
   \volume_{\widehat{\tr}}  \ H^1_{\cM}(\Ad^* \pi, \Q(1))  &\sim_{\Q^{\times}} &   L^*(0, \Ad  \pi)  \frac{ \vol_{\widehat{\tr}} \ H_{\B}(\left( \Ad^* \pi\right)_{\R}, \Q)     }{ \vol_{\widehat{\tr}}  \  F^{1} H_{\dR} (\Ad^* \pi) }
  \\ &  \sim_{\sqrt{\Q^*}}&
 \frac{   L^*(0, \Ad \  \pi)   }{ \vol_p  \  F^{1} H_{\dR} (\Ad \pi) } \end{eqnarray*}
 where we also used, at the last step, the fact that $\vol_{\widehat{\tr}}F^{1} H_{\dR} (\Ad^* \pi) \sim_{\sqrt{\Q^*}} \vol_{p} F^1 H_{\dR}(\Ad \pi)$,
 beause $\Ad \pi$ and $\Ad^* \pi$ are abstractly isomorphic and $\vol_{S} F^1$ is independent, up to $\sqrt{\Q^*}$, of the choice of weak polarization $S$ (again, Lemma \ref{lem:vol-indep}). 
 Using \eqref{tr-tr} and \eqref{recall},   we see that proving \eqref{Frodo2} or \eqref{Frodo3}  is equivalent to verifying 
\begin{equation} c_{\infty} 
 \frac{L (\frac{1}{2}, \pi; \rho)}{L (1, \Ad, \pi)}  \frac{   L^*(0, \Ad \  \pi)   }{ \vol_p  \  F^{1} H_{\dR} (\Ad \pi) }  \sim_{\sqrt{\Q^*}} 1. \end{equation} 
The functional equation means that
$L^*(0,\Ad,\pi) = \sqrt{\Delta_{\Ad}}\frac{L_{\infty}(1, \Ad, \pi)}{L_{\infty}(0, \Ad, \pi)}L(1, \Ad, \pi),$ where  $\Delta_{\Ad} \in \Q^*$ the conductor of the adjoint $L$-function; so, substituting the expression for $c_{\infty}$
from Theorem  \ref{period_theorem}, we must check
 \begin{equation} \label{ghastly}   
  \underbrace{\frac{L_{\infty}(\frac{1}{2}, \pi; \rho) }{ L_{\infty}(0, \Ad, \pi)}}_{\gamma_1'}  
 \underbrace{\left(  \frac{\Delta_K}{\Delta_U^2} \right)   \left( \frac{ \Delta_{G,\infty}}{\Delta_{H,\infty}^2} \right)   }_{\gamma_2'} 
  \cdot  \frac{  L(\frac{1}{2}, \pi; \rho) }{ \vol_p(F^1 H_{\dR})}
\in \sqrt{\Q}.
 \end{equation} 

Now  computing case-by-case (see  Table \ref{table:gamma1} below):   
\begin{equation}
\label{eqn:gamma2p}
 \gamma'_1 \sim_{\Q^*} (2\pi i )^{-m}, \gamma'_2 \sim_{\Q^*} 1,
\end{equation} 
 where 
  \begin{equation} \label{mdef}
  m = \begin{cases}  n(n+1), \qquad \text{ if } G = \PGL_n \times \PGL_{n+1};  \\  n(n+1), \quad \text{ if } G=\Res_{E/\Q} (\PGL_{n} \times \PGL_{n+1}); \\ 
  2n^2, \quad \text{ if } G=\Res_{E/\Q} (\SO_{2n} \times \SO_{2n+1});  \\ 2n(n+1) \quad \text{ if } G=\Res_{E/\Q} (\SO_{2n+1} \times \SO_{2n+2}).\end{cases}
  \end{equation} 
 
Moreover,   (assuming Deligne's conjecture \cite{Deligne}, which  is   a special case of Beilinson's conjecture):    \begin{equation} \label{eqn:GGPcomp}     \frac{L(\frac{1}{2}, \pi; \rho)}{\vol_p(F^1 H_{\dR}(\Ad \pi))}  \in \sqrt{\Q} \cdot (2 \pi i)^m. \end{equation} 
 with $m$ the same integer as above.   Equation \eqref{eqn:GGPcomp}  requires an argument, and is in fact  quite surprising: the numerator is related to the Rankin-Selberg $L$-function and the denominator to the adjoint $L$-function,
 and so it is not apparent they should cancel.   This is the surprising cancellation that we have referred to in the introduction, and  we prove it in the next section. 
 
 The final assertion of Theorem  \ref{MainPeriodTheorem}  follows immediately from the corresponding assertion in  Theorem \ref{period_theorem}.

\qed

{\small 	    \begin{table}[ht]
\caption{Collates data  about the various cases; repeatedly uses 
 $
 \sum_{i=1}^{m} i(m+1-i) =\frac{1}{6} m(m+1)(m+2).
 $ 
``same'' means ``same as the other entry in the same row.'' ``sym'' means ``extend by symmetry.''}
\label{table:gamma1}

\bigskip
 \begin{tabular}{|c||c|c|c|}
 \hline
$G_{\infty} $   	 						&   $\SO_{2n}  \times \SO_{2n+1} /_{\C}$ 									&  $\SO_{2n+1}  \times \SO_{2n+2} /_{\C}$  			\\ \hline 
$H_\infty$ 							& $\SO_{2n} /_{\C}$ 													& $\SO_{2n+1} /_{\C}$ 							\\ \hline 
 $(d_K+r_K)$   				 			& $4n^2+2n$ 														& $4n^2+6n+2$  								\\ \hline 
 $ (d_U+r_U)$  				 			& $2n^2$  														& $2n^2+2n$ 									\\ \hline
  $\Delta_K/\Delta_U^2$   	 				& $ (\sqrt{\pi})^{(2n)} $ 												& $(\sqrt{\pi})^{2n+2}$ 	 						\\ \hline
    $\Delta_{G,\infty}$  					 & $\prod_{i=1}^{n-1} \Gamma_{\C}(2i)^2\Gamma_{\C}(n)\Gamma_{\C}(2n)$ 		&$ \prod_{i=1}^{n} \Gamma_{\C}(2i)^2 \Gamma_{\C}(n+1) $ \\ \hline
$\Delta_{H, \infty}$   		 				 &  $\prod_{i=1}^{n-1} \Gamma_{\C}(2i) \Gamma_{\C}(n)$						 &     $\prod_{i=1}^n \Gamma_{\C}(2i)$ 				\\ \hline
  $\Delta_{G,\infty}/\Delta_{H,\infty}^2$   			& $\Gamma_{\C}(2n)/\Gamma_{\C}(n) = \pi^{-n}$  						&  $\Gamma_{\C}(n+1) = \pi^{-n-1}$ 					\\ \hline
  $L(1/2,\pi_\infty,\rho)$    					& $\sim_{\Q^*}\pi^{-\frac{1}{3} (2n-1)(2n)(2n+1) - n(n+1) }$ 					 & $\sim_{\Q^*} \pi^{-\frac{1}{3} (2n)(2n+1)(2n+2) - n(n+1)}$  \\ \hline
 $L^* (0,\pi_\infty, \Ad)$  					& $\sim_{\Q^*} \pi^{-\frac{8}{3}(n-1)n(n+1)+n^2-3n}$							& $\pi^{-\frac{4}{3} n(n+1)(2n+1) +n(n+1)}$			 \\ \hline
 $\frac{L(1/2,\pi_\infty,\rho)}{L^* (0,\pi_\infty, \Ad)}$  & $\pi^{-2n^2} $													& $\pi^{-2n(n+1)}$								\\ \hline
  $M$ 									& {\tiny $ (2n-2,0)^1, \ldots, (n,n-2)^1, (n-1,n-1)^2, \mbox{sym.}	$}			&  \tiny{$(2n,0), \ldots, (n+1,n-1), (n,n)^2,  \mbox{sym.}$ 	}	\\ \hline
 $N$ 									& {\tiny $(2n-1, 0), (2n-2, 1), \ldots, (0,2n-1)$ 	}							&\tiny{$  (2n-1, 0), (2n-2, 1), \ldots, (0,2n-1)$ 	}					\\ \hline
 $ M \otimes N$ 							&  {\tiny $(4n-3,0)^1, (4n-4, 1)^2 \ldots, (3n-1,n-2)^{n-1}$ }  				&  {\tiny $  (4n-1,0)^1, (4n-2, 1)^2,  \ldots, (3n,n-1)^n$ }				\\  
  										& {\tiny $ (3n-2,n-1)^{n+1}, \ldots , (2n-1, 2n-2)^{2n}, \mathrm{sym}. $} 		& {\tiny $ (3n-1,n)^{n+2}, \ldots , (2n, 2n-1)^{2n+1}, \mbox{sym.} $}		\\ \hline
{\tiny $L(s, \Res_{E/\Q} M \otimes N)$ }				& {\tiny $\left( \prod_{i=1}^{n-1} \Gamma_\C (s-i+1)^i \cdot \prod_{i=n}^{2n-1} \Gamma_\C (s-i+1)^{i+1}\right)^2$}
																								&	 {\tiny $\left( \prod_{i=1}^{n} \Gamma_\C (s-i+1)^i \cdot \prod_{i=n+1}^{2n} \Gamma_\C (s-i+1)^{i+1} \right)^2$		}																																							\\ \hline		
 $ \Ad(M)$ 								&   \mbox{see text} 											 & 	 \mbox{see text}												\\ \hline
 $ \Ad(N)$ 								&  	{\tiny $ (2n-1,1-2n)^1, (2n-2,2-2n)^1,  (2n-3), 3-2n)^2, $}		&\mbox{same}													\\
 										&	{\tiny $ ((2n-4), -(2n-4))^2, \dots, (1,-1)^n, (0,0)^n, \mbox{sym} )$}		 & 																\\ \hline
\hline
\end{tabular} 
  \label{inner3}

\bigskip
 \begin{tabular}{|c||c|c|}
 \hline
$G_{\infty} $   								& $\PGL_n  \times \PGL_{n+1} /_{\C}$  							&   $\left(\PGL_n \times \PGL_{n+1}/_{\R}\right)^2$ 								\\ \hline 
$H_\infty$ 								&  $\GL_n /_{\C}$  											& $(\GL_n/_{\R})^2$ 												\\ \hline 
 $(d_K+r_K)$   								& $2n^2+4n-2 $ 											& $2n^2+2n$													 	 \\ \hline 
 $ (d_U+r_U)$  			 					&  $n^2+n$  												&	$n(n-1) + 2[n/2]$  														\\ \hline
     $\Delta_K/\Delta_U^2$   					&  $(\sqrt{\pi})^{2n-2}$  										&  $\pi^{2n-2[n/2]}$											\\ \hline
  $\Delta_{G,\infty}$  							&  $ (\prod_{i=2}^{n} \Gamma_{\C}(i))^2  \Gamma_{\C}(n+1)$ 			& $ (\prod_{i=2}^{n} \Gamma_{\R}(i))^4  \Gamma_{\R}(n+1)^2$					\\ \hline
      $\Delta_{H, \infty}$   						&  $(\prod_{i=1}^{n} \Gamma_{\C}(i)) $ 							 &  $ \prod_{i=1}^{n} \Gamma_{\R}(i))^2$													   \\ \hline
 $\Delta_{G,\infty}/\Delta_{H,\infty}^2$ 			& $\Gamma_{\C}(n+1)/ \Gamma_{\C}(1)^2\sim  \pi^{1-n}$ 				&  $\Gamma_{\R}(n+1)^2/\Gamma_{\R}(1)^4 \sim \pi^{2([n/2]-n)}$						 							\\ \hline 
 $L(1/2,\pi_\infty,\rho)$ 						&  $\sim_{\Q^*}\pi^{-\frac{2}{3} n(n+1)(n+2)}$  						&  $\mbox{same}$						  				 	   \\ \hline
 $L^* (0,\pi_\infty, \Ad)$ 						& $\sim_{\Q^*} \pi^{-\frac{1}{3} n(n+1)(2n+1)}$						& $\mbox{same}$											\\ \hline
 $\frac{L(1/2,\pi_\infty,\rho)}{L^* (0,\pi_\infty, \Ad)}$ 	& $\sim_{\Q^*} \pi^{-n(n+1)}$ 									& $\mbox{same}$															\\ \hline 
 $M$ 									& $ (n-1,0), (n-2,1) , \ldots , (0,n-1)$ 									& $\mbox{same}$	 																\\ \hline
 $N$ 									&$ (n,0), (n-1,1) , \ldots , (0,n)$ 	 							& $\mbox{same}$	  																\\\hline
 $ M \otimes N$ 							& $ (2n-1,0)^1, (2n-2,1)^2, \ldots, (n,n-1)^{n},  \mbox{sym.}	$			 & 	$\mbox{same}$												\\ \hline
 $L(s, \Res_{E/\Q} M \otimes N)$ \footnote{$L(s, M \otimes N) L(s, M' \otimes N)$ in the $E=\Q$ case}
 										& $(\Gamma_\C (s)^1 \Gamma_\C (s-1)^2 \cdots \Gamma_\C (s-n+1)^{n})^2$&$\mbox{same}$													 \\ \hline
 $ \Ad(M)$ 								&  $ (n-1,1-n)^1, \ldots, (1,-1)^{n-1}, (0,0)^{n-1},   \mbox{sym.}	$				 & $\mbox{same}$													\\ \hline
 $ \Ad(N)$ 								&  $ (n,-n)^1, \ldots, (1,-1)^{n}, (0,0)^{n},    \mbox{sym.} $		& $\mbox{same}$													\\ \hline					
  $L(s, \Ad, \Pi)$							& $ \left(\prod_{i=1}^j \Gamma_\C (s+i)^{j+1-i} \cdot \prod_{i=1}^{j+1} \Gamma_\C (s+i)^{j+2-i} \right)^2$ 
 																								& $\mbox{same}$											 \\ \hline

\end{tabular} 	
 \end{table} 
  }

\newcommand{\yuk}{\mathrm{yuk}}
\newcommand{\bg}{\mathrm{big}}
\newcommand{\sml}{\mathrm{small}}
 \section{Hodge linear algebra related to the Ichino--Ikeda conjecture} \label{GGP2}

	In this section, we will prove   
	most brutally     \eqref{eqn:GGPcomp}
	from the prior section.  To recapitulate,  and unpack some notation,  
     this  asserts
    that, for an automorphic cohomological representation $\Pi$ of $\GG$ as in Theorem \ref{MainPeriodTheorem}, we have
     \begin{equation} \label{eqn:recapitulate}   \sqrt{\Q} (2 \pi i)^m   \ni \begin{cases}   
       \frac{L(\frac{1}{2}, \Pi)}{\vol_p(F^1 H_{\dR}(\Ad \Pi))},  \GG = \SO_n \times \SO_{n+1}/E \mbox{ or}   \\  
        \frac{L(\frac{1}{2}, \Pi)^2 }{\vol_p(F^1 H_{\dR}(\Ad \Pi))},  \GG = \PGL_n \times \PGL_{n+1}/E \mbox{ or} \\
       \frac{L(\frac{1}{2}, \Pi)^2}{\vol_p(F^1 H_{\dR}(\Ad \Pi))}  \frac{L(\frac{1}{2}, \Pi^{\psi})^2}{\vol_p(F^1 H_{\dR}(\Ad \Pi))},  \GG = \PGL_n \times \PGL_{n+1}/\Q \end{cases}  \end{equation} 
wher $m$ is in \eqref{mdef}, and, in the last equation, $\psi$ is a quadratic character as in \S \ref{psidef},
and $\Pi^{\psi}$ is as in \eqref{Pipsidef}.   In all cases the $L$-function above is now the Rankin--Selberg $L$-function. 

 This    will follow (as explained below) from
 \eqref{eqn:ggp-pglnR}, \eqref{eqn:ggp-pglnC}, \eqref{eqn:ggp-so2n}, \eqref{eqn:ggp-so2n+1}
in the four cases.

 We note that Yoshida \cite{Yoshida} has given an elegant ``invariant-theoretic'' framework for doing computations of the type that we carry out here.   However we will follow a fairly direct approach, along the lines
 taken by M. Harris \cite{HarrisAdjointMotive}. In any case the main point is similar: the  period invariants described
 in \S  \ref{generalities} behave quite well under functorial operations.

	\subsection{Preliminaries}
  In all the cases,  the group $\GG$ is  the product of two classical groups  
   $$\GG=\Res_{E/\Q} (G_1 \times G_2),$$
   where $E$ is either $\Q$ or a quadratic imaginary extension of $\Q$, and $G_1, G_2$ are reductive $E$-groups. 
   
   There is a choice of whether we take $G_1$ to be the larger or smaller group.   In the case of $\PGL_n \times \PGL_{n+1}$,  we 
  take $G_1 = \PGL_n, G_2 = \PGL_{n+1}$.  In the cases involving $\SO_n \times \SO_{n+1}$
  we take $G_1$ to be the even orthogonal group, $G_2$ to be the odd orthogonal group and $E$ the imaginary quadratic field. 
 Then we may factor $\Pi$ into automorphic representations $\pi_i$ on $G_i$:
  $$\Pi =  (\pi_1 \boxtimes \pi_2)$$ 
 
  {\em  We will often use the abbreviation $j=n-1$ in the $\PGL_n \times \PGL_{n+1}$ cases.}  
  
 First of all, let us describe how to construct a $\Q$-motive whose $L$-function agrees  with the $L$-function appearing in \eqref{eqn:recapitulate}: 
   
The dual groups of the algebraic $E$-groups $G_1$ and $G_2$  are classical groups, and as such their $c$-groups have a ``standard'' representation:
 standard on the dual group factor, and we fix the $\mathbb{G}_m$ factor
 so that the weight of the associated motive is given by  $n-1$ in the $\PGL_n$ cases and $k-2$ in the $\SO_k$ cases. 
   The reader is referred to Appendix \ref{sec:archpar} for more detail on these standard representations.

Conjecture  \ref{Conj Pimotive} of \S \ref{Pimotive Appendix} states that  attached to $\pi_1, \pi_2$ are 
systems of motives indexed by representations of the $c$-group;
in particular, attached to the ``standard representations'' just mentioned,
we get motives $M$  (for $\pi_1$) and $N$ (for $\pi_2$).  

Here a subtlety arises, similar to  that discussed in \S \ref{sec: fake fake}:  the morphisms from the motivic Galois group to the $c$-group
(from Conjecture \ref{Conj Pimotive})  \label{c group inner twist}
are not necessarily defined over $\Q$.  Thus, in general, we can construct the motives $M, N$ only with $\overline{\Q}$-coefficients, and they cannot
   be descended to $\Q$-motives. {\em For the moment, however, we suppose they can be descended to $\Q$-coefficients,} and 
   write $M$ and $N$ for the motives  with $\Q$-coefficients thus attached to
  $\pi_1$ 
 and $\pi_2$ respectively. 
This italicized assumption is not necessary: the argument can be adapted to the general case by using an auxiliary coefficient field;   for expositional ease we postpone this argument to   Sec. \ref{sec:mot-coeffs}. 

Proceeding under the italicized assumption for the moment, then, we obtain $E$-motives $M$ and $N$ with $\Q$-coefficients, whose $L$-functions coincide with the $L$-functions of the standard representations of $\pi_1$ and $\pi_2$, 
shifted by a factor of one-half of the weight of the motive.  
By computing the determinant of the standard representations, we verify 
\begin{equation} \label{detcomp1} \det(M) \simeq \Q(-n(n-1)/2) \mbox{ and}  \det(N) \simeq \Q(-n(n+1)/2)\end{equation}
in the $\PGL$ cases, and 
\begin{equation} \label{detcomp2}   \det(M) \simeq \Q(-2n(n-1))^{\chi}  \ \ (\SO_{2n}) \mbox{ and }   \det(N) \simeq \Q(-n(2n-1)) \ (\SO_{2n+1}).\end{equation}
where  $\chi$ is the quadratic character of $E$ that arises from the action on the Dynkin diagram of $\SO_{2n}$. 
These equalities will be used to evaluate period determinants attached to $M$ and $N$. 

Moreover, if we write  
  $$\bM = \Res_{E/\Q} (M\otimes N)$$
  then we have an equality of $L$-functions:
 \begin{equation}
 \label{eqn:lfunc-equal}
 L(s+\frac{1}{2},\Pi) = L(s+r,\bM) =L(s,\bM (r)),
 \end{equation}
where 
\begin{align*}
r = \begin{cases} & n,  \qquad \text{ if } G = \PGL_n \times \PGL_{n+1}  \\ 
& n,  \quad \text{ if } G=\Res_{E/\Q} (\PGL_{n} \times \PGL_{n+1}) \\ 
& 2n-1,  \quad \text{ if } G=\Res_{E/\Q} (\SO_{2n} \times \SO_{2n+1}) \\ 
& 2n,  \quad \text{ if } G=\Res_{E/\Q} (\SO_{2n+1} \times \SO_{2n+2}).
\end{cases}
\end{align*}

In the case $\PGL_n \times \PGL_{n+1}$ over $\Q$  it is also useful to note that
$$ L(s+\frac{1}{2}, \Pi^{\psi}) = L(s+r, \bM^{\psi}) = L(s, \bM^{\psi}(r))$$
with $\Pi^{\psi}$ as in \eqref{Pipsidef} and one can express $\bM^{\psi}$ either as  $M^{\psi} \otimes N$ or $M \otimes N^{\psi}$;
here, in all cases, the superscript $\psi$ on a motive means that we tensor by the one-dimensional Artin motive corresponding to $\psi$.   In general twisting by $\psi$  can change the determinant, 
so that the twisted motive $M^\psi$ (or $N^\psi$) may only correspond to an automorphic form on 
$\GL_n$ (or $\GL_{n+1}$) rather than $\PGL_n$ (or $\PGL_{n+1}$); however this does not affect the computations below, and because of $M^{\psi} \otimes N = M \otimes N^{\psi}$
we can freely twist whichever factor is most convenient for the computation. 

We will  freely use the  $c^+, c^-, \delta$  periods 
of a $\Q$-motive; these are defined in \S \ref{deltacc}.

To avoid very heavy notation, we shall write: 
\begin{align*} \mathcal{L}_M &= F^1 H^0_{\dR} (\Res_{E/\Q} \Ad (M)),\\
 \mathcal{L}_N &= F^1 H^0_{\dR} (\Res_{E/\Q} \Ad (N)).  
 \end{align*} 
These are $\Q$-vector spaces. If the motives in question are equipped with a weak polarization, we may compute the volumes of $\mathcal{L}_M, \mathcal{L}_N$
 according to this polarization, as in \S \ref{volLcomp}.  However, these volumes can be defined
 intrinsically, as in the proof of Lemma \ref{lem:vol-indep}. Thus if we write $\vol(\mathcal{L}_M)$
 without specifying a polarization, we mean the class in $\C^*/\sqrt{\Q^*}$ defined as in that Lemma. 
 
 Also, observe that the adjoint motives for $M$ and $M^{\psi}$ are canonically identified, so we do not need to distinguish
 between $\mathcal{L}_M$ and $\mathcal{L}_{M^{\psi}}$. 
 
 We note that  the adjoint motive for $\Pi$ is identified with $\Res_{E/\Q} \Ad(M) \oplus \Res_{E/\Q} \Ad(N)$, and so  
  \begin{equation}
 \label{eqn:volMN}
 \vol (F^1 H^0_{\dR} (\Ad \Pi))= \vol (\mathcal{L}_M) \vol (\mathcal{L}_N),
 \end{equation} 
 the equality being of complex numbers up to $\sqrt{\Q^*}$. 
Moreover assuming Deligne's conjecture \cite{Deligne}, which is a special case of Beilinson's conjecture,  for the motive $\bM(r)$, we have:  
\begin{equation}
\label{eqn:deligne-conj}
\frac{L(0,\bM(r))}{c^+ (\bM(r))} \in \Q. 
\end{equation}
 
 Now combining \eqref{eqn:volMN}, \eqref{eqn:lfunc-equal}
 and \eqref{eqn:deligne-conj}, we see that the sought after relation \eqref{eqn:recapitulate} reduces 
 to a relation between $c^+ (\bM(r))$, $\vol (\mathcal{L}_M)$, $\vol (\mathcal{L}_N)$, namely
 \begin{equation} 
\label{eqn:condensate} 
 \frac{c^+(\bM(r))^{e}}{ \vol(\mathcal{L}_M) \vol(\mathcal{L}_N)  } \sim_{\sqrt{\Q^*}} (2\pi i )^{m}
\end{equation}
in the $\SO$ cases (with $e=1$) or the $\PGL$ over $E$ case (with $e=2$), or in the remaining case:
   \begin{equation} 
\label{eqn:condensate2} 
 \frac{c^+(\bM(r))^2}{ \vol(\mathcal{L}_M) \vol(\mathcal{L}_N)  }
  \frac{c^+(\bM^{\psi}(r))^2}{ \vol(\mathcal{L}_M) \vol(\mathcal{L}_N) }  \sim_{\sqrt{\Q^*}} (2\pi i )^{m}.
\end{equation}

We verify these statements case-by-case in 
 \eqref{eqn:ggp-pglnR}, \eqref{eqn:ggp-pglnC}, \eqref{eqn:ggp-so2n}, \eqref{eqn:ggp-so2n+1} below.

\subsection{Period invariants of motives} \label{generalities}
   
  Our proof of \eqref{eqn:condensate} and \eqref{eqn:condensate2} will be to write both sides in terms of certain elementary ``period invariants''
  attached to the motives $M$ and $N$.     Such period invariants have been previously considered by   M. Harris \cite{HarrisCrelle}. 
  More precisely we attach an invariant $\cQ_p \in \mathbb{C}^*/E^*$   to the motive $M$, any integer  $p$ for which $F^pH_{\dR}(M)/F^{p+1} H_{\dR}(M)$ is one-dimensional,
  and an embedding $\sigma: E \hookrightarrow \C$.  The main point is that these invariants  behave quite nicely under functorial operations. 
   
  \subsubsection{Bases}

  Let $M$ be a pure motive over $\Q$ of weight $j$ and let $V$ denote the $\Q$-Hodge structure $H^*_B(M_\C, \Q)$.  
Let $V_\C = \oplus_p V^{p,j-p}$ be the Hodge decomposition. We have natural isomorphisms 
\begin{equation}  \label{Qstruc}
V^{p,j-p} \simeq \frac{F^p H^*_{\dR} (M)}{F^{p+1} H^*_{\dR} (M)} \otimes \C.
\end{equation} 
This isomorphism gives a $\Q$-structure on $V^{p,j-p}$, namely that coming from $\frac{F^p H^*_{\dR} (M) }{F^{p+1} H^*_{\dR} (M)}$.

\begin{lemma*}  
Let $\omega_p$ be any element of $V^{p,j-p}$ that is $\Q$-rational for the $\Q$-structure defined above. Then $c_{\dR} (\omega_p) = \omega_p$. Equivalently,
$F_\infty (\omega_p) =c_{\B}(\omega_p)= \bar{\omega}_p$. 
\end{lemma*}

\proof  
The element $\omega_p$  corresponds via the isomorphism above to an element $\omegat_p$ in $F^p H^j_{\dR} (M) $ that is well defined up to elements of $F^{p+1} H^j_{\dR} (M)$. Let us fix once and for all such an $\omegat_p$ so that 
$$ \eta_p:=\omega_p - \omegat_p \in F^{p+1} H^j_{\dR} (M) \otimes \C.$$
Then
  $$c_{\dR} \omega_p - \omega_p = c_{\dR} \eta_p - \eta_p \in  F^{p+1} H^j_{\dR} (M) \otimes \C $$
  (since $c_{\dR}$ preserves the Hodge filtration).  Since $c_{\dR}$ preserves
  the spaces $V^{p,j-p}$ and $V^{p, j-p}$ injects into $H^j_{\dR}(M) /F^{p+1} H^j_{\dR}(M) \otimes \C$, 
  we deduce that $c_{\dR} \omega_p = \omega_p$, as claimed.  
  \qed

  \subsubsection{Motives over $E$} \label{E motive}
   Now suppose that $M$ is a motive over $E$; for this subsection, suppose that $E$ is an imaginary quadratic field.
   
   Let $\sigma$ denote the 
given embedding of $E$ in $\C$ and $\bar{\sigma}$ the complex conjugate of $\sigma$. Then the interaction between the Betti-de Rham comparison isomorphisms and complex conjugation is described by the commutativity of the following diagram:

\xymatrix{
& & & H^*_{\dR}(M) \otimes_{E,\sigma} \C  \ar[r]^{\simeq}_{\varphi_\sigma} \ar[d]_{c_{\dR}} & H^*_B(M_{\sigma,\C}) \otimes \C =: V_\sigma  \ar[d]^{F_\infty \cdot c_{\B}}  \\
& & & H^*_{\dR}(M) \otimes_{E,\sigmab} \C  \ar[r]^{\simeq}_{\varphi_{\bar{\sigma}}} & H^*_B(M_{\sigmab,\C}) \otimes \C =: V_{\bar{\sigma}}  \\
}

Here $c_{\dR}$ is complex conjugation on the second factor, 
$c_{\B}$ is complex conjugation on the second factor, 
$F_\infty$ denotes the map on $H^*_B$ induced 
by complex conjugation on the underlying analytic spaces, and $\varphi_\sigma$, $\varphi_{\bar{\sigma}}$ denote the comparison isomorphisms. 
For $\omegat$ any element in $H^*_{\dR} (M)$, we denote 
by $\omegat^\sigma$ and $\omegat^{\bar{\sigma}}$ the images of $\omegat$ under $\varphi_\sigma$ and $\varphi_{\bar{\sigma}}$ respectively.

Note that 
$$
F_\infty: V_{\sigma}^{p,q} \rightarrow V_{\bar{\sigma}}^{q,p}, \quad c_{\B}: V_{\sigma}^{p,q} \rightarrow V_{\sigma}^{q,p}
\mbox{ and so }
F_\infty c_{\B} : V_{\sigma}^{p,q} \rightarrow V_{\bar{\sigma}}^{p,q}.
$$

Now the map
$$ \varphi_\sigma: F^p H^*_{\dR} (M) \otimes \C \rightarrow \oplus_{i \ge p} V_\sigma^{i, j-i} $$
induces an isomorphism
\begin{equation}
\label{eqn:dRB-isom-sigma}
\frac{F^p H^*_{\dR} (M)}{F^{p+1} H^*_{\dR} (M)} \otimes \C \simeq V_\sigma^{p,j-p},
\end{equation}
and likewise with $\sigma$ replaced by $\bar{\sigma}$.

Next, we discuss how  restriction of scalars (taken in the sense of \cite[Example 0.1.1]{Deligne}) interacts with cohomology.   If $M$ is any motive over $E$, and $\bM := \Res_{E/\Q} (M)$, then 
\begin{equation} \label{restriction}H^*_{\dR} (\bM) = H^*_{\dR} (M),\end{equation}
viewed as a $\Q$-vector space, and 
\begin{equation} \label{Betti res}
H^*_B(\bM) = H^*_{B,\sigma} (M) \oplus H^*_{B,\bar{\sigma}} (M).
\end{equation}

  \subsubsection{Standard elements $\omegat, \omega$.} \label{NIST}
  We return to allowing $E$ to be either $\Q$ or a quadratic imaginary field.
    
    Now, we will use the following notation. For the various $E$-motives $M$ we will consider, 
    let $p$ be any integer  such that   $\dim F^p/F^{p+1}=1$ and $p^*$ the dual integer, so that $p+p^*$ equals the weight of $M$.
    
 We denote by
  $$\omegat_p \in F^p H^*_{\dR}(M)$$  any element that spans the one-dimensional quotient $F^p H^*_{\dR} (M) /F^{p+1} H^*_{\dR}(M)$.
  For  $\sigma: E \hookrightarrow \C$ an embedding we define 
 $$\omega_p^{\sigma} \in H_{\B}(M_{\sigma}, \C)^{p,p^*}$$ the element corresponding to $\omegat_p$ via the 
 isomorphism \eqref{eqn:dRB-isom-sigma}.
 If $E =\Q$ we will omit the $\sigma$.  Observe that
\begin{equation} \label{Gove} F_\infty c_{\B} (\omega_p^\sigma) = \omega_p^{\bar{\sigma}}.\end{equation}
  
Whenever $\omega_p$ and $\omega_{p^*}$ are defined, we define complex scalars $\cQ_p^{\sigma}$ by the rule 
 \begin{equation} \label{rel1} \overline{\omega^\sigma_p}   = c_{\B}(\omega^\sigma_p)  = \cQ_p^{\sigma}  \omega^\sigma_{p^*} \cdot  \begin{cases} 1, & p < p^* \\ 1 = (-1)^w, & p =p^*, \\  (-1)^{w}, & p > p^*  \end{cases}.  \end{equation}
Observe that
\begin{equation} \label{rel2}  \overline{\cQ_p^{\sigma}} \cQ_{p^*}^{\sigma} = (-1)^w.\end{equation}
This invariant is compatible with complex conjugation: 
 \begin{lemma} \label{821lemma}
$\overline{\cQ_p^{\sigma}} = \cQ_p^{\bar{\sigma}}$.   
\end{lemma} 
\proof
 We have for $p \leq p^*$
$$  
F_\infty (  \cQ_p^\sigma \omega_{p^*}^\sigma) \stackrel{\eqref{Gove}}{=}   \omega_p^{\bar{\sigma}} \quad \text{ and } \quad  F_\infty (  \cQ_{p^*}^{\bar{\sigma}} \omega_{p}^{\bar{\sigma}}) \stackrel{\eqref{Gove}}{=}  (-1)^w \omega_{p^*}^{\sigma}. $$ 
which together imply that  $\cQ_p^{\sigma} \cQ_{p^*}^{\bar{\sigma}}=(-1)^w$; now compare with\eqref{rel2}. \qed

 As a result, we will sometimes write
 $$ |\cQ_p|^2 = \cQ_p^{\sigma} \cdot \cQ_p^{\sigmabar},$$
 noting that the right-hand side doesn't depend on $\sigma$, and equals $1$  if $p=p^*$.  
 In particular, in the case when $E=\Q$ so that $\cQ_p^{\sigma} =\cQ_p^{\sigmabar}$ we have $\cQ_{p}=\pm 1$
 in middle dimension $p=p^*$.

  \subsection{The case of $\PGL_n \times \PGL_{n+1}$ over $\Q$} \label{PGLnQ}

In this case, the dimension of each graded piece of the Hodge filtration, for both $M$ and $N$, equals $1$. 
Recall that we write $j=n-1$. 
Therefore, let  $\omega_i, \omegat_i$, $0\le i \le j$
be the standard elements attached to $M$,  as in \S \ref{NIST}, and $\cQ_i$, $0\le i \le j$ the 
associated quadratic period invariants, as in \S \ref{NIST}. The corresponding elements attached to $N$ will be denoted 
$\eta_i, \etat_i, \cR_i$ for $0\le i \le j+1$. 

We may form the dual bases $\omegat_p^{\vee} \in H_{\dR}(M^{\vee}) = H_{\dR}(M)^{\vee}$ and $\omega_p^{\vee} \in H_{B}(M^{\vee},\C) = H_{\B}(M, \C)^{\vee}$, defined as usual by the rule
$$ \langle \omegat_a, \omegat_b^{\vee} \rangle = \delta_{ab}.$$ 
 Then 
  $\omegat_p^{\vee}$ gives a basis for $F^{-p} H_{\dR}(M^{\vee})/F^{1-p} H_{\dR}(M^{\vee})$
and is associated to the element $\omega_p^{\vee} \in H^{-p, -p^*}(M^{\vee},\C)$ under the isomorphism \eqref{Qstruc}, but now for $M^{\vee}$.
Defining period invariants $\cQ^{\vee}$ for $M^{\vee}$ using this basis, we get
$$\cQ_{p}^{\vee} = \pm \cQ_{p^*}.$$
Write  $\omega_{p,q}  = \omega_p \otimes \omega_{j-q}^{\vee} \in H^0_{B}(M \otimes M^{\vee},\C)$ and $\omegat_{p,q} = \omegat_p \otimes \omegat_{j-q}^{\vee} \in H^0_{\dR}(M \otimes M^{\vee})$.   

The subspace $F^1 H^0_{\dR} (\Ad (M))$ has as a $\Q$-basis the elements
 \begin{equation}
 \label{eqn:basisforF1}
  \omegat_{p,q}, \quad  \ p+q \ge j+1.
  \end{equation}
 
  Recall, from the proof of  Lemma \ref{lem:vol-indep}, that the square of $\vol F^1 H_{\dR}(\Ad M)$ can be computed via computing the image
  of a generator of $\det F^1 H_{\dR}(\Ad M)$ under the complex conjugation map
  to $\det (H_{\dR}/F^0 H_{\dR})$.  (See in particular \eqref{lambdadef}). 
   In the case at hand, a generator for $\det F^1 H_{\dR}$ is given by
  $$ \bigwedge_{p + q \geq j+1}\omegat_{p,q} = \bigwedge_{p+q \geq j+1} \omega_{p,q},$$
  and its complex conjugate is given by 
  $$ \left( \prod_{p+q \geq j+1} \cQ_p \cQ_q \right) \bigwedge_{p+q \geq j+1} \omega_{p^*, q^*}  =  \left( \prod_{p+q \geq j+1} \cQ_p \cQ_q \right)  \bigwedge_{p+q \geq j+1} \omegat_{p^*, q^*}$$
  where the last equality is valid in the determinant of $H^0_{\dR}(\Ad M)/F^0 H_{\dR}$.  Therefore
 
\begin{equation} \label{LMvolQ}
 \vol(\mathcal{L}_M) \sim_{\sqrt{\Q^*}} \prod_{0\le p \le j} \mathcal{Q}_p^{p}.
 \end{equation} 
  Likewise, for $N$, we get:
\begin{equation} \label{LNvolQ}
\vol(\mathcal{L}_N) \sim_{\sqrt{\Q^*}} \prod_{0\le p \le j+1} \mathcal{R}_p^{p}.
\end{equation}

Now $M\otimes N$ has a unique critical point, namely $s=j+1$.
We will now compute {\it square} of the period $$c^+ (M\otimes N( j+1))$$ 
in the case $j=2t$ is even; the case $j$ odd is exactly similar.

We first note that since $M$ is attached to a form on $\PGL_n$,  $F_\infty$ acts on $H^{tt}(M)$ by $+1$.  
Let $e_0^+, \ldots ,e_t^+$ be a $\Q$-basis of $H_B(M)^+$ and $e_{t+1}^-, \ldots, e_{2t}^-$ a $\Q$-basis of 
$H_B(M)^-$; here $+$ and $-$ refer to the $F_{\infty}$-eigenvalue. Then 
$$
(e_0^+  \cdots  e_t^+ \ e_{t+1}^- \cdots e_{2t}^- ) = (\omega_0 \cdots \omega_t  \ \omega_{2t} \cdots \omega_{t+1} ) 
\begin{pmatrix} A_M & B_M \\ C_M & D_M \end{pmatrix} 
$$ 
where $A_M$, $B_M$, $C_M$ and $D_M$ are of sizes $(t+1)\times (t+1)$, $(t+1)\times t$, $t\times (t+1)$ and $t\times t$ respectively. 
Likewise let $f_0^+, \ldots, f_t^+$ be a $\Q$-basis of $H_B(N)^+$ and $f_{t+1}^-,\ldots, f_{2t+1}^-$ a $\Q$-basis of $H_B(N)^-$. 
Then 
$$
(f_0^+  \cdots  f_t^+ \ f_{t+1}^- \cdots f_{2t+1}^- ) = (\eta_0 \cdots \eta_t  \ \eta_{2t+1} \cdots \eta_{t+1} ) 
\begin{pmatrix} A_N & B_N \\ C_N & D_N \end{pmatrix} 
$$ 
where $A_N$, $B_N$, $C_N$ and $D_N$ all have size $(t+1)\times (t+1)$. Note that the $i$th row of $C_M$ (resp. of $D_M$) is equal to 
$\mathcal{Q}_i$ (resp. $-\mathcal{Q}_i$) times the $i$th row of $A_M$ (resp. of $B_M$). Likewise the $i$th row of $C_N$ (resp. of $D_N$) is equal to 
$\mathcal{R}_i$ (resp. $-\mathcal{R}_i$) times the $i$th row of $A_N$ (resp. of $B_N$).

Let us compute both $c^\pm (M\otimes N)$ in terms of $c^\pm(M)$ and $c^\pm (N)$. 
Since $H_B (M \otimes N)^+ = (H_B(M)^+ \otimes H_B(N)^+) \oplus (H_B(M)^- \otimes H_B(N)^-)$ and
(with notation $F^{\pm}$ as in \S \ref{deltacc})  
$$F^{\pm} H_{\dR}(M \otimes N) = \oplus_{p+q \ge j+1} \Q \cdot \omega_p \otimes \eta_q,$$ we get
$c^+(M\otimes N) = \det (X)$, where 
$$
\begin{pmatrix} (e_i^+ \otimes f_{k}^+)_{i,k} & (e_{i'}^- \otimes f_{k'}^- )_{i',k'} \end{pmatrix} = \begin{pmatrix}( \omega_p \otimes \eta_q)_{p,q} & (\omega_{p'} \otimes \eta_{q'})_{p',q'} \end{pmatrix} X,
$$
and the indices $i,k,i',k', p,q$ range over $0\le i \le t$, $0\le k \le t$, $2t \ge i' \ge t+1$, $2t+1 \ge k' \ge t+1$, $0\le p\le t$, $0\le q \le t$ and $(p',q')$ ranges over pairs such that $p'>t$ or $q'>t$
but $p'+q'\le 2t$. 
Note that if $p'>t$ then $q'\le t$ and $0\le 2t-p' < t$. Likewise, if $q'>t$, then $p'\le t$ and $0\le 2t+1-q' \le t$. 
Let $A_M^*$ and $B_M^*$ be the matrices obtained from $A_M$ and $B_M$ by deleting the last row. 
Using the relations $\omega_{2t-p} =   \mathcal{Q}_p^{-1} F_\infty( \omega_p)$ and $\eta_{2t+1-q}=\mathcal{R}_q^{-1} F_\infty( \eta_q)$, we see that 
\begin{eqnarray*}
c^+ (M\otimes N) &=& \prod_{0\le p <t} \mathcal{Q}_p^{p+1} \prod_{0\le q \le t} \mathcal{R}_q^{q} \cdot \det \begin{pmatrix} A_M \otimes A_N & B_M \otimes B_N \\ A_M^* \otimes A_N & - B_M^* \otimes B_N \end{pmatrix} \\
&=&\prod_{0\le p <t} \mathcal{Q}_p^{p+1} \prod_{0\le q \le t} \mathcal{R}_q^{q} \cdot \det (A_M\otimes A_N) \cdot \det(-2 B_M^* \otimes B_N) \\
&\sim_{\Q^*} &  \prod_{0\le p <t} \mathcal{Q}_p^{p+1} \prod_{0\le q \le t} \mathcal{R}_q^{q} \cdot  \det(A_M)^{t+1} \det (A_N)^{t+1} \det(B_M^*)^{t+1} \det(B_N)^t \\
&=&  \prod_{0\le p <t} \mathcal{Q}_p^{p+1} \prod_{0\le q \le t} \mathcal{R}_q^{q} \cdot (c^+(M) c^-(M))^{t+1} \cdot c^+(N)^{t+1} c^-(N)^t.
\end{eqnarray*}
Now 
\begin{align*}
\delta(M) &= \det \begin{pmatrix} A_M & B_M \\ C_M & D_M \end{pmatrix}  = \prod_{0\le p <t} \mathcal{Q}_p \cdot \det \begin{pmatrix} A_M & B_M \\ A_M^* & -B_M^* \end{pmatrix} \\
&=\prod_{0\le p <t} \mathcal{Q}_p \cdot \det \begin{pmatrix} A_M & B_M \\ 0 & -2B_M^* \end{pmatrix} \\
&\sim_{\Q^*} \prod_{0\le p <t} \mathcal{Q}_p \cdot c^+(M) c^-(M).
\end{align*}
Likewise,
\begin{equation}
\label{eqn:deltan}
 \delta(N) \sim_{\Q^*} \prod_{0\le q \le t} \mathcal{R}_q \cdot c^+(N) c^-(N).
\end{equation}
Thus, up to $\Q^*$ factors, $c^+ (M\otimes N) $ equals   
{\small \begin{align*}
  \prod_{0\le p <t} \mathcal{Q}_p^{p+1} \prod_{0\le q \le t} \mathcal{R}_q^{q} \cdot \left(\delta(M) \prod_{0\le p <t} \mathcal{Q}_p^{-1}\right)^{t+1} \cdot \left(\delta(N) \prod_{0\le q \le t} \mathcal{R}_q^{-1}\right)^{t}\cdot c^+(N)\\
 = \delta(M)^{t+1} \delta(N)^t \cdot \prod_{0\le p <t} \mathcal{Q}_p^{p-t} \cdot \prod_{0\le q \le t} \mathcal{R}_q^{q-t} \cdot c^+(N)
\end{align*}}
We will also need the same result when we do not assume that $F_{\infty}$ acts on $H^{tt}(M)$ as $+1$,
for example if we replace $M$ by $M \otimes \psi$.
 A similar computation 
 shows: 
 \begin{prop}
 Suppose that $\chi$ is a quadratic idele character for $\Q$; write $\mathrm{sign}(\chi) = \pm 1$ according to whether $\chi$ is trivial or not on $\mathbf{R}^*$. 
Then $$ c^\pm (M^{\chi} \otimes N) \sim_{\Q^*}\delta(M)^{t+1} \delta(N)^t \cdot \prod_{0\le p <t} \mathcal{Q}_p^{p-t} \cdot \prod_{0\le q \le t} \mathcal{R}_q^{q-t} \cdot c^{ \pm \mathrm{sign}(\chi)} (N) $$
\end{prop}

Let $R^\pm (M,N)$ be the ratio defined by:
$$
R^\pm (M,N):= \frac{c^{\pm} (M\otimes N(j+1))^2}{\vol (F^1(\Ad (M))) \cdot \vol (F^1 (\Ad (N)))}.
$$
Since $j+1=2t+1$ is odd, we have 
$$ 
c^\pm (M\otimes N (j+1)) = (2\pi i )^{\frac{1}{2}(j+1) \cdot \mathrm{rank} (M\otimes N)} c^{\mp} (M\otimes N) =(2\pi i )^{\frac{1}{2}(j+1)^2(j+2)} c^{\mp} (M\otimes N).
$$
Therefore,  the Proposition above,  together with the properties of period invariants given in  \eqref{rel2} and Lemma  \ref{rel2}, and the evaluations  \eqref{LMvolQ} and \eqref{LNvolQ}  of the volumes of the $\mathcal{L}$ subspaces, give
$$
R^\pm (M^{\psi},N) \sim_{\Q^*} (2\pi i )^ {(j+1)^2 (j+2)}  \delta(M^{\psi})^{2t+2} \delta(N)^{2t} \cdot  \prod_{0\le q \le t }\mathcal{R}_q \cdot c^{\mp \mathrm{sign}(\chi)} (N)^2. 
$$ 
By \eqref{detcomp1} we have 
$$
\delta(M)^2 \mbox{ and }\delta(M^{\psi})^2 \in (2\pi i)^{-j(j+1)} \cdot (\Q^*)^2 \quad \text{and} \quad \delta(N)^2 \in (2\pi i )^{-(j+1)(j+2)} \cdot \Q^*,
$$
where the computation for $\delta(M^{\psi})$ comes from \cite[Proposition 6.5]{Deligne}. 
We now  get from   \eqref{eqn:deltan} our desired result, namely, if $\psi$ has sign $-1$, then 
  \begin{align} \label{eqn:ggp-pglnR}
 R^\pm (M,N)\cdot R^\pm (M^{\psi},N) \sim_{\Q^*} (2\pi i )^{(j+1)(j+2)}.
\end{align}

  \subsection{The case $\PGL_n \times \PGL_{n+1}$ over imaginary quadratic $E$}

Again  $M$ has weight $j$ and rank $j+1$.  Just as in the prior case, each graded step of the Hodge filtration has dimension $1$, both for $M$ and for $N$.

Let $\omegat_0, \ldots, \omegat_j$ be a $E$-basis for $H^*_{\dR} (M)$, 
chosen as in \S \ref{NIST}, and with associated invariants $\cQ_p^{\sigma}$ as in  \S \ref{NIST}. 
Just as at the start of \S \ref{PGLnQ}, but now keeping track of embeddings, we  form $\omega_p^{\sigma} \in H^*_B(M_{\sigma},\C)$,
and also the dual bases $\omegat^{\vee}_p, \omega_p^{\vee, \sigma}$, and put
$$ \omegat_{p,q} = \omegat_p \otimes \omegat^{\vee}_{j-q} \in H^*_{\dR}(M \otimes M^{\vee})$$
and similarly $\omega_{p,q} \in H^*_{\B}(M_{\sigma} \otimes M_{\sigma}^{\vee}, \C)$.

   We may compute the volume of $\mathcal{L}_M$ in a very similar way to the previous discussion. 
  In the case at hand, a generator for $\det F^1 H_{\dR}$ is given by
  $$ \bigwedge_{p + q \geq j+1}\omegat_{p,q}  \wedge \sqrt{-D} \omegat_{p,q} \sim \bigwedge_{p+q \geq j+1} (\omega_{p,q}^{\sigma}, \omega_{p,q}^{\bar{\sigma}}) \wedge (\sqrt{-D} \omega_{p,q}^{\sigma}, -\sqrt{-D} \omega_{p,q}^{\overline{\sigma}}),$$
  where we used the isomorphism  from \eqref{Betti res} to go from left to right. 
The  complex conjugate of the above element is  given by 
  $$ \left( \prod_{p+q \geq j+1} |\cQ_p|^2  |\cQ_q|^2 \right)  \underbrace{ \bigwedge_{p+q \geq j+1}  \left( \mbox{ same, replacing $p,q$ by $p^*, q^*$}\right)}_{\sim_{\Q^*} \det(H_{\dR}/F^1 H_{\dR})}.$$

Therefore,  
\begin{equation} \label{LMCvol}
\vol (\mathcal{L}_M) \sim_{\sqrt{\Q^*}}   \prod_{0\le p \le j} |\cQ_p|^{2p}. 
\end{equation}
There is an identical expression for the volume of $\mathcal{L}_N$,  simply replacing $j$ by $j+1$ and $\cQ$ by $\cR$.

{\em For the remainder of this subsection, we fix an embedding $\sigma: E \hookrightarrow \C$,  
and when we write $\cQ, \cR$ etc. we mean $\cQ^{\sigma}, \cR^{\sigma}$, etc.}

We shall now compute the
Deligne periods $c^\pm (\Res_{E/\Q} (M\otimes N))$. 
Instead of using the basis consisting of $\tilde{\omega}_i^\sigma$, we can work with the $\omega_i^\sigma$. 
Suppose that  $A$ is the $(j+1) \times (j+1)$ complex matrix defined by 
\begin{equation}
\label{eqn:comp-glnc-1}
(e_0 \cdots e_j) = (\omega_0^\sigma \cdots \omega_j^\sigma) \cdot A.
\end{equation}
Note that this depends on the choice of $\sigma$, but we fixed one above.

 Note that
$$
F_\infty c_{\B} \cdot e_i = F_\infty e_i \quad \text{ and} \quad F_\infty c_{\B} \cdot \omega_i^\sigma = \omega_i^{\bar{\sigma}}.
$$
Thus applying $F_\infty c_{\B}$ to \eqref{eqn:comp-glnc-1}, we get
$$
(F_\infty e_0 \cdots F_\infty e_j) = (\omega_0^{\bar{\sigma}} \cdots \omega_j^{\bar{\sigma}}) \cdot \bar{A}.
$$
Likewise, let $f_0, \ldots, f_{j+1}$ denote a basis of $H^*_{B,\sigma} (N)$ and let $B$ be the $(j+2)\times (j+2)$  complex matrix defined by
\begin{equation}
\label{eqn:comp-glnc-2}
(f_0 \cdots f_{j+1}) = (\eta_0^\sigma \cdots \eta_{j+1}^\sigma) \cdot B,
\end{equation}
where $(\eta_0, \ldots, \eta_{j+1})$ is a $E$-basis for $H^*_{\dR} (M)$. Note that 
\begin{equation} 
\label{eqn:ab-bar}
\overline{a_{t,i}} = \cQ_t^{-1} a_{j-t,i} \quad \text{and} \quad \overline{b_{t',i'}}=\mathcal{R}_{t'}^{-1} b_{j+1-t',i'}.
\end{equation}
where we repeat that $\cQ_t$ really means $\cQ_t^{\sigma}$, with the same choice of $\sigma$ as fixed above. 

Now we need to compute the change of basis matrix $X$ between the bases:
\begin{equation}
\label{eqn:ii'}
e_i \otimes f_{i'} \pm  F_\infty (e_i \otimes f_{i'}) , \quad 0\le i \le j, 0\le i' \le j+1 
\end{equation}
and
\begin{equation}
\label{eqn:tt'}
(\varphi_\sigma, \varphi_{\bar{\sigma}} ) (\omega_t \otimes \eta_{t'}), \ \ (\varphi_\sigma, \varphi_{\bar{\sigma}} ) (\sqrt{-D} \omega_t \otimes \eta_{t'}),  0\le t+t' \le j
\end{equation}
of the complex vector spaces 
$$ (H_B(\Res_{E/\Q} (M\otimes N)) \otimes \C ) ^\pm \simeq H^*_{\dR} (\Res_{E/\Q} (M\otimes N) \otimes \C )/ F^{\mp}.$$ 
Note that
$$ 
(\varphi_\sigma, \varphi_{\bar{\sigma}} ) (\omega_t \otimes \eta_{t'}) = (\omega_t^\sigma \otimes \eta_{t'}^\sigma, \omega_t^{\bar{\sigma}} \otimes \eta_{t'}^{\bar{\sigma}} ),$$
while
$$(\varphi_\sigma, \varphi_{\bar{\sigma}} ) (\sqrt{-D} \cdot \omega_t \otimes \eta_{t'})  =   \sqrt{-D} (\omega_t^\sigma \otimes \eta_{t'}^\sigma,  - \omega_t^{\bar{\sigma}} \otimes \eta_{t'}^{\bar{\sigma}} )
$$
Thus the entries in the $(i,i')$th column of $X$ corresponding to the elements in \eqref{eqn:ii'} and \eqref{eqn:tt'} are 
$$
\begin{pmatrix} \frac{a_{t,i} b_{t',i'} + \overline{a_{t,i} b_{t',i'} }}{2}  \\  \frac{\pm(a_{t,i} b_{t',i'} - \overline{a_{t,i} b_{t',i'} })}{2 \sqrt{-D} } \end{pmatrix} 
$$ 
Then 
$$ \det(X) \sim_{\Q^*}  \frac{1}{\sqrt{-D}^{\frac{(j+1)(j+2)}{2} }} \det(Y)$$
where $Y$ is the matrix whose entries in the $(i,i')$th column corresponding to $(t,t')$ are  
$$ 
\begin{pmatrix} a_{t,i} \cdot b_{t',i} \\  \\ \overline{a_{t,i}} \cdot \overline{ b_{t',i'}} \end{pmatrix}  \stackrel{\eqref{eqn:ab-bar}}{=} \begin{pmatrix} a_{t,i} \cdot b_{t',i'} \\  \\ \cQ_t^{-1} \cR_{t'}^{-1} a_{j-t,i} \cdot b_{j+1-t', i'} \end{pmatrix} 
$$
As $(t,t')$ vary over all pairs such that $t+t' \le j$, the pairs $(t^*, (t')^*):= (j-t, j+1 -t')$ vary over all pairs such that $t^*+ (t')^* \ge j+1$. 
Thus 
$$
\det (Y) = \left(\prod_{0 \le t+t' \le j} \cQ_t^{-1} \cR_{t'}^{-1} \right) \cdot \det (Z)
$$
where up to a permutation of the rows, the matrix $Z$ is just $A \otimes B$. 
Then
\begin{align*}
c^{\pm} (\Res_{E/\Q} (M\otimes N))&\sim_{\Q^*} \frac{1}{\sqrt{-D}^{\frac{1}{2}(j+1)(j+2)} } \cdot \cQ_0^{-(j+1)} \cQ_1^{-j} \cdots \cQ_j^{-1}  \\
& \hspace{20mm} \cdot \cR_0^{-(j+1)} \cR_1^{-j} \cdots \cR_{j+1}^{0} \cdot\det(A)^{j+2} \det(B)^{j+1} \\
\end{align*}

Now we note that \eqref{detcomp1} implies that  
\begin{equation} \label{DetAcomp} \det (A)^2 \sim_{\Q^*} (2\pi i)^{-j(j+1)} \cdot \prod_{i=0}^{j} \cQ_p,\end{equation}
 and  in fact that $\prod_{i=0}^{j} \cQ_p$  is an element in $E$ of norm $1$. 
 
Indeed $\det(M)$ is a Tate motive, as observed in \eqref{detcomp1};   if we denote by  $H_{\dR}(\det M)_{\Q}$ a generator of the canonical $\Q$-line
inside its de Rham cohomology, 
arising from a $\Q$-rational differential form on $\mathbb{G}_m$, we may write
$$\omegat_0 \wedge  \omegat_1 \wedge \dots \wedge \omegat_j = \lambda \cdot H_{\dR}(\det M)_{\Q}$$
for some $\lambda \in E^*$  %
and  computing periods we see that
$$ \det(A) \sim_{\Q^*}  \lambda^{-1} (2 \pi i)^{-j(j+1)/2}.$$
On the other hand, we have
$ \omega_0 \wedge \omega_1 \wedge \dots \wedge \omega_j =  \omegat_0 \wedge \omegat_1 \wedge \dots \wedge \omegat_j$,
and comparing this element with its complex conjugate we find 
$\bar{\lambda} = \pm  \lambda \cdot \prod_{i=0}^j \cQ_j$ (for an explicit, but unimportant, choice of sign). This relation determines $\lambda$ up to $\Q^*$, and we have
$$\det(A)^2 \sim_{\Q^*} (2 \pi i)^{-j(j+1)} \lambda^{-2} \sim_{\Q^*} (2 \pi i)^{-j(j+1)} |\lambda|^{-2} \cdot   \bar{\lambda}/\lambda,$$
which proves \eqref{DetAcomp}. 
 Likewise, $\det(B)^2 \sim_{\Q^*} (2\pi i )^{-(j+1)(j+2)} \prod_{q=0}^{j+1} \cR_q^{-1}$, where again $\prod_{q=0}^{j+1} \cR_q^{-1}$ is an element of $E$ of norm $1$. 
 
 We may thereby simplify the expression above to
 \begin{align*}
c^{\pm} (\Res_{E/\Q} (M\otimes N))^2   & \cdot (2 \pi i)^{j(j+1)(j+2) + (j+1)^2 (j+2)}& 
\\ & \sim_{\Q^*}  \left( \cQ_0^{-j}  \cQ_1^{2-j} \dots \cQ_j^{+j} \right) \cdot \left( \cR_0^{-j-1} \cR_1^{-j+1} \dots \cR_{j+1}^{j+1} \right)
\\ & \sim_{\Q^*}   |\cQ_j|^{2j} |\cQ_{j-1}|^{2(j-2)} \cdots  \times |\cR_{j+1}|^{2(j+1)} |\cR_j|^{2j} \cdots 
\\ & \sim_{\Q^*} \prod_{p=0}^j |\cQ_p|^{2p} \cdot \prod_{q=0}^{j+1} |\cR_q|^{2q}.
\end{align*}

Using \eqref{LMCvol} and the relation
\begin{align*}
c^+ (\Res_{E/\Q}(M \otimes N) (j+1)) &= (2\pi i )^{ (j+1) \cdot \frac{1}{2} \mathrm{rank}  \ \Res_{E/\Q}(M \otimes N)} \cdot c^+ (\Res_{E/\Q}(M \otimes N)) \\
&= (2\pi i ) ^{(j+1)^2(j+2)} c^+ (\Res_{E/\Q}(M \otimes N),
\end{align*}
we find at last
\begin{equation} 
\label{eqn:ggp-pglnC} 
 \frac{c^+(\Res_{E/\Q} (M\otimes N) (j+1) )^2}{ \vol(\mathcal{L}_M)) \vol(\mathcal{L}_N)  } \sim_{\sqrt{\Q^*}} (2\pi i )^{ (j+1)(j+2)}. 
\end{equation}

  \subsection{Polarizations} \label{polarizations section 8}
In the remaining orthogonal cases,  the motives $M$ and $N$ over the imaginary quadratic field $E$ are equipped with (weak) polarizations, as follows from the discussion in the Appendix; 
these arise from the (orthogonal or symplectic) duality on the standard representations used to define $M$ and $N$. 

We will make use of these polarizations
for our analysis, and thus we summarize here some useful properties:

We denote by $S$ the weak polarization on $M$, i.e. 
$
S: M \otimes M \rightarrow \Q(-w),
$ with $w$ the weight of $M$. 
As usual, we write \begin{equation} \label{QandS}Q = (2\pi \sqrt{-1})^w S.\end{equation}
 Thus the form $Q$ is $\Q$-valued on $H^*_{\B}(M_{\sigma},\Q)$  
 (we shall denote this form by $Q_{\sigma}$, and write $S_{\sigma} = (2\pi \sqrt{-1})^{-w} Q_{\sigma}$ on the same space)
 whereas the form $S$ is $E$-valued on $H^*_{\dR}(M)$. 
We denote by the same letter $S$ the weak polarization on $N$.

These polarizations induce also polarizations on $\Ad(M), \Ad(N), M \otimes N$ by transport of structure, and also on the restriction of scalars from $E$ to $\Q$ of any of these motives; we will again  
denote these by the same letters, or by (e.g.) $S^{\mathrm{Ad}}$ if we want to emphasize that we are working with the adjoint motive.  We denote similarly (e.g.) $Q^{\mathrm{Ad}}_{\sigma}, S^{\mathrm{Ad}}_{\sigma}$ for the forms on the $\sigma$-Betti realizations, just as above.  

\subsubsection{Polarizations and restriction of scalars}

For a moment, let $X$ denote an  a polarized $E$-motive and $\mathbf{X}  := \Res_{E/\Q} X$.  Then $\mathbf{X}$ inherits a polarization from $X$. The corresponding bilinear form $Q$ on $H^*_B(\mathbf{X})$ is 
just the sum of the forms  $Q_\sigma$  and $Q_{\bar{\sigma}}$ on $V_\sigma=H^*_{B,\sigma} (X)$ and $V_{\bar{\sigma}}=H^*_{B,\bar{\sigma}} (X)$
 respectively. On 
the de Rham realization, the form is just the trace from $E$ to $\Q$ of the $E$-valued form on $H^*_{\dR} (X)$. 
Further, the $\C$-antilinear isomorphism $F_\infty c_{\B} $ from $V_\sigma$ to $V_{\bar{\sigma}}$ 
identifies $Q_\sigma$ and $Q_{\bar{\sigma}}$ with complex conjugates of each other. 
In particular, to compute the form on $H^*_{\dR} (\mathbf{X})$, we may embed $H^*_{\dR} (\mathbf{X})$ in $V_\sigma$ for 
instance (via $\varphi_\sigma$) and take the trace (from $\C$ to $\R$) of the form $Q_\sigma$.

\subsubsection{The adjoint motive: polarized case}

 Next, some comments on the adjoint motive.  Let $w$ be the weight of the polarized motive $M$.  
 
Since  $\Ad(M) \subset \Hom(M,M) \simeq M \otimes M^\vee$, and since $M^\vee \simeq M(w)$ via the 
polarization, we may view $\Ad (M)$ as a sub-motive of $M\otimes M(w)$. 
Now  
\begin{equation}
\label{BJdR} H_{\dR}(\Ad(M))   \subset  H_{\dR}(M \otimes M^{\vee}) = H_{\dR}(M)^{\otimes 2} \otimes H_{\dR}(\Q(w))  \stackrel{\eqref{TT}}{\simeq}  H_{\dR}(M)^{\otimes 2} \end{equation}
In this way, we can regard $\eta \otimes \eta'$ as an element of $H_{\dR}(M \otimes M^{\vee})$ when $\eta, \eta' \in H_{\dR}(M)$. 
Under the above identification the form $Q_{\Ad}$ induced on the adjoint corresponds to $(2 \pi \sqrt{-1})^{-2w}  Q^{\otimes 2}$,   whereas $S_{\Ad}$ corresponds to $S^{\otimes 2}$.

 Similarly, for $\sigma$ an embedding of $E$ into $\C$, we have \begin{align}
\label{BJB} H_{\B}(\Ad(M)_{\sigma},\C) &\subset  H_{\B}(M_\sigma \otimes M_\sigma^{\vee}, \C) = H_{\B}(M_{\sigma})^{\otimes 2}  \otimes H_{\B}(\Q(w), \C)  \stackrel{\eqref{TT}}{=} H_{\B}(M_{\sigma},\C)^{\otimes 2}. 
\end{align}
  Under this identification $Q_{\Ad}$ corresponds to $  Q^{\otimes 2}$, 
and $S_{\Ad}$ corresponds to $(2 \pi \sqrt{-1})^{2w}  S^{\otimes 2}$.

\subsubsection{} \label{Gram form}
In what follows, we  will compute the volume of $\mathcal{L}_M$ with respect to the polarization, as described
in \S \ref{volLcomp}.

 In other words, we compute the volume  on $\Q$-vector space   $\mathcal{L}_M$ with reference to the quadratic form
 obtained by pulling back the polarization under the map
$$\mathcal{L}_M \rightarrow H^*_{\B}(\Res_{E/\Q} \Ad (M), \R)$$
given by $x \mapsto \frac{1}{2}(x+\bar{x})$. 

If we regard the target above as   
 $H^*_{\B,\sigma}(\Ad (M), \R) \oplus H^*_{\B,\bar{\sigma}}(\Ad M,\R)$
the map is given by $ \frac{1}{2}(\varphi_{\sigma} + \overline{\varphi_{\sigma}}, \varphi_{\bar{\sigma}}+ \overline{ \varphi_{\bar{\sigma}}}) $.
Here $\varphi_{\sigma}$ is as in  \S \ref{E motive}. In other words, the form  on $\mathcal{L}_M$ is given by    
$$
(x,y) :=  \tr_{\C/\R} \ S_\sigma^{\Ad} (\frac{1}{2}(x^{\sigma}+\overline{x^{\sigma}}), \frac{1}{2}(\overline{y^{\sigma}+ \overline{y^{\sigma}}}))= \frac{1}{2} (\tr_{\C/\R} \ S_\sigma^{\Ad} (x^{\sigma}, y^{\sigma}) + \tr_{\C/\R} \ S_\sigma^{\Ad} (x^{\sigma},\overline{y^{\sigma}}) ).
$$

\subsubsection{Period invariants, revisited}

In this case, the previous discussion of period invariants can be slightly simplified.  In \S \ref{NIST}
we have  introduced elements $\omegat_p \in H_{\dR}(M)$ for each integer $p$ with $\dim F^p/F^{p+1} = 1$.   
In the cases with a polarization    we can and  will choose the elements $\omegat_p$  to be self-dual, in that  
\begin{equation} \label{S1} S(\omegat_p, \omegat_{p^*}) = 1 = S_{\sigma}(\omega^\sigma_p, \omega^\sigma_{p^*})  \ \ (p < p^*) \end{equation}
   whenever both $\omegat_p,\omegat_{p^*}$ are both chosen.  (The second equality follows from the first.)   
   The same quantity then equals $(-1)^w$ for $p > p^*$. 

 If $p=p^*$, which only occurs in even weight $w$, we cannot guarantee \eqref{S1}; here  
 $S_{\sigma}(\omega_p^{\sigma},\omega_p^{\sigma}) = \sigma(S(\omegat_p,\omegat_p)) $ lies in $E^\times$ and its class mod $(E^\times)^2$ is independent of the choice of $\omegat_p$.   Define therefore
\begin{equation} \label{alphadRdef}
\alpha_{\dR} (M) =  S(\omegat_p,\omegat_p).
\end{equation}
If the weight $j$ is odd, we set $\alpha_{\dR}(M) =1$.  In all cases, this is an element of $E^{\times}$ whose square-class is independent of choices. 
 
   We may then evaluate the $\cQ_p^{\sigma}$ in terms of the polarization. It follows from \eqref{rel1} that 
   $$ \cQ_p^{\sigma} = \begin{cases} S_{\sigma}(\omega^{\sigma}_p, \overline{\omega^{\sigma}_p}), & p < p^* \\ 
    \sigma( \alpha_{\dR}(M))^{-1} S_{\sigma}(\omega^{\sigma}_p, \overline{\omega^{\sigma}_p}), &  p =p^*. 
    \end{cases} 
    $$
   
   Note that $\cQ_p^{\sigma}$ belongs to $\R^*$ if $p \neq p^*$; thus, 
   when $\omegat_p$ are normalized above, we have $\cQ_p^{\sigma} = \cQ_p^{\sigmabar}$, and we may simply refer to $\cQ_p$. 
For $p=p^*$ we have  $\cQ_p^{\sigma} \in \sigma(\alpha_{\dR}(M))^{-1} \mathbf{R}^*$.     

Finally, if $\omegat_p, \omegat_q$ are both defined,  we denote by 
  $$ \omegat_{p,q} \in H^*_{\dR}(M \otimes M^{\vee}), \omega_{p,q}^{\sigma} \in H_{\B}( (M \otimes M^{\vee})_{\sigma}, \C)$$
  the image of $\omegat_p \otimes \omegat_q$ and $\omega_p \otimes \omega_q$, respectively, under the identifications of \eqref{BJdR} and \eqref{BJB}, respectively.

\subsection{$\SO_{2n} \times \SO_{2n+1}$ over $E$ imaginary quadratic}
\label{sec:so2nx2n+1}
  
 Recall that $M$ is the motive attached to automorphic form on $\SO_{2n}$, and $N$ the motive attached to the automorphic form on $\SO_{2n+1}$, and 
 we have fixed polarizations in \S \ref{polarizations section 8}.  
\subsubsection{Computation of archimedean $L$-factors} 

  In this case, the Hodge numbers  for $\Ad M$ are  somewhat irregular, so we will discuss the archimedean computation by hand. 
   We have
{\small  $$ 
  L_\infty (s, \Res_{E/\Q} \Ad N) = \left( \Gamma_\C (s+2n-1)^1 \Gamma_\C (s+2n-2)^1 \cdots \Gamma_\C(s+3)^{n-1} \Gamma_\C(s+2)^{n-1} \Gamma_\C(s+1)^n \right)^2  \Gamma_\C (s)^n
 $$}
 and
{\small  \begin{align*}
 L^*_\infty (0, \Res_{E/\Q} \Ad N) &= \left( \Gamma_\C (2n-1)^1 \Gamma_\C (2n-2)^1 \cdots \Gamma_\C(3)^{n-1} \Gamma_\C(2)^{n-1} \Gamma_\C(1)^n \right)^2  \Gamma_\C^* (0)^n\\
 & \sim_{\Q^*} \pi^{-2 \left[ n\cdot 1 + (n-1) \cdot (2+3) + \cdots + 1 \cdot (2n-1 + 2n-2) \right]
}= \pi^{-2 \left[ n + \sum_{i=1}^{n-1} i (2n-2i + 2n-2i +1) \right]} \\&= \pi^{-2 \left[ \sum_{i=1}^{n} i + 4 \sum_{i=1}^{n-1} i(n-i) \right]} = \pi^{-\frac{4}{3} n(n-1)(n+1)-n(n+1)}.
 \end{align*}}
 
 For $\Ad M$, 
the Hodge numbers range from
 $(2n-3,-(2n-3))$ to $(-(2n-3),(2n-3))$; the multiplicities are given by
 $$
 1,1,\cdots ,\overline{t,t}, \cdots ,n-1,n-1, n,n,n,n-1,n-1, \cdots ,\overline{t,t}, \cdots,1,1, 
 $$
 if $n=2t$ is even, and by
 $$
 1,1,\cdots , \overline{t,t+1},  \cdots n-1,n-1,n,n,n,n-1,n-1, \cdots  \overline{t+1,t},  \cdots,1,1,
 $$
if $n=2t+1$ is odd. (Here the bar indicates that those terms are skipped.)
In the first case, 
{\small \begin{align*}
L_\infty (s, \Res_{E/\Q} \Ad M) &=  \left( \Gamma_\C (s +4t-3)^1 \Gamma_\C (s+4t-4)^1 \cdots \Gamma_\C (s+2t+1)^{t-1} \Gamma_\C (s+2t)^{t-1} \cdot \right. \\
& \left. \Gamma_\C (s+2t-1)^{t+1} \Gamma_\C (s+2t-2)^{t+1} \cdots  \Gamma_\C (s+3)^{2t-1} \Gamma_\C (s+2)^{2t-1} \Gamma_\C (s+1)^{2t} \right)^2 \Gamma_\C (s)^{2t}
\end{align*}}
and
{\small \begin{align*}
L_\infty^* (0, \Res_{E/\Q} \Ad M) &= \left( \Gamma_\C (4t-3)^1 \Gamma_\C (4t-4)^1 \cdots \Gamma_\C (2t+1)^{t-1} \Gamma_\C (2t)^{t-1} \cdot \right. \\
& \left. \Gamma_\C (2t-1)^{t+1} \Gamma_\C (2t-2)^{t+1} \cdots  \Gamma_\C (3)^{2t-1} \Gamma_\C (2)^{2t-1} \Gamma_\C (1)^{2t} \right)^2 \Gamma_\C^* (0)^{2t}\\
& \sim_{\Q^*} \pi^{-2[ 2t + (2t-1) (2+3) + \cdots + (t+1) (2t-1 + 2t-2) + 
(t-1)( 2t+1 + 2t) + \cdots + 1(4t-3 + 4t-4) ]} \\
&= \pi^{-2 [ \sum_{i=1}^{2t-1} i(4t-2i-1+4t-2i-2) + \sum_{i=1}^{2t-1} i ] } = \pi^{-2 [ 4 \sum_{i=1}^{2t-1} i(2t-i) -2 \sum_{i=1}^{2t-1} i ] } \\
&= \pi^{-2 [ 4 \sum_{i=1}^{n-1} i(n-i) -2 \sum_{i=1}^{n-1} i ]} = \pi^{-\frac{4}{3}n(n-1)(n+1) + 2n(n-1)}.
\end{align*}}

Similarly, if $n=2t+1$, we have:
{\small \begin{align*}
L_\infty (s, \Res_{E/\Q} \Ad M) &= \left( \Gamma_\C (s +4t-1)^1 \Gamma_\C (s+4t-2)^1 \cdots \Gamma_\C (s+2t+3)^{t-1} \Gamma_\C (s+2t+2)^{t-1} \cdot \right. \\
& \left. \Gamma_\C (s+2t+1)^{t} \Gamma_\C (s+2t)^{t+1} \cdots  \Gamma_\C (s+3)^{2t} \Gamma_\C (s+2)^{2t} \Gamma_\C (s+1)^{2t+1} \right)^2 \Gamma_\C (s)^{2t+1}
\end{align*}}
and
{\small \begin{align*}
L_\infty^* (0, \Res_{E/\Q} \Ad M) &= \left( \Gamma_\C (4t-1)^1 \Gamma_\C (4t-2)^1 \cdots \Gamma_\C (2t+3)^{t-1} \Gamma_\C (2t+2)^{t-1} \cdot \right. \\
& \left. \Gamma_\C (2t+1)^{t} \Gamma_\C (2t)^{t+1} \cdots  \Gamma_\C (3)^{2t} \Gamma_\C (2)^{2t} \Gamma_\C (1)^{2t+1} \right)^2 \Gamma_\C^* (0)^{2t+1}\\
& \sim_{\Q^*} \pi^{ -2 [ (2t+1)(0+1) + 2t(2+3) + \cdots + (t+1)(2t) + t (2t+1) + (t-1) (2t+2 + 2t+3)  + \cdots + 1(4t-2 + 4t-1) ]} \\
&= \pi^{ -2[ \sum_{i=1}^{2t} i (4t-2i+4t-2i+1) + \sum_{i=1}^{2t} i ] } = \pi^{ -2[4\sum_{i=1}^{2t} i(2t-i) + 2 \sum_{i=1}^{2t} i ]} \\
&= \pi^{-2[4\sum_{i=1}^{2t} i(2t+1-i) - 2 \sum_{i=1}^{2t} i ]} = \pi^{-2 [ 4 \sum_{i=1}^{n-1} i(n-i) -2 \sum_{i=1}^{n-1} i ]} = \pi^{-\frac{4}{3}n(n-1)(n+1) + 2n(n-1)},
\end{align*}}
which is the same expression as in the case $n=2t$.

Thus, in either case, we have:
\begin{equation}
\label{eqn:gamma-so2n-ad}
L^* (0, \pi_\infty, \Ad) = L_\infty^*(0, \Res_{E/\Q} \Ad M) L_\infty^*(0, \Res_{E/\Q} \Ad N) 
 \sim_{\Q^*} \pi^{-\frac{8}{3} n(n-1)(n+1) +n^2 -3n}. 
\end{equation}  

\subsubsection{Volume computation}
 
We first compute the volume term for $\Ad(N)$.   
As in the $\PGL$ cases, all the graded pieces of the Hodge filtration for $N$ are one-dimensional. Let $\tilde{\eta}_0, \ldots, \tilde{\eta}_{2n-1}$ be a basis of $H^*_{\dR} (N)$, chosen as before;
these define invariants $\cR_0$,$\ldots$, $\cR_{2n-1}$ as well as a basis $\mathcal{B}_N =( \eta_0, \cdots, \eta_{2n-1} )$ of 
$H^*_{\dR} (N) \otimes_{E,\sigma} \C$. 
(As before, we fix an embedding $\sigma: E \hookrightarrow \C$ and when we write $\cQ, \cR$ etc. we mean $\cQ^{\sigma}, \cR^{\sigma}$, etc.)
To compute $\vol (F^1 \Ad (N)) $, we first 
write down an explicitly a basis for $H^*_{\dR} \Res_{E/\Q} \Ad(N)\otimes \C$. 
Here  $\Ad(N)$ is the  $n(2n+1)$-dimensional subobject of $\Hom(N,N)$ consisting of the those endomorphisms $L$
satisfying 
$ Q(Lx,y) + Q(x,Ly)=0,$ where $Q$ is the symplectic form on $N$.  
 
A basis for $F^m H^*_{\dR}(\Ad (N))\otimes_{E,\sigma} \C /F^{m+1} $ is indexed by 
{\it unordered} pairs $(i,j)$ such that $i+j =m+(2n-1)$ and is given by 
$$ \{ \eta_i \otimes \eta_j + \eta_j \otimes \eta_i:  \quad i+j =m+(2n-1) \} $$
or more precisely the image of these elements under the identifications of \eqref{BJdR}. 

If we replace $\Ad(N)$ by $\Res_{E/\Q} (\Ad(N))$, then we also need to throw in $\sqrt{-D}$ times 
the basis vectors above. The union of the elements above with $m\ge 1$ is then 
a $\C$-basis for $F^1 H^*_{\dR} (\Res_{E/\Q} \Ad (N)) \otimes \C$. While it is not 
a $\Q$-basis of the natural rational structure on this space, it is a $\Q$-basis of the corresponding graded 
for the Hodge filtration, so to compute the volume we may as well work with this basis.

In a similar fashion to our previous computations, we get
$$
\vol (\mathcal{L}_N)^2 \sim_{(\Q^*)^2} D^{n^2} \cdot \left( \cR_{2n-1}^{2n} \cR_{2n-2}^{2n-1} \cdots \cR_n^{n+1} \cR_{n-1}^{n-1} \cdots \cR_2^2 \cR_1\right)^2 
$$
and using $\cR_i \cR_{2n-1-i}=  (-1)$, that
$$ 
\vol (\mathcal{L}_N) \sim_{\Q^*} D^{\frac{1}{2} n^2} \cdot \cR_0^{-2n} \cR_1^{-(2n-2)} \cdots \cR_{n-1}^{-2}.
$$

We now turn to $\Ad(M)$. 
For $i=0, \ldots, n-2, n, \ldots 2n-2$  pick elements $\tilde{\omega}_i \in F^i H^*_{\dR} (M)$ 
according to the discussion of \eqref{NIST}, obtaining invariants $\cQ_i$ as explained there.

For the two dimensional space $F^n H^*_{\dR}/F^{n+1}$, there is no natural basis, 
so we just pick any orthogonal basis $\{ \tilde{\omega}^+_{n-1}, \tilde{\omega}_{n-1}^- \}$ for the form $S$.
Let $\tilde{\mathcal{B}}_M = \{ \tilde{\omega}_i \} \cup \{ \tilde{\omega}_{n-1}^+, \tilde{\omega}_{n-1}^- \}$. 
Let $\omega_{n-1}^+$, $\omega_{n-1}^-$ be the images of $\tilde{\omega}^+_{n-1}$, $\tilde{\omega}_{n-1}^-$ 
respectively in $H^{n-1,n-1}_\sigma (M)$. 
Suppose that 
\begin{align*}
\overline{\omega^+_{n-1}} &= \cQ_{11} \omega^+_{n-1} + \cQ_{12} \omega^-_{n-1}, \\
\overline{\omega^-_{n-1}} &= \cQ_{21} \omega^+_{n-1} + \cQ_{22} \omega^-_{n-1}.
\end{align*}

Let $\mathcal{B}_M$ denote the basis $\{ \omega_0, \cdots, \omega_{n-2},\omega_{n-1}^+,\omega_{n-1}^-, \omega_n \cdots , \omega_{2n-2} \}$  of $H_{\dR}(M)\otimes_{E,\sigma} \C$.  
As before
 $\Ad(M)$ is the $n(2n-1)$-dimensional sub-object of $\Hom(M,M)$ consisting of the those endomorphisms $L$
satisfying 
$$ Q(Lx,y) + Q(x,Ly)=0,$$
where $Q$ is the symmetric form on $M$. 

A basis for $F^m H^*_{\dR}(\Ad (M))\otimes_{E,\sigma} \C /F^{m+1} $ is indexed by 
{\it unordered} pairs $(\omega_i,\omega_j)$, $\omega_i, \omega_j \in \mathcal{B}_M$,  such that $i+j =m+(2n-2)$ with $i\neq j$ and is given by 
$$ \{ \omega_i \otimes \omega_j - \omega_j \otimes \omega_i:  \quad i+j =m+(2n-2) \} $$
again with reference to the isomorphism \eqref{BJdR}. 

We will compute $(\vol \cL_M)^2$ as the determinant of the Gram matrix of the form described in  \S \ref{Gram form}.
The only tricky part is the contribution of terms involving $\omega_{n-1}^\pm$. Let  
$$
x^\pm = \omega_{n-1}^\pm \otimes \omega_j - \omega_j \otimes \omega_{n-1}^\pm,
$$
where $j$ lies in the range $n\le j \le 2n-2$. 
Consider the $4\times 4$-matrix $X$ of inner products $(x,y)$ where $x,y$ run over the elements $x^\pm ,\sqrt{-D} x^\pm$.  Set
$$ \cQ_+ = S_\sigma( \omega_{n-1}^+, \overline{\omega_{n-1}^+} ), \quad  \cQ_- =  S_\sigma( \omega_{n-1}^-, \overline{\omega_{n-1}^-} )
$$ 
and 
$$ 
\A+ i \B =    S_\sigma (  \omega_{n-1}^+, \overline{\omega_{n-1}^-} ), \quad \A,\B \in \R,
$$
Note for example that, using \eqref{BJdR}  
$$ 
(x^+,\sqrt{-D}x^+) = \frac{1}{2} \left(\tr \ S_\sigma^{\Ad} ( x^+,\sqrt{-D}x^+) + \tr \ S_\sigma^{\Ad} ( x^+,\overline{\sqrt{-D}x^+} ) \right)= 0, 
$$
while 
$$ (x^+,x^-) = \frac{1}{2} \left(\tr \ S_\sigma^{\Ad} (x^+,x^-) + \tr \ S_\sigma^{\Ad} ( x^+,\overline{x^-} ) \right) = 2\A  \cQ_j.
$$
and 
$$
(x^+,\sqrt{-D}x^-) = \frac{1}{2} \left( \tr \ S_\sigma^{\Ad} ( x^+,\sqrt{-D}x^-) + \tr \ S_\sigma^{\Ad} ( x^+,\overline{\sqrt{-D}x^-} ) \right) = 2\sqrt{D} \B  \cQ_j.
$$
Then
\begin{align*}
\det (X) &= (2\cQ_j)^4 \cdot \det \begin{pmatrix} 
\cQ_+ & \A & 0 & \sqrt{D} \B \\
\A & \cQ_- & -\sqrt{D} \B & 0 \\
0 & -\sqrt{D} \B & D \cQ_+ & D\A \\
\sqrt{D} \B & 0 & D\A & D \cQ_- 
\end{pmatrix} \\
& =  (2\cQ_j)^4 \cdot D^2 (\cQ_+ \cQ_- - \A^2 - \B^2)^2.
\end{align*}
Note that 
$$
\cQ_+ \cQ_- - \A^2 - \B^2 = \det (\Gamma) = \Delta \cdot \det (\Xi),
$$
where
$$\Gamma:= \begin{pmatrix} \cQ_+ & \A+i \B \\ \A- i\B & \cQ_- \end{pmatrix}, \quad \Delta := S_\sigma( \omega_{n-1}^+, \omega_{n-1}^+ ) S_\sigma ( \omega_{n-1}^-, \omega_{n-1}^- ), \quad \Xi:= \begin{pmatrix} \cQ_{11} & \cQ_{12} \\ 
\cQ_{21} & \cQ _{22} \end{pmatrix}.$$
We remark that $\det(\Gamma)$ lies in $\R^\times$, $\Xi \bar{\Xi} = I$ and $\Delta$ lies in $E^\times$, hence 
$$
\det(\Gamma)^2 = \Delta \bar{\Delta} \in \Q^\times.
$$
Combining the above computation with a routine computation of the contribution from terms not involving $\omega_{n-1}^\pm$, we find 
$$
\vol (\mathcal{L}_M)^2 \sim_{(\Q^*)^2} D^{n^2-n} \cdot \left( \cQ_{2n-2}^{2n-2}  \cdots \cQ_{n}^{n} \cQ_{n-2}^{n-2} \cdots  \cQ_2^2 \cQ_1\right)^2 \cdot \det(\Gamma) ^{2(n-1)},
$$
and using $\cQ_i \cQ_{2n-2-i}= 1$, that 
$$
\vol (\mathcal{L}_M) \sim_{\Q^*} 
 \cQ_0^{-(2n-2)} \cQ_1^{-(2n-4)} \cdots \cQ_{n-2}^{-2}\cdot \Delta^{n-1} \cdot \det (\Xi)^{n-1}.
$$

Let $\{ e_0 ,\ldots, e_{2n-1} \}$ and $\{ f_0, \ldots, f_{2n-1} \}$ be $\Q$-bases for $H^*_{B,\sigma} (M)$ and $H^*_{B,\sigma} (N)$ respectively. Then $c^+ (\Res_{E/\Q}(M\otimes N))$ is the determinant of the change of basis matrix between
$$
\{ e_i \otimes f_{i'} + F_\infty (e_i \otimes f_{i'}) \}, \quad 0\le i,i'\le 2n-1,
$$ 
and
$$ 
\{ (\varphi_\sigma, \varphi_{\bar{\sigma}}) \omega \otimes \eta \}, \ \{ (\varphi_\sigma, \varphi_{\bar{\sigma}}) \sqrt{-D} \omega \otimes \eta \}
$$
where $ \omega \in \mathcal{B}_M \cap F^t H^*_{\dR}(M), \ \eta \in \mathcal{B}_N \cap F^{t'}H^*_{\dR} (N), 0\le t+t' \le  2n-2$.  
 We find as in the previous section that 
\begin{align*}
c^\pm(\Res_{E/\Q} M\otimes N) & \sim_{\Q^*} \sqrt{-D}^{-2n^2} \cdot \cQ_0^{-(2n-2)} \cQ_1^{-(2n-4)} \cdots \cQ_{n-2}^{-2} \cdot \\ & \cR_0^{-2n} \cR_1^{-(2n-2)} \cdots \cR_{n-1}^{-2} \cdot 
  \det(\Xi)^{-n}\cdot   \det (A\otimes B),
\end{align*}
where $A$ and $B$ are the period matrices given by
$$
(e_0 \ e_1 \cdots e_{2n-1}) = \mathcal{B}_M \cdot A, \quad (f_0 \ f_1 \cdots f_{2n-1}) = \mathcal{B}_N \cdot B.
$$
Computing the Gram matrices of the bases $e_i$ and $\mathcal{B}_M$ with respect to the polarization
and taking determinants, we may compute $\det(A)$ and $\det(B)$: 
$$ 
\det(A)^2 \sim_{\Q^*} \Delta^{-1} \cdot (2\pi i )^{-2n(2n-2)}, \quad \det (B)^2 \sim_{\Q^*}  (2\pi i )^{-2n(2n-1)},
$$ 
so
$$ 
\det(A\otimes B) = \det(A)^{2n} \det(B)^{2n} \sim_{\Q^*} \Delta^{-n} (2\pi i)^{-2n^2(4n-3)}.
$$
Finally the center is the point $s=2n-1$ and 
$$ c^\pm ((\Res_{E/\Q} M\otimes N) (2n-1)) = c^\mp(\Res_{E/\Q} M\otimes N) \cdot (2\pi i )^{4n^2 (2n-1)}.$$ 
 
Putting all of the above together yields:
\begin{equation} 
\label{eqn:ggp-so2n}
 \frac{c^+((\Res_{E/\Q} M\otimes N) (2n-1))}{\vol (\mathcal{L}_M) \vol (\mathcal{L}_N)} \sim_{\Q^*} (2\pi i)^{2n^2} \cdot  \underbrace{ \sqrt{\Delta \bar{\Delta}}}_{\in \sqrt{\Q^*}}  \cdot \sqrt{D}^n.
 \end{equation}

\subsection{$\SO_{2n+1} \times \SO_{2n+2}$ over $E$ imaginary quadratic}
Recall that here $N$ is associated with $\SO_{2n+1}$ and $M$ with $\SO_{2n+2}$. 
 We will be brief for all the computations  are very similar to the prior section, e.g. 
 the term $L^*(0,\Ad N)$ is the same as in the previous section, while the formula for 
  $L^*(0, \Ad(M))$ is obtained by replacing $n$ by $n+1$ in the formula from the previous section. 

The volume computations are also similar: we have 
$$ 
\vol (F^1 \Res_{E/\Q} (\Ad N)) \sim_{\Q^*} D^{\frac{1}{2} n^2} \cdot \cR_0^{-2n} \cR_1^{-(2n-2)} \cdots \cR_{n-1}^{-2}.
$$ 

$$
\vol (F^1 \Res_{E/\Q} (\Ad M)) \sim_{\Q^*} 
 \cQ_0^{-2n} \cQ_1^{-(2n-2)} \cdots \cQ_{n-1}^{-2}\cdot \Delta^{n} \cdot \det (\Xi)^{n},
$$
where $\Delta$, $\Xi$ are defined similarly. 
\begin{align*}
c^\pm(\Res_{E/\Q} M\otimes N) & \sim_{\Q^*} \sqrt{-D}^{-2n(n+1)} \cdot \cQ_0^{-2n} \cQ_1^{-(2n-2)} \cdots \cQ_{n-1}^{-2} \cdot \\ & \cR_0^{-2n} \cR_1^{-(2n-2)} \cdots \cR_{n-1}^{-2} \cdot 
\det(\Xi)^{-n}\cdot   \det (A\otimes B),
\end{align*}
where $A$ and $B$ are the period matrices as before. 
Now,  computing with Gram matrices as  before  shows  
$$ 
\det(A)^2 \sim_{\Q^*} \Delta^{-1} \cdot (2\pi i )^{-2n(2n+2)}, \quad \det (B)^2 \sim_{\Q^*}  (2\pi i )^{-2n(2n-1)},
$$ 
so
$$ 
\det(A\otimes B) = \det(A)^{2n} \det(B)^{2n+2} \sim_{\Q^*} \Delta^{-n} (2\pi i)^{-n(2n+2)(4n-1)}.
$$
The center is the point $s=2n$ and 
$$ c^\pm ((\Res_{E/\Q} M\otimes N) (2n)) = c^\mp(\Res_{E/\Q} M\otimes N) \cdot (2\pi i )^{4n^2 (2n+2)}.$$

Putting all of the above together yields:
 \begin{equation}
 \label{eqn:ggp-so2n+1}
 \frac{c^+((\Res_{E/\Q} M\otimes N) (2n))}{\vol (\mathcal{L}_M) \vol (\mathcal{L}_N)} \sim_{\Q^*} (2\pi i)^{2n(n+1)} \cdot \sqrt{D}^n.
\end{equation}

\subsection{Motives with coefficients}  
\label{sec:mot-coeffs}
We return to the issue mentioned on page \pageref{c group inner twist}, namely, 
the morphism from the motivic Galois group to the $c$-group of $G_1$ or $G_2$ might not be defined over $\Q$. 
In this remark we outline a modification of the argument above that
accounts for this possibility. We will explain this in the case $G=\Res_{E/\Q}\SO(2n) \times \SO(2n+1)$ 
for an imaginary quadratic $E$, the other cases being similar. The reader is 
referred to \cite{Deligne} Sec. 2 for a survey of motives with coefficients and 
for the formulation of Deligne's conjecture in that setting, which we use below.

 Choose a large enough number field $K$ over which the $\widehat{G_i}$-motives attached to $\pi_1, \pi_2$ are defined, i.e.
so that the associated morphisms from the motivic Galois group to the $c$-group of $G_i$ are defined over $K$. 
 
We get 
motives attached to $\pi_1$ and $\pi_2$ over $E$ with coefficients in $K$, denoted 
$M_K$ and $N_K$ respectively. Attached to $\Pi$ one has the motive
$\bM_K = \Res_{E/\Q} (M_K \otimes N_K)$.  
Then 
$$
L(2n-1, \bM_K)\in (K\otimes \C), \quad c^+ (\bM_K (2n-1)) \in (K\otimes \C)^*/K^*.
$$
where all the tensor products are taken over $\Q$;  Deligne's conjecture states that 
\begin{equation}
\label{eqn:del-coeff}
\frac{L(2n-1, \bM_K)}{c^+ (\bM_K (2n-1))} \in K \hookrightarrow (K\otimes \C).
\end{equation}

Let $\Ad M_K$ and $\Ad N_K$ be defined as above as sub-motives of $M_K \otimes M_K (2n-2)$ and 
of $N_K \otimes N_K (2n-1)$ respectively;  by the general formalism of Appendix  \S \ref{Pimotive Appendix} these are equipped with
polarizations (in the category of motives with $K$-coefficients).\footnote{ It is plausible that this fails in some $\PGL$ cases, but there our proofs
never used polarizations and with minor modifications one proceeds without them.}
 Then we can define the volumes
 $$
\vol  \cL_M, \vol \cL_N \in (K\otimes \C)^\times/ K^\times,
$$
generalizing in the obvious way the definition in \eqref{voldef}, 
and 
$$
\vol F^1 H^0_{\dR} \Ad \bM_K = \vol \cL_M \vol \cL_N.
$$
Moreover the computations 
in Sec. \ref{sec:so2nx2n+1} can be easily modified to show that 
the following variant of \eqref{eqn:ggp-so2n} remains valid:
\begin{equation}
\label{eqn:so2n-coeff}
\frac{c^+ (\bM_K (2n-1))}{\vol F^1 H_{\dR} \Ad \bM_K \cdot (2\pi i)^{2n^2} } \in   \sqrt{(K\otimes \Q)^*}.
\end{equation}
(One uses that the $K$-action on $H_B(M_K)$ and $H_B (\bM_K)$ commutes with the action of $\C^*$ and $W_\R$ respectively.)

Now we have an equality 
\begin{equation}
\label{eqn:lfunc-eq-coeff}
L(\frac{1}{2}, \Pi)= L(2n-1, \bM_K)
\end{equation}
(Rankin-Selberg $L$-function on the left) 
which in fact shows that the RHS lies in $(\Q \otimes \C)  \hookrightarrow (K\otimes \C)$.

 Finally, we note that there is a natural functor 
  $$ \text{Motives with $\Q$-coefficients} \mapsto \text{Motives with $K$-coefficients} $$
  denoted $X \mapsto X_K$ and we have the relation
  $$
  (\Ad \Pi)_K\simeq \Ad \bM_K
  $$
  where $\Ad \Pi$ is the conjectural adjoint motive with $\Q$-coefficients attached to $\Pi$.  
 The proof of Lemma \ref{lem:vol-indep} shows that the square of the volume 
 $\vol_S F^1 H_{\dR} (\Ad \Pi)$ (for any weak polarization $S$ on $\Ad \Pi$),
is (up to $\Q^*$) independent of the choice of $S$. Moreover,  \begin{equation}
\label{eqn:vol-eq-coeff}
\vol_S F^1 H_{\dR} (\Ad \Pi) = \vol_{S} F^1 H_{\dR} (\Ad \bM_K) 
\end{equation}
where the LHS lies in $\C^* / \Q^*$, the RHS in $(K\otimes \C)^*/ K^*$ and the equality 
must be viewed as saying the LHS maps to the RHS under the natural map 
$\C^*/\Q^* \rightarrow (K\otimes \C)^*/ K^*$.
Putting everything together (i.e. \eqref{eqn:del-coeff}, \eqref{eqn:so2n-coeff}, \eqref{eqn:lfunc-eq-coeff} and \eqref{eqn:vol-eq-coeff}) gives
\begin{equation}  \label{oym}
\frac{ L(\frac{1}{2}, \Pi)}{\vol_S F^1 H_{\dR} (\Ad \Pi)\cdot (2\pi i)^{2n^2}} \in (\Q \otimes \C) \cap \sqrt{(K\otimes \Q)^*},
\end{equation}
in particular, the square of the left-hand side lies in $(\Q \otimes \C) \cap (K \otimes \Q) = \Q$, as desired. 

 \section{A   case with $\delta=3$}
 
In this section we offer  what is perhaps the most interesting evidence for our conjecture,
in a case where $Y(K)$ is a $9$-manifold. Namely,
we verify some of the numerical predictions in a cohomological degree
that is neither minimal nor maximal; indeed, 
  we cannot even produce explicit cycles!

What we check
is the following:
  our conjecture relating $H^3$ to $H^4, H^5, H^6$ holds,
``up to rotation'' (see Theorem \ref{SecondMainTheorem} for the precise statement). 
That theorem is phrased as conditional on Beilinson's conjectures, but what we actually do is unconditional: 
we compute many numerical invariants of the lattices $H^*$, and we only need Beilinson's conjectures
to compare these computations with our conjecture. 
We also verify Prediction \ref{pred3} unconditionally (at least up to some factors in $\sqrt{\Q^*}$). 
 It would be interesting to analyze the square classes that appear in our argument, in order to eliminate these factors of $\sqrt{\Q^*}$.

A critical input into our result is the work of M. Lipnowski \cite{Lipnowski}, who combines
the ideas of equivariant analytic torsion with base change. 
 
 \subsection{Notation and assumptions} 
  \begin{itemize}
 \item[-] Let $F$ be an imaginary quadratic field (we will regard it as embedded in $\C$) and $E\supset F$ a   cyclic extension of degree $3$;
 let $\sigma$ be a generator for the Galois group of $E/F$, so that
 $$\Gal(E/F) = \langle \sigma \rangle =  \{ 1, \sigma, \sigma^2 \}.$$   
 
 We will assume $E/F$ to be {\em unramified}, but this is only so we can apply the work of \cite{Lipnowski}
 in the simplest form; the reader can easily verify that the same idea would apply for $E/F$ unramified  at  primes above $3$, for example,
 using the refined theorems later in \cite{Lipnowski}.

  \item[-]   Choose a non-split quaternion algebra $D$ over $F$, and
 let $\GG_F $ be the (algebraic) group (underlying) $D^{\times}/F^{\times}$. 
 Let $\GG_E$ be the base change of $\GG$ to $E$; and let
$$G = \GG_E( E \otimes \C) = \PGL_2(\C) \times \PGL_2(\C) \times \PGL_2(\C)$$
 be the archimedean group at $\infty$ associated to $\GG_E$. 
 
 \item[-] 
 We denote by $ \mathfrak{a}_{\GG_F(\R)} $ or $\mathfrak{a}_F$
 for short the (one-dimensional) complex vector space attached to the real group $\GG_F(\R)$ (see \S \ref{VoganZuckerman}). Similarly 
 we define $\mathfrak{a}_{\GG_E(\R)} = \mathfrak{a}_E \mbox{ for short}$, a three-dimensional complex vector space. 
 Note that we may naturally identify
 $$ \mathfrak{a}_E \simeq \mathfrak{a}_F^{\Sigma}$$
 where $\Sigma$ is the set of  embeddings $E \hookrightarrow \C$
 extending the given embedding of $F$.

  \item[-] Let $\pi$ be an infinite-dimensional automorphic   representation for $\mathbf{G}_F$, cohomological at $\infty$,  and  let $\Pi$ be the base-change of $\pi$ to $\mathbf{G}_E$.
  
  \item[-] 
We suppose that $\pi$ is  trivial at each ramified place for $D$, and with conductor $\mathfrak{p}^{f(\mathfrak{p})}$ at each 
prime $\mathfrak{p}$ that is unramified for $D$. Put $\mathfrak{n} =  \prod_{\mathfrak{p}} \mathfrak{p}^{f(\mathfrak{p})}$.
(If one allows the case where $E/F$ is ramified, we should additionally assume that 
   $\mathfrak{n}$ is  relatively prime to the discriminant of $E/F$.)
  \item[-] 
  Let $K_F$ be  the level structure for $\GG_F$ of ``level $\Gamma_0(\mathfrak{n})$.'' 
  By this we mean $K_F = \prod_{v} K_v$, the product over all finite places $v$, where
  \begin{itemize}
  \item[(a)]  If $v$ is ramified for $D$, then take $K_v =  \mathcal{O}_{D_v}^* F_v^*/F_v^*$ where $\mathcal{O}_{D_v} \subset D_v$ is the maximal order. 
  \item[(b)]  If $v$ is unramified for $D$, fix an isomorphism $D_v \simeq \PGL_2(F_v)$;
  then $K_v$ is given by the preimage of the matrices $\small{\left( \begin{array}{cc} a & b \\ c & d \end{array}\right)} \in \PGL_2(\mathcal{O}_v)$
  where the valuation of $c$ is at least $f(v)$. 
  \end{itemize}
  
  We define similarly $K_E$  to be the level structure for $\GG$
  ``of level $\Gamma_0(\mathfrak{n} \cdot \OO_E)$,'' where we choose the isomorphisms in (b) in such a way that $K_E$ is $\sigma$-invariant.

 \item[-]  Let \begin{equation} \label{YYbardef}  Y =Y(K_E),  \ \ \overline{Y} = Y(K_F)\end{equation}
 be the corresponding arithmetic manifolds for $\mathbf{G}_E$ and $\mathbf{G}_F$, respectively; thus $Y$ is nine-dimensional and $\overline{Y}$ is three-dimensional.
 Moreover there is a natural $\Gal(E/F) = \langle \sigma \rangle$-action on $Y$ (arising from the $\sigma$-action on $\mathbf{G}_E$, which preserves the level structure).
 The inclusion $\mathbf{G}_F \hookrightarrow \mathbf{G}_E$ gives rise to a map $\overline{Y} \rightarrow Y^{\sigma}$
 of $\overline{Y}$ into the $\sigma$-fixed subspace on $Y$. 
  
  We equip $\overline{Y}$ with the Riemannian metric arising from the standard Riemannian metric
  on hyperbolic $3$-space $\mathbb{H}^3$, and we equip $Y$ with the Riemannian metric arising from the standard Riemannian metric
  on $\mathbb{H}^3 \times \mathbb{H}^3 \times \mathbb{H}^3$.   
 \item[-] We suppose that  
 \begin{equation} \label{1dassumption} \dim H^3_{\mathrm{cusp}}(Y, \C) = 1.\end{equation}
 Here the notation ``cusp'' should be understood as meaning the contribution of all infinite-dimensional automorphic representations to cohomology.
 
\eqref{1dassumption} implies  firstly  that $\dim H^1_{\mathrm{cusp}}(\overline{Y}, \C) =  1$, because
 of base change, and secondly that $Y, \overline{Y}$ have only one connected component (which is equivalent to asking that the class numbers of $E$ and $F$ are odd). 
 It also implies that $\pi$ is the only  nontrivial automorphic representation which contributes to the cohomology of $\overline{Y}$,
 and similarly $\Pi$ is the only   nontrivial automorphic representation which contributes to the cohomology of $Y$.

\item[-]Let $L_{\Pi}$ be the coadjoint motivic cohomology $H^1_{\m}(\coAd \Pi, \Q(1))$
as in \eqref{Ldefagain}; 
let $L_{\Pi} \otimes \C \rightarrow \mathfrak{a}_E$ be the Beilinson regulator, as in  \eqref{Lreg}.
  We define similarly
$L_{\pi}$ with  its regulator map $L_{\pi} \otimes \C \rightarrow \mathfrak{a}_F$. 
  There is a natural action of $\langle \sigma \rangle \simeq \Z/3\Z$ on $L_{\Pi}$,  and   an identification
    \begin{equation} \label{Lpifix} L_{\pi}  \stackrel{\sim}{\rightarrow} L_{\Pi}^{\sigma}.\end{equation}
  \end{itemize} 

 Before we give the theorem statement, let us comment a little on the assumptions.
 Although we do not have any numerical examples, we expect that situations like the above should be very easy to find
{\em given} an effective ability to compute $H^3(Y, \C)$  numerically. 
 In particular, it is very common (see discussion in \cite{BSV}) that  the cuspidal cohomology of $\overline{Y}$ is one-dimensional.
When that is so,  we would expect that the cuspidal contribution to $H^3(Y, \C)$ {\em also}  is  one-dimensional, comprising solely the base-change forms -- in situations
with $\delta > 0$, cuspidal cohomology  in characteristic zero that does not arise via a lift from another group  is considered to be very rare.

 Before we formulate the theorem, 
note that 
 $\Gal(E/F)$, and thus the real group algebra $\R[ \Gal(E/F)]$,  acts on $H^*(Y(K), \C)$. 
By a {\em rotation} in the group algebra $\R[\Gal(E/F) ] \simeq \mathbf{R} \times \mathbf{C}$
we mean  an element of the form $(1, z)$ where $|z|=1$.

 \begin{theorem} \label{SecondMainTheorem}
With the assumptions above,  Prediction \ref{pred3}  (more precisely
equation \eqref{frodo2})
holds up to $\sqrt{\Q^*}$.

Moreover, assume Beilinson's conjectures, as formulated in Conjecture
\ref{conj:Beilinson}, and the existence of a $2$-dimensional motive associated to $\pi$ (so also $\Pi$).  Let $\aGs$, and so also $L_{\Pi}^*$, act on $H^*(Y(K), \C)_{\Pi}$ by means of the action constructed in \S \ref{VoganZuckerman}.

Then there are rotations $r_i\in \R[\Gal(E/F)]$, for $1 \leq i \leq 3$, such that 
\begin{equation} \label{j4}   H^3(Y(K), \overline{\Q})_{\Pi}   \cdot \wedge^i L_{\Pi}^* =  r_i  H^{3+i}(Y(K), \overline{\Q})_{\Pi} . \end{equation} 
   In other words, the main Conjecture \ref{mainconjecture} holds, up to replacing $\Q$ by $\overline{\Q}$
 and up to a rotation in $\R[\Gal(E/F)]$.  (In fact, $\overline{\Q}$ can be replaced by  an extension of the form $\Q(\sqrt{a}, b^{1/4})$ for $a,b \in \Q^*$, and $r_3$ can be taken trivial.)
  \end{theorem}

Here the tempered cohomology contributes in degrees $3$ to $6$.  The groups $H^4$ and $H^5$ are  ``inaccessible'', because it appears
to be very difficult to directly construct rational cohomology classes of this degree.  Our method of proof  is in fact quite indirect, going through analytic torsion.
 
We need some setup first on metrized lattices (\S \ref{metrizedlattices})
and then on Reidemeister torsion (\S \ref{Rtorsion}). This setup will allow us to check
that Prediction \ref{pred3} holds in \S \ref{pred3proof}. The full  Theorem  above will follow from a more detailed analysis, which we carry out 
in the remainder of the section.    This final analysis uses many of the results of  this paper: it uses the results of 
Theorem \ref{MainPeriodTheorem} both over $F$ and over $E$, the compatibility with Poincar{\'e} duality (Proposition \ref{basic properties}),  
and the study of analytic torsion over $F$ and over $E$ (both usual and $\sigma$-equivariant).

\subsection{Volumes and functoriality} \label{metrizedlattices} 
 Some brief remarks about the behavior of volumes under functoriality:
 Let $V$ be a $\Q$-vector space  equipped with a metric, i.e.,  $V \otimes \R$ is equipped
 with an inner product.  We define its volume as in  \eqref{voldef}. Then 
 $$V^* := \Hom(V, \Q), \Sym^k  V, \wedge^k V$$
 all obtain metrics; similarly, if $V, W$ are $\Q$-vector spaces
 with metrics, then $V \otimes W$ inherits a metric. 
 We have a natural metrized isomorphism $ \wedge^k V \simeq \left( \wedge^{d-k} V \right)^* \otimes (\det V)$,
 where we wrote $\det(V) = \wedge^{\dim(V)} V$.

Fix an isomorphism $f: (V \otimes \R, \mbox{metric}) \rightarrow \left(\R^n, \mbox{Euclidean inner product}\right)$. 
If we write $f(V) = \Q^n g$ for some $g  \in \GL_n(\R)$, 
we have $\vol(V) = \det(g)$.   Using this  it is easy to check the following identities:
\begin{equation} \label{tensorvolume}  \vol(V_1 \otimes V_2)  =  (\vol V_1)^{\dim V_2} (\vol V_2)^{\dim V_1}\end{equation} \begin{equation} \vol(V^*) = \vol(V)^{-1}, \ \  \prod_i  (\vol \wedge^i V)^{ (-1)^i} = 1 \ \ (\dim V \geq 2),\end{equation} 
where all equalities are in $\R^*/\Q^*$. 

 If $\sigma$ is an automorphism of $V$ with prime order, then
 we denote by $V^{\sigma}$ the fixed point space;  we denote by $V_{\sigma}$ the coinvariants $V/V^{\sigma}$.
 It will be convenient to abridge
 $$\vol^{\sigma}(V) := \vol(V^{\sigma}),$$
 the volume of the $\sigma$-invariants with respect to the induced metric. 
 
Finally, it will be convenient to make the following notation: If $V_i$ are a collection of $\Q$-vector spaces
with metrics, indexed by  the integers, and only finitely many $V_i$ are nonzero, we denote  by
\begin{equation} \label{volproddef} \vol V_* = \prod (\vol \ V_i)^{(-1)^i},\end{equation}  
the {\em alternating product} of the volumes. We will often apply this notation when $V_i$
is the $i$th cohomology group of a Riemannian manifold, equipped with the metric that arises
from its identification with harmonic forms.

\subsection{Analytic torsion and equivariant analytic torsion. The theorems of Moscovici-Stanton and Lipnowski}  \label{Rtorsion}

  As a reference on this topic see \cite{Cheeger} (for the general case) and \cite{LottRothenberg} (for the equivariant case); we briefly summarize the important points. 
  
 Let $M$ be a compact Riemannian manifold, $\sigma$ an automorphism of $M$ of prime order $p$, $G = \langle \sigma\rangle$ the group generated by $\sigma$.   We shall suppose that $\dim(M)$ and $\dim(M^{\sigma})$ are both odd. 
 We may find a $G$-stable triangulation of $M$, by \cite{Illman},
 and it may be assumed to be regular (see \cite[Chapter III]{Bredon}).

If $W$ is a real vector space, let $\det(W)$ be the line (= one-dimensional real vector space)
given by $\wedge^{\dim(W)} W$. If $W$ has a  Euclidean metric, then $\det(W)$ has a metric too; this normalizes
an element of $\det(W)$ up to sign, the element of norm $1$. 
   If $W_{\bullet}$ is a complex of real vector spaces, 
   define $\det W_{\bullet} = \otimes_{i} (\det W_i)^{(-1)^i}$, a one-dimensional  real vector space.
   (Here, $L^{-1}$ denotes the dual of $L$, if $L$ is one-dimensional.)  
   There is a natural isomorphism $\det W_{\bullet} \simeq \det H^*(W_{\bullet})$,
   where we regard the cohomology as a complex of  vector spaces with zero differential.

In particular, writing $C^*(M, \R)$ for the cochain complex of $M$ with respect to the fixed triangulation, we get an isomorphism 
 \begin{equation} \label{canonical id} \det C^*(M, \R) \simeq \det H^*(M, \R)\end{equation}
        Equip the chain complex $C_*(M, \R)$ with the metric where the characteristic functions of cells form an orthonormal basis;
        give $C^*(M, \R)$ the dual metric. 
Equip the cohomology $H^*(M, \R)$ with the metric that arises from its identifications with harmonic forms. 
These metrics induce
metrics on the one-dimensional vector spaces $\det C^*(M, \R)$ and $\det H^*(M, \R)$ respectively. 

We define the Reidemeister torsion of $M$ (with reference to the given triangulation) by comparing these metrics, using  
the identification \eqref{canonical id}: 
\begin{equation} \label{RTdef} \mathrm{RT}(M) \cdot \| \cdot \|_{C^*} = \| \cdot \|_{H^*}.\end{equation}
Evaluate the resulting equality on an element $c \in \det C^*(M, \Q)$;
then $\|c\|_{C^*}$ is easily seen to lie in $\Q^*$,  whereas $\|c\|_{H^*} \sim_{\Q^*}
 \vol H^*(M, \Q)$, where the right-hand side is defined as  an alternating product as in \eqref{volproddef}.  Therefore, 
\begin{equation} \label{RTalt} \mathrm{RT}(M) \sim  \vol H^*(M, \Q),\end{equation}
     
 We also need an equivariant version of the same discussion. 
The complex of invariants
 $C^*(M, \R)^{\sigma}$  has cohomology identified with $H^*(M,\R)^{\sigma}$; we get
\begin{equation} \label{secondiso} \det C^*(M, \R)^{\sigma} \simeq \det H^*(M, \R)^{\sigma}.\end{equation}
 These too have metrics, induced from $C^*(M, \R)$ and $H^*(M, \R)$; we define the 
``invariant part''  $\mathrm{RT}^{\sigma}(M)$ of the Reidemeister torsion  via the same
 rule \eqref{RTdef}, now applied to \eqref{secondiso}. 
 An orthogonal basis for $C_*(M, \Q)^{\sigma}$  is obtained by taking 
 all $\sigma$-invariant cells, and the $\sigma$-orbits of cells that are not invariant; 
 we have a similar (dual) basis for $C^*(M, \Q)^{\sigma}$. 
 The elements of the resulting basis are orthogonal, and their lengths are either $1$  or $\sqrt{p}$,
 where $p$ is the order of $\sigma$. 
Writing $\varepsilon_j$ for the number of $j$-dimensional simplices that are not invariant, we see
 $ \vol C^j(M, \Q)^{\sigma} \sim  p^{\varepsilon_j/2}.$
However, modulo $2$, $\sum \varepsilon_j = \sum (-1)^j \varepsilon_j = \chi(M)- \chi(M^{\sigma})$. 
Both Euler characteristics are zero (we are dealing with odd-dimensional manifolds).  
Proceeding as above, we get 
$$ \mathrm{RT}^{\sigma}(M) \sim   \vol H^*(M, \Q)^{\sigma}.$$

The main theorem of \cite{Cheeger} is an equality between $\mathrm{RT}$ and an analytic invariant, the analytic torsion;
the main theorem of \cite{LottRothenberg} is a corresponding equality for $\mathrm{RT}^{\sigma}$. We do not need to recall these results in full here.

All that is important for us are the following two statements, in the case when $M = Y$ 
from \eqref{YYbardef}, and $\sigma$ is given by the action of a generator of $\Gal(E/F)$ on $Y$:
\begin{equation}  \label{rt1} \mathrm{RT}(Y) = 1 \end{equation}
\begin{equation} \label{rt2}  \mathrm{RT}^{\sigma}(Y) = \mathrm{RT}(\overline{Y})^2. \end{equation} 

 These statements are proved by studying the analytic torsion. 
For \eqref{rt1} see  \cite{MS} or \cite{speh}: the idea is, roughly speaking, that the
product decomposition of the universal cover of $Y$ means that every Laplacian eigenvalue occurs in several cohomological degrees,
leading to a mass cancellation in the  analytic torsion.

As for \eqref{rt2}, this key relationship is due to Lipnowski \cite[\S 0.2, ``Sample Theorem'']{Lipnowski}. 
 Lipnowski's results are deduced
from the theory of base change: the analytic torsion counterparts of $\mathrm{RT}(Y)$ and $\mathrm{RT}^{\sigma}(Y)$ are defined
in terms of a regularized trace of $\log \Delta$, acting on $Y$, and possibly twisted by a power of $\sigma$;
however the theory of base change precisely allows one to relate this to corresponding computations on $\overline{Y}$. 
\footnote{Here are some notes regarding the translation of Lipnowski's theorem to the form above: Lipnowski
works in a situation with a Galois group $\langle \sigma \rangle$ of order $p$ and  shows that $\tau_{\sigma} = \tau^p$.
Here $\tau$ is exactly $\mathrm{RT}(\overline{Y})$, for suitable choices of data,  but $\tau_{\sigma}$ takes some translation: its logarithm is 
 the   logarithmic determinant of the de Rham Laplacian on $Y$
  twisted by $\sigma$. One obtains the same logarithmic determinant if we twist by 
 $\sigma^i$ for any $1 \leq i \leq p-1$.
Add up  over $1   \leq i \leq p-1$ and apply the main theorem  of \cite{LottRothenberg} to the representation of 
 $\langle \sigma \rangle$ which is the difference of the regular representation and $p$ copies of the trivial representation. 
We find $\mathrm{RT}(Y) \cdot \tau_{\sigma}^{p-1} = 
 \mathrm{RT}^{\sigma}(Y)^p$; therefore, in our case with $p=3$, we find  $\frac{\mathrm{RT}^{\sigma}(Y)^3}{\mathrm{RT}(Y)} = \tau_{\sigma}^{2} = \tau^{6}$. To conclude we apply \eqref{rt1}.}
\subsection{Volumes of cohomology groups for $Y$ and $\bar{Y}$} 
We gather some preliminary results related to the volumes of groups $H^j(Y, \Q)$ and $H^j(\overline{Y},\Q)$,
measured as always with respect to the metric induced by the $L^2$-norm on harmonic forms. 

We have
\begin{equation} \label{factorization} \vol H^i(Y, \Q) = \vol H^i_{\Pi}(Y, \Q) \cdot \vol H^i_{\triv}(Y, \Q), \end{equation}
(equality in $\R^*/\Q^*$) by virtue of our assumption that the only cohomological automorphic representations at level $K$
are the trivial representation and $\Pi$: the splitting $H^i = H^i_{\Pi} \oplus H^i_{\triv}$ is both orthogonal
and defined over $\Q$.  Poincar{\'e} duality induces a metric isomorphism $H^i(Y, \Q) \simeq H^{i^*}(Y,\Q)^*$, 
where $i+i^* = 9$, and thus
$$\vol H^i(Y, \Q) \cdot \vol H^{i^*} (Y, \Q) \sim 1$$
and the same result holds for the trivial and $\Pi$ parts individually.
We have similar results for the $\sigma$-invariant volumes, and also a similar equality for $\overline{Y}$:
$$\vol H^i(\overline{Y}, \Q) = \vol H^i_{\pi}(\overline{Y}, \Q) \cdot \vol H^i_{\triv}(\overline{Y}, \Q),$$

We now compute the various volume terms related to the trivial representation.

Observe that
$$ \dim H^i_{\triv}(\overline{Y}, \Q) = \begin{cases} 1, i \in \{0,3\};\\ 0,  \mbox{else}, \end{cases} \mbox{ and   }\dim H^{i}_{\triv}(Y, \Q) = \begin{cases} 1, i \in \{0,9\}; \\ 3, i \in \{3,6\}; \\ 0,  \mbox{else}. \end{cases}$$
Explicitly speaking, harmonic representatives for $H^3_{\triv}(Y,\R)$  
are obtained from the pullbacks $\pi^* \nu$ under the coordinate projections
\begin{equation} \label{threeproj} \mathbb{H} \times \mathbb{H} \times \mathbb{H} \rightarrow \mathbb{H},\end{equation}
here $\mathbb{H}$ is the hyperbolic $3$-space, and $\nu$ the volume form on it.
Moreover,   cup product gives an isomorphism
$$ \wedge^3 H^3(Y, \Q)_{\triv}  \simeq H^9(Y, \Q) = H^9(Y, \Q)_{\triv}.$$ 

\begin{lemma}
\begin{equation} \label{Trivial_Volume} \vol  H^*(Y, \Q)_{\triv}  \sim 1\end{equation} 
\end{lemma}
\proof 
It is enough to show that \begin{equation} \label{needed} \vol(H^3_{\triv}(Y, \Q)) \vol(H^9_{\triv}(Y, \Q))  \sim 1.\end{equation}
because then Poincar{\'e} duality gives
$\vol(H^6_{\triv}(Y, \Q))_{\triv}   \vol(H^0_{\triv}(Y, \Q)) \sim 1$, and that gives the Lemma.
 To verify  \eqref{needed},  take an orthonormal basis $\omega_1, \omega_2, \omega_3$ for harmonic $3$-forms spanning $H^3(Y, \R)_{\triv}$. 
The  norm of each one  at  every point of $Y(K)$ (where the norm is that induced by the Riemannian structure)  equals $1/\sqrt{\vol(Y)}$, 
where we measure the volume of $Y$ with respect to the Riemannian measure. 
 The volume of $H^3(Y, \Q)$ equals (up to $\Q^*$, as usual)
$$ \frac{1}{ \int_{Y} \omega_1 \wedge \omega_2 \wedge \omega_3 }  = \sqrt{\vol (Y)}, $$
 since $\omega_1 \wedge \omega_2 \wedge \omega_3$ is a multiple of the volume form,   and its norm at each point
is $\vol(Y)^{-3/2}$.  On the other hand, the volume of $H^9(Y, \Q)$ equals the $L^2$-norm
of $\frac{d(\vol)}{\vol(Y)}$, with $d(\vol)$ the Riemannian volume form, i.e., $\vol(Y)^{-1/2}$.  That proves \eqref{needed}.   
\qed

\begin{lemma}\label{l2}
\begin{equation} \label{trivvolume}
  \vol^{\sigma}  \ H^*(Y, \Q)_{\triv} 
\sim \vol(\overline{Y})^2.\end{equation}
\end{lemma}
\proof  Notation as in \eqref{threeproj}, a generator $\omega_3$ for $H^3(Y, \Q)_{\triv}^{\sigma}$
is given as $\frac{ \pi_1^* \nu + \pi_2^* \nu + \pi_3^* \nu}{\vol \overline{Y}}$. To verify this,
recall that we have a map $\overline{Y} \rightarrow Y^{\sigma}$ (it is possible that this map is not surjective but it doesn't matter).   Each $\pi_i^* \nu$ pulls back to 
$\nu$ on $Y$, and in particular integrates to $\vol (\overline{Y})$, where the volume is measured w.r.t. $\nu$. 
Therefore $\int_{\overline{Y}} \omega_3 = 3$, so $\omega_3$ really does belong to $H^3(Y, \Q)$. 
The $L^2$-norm of $\omega_3$  is given by $\sqrt{  3 \cdot \frac{\vol Y}{\vol(\overline{Y})^2}}$.   Therefore, the left hand side of \eqref{trivvolume} is 
$\sim   \vol (Y)^{1/2} \cdot   \frac{ \vol(\overline{Y})}{\vol(Y)^{1/2}} \cdot  \frac{ \vol(\overline{Y})}{\vol(Y)^{1/2}}  \cdot \vol (Y)^{1/2} = \vol(\overline{Y})^2$. 
 \qed

\subsection{Proof of Prediction \ref{pred3}}  \label{pred3proof}
In what follows we abbreviate $$H^i_{\Pi} := H^i(Y, \Q)_{\Pi}$$
for the $\Pi$-summand of cohomology.

Combining  \eqref{RTalt}, \eqref{rt1}, \eqref{factorization}, \eqref{Trivial_Volume}  and Poincar{\'e} duality we get
 \begin{equation} \label{oinkA} \vol \  H^4_{\Pi} \in \sqrt{\Q^*} \cdot \vol  \ H^3_{\Pi}. \end{equation}
  Next, we have 
  \begin{equation} \label{oink1}   \vol^{\sigma} H^*_{\Pi}  \sim  \left( \vol H^*_{\pi} \right)^2 \ \end{equation}
  since from $ \mathrm{RT}^{\sigma}(Y) \stackrel{\eqref{rt2}}{=} \mathrm{RT}(\overline{Y})^2$ 
we get   \begin{equation} \label{oink3} \vol^{\sigma} H^*_{\Pi}(Y, \Q)  \cdot    \vol^{\sigma} H^*_{\triv}(Y, \Q)  \sim  \left( \vol H^*_{\pi}(\overline{Y} ,\Q) \cdot  \prod \vol H^*_{\triv}(\overline{Y}, \Q)\right)^2 \end{equation}
but  Lemma \ref{l2} implies that the contribution of the trivial representation on left and right cancel.
  
  Expanding \eqref{oink1}, noting that $H^3_{\Pi}$ is $\sigma$-fixed, and using Poincar{\'e} duality, we see  
$$ \left( \frac{  \vol^{\sigma} H^4_{\Pi} }{\vol H^3_{\Pi}} \right)^2 \sim \left( \frac{1}{\vol H^1_{\pi}}\right)^4,$$
that is to say
\begin{equation} \label{buggerme} \frac{ \vol^{\sigma} H^4_{\Pi}  \left(\vol H^1_{\pi}\right)^2 }{\vol H^3_{\Pi}} = \sqrt{q'},\end{equation}
for some $q' \in \Q^*$. This indeed verifies Prediction \ref{pred3}  up to $\sqrt{\Q^*}$. 

\subsection{Computation of $\vol H^3_{\Pi}$ and $\vol H^1_{\pi}$.}
In this case we know  \eqref{frodo}  both over $E$ and $F$: 
\begin{equation} \label{KP1} \vol(H^3_{\Pi})^2 \cdot\vol(L_{\Pi}^*) \sim \sqrt{q_1}.\end{equation} 
  \begin{equation} \label{KP2} \vol(H^1_{\pi})^2 \cdot \vol(L_{\pi}^*) \sim \sqrt{q_2}.\end{equation}
for $q_i \in \Q^*$. 

The computation of the periods  of cohomological forms on inner forms of $\GL(2)$ in minimal cohomological degree,  in terms of associated $L$-functions, was in essence done by Waldspurger \cite{WaldspurgerA, WaldspurgerB},
and \eqref{KP1}, \eqref{KP2} can be deduced from this computation, together with a  computation along the lines of \S \ref{GGP2} relating these $L$-functions  to $L_{\Pi}$ and $L_{\pi}$. 

However, we will now briefly outline how to deduce  \eqref{KP1} and \eqref{KP2}  directly from some mild variants of
 Theorem \ref{MainPeriodTheorem}, because that Theorem already has done all the appropriate normalizations and Hodge--linear algebra needed 
 to get the result in the desired form. We will focus on \eqref{KP2}; all steps of the proof of Theorem \ref{MainPeriodTheorem}, and the variant
we will  need below, go through with $F$ replaced by $E$ or indeed any CM field, and that  will give \eqref{KP1}. 
Besides this issue of working over $E$ rather than $F$, the reason we need  ``variants'' of Theorem \ref{MainPeriodTheorem} is to provide  enough flexibility to ensure that the $L$-values occuring are not zero.
 One pleasant feature of the current case is that the hypotheses of \S \ref{sec:periodintegrals}  
 are all known here. 
 
 We apply Theorem \ref{MainPeriodTheorem} with:
\begin{itemize} \item[-] $\GG$ the form of $\SO(3)$
defined by the reduced norm on the trace-free part of $D$; in particular $\GG(F) = D^*/F^*$. 
\item[-] $\HH \subset \GG$ the $\SO(2)$-subgroup defined by a  subfield $\tilde{F} \subset D$, quadratic over $F$, i.e. we have $\HH(F) = \tilde{F}^*/F^*$. 
\item[-] The cycle $Z(U)$ will be twisted, as in \S \ref{psidef}, by a quadratic idele class character $\psi$ of $\tilde{F}$, trivial on $F$.  
 \end{itemize}
 The twist mentioned was not used in Theorem \ref{MainPeriodTheorem}, but all steps of the proof go through.
 The only change is in the nonvanishing criterion in the last paragraph: one must replace the 
  Rankin-Selberg $L$-function by its $\psi$-twist. 
 
It is possible,  by Theorem  \cite[Theorem 4, page 288]{WaldspurgerShimura} and a local argument, given below,  to choose such  $\tilde{F},\psi$ in such a fashion that:
  \begin{itemize}
  \item[(a)] 
 $L(\frac{1}{2}, \mathrm{BC}_{F}^{\tilde{F}} \pi \otimes \psi) \neq 0$, and
 \item[(b)]   
 For $v$ a place of $F$ which remains inert in $\tilde{F}$, the local $\varepsilon$-factor $\varepsilon_v( \mathrm{BC}^{\tilde{F}}_F \pi \otimes \psi)$
 equals $-1$  when $D$ is ramified and otherwise $1$. 
 \end{itemize} 
 In both cases $\mathrm{BC}^{\tilde{F}}_F$ means base change (global or local) from $F$ to $\tilde{F}$. 
According to the last paragraph of Theorem \ref{MainPeriodTheorem}, together with the work of Tunnell--Saito \cite{Tunnell, Saito, prasad} relating invariant linear forms to $\varepsilon$-factors,  conditions (a) and (b) imply that the $\sqrt{\Q}$ ambiguity of the theorem statement
is actually nonzero, giving \eqref{KP1}.

Finally, we describe the local argument alluded to above.  We will find 
a pair of distinct quadratic idele class characters $\chi_1, \chi_2$ of $F$, and then construct $\tilde{F}, \psi$ from them, so that 
there is an equality of $L$-functions $L(\tilde{F}, \psi) = L(F, \chi_1) L(F, \chi_2)$.   (Thus, if $\chi_i$ corresponds to the quadratic extension
$F(\sqrt{d_i})$, we take $\tilde{F} = F(\sqrt{d_1d_2})$, and $\psi$ to correspond to the quadratic extension $F(\sqrt{d_1},\sqrt{d_2})$ over $\tilde{F}$). 

 Let $T$ be the set of ramified places for $D$. Let $S$ be the set of all places not in $T$ where $\pi$ is ramified, together with the archimedean places.
 Let $R$ be the remaining places. 
Our requirements (a) and (b) 
then translate to:
\begin{itemize} 
\item[(a)'] $L(\frac{1}{2}, \pi \times \chi_1) L(\frac{1}{2}, \pi \times \chi_2) \neq 0$, and 
\item[(b)']  $\varepsilon_v(\pi \times \chi_1) \varepsilon_v(\pi \times \chi_2)  \chi_1 \chi_2(-1) = \begin{cases} -1, v \in T \\ 1, v \in S \coprod R \end{cases}. $
\end{itemize}
 
 Let us recall  (see e.g. the summary in \cite[\S 1]{Tunnell}) that for $k$ a local field and $\sigma$ a representation of $\PGL_2(k)$, 
 the local epsilon factor $\varepsilon(\sigma,  \psi, 1/2) = \varepsilon(\sigma)$
 is independent of additive character $\psi$. Moreover, if $\sigma$ is a principal series,
 induced from the character $\alpha$ of $k^*$, we have $\varepsilon(\sigma) = \alpha(-1)$;
 if $\sigma$ is the Steinberg representation we have $\varepsilon(\sigma) = -1$,
 and for the unramified quadratic twist of the Steinberg representation have $\varepsilon(\sigma) = 1$.

 If $\chi$ is a quadratic idele class character of $F$ that is unramified at $T$ and trivial at $S$, 
 the global root number of the $\chi$-twist satisfies
 $$ \frac{ \varepsilon(\pi \times \chi) }{\varepsilon(\pi)} = \prod_{v \in T} \chi_v(\varpi_v) \cdot  \underbrace{ \prod_{v \in R} \chi_v(-1)}_{= \prod_{v \in S \coprod T} \chi_v(-1) = 1} = \prod_{v \in T} \chi_v(\varpi_v).$$
 In other words, twisting by such a $\chi$ changes the  global root number by a factor $(-1)^{t}$, where $t$ is the number of places in $T$ where $\chi$ is nontrivial.
  
 Choose $\chi_{1}$ and $\chi_2$  of this type such that $\chi_1$ and $\chi_2$ are ``opposed'' at each place of $T$
 (i.e. one is trivial and one is the nontrivial quadratic unramified character), and such that $\chi_1$ and $\chi_2$ are both trivial at each place of $S$.  Then 
\begin{equation} \label{oinkoink} \varepsilon_v(\pi \times \chi_1) \varepsilon_v(\pi \times \chi_2)  \chi_1 \chi_2(-1) = \begin{cases}  -1, v \in T \\ 1, v \in S \coprod R \end{cases}   \end{equation}
    
  The global root numbers of $\pi \times \chi_i$ is given by $\varepsilon(\pi) \cdot (-1)^{t}$, 
  where $t$ is the number of nontrivial places in $T$ for $\chi_1$ or $\chi_2$ (they have the same parity). Choosing $t$ appropriately
  we arrange that $\varepsilon(\pi \times \chi_1) = \varepsilon(\pi \times \chi_2) = 1$.

  Waldspurger's result implies that we may now find twists $\chi_1', \chi_2'$ of $\chi_1, \chi_2$, 
coinciding with $\chi_1, \chi_2$ at all places of $T \coprod S$, such that $L(\frac{1}{2}, \pi \times \chi_i') \neq 0$. 
The condition \eqref{oinkoink} continues to hold for the $\chi_i'$, so we  have achieved (a)' and (b)' as required.  

 \subsection{Proof of the remainder of Theorem \ref{SecondMainTheorem}} 
 
We must verify \eqref{j4} for $1 \leq i \leq 3$. 
Let us compute volumes of everything in sight in terms of the volumes of $L_{\Pi}$ and $L_{\pi}$.

First of all, 
\begin{equation} \label{brodo}   \vol(H^3_{\Pi} \otimes L_{\Pi}^{*})^2 \stackrel{\eqref{tensorvolume}}{\sim} \vol(H^3_{\Pi})^6 \cdot \vol(L_{\Pi})^{-2}  \stackrel{\eqref{KP1}}{\sim}    \vol(H^3_{\Pi})^2
\stackrel{\eqref{oinkA}}{\sim}   \vol(H^4_{\Pi})^2  \end{equation}
Also we have
\begin{align} \label{brodo1} \vol^{\sigma} (H^3_{\Pi} \otimes L_{\Pi}^*)^2   &= \vol^{\sigma}(L_{\Pi}^*)^2 \vol(H^3_{\Pi})^2  \stackrel{\eqref{Lpifix}}{\sim}  \vol(L_{\pi}^*)^2 \vol(H^3_{\Pi})^2\\
 \stackrel{\eqref{KP2}}{\sim}& \frac{(\vol H^3_{\Pi})^2}{(\vol H^1_{\pi})^4} \stackrel{\eqref{buggerme}}{\sim} \left( \vol^{\sigma} H^4_{\Pi}\right)^2.  \end{align}

We can now deduce the conclusions of the Theorem. 
First, it follows from \eqref{KP1} that 
 $$H^3_{\Pi} \cdot  \wedge^3 L_{\Pi}^* =  \sqrt{q_1} H^6_{\Pi},$$
 by comparing volumes -- both sides above are one dimensional $\Q$-vector spaces;  that proves \eqref{j4} for $i=3$ . 
 For $i=4$ we use the following lemma, applied with $L_1$ the image of  $H^3_{\Pi} \otimes L_{\Pi}^{*}$ in $H^4(Y, \R)_{\Pi}$, and
  $L_2 := H^4_{\Pi}$. 
 \begin{lemma*}
 Let $V_{\R}$ be a three-dimensional real vector space with metric,  equipped 
 with an isometric action of $\langle \sigma \rangle \simeq \Z/3\Z$, with $\dim V_{\R}^{\sigma} = 1$. 
 Suppose $V_1,  V_2 \subset V_{\R}$ are two different $\Q$-structures, both stable under $\sigma$. If
\begin{equation} \label{vol equ}  \vol(V_1) =  \vol(V_2), \ \ \vol^{\sigma}(V_1) = \vol^{\sigma}(V_2),\end{equation}
 then we have  $$V_1  \otimes \Q(\sqrt{b}) = \alpha (V_2 \otimes \Q(\sqrt{b}))$$ for a rotation $\alpha \in \R[\sigma]^*$ and some positive $b\in \Q^*$.
\end{lemma*} 

\proof We have an isomorphism $\Q[\sigma] \simeq \Q \oplus \Q[\zeta_3]$
and correspondingly we may split orthogonally
$$V_i = V_i^{\sigma} \oplus (V_i)_{\sigma}.$$
Since $V_1^{\sigma}, V_2^{\sigma}$ have the same volume, they are equal. 
On the other hand, $(V_1)_{\sigma} \otimes \R = (V_2)_{\sigma} \otimes \R$,
and these spaces are both isometric to $\R[\zeta_3]$ equipped with the 
  the standard quadratic form $|x+iy|^2 = x^2+y^2$. The images of $V_1, V_2$
  in $\R[\zeta_3]$ must be of the form   $\alpha_i \cdot \Q[\zeta_3]$ for some $\alpha \in \R[\zeta_3]^* \simeq \C^*$;
  since the volumes of these spaces coincide  in $\R^*/\Q^*$ we get   $|\alpha_1|^2  = b |\alpha_2|^2$ for some $b \in \Q^*$.  Therefore
  $V_1 \otimes \Q(\sqrt{b}), V_2 \otimes \Q(\sqrt{b})$ differ by a rotation as claimed.  \qed 

In our case we do not have the exact equality of volumes as in \eqref{vol equ}, but only equality up to certain factors in $\sqrt{\Q^*}$.
Correspondingly, we get $L_1= \alpha L_2$ only after first extending scalars to a field of the form $\Q(\sqrt{a_1}, b^{1/4})$.  
  This implies the case $i=4$ in the theorem. 
  
 Finally,  the case of $i=5$ follows from Poincar{\'e} duality:
   take $h, h' \in H^3_{\Pi}$ and $a \in \wedge^2 L_{\Pi}^{*}, a' \in    L_{\Pi}^{*}$. Then Lemma \ref{adjointness2} implies        $$ \langle h \cdot a, h' \cdot a' \rangle =    \langle h \cdot a a', h' \rangle \in \Q \cdot \sqrt{q_1} $$
   where $\langle -, - \rangle$ is the Poincar{\'e} duality pairing on $H^*(Y, \R)$, and we used \eqref{KP1} at the last step. 
Therefore, the three-dimensional  vector spaces $H^3_{\Pi} \cdot L_{\Pi}^{*} \otimes \overline{\Q}$
and $H^3_{\Pi} \cdot \wedge^2 L_{\Pi}^{*} \otimes \overline{\Q}$ are  dual to one another under the Poincar{\'e} duality pairing. 
    Since  the former space is a rotation of $H^4_{\Pi} \otimes \overline{\Q}$, as explained above, 
    we deduce that the latter space is a rotation of $H^5_{\Pi} \otimes \overline{\Q}$. This concludes the proof of the theorem.

\appendix

 \section{The motive of a cohomological automorphic representation} \label{Pimotive Appendix}

 In this appendix,  for lack of a sufficiently general reference, we shall formulate the precise conjectures relating
 cohomological automorphic representations to motives. 
 \subsection{ The notion of a $\Gh$-motive.}   \label{motive_general}  

\subsubsection{The motivic Galois group} 
Let $F$ be a number field. 
Assuming standard conjectures, the category $\cM_{F}$ of  Grothendieck motives over $F$ (with $\Q$-coefficients) is a neutral Tannakian category,  
 with fiber functor   sending the motive $M$ to the Betti cohomology of 
 $ M_v := 
 M \times_v \C$ for an embedding $v:F  \hookrightarrow \C$. 
 (See \S \ref{sec:purity}). 
 Fixing such $v$ gives a motivic Galois group (the automorphisms of this fiber functor),
 denoted $G_{\Mot}$. It is a pro-algebraic group over $\Q$; it depends on the choice of $v$, but we will suppress this dependence in our notation.

  For any object $M$ of $\cM_F$, we let $G_M$ denote the algebraic group over $\Q$ defined similarly but 
  with $\cM_F$ replaced by the smallest Tannakian subcategory containing $M$. Then 
  $G_M$ is of finite type and 
  \begin{equation}
  \label{eqn:gmot=invlimit}
  G_{\Mot} = \varprojlim_M G_M.
 \end{equation}
   The natural map 
$$ 
\rho_{M,\ell}: \Gamma_F \rightarrow \GL (H_\et^* (M_v, \Q_\ell )) =\GL (H_B^* (M_v) \otimes \Q_\ell) 
$$ 
factors through $G_M (\Q_\ell)$. Conjecturally the image of this map is Zariski dense in 
$G_M(\Q_\ell)$  \cite{serre} \S 3.2? (sic), and we will assume this in our discussion.

The groups $G_M$ and  $G_{\Mot}$ sit in exact sequences:  
$$
1 \rightarrow G_M^0 \rightarrow G_M \rightarrow \Gamma_M \rightarrow 1
$$
and 
$$
1 \rightarrow G_{\Mot}^0 \rightarrow G_{\Mot} \rightarrow \Gamma_F \rightarrow 1
$$
where $G_M^0$ and $G_{\Mot}^0$ denotes the identity components of $G_M$ and $G_{\Mot}$ respectively.
The group $\Gamma_F =\Gal (\overline{F}/F)$ 
may be viewed as the Tannakian group associated with the Tannakian category of Artin motives over $F$.

The Galois representations $\rho_{M,\ell}: \Gamma_F \rightarrow G_{M}(\Q_{\ell})$  yields, in the inverse limit, a map \begin{equation} \label{rholdef}
\rho_\ell: \Gamma_F \rightarrow G_{\Mot} (\Q_\ell)
\end{equation}
with the property that the composite map $\Gamma_F \rightarrow G_{\Mot} (\Q_\ell) \rightarrow \Gamma_F$ is the identity. 

\subsubsection{The group $\GC$} 
We will use the $C$-group defined in \cite{BG}, see in particular Proposition 5.3.3 therein. Let  $\Gth = (\widehat{G} \times \mathbb{G}_m)/\langle \Sigma \rangle,$
where $\Sigma$ is the order $2$ element defined by $(\Sigma_{\widehat{G}}(-1), -1)$, and $\Sigma_{\widehat{G}}$
is the co-character of $\widehat{T} \subset \widehat{G}$ corresponding to the sum of all positive roots for $\GG$. 
This has the property that the cocharacter $x \mapsto (x, \Sigma_{\widehat{G}}(x))$ from $\mathbb{G}_m \rightarrow  \widehat{G} \times \mathbb{G}_m$
admits a square root when projected to $\Gth$; this square root will be denoted by $\varpi$:
\begin{equation} \label{weightchardef} \varpi: \mathbb{G}_m \rightarrow \Gth,\end{equation}
so that we may informally write $\varpi(x)= ( \sqrt{x}, \Sigma_{\widehat{G}}(\sqrt{x}))$.

We define the $C$-group as the semidirect product 
 $$ 
\GC = \Gth \rtimes \Gamma_F
$$
where $\Gamma_F$ acts on $\widehat{G}$ in the natural way and on $\mathbb{G}_m$ trivially.  The action of $\Gamma_F$ on $\Gth$ factors through a finite quotient of $\Gamma_F$. 
We understand $\GC$ to be a pro-algebraic group defined over $\Q$. 

Note that, parallel to the structure of $G_{\Mot}$ noted above, there is an exact sequence $$
1 \rightarrow \Gth \rightarrow{}  \GC \rightarrow \Gamma_F \rightarrow 1.$$

  Just as for $\Gh$ itself, the complex algebraic groups $\Gth, \ \  \GC$ can be  descended to algebraic groups $\Gth, \  \GC$ over  $\Z$,    using the split Chevalley model of $\Gh$; thus their $R$-points make sense for any ring $R$ and, by a slight abuse of notation,
  we will allow ourselves to write $\Gth(R),  \ \GC(R)$ for these $R$-points. We also write $\Gth_R,  \ \GC_R$
  for the corresponding $R$-algebraic groups. 
   
    \subsubsection{$\Gh$-motives}  
A {\em $\Gh$-motive $X$ (over $F$)} will by definition be a   homomorphism
\begin{equation}\label{iotaX1} \iota_X: G_{\Mot, \overline{\Q}} \rightarrow {} \GC_{\overline{\Q}},\end{equation}
commuting with the projections to $\Gamma_F$ ,and whose projection to $\mathbb{G}_m/\{\pm 1\} \simeq \mathbb{G}_m$ 
gives the representation associated to the Tate motive $\Q(-1)$.

Here the subscripts refer to base extensions of these algebraic groups to $\overline{\Q}$.  
The morphisms between $G$-motives $X, Y$ will be understood to be 
the elements of 
$\Gth (\Qbar)$ conjugating $\iota_X$ to $\iota_Y$; in particular, the isomorphism class of $X$ depends only on the $\Gth(\Qbar)$-conjugacy class of $\iota_X$.

Then $X$ defines a functor (also denoted $X$) from finite-dimensional $\GC$-representations over $\overline{\Q}$ to the category of motives over $F$ with coefficients in $\Qbar$. In fact, this functor is a more intrinsic presentation of a $\Gh$-motive, because, after all, the motivic Galois group depends on a choice of fiber functor to begin with. 

Composing $\iota_X$ with $\rho_\ell$ (see \eqref{rholdef}) gives a map
$$
\rho_{X,\ell}: \Gamma_F \xrightarrow{\rho_\ell} G_{\Mot} (\Q_\ell) \rightarrow G_{\Mot,\Qbar} (\Q_\ell \otimes \Qbar) 
\xrightarrow{\iota_X}     {} \GC_{\overline{\Q}} (\Q_\ell \otimes \Qbar).
$$ 
Thus we get a representation $\rho_{X, \lambda}: \Gamma_F \rightarrow{}  \GC_{\overline{\Q}} (\overline{\Q}_{\lambda})$
for each prime $\lambda$ of $\overline{\Q}$ above $\ell$, with the property that the composite of this 
map with the projection $\GC \rightarrow \Gamma_F$ is the identity. 

\begin{lemma} [The Galois representation determines the motive] 
If $\rho_{X,\lambda}$ and $\rho_{Y,\lambda}$ are conjugate  under $\Gth(\Qbar_{\lambda})$ for some $\lambda$, then also $X,Y$ are isomorphic -- i.e., $\iota_X, \iota_Y$ define
the same $\Gth(\Qbar)$-conjugacy  class of maps.   \end{lemma}

\proof  
If   $\rho_{X, \lambda}$ and $\rho_{Y, \lambda}$ are conjugate, then
 $\iota_X$ and $\iota_Y$, considered as maps of $\overline{\Q}_{\lambda}$-algebraic groups, are conjugate on a Zariski-dense subset
of the source (by our assumption that the image of $\Gamma_F$ in $G_{\Mot}(\Q_{\ell})$ is dense).    Thus $\iota_X$ and $\iota_Y$ are conjugate over ${\overline{\Q}}_{\lambda}$. But then they are also conjugate over $\overline{\Q}$.
 \qed

If $(\rho,V_\rho)$ is a $\GC$-representation over $\Qbar$,  
we write $X_{\rho}$ for the associated motive, i.e. the motive  with $\Qbar$ coefficients associated to the composite $\rho \circ \iota_X$. 
There is a tautological isomorphism \begin{equation} \label{tau} H_B(X_{\rho} \times_{v} \C, \overline{\Q})   \simeq V_{\rho}. \end{equation} 

  \subsection{The $\widehat{G}$ motive attached to a cohomological automorphic representation} \label{GhC}
  
  Now let $F=\Q$; we will formulate the precise connections between cohomological automorphic representation for $\G$, 
  and $\Gh$-motives. 

 It is convenient to start with a character
$\chi: \mathscr{H} \rightarrow \overline{\Q}$ of the cohomological Hecke algebra for $Y(K)$, as in \S \ref{sec:numerical_invariants}  but allowing $\overline{\Q}$ values. 
Attached to each embedding $\lambda: \overline{\Q} \hookrightarrow \C$ there is a near-equivalence class of cohomological automorphic representation
$\Pi^{\lambda}$ whose Hecke eigenavalues coincide with  $\lambda \circ \chi$.

Attached to $\chi$ there should be a compatible system of Galois representations to $\GC$ in the following sense:
 For each nonarchimedean place $\lambda$ of
 $\overline{\Q}$ we should have \cite[Conjecture 5.3.4]{BG} attached a distinguished conjugacy class of maps
  \begin{equation}  \label{rho1} \rho_{\lambda} :  \Gal(\overline{\Q}/\Q)  \longrightarrow{}  \GC(\Qbar_{\lambda} ) \   \ \ \mbox{ $\lambda$ nonarchimedean};  \end{equation}
which matches with $\lambda \circ \chi$ under the Satake correspondence,
(see {\em loc. cit.} for details). 

The basic conjecture concerning the existence of motives  (cf. the discussion at the end of \cite[\S 2]{LanglandsMarchen})\  is then the following: \ \begin{conj} \label{Conj Pimotive}
Given a cohomology class as above, there exists a $\Gh$-motive $X$ over $\Q$,  with the property that
for each nonarchimedean $\lambda$ the Galois representation
$\rho_{\lambda}$ attached to the cohomology class  is isomorphic to the Galois representation $\rho_{X, \lambda}$ arising from $X$. 
\end{conj}

\subsection{Descent of the coefficient field for a $\widehat{G}$-motive }
In \S \ref{GhC} we have formulated the conjectures over $\overline{\Q}$. However if the Hecke character $\chi$ takes values in a
subfield $E \subset \overline{\Q}$ it is of course preferable to work over $E$.  In the current section, we outline
how to do this, i.e. how to descend the coefficient field of a $\Gh$-motive, at the cost of replacing $\widehat{G}$ by an inner form. 
 
\subsubsection{Twisting  a Galois representation}  \label{Twisting} 
Let us first recall how to ``apply a Galois automorphism to a representation.''

Suppose that $H$ is an algebraic group over $\overline{\Q}$,
and $\sigma$ is an automorphism of $\overline{\Q}$. 
We can define the $\sigma$-twist $H^{\sigma}$: 
if $H$ is defined by various equations $f_i = 0$, then $H^{\sigma}$
is defined by the equations $f_i^{\sigma} = 0$, and so on;
if $H$ is defined over $\Q$ there is a canonical isomorphism $H \simeq H^{\sigma}$. 
Also $\sigma$ induces a bijection $H(\overline{\Q}) \rightarrow H^{\sigma}(\overline{\Q})$
denoted by $h \mapsto h^{\sigma}$. 

In particular, given a homomorphism $\pi: H \rightarrow H'$ of $\overline{\Q}$-algebraic groups,
we obtain a twisted morphism $\pi^{\sigma}: H^{\sigma} \rightarrow (H')^{\sigma}$,
with the property that $\pi^{\sigma}(h^{\sigma}) = \pi(h)^{\sigma}$.

\subsubsection{Descent of coefficients for a motive} \label{shrinky coeff}
Let $X$ be a $\Gh$-motive. 
For  $\sigma \in \Gal(\overline{\Q}/\Q)$, we can form a new motive $X^{\sigma}$ by the rule $$\iota_{X^\sigma} = (\iota_X)^\sigma.$$
Informally, $X^{\sigma}$ applies $\sigma$ to the coefficients of the system of motives defined by $X$. 

Now let $E$ be a finite extension of $\Q$, and suppose that $X^{\sigma} \simeq X$ for all $\sigma \in \Gal(\overline{\Q}/E)$.   In particular,
there exists an element $g_{\sigma} \in \Gth(\overline{\Q})$ with the property that
$$ \Ad(g_{\sigma}) \iota_X = \iota_{X^{\sigma}}.$$ 
Explicitly, this means that for $g \in G_{\Mot}(\overline{\Q})$ we have
$\Ad(g_{\sigma}) \iota_X(g^{\sigma}) = \iota_{X}(g)^{\sigma}$, 
so that the image of $G_{\Mot}(\Q)$ is fixed under $z \mapsto \Ad(g_{\sigma}^{-1})  z^{\sigma}$.

The element $g_{\sigma}$ is determined up to $\overline{\Q}$-points of $Z(\iota_X)$, the centralizer of $\iota_X$ inside $\Gth_{\overline{\Q}}$.  In particular, 
if the centralizer of $\iota_X$ coincides with the center of $\Gth_{\overline{\Q}}$, the rule $\sigma \mapsto g_{\sigma}$ 
defines a cocycle; its cohomology class lies in $$H^1(\Gal(\overline{\Q}/E), \Gth(\overline{\Q}) \mbox{ modulo center}) = H^1(E,  \widehat{G} \mbox{ modulo center}),$$
where we use the usual notation for Galois cohomology on the right. 

This cocycle can be used to descend  $\widehat{G}_{\overline{\Q}},\Gth_{\overline{\Q}}$ and $\GC_{\overline{\Q}}$ 
to $\Q$-forms $\widehat{G}_*, \Gth_*,  \ \GC_*$, described as the fixed points of $ z \mapsto \Ad(g_{\sigma}^{-1}) z^{\sigma}$ on the respective (pro)-groups.  
We may then descend $\iota_X$ to a morphism
\begin{equation} \label{iotaXdescent} \iota_X :  G_{\Mot} \longrightarrow \  \GC_*  \ \ \ \ \ \mbox{ (morphism of $E$-groups) }\end{equation}
Composition with the adjoint representation of $\GC_*$ should then yield the adjoint motive described in Definition  \ref{AdjointMotiveDefinition}.

 \subsection{Standard representations of the $c$-group for $\PGL$ and $\SO$}
  \label{sec:archpar}
  
 According to our prior discussion, a cohomological form for $\GG$ gives rise to a $\widehat{G}$-motive with $\overline{\Q}$ coefficients;  
 in particular, a representation of $\GC$ gives rise to a usual motive with $\Qbar$ coefficients.  
 The Hodge weights of the resulting motive are given by the  eigenvalues of the weight
 cocharacter \eqref{weightchardef}. 
 
 In the remainder of this section,
 we specify a standard representation of the $c$-group in the cases
 of interest, namely, $\GG =\PGL_n$ and $\GG = \SO_m$.    
 We will compute the Hodge numbers both for this motive (denoted $M$) 
 and for the motive associated to the adjoint representation of $\GC$ (denoted $\Ad M$). 
 We work over an arbitrary number field $F$;
 in the text, $F$ will sometimes be an imaginary quadratic extension of $\Q$.

  \begin{itemize}
  \item[-] $G=\PGL_{n}, \Gth = \SL_n \times \mathbb{G}_m/ (  (-1)^{n+1} \mathrm{Id}_n, -1)$. 
 
 Here \begin{equation}
 \label{varpidefPGL} \varpi(x) =    (  \Sym^{n-1} \begin{bmatrix}  \sqrt{x} & \\ &  1/\sqrt{x} \end{bmatrix},\sqrt{x}),\end{equation}
and we define the standard representation of $\Gth$ to be the tensor product of the character  $x \mapsto x^{n-1}$
  on $\mathbb{G}_m$ with the standard representation of $\SL_n$.  This extends to $\GC$, by extending trivially on $\Gamma_F$.

  Thus the Hodge numbers of $M$ are 
  $$
  (n-1,0), (n-2,1) \ldots, (1, n-2), (0,n-1)
  $$
  each with multiplicity one, and the Hodge numbers of $\Ad M$ are 
\begin{align*}  (n-1, -(n-1))^1,  (n-2, -(n-2))^2 \ldots, (1,-1)^{n-1}, (0,0)^{n-1}, \\  (-1,1)^{n-1} \ldots, (-(n-2),n-2)^2, (-(n-1),n-1)^1
  \end{align*}
where we wrote the multiplicities as superscripts.
 
 \bigskip

\item[-] $G=  \SO_{2n}, \Gth = \mathrm{SO}_{2n} \times \mathbb{G}_m/(\mathrm{Id}_n,-1)$.

 Here  
   $$\varpi(x)  =   \left( \Sym^{2n-2} \begin{bmatrix}  \sqrt{x} & \\ &  1/\sqrt{x}   \end{bmatrix} \oplus \mathrm{Id}_1, \sqrt{x} \right).$$  
where $\mathrm{Id}_1$ is the identity matrix in one dimension, and we define the standard representation of $\Gth$ to be the tensor product of the standard representation 
on $\SO_{2n}$ and the character $x \mapsto x^{2n-2}$ on $\mathbb{G}_m$.  This extends to $\GC$:  first extend it to $\mathrm{O}_{2n} \times \mathbb{G}_m/(1,-1)$, and 
then use the map $\GC \rightarrow  \mathrm{O}_{2n} \times \mathbb{G}_m/(1,-1)$
extending the inclusion of $\Gth$; here the map $\Gamma_F \rightarrow \mathrm{O}_{2n}$
should induce the natural action of $\Gamma_F $ on $\mathrm{SO}_{2n} = \widehat{G}$ by pinned automorphisms.   
 
 Thus the Hodge numbers of $M$ are 
$$ (2n-2,0)^1, (2n-3,1)^1, \cdots , (n-1,n-1)^2, \cdots, (1,2n-3)^1, (0,2n-2)^1 $$
and the  Hodge numbers of $\Ad(M)$ (which is of rank
$n(2n-1)$) range from $(2n-3,-(2n-3))$ to $(-(2n-3),(2n-3))$ and admit a 
pattern that depends on the parity of $n$. If $n=2t$, the multiplicities are given by
 $$
 1,1,\cdots ,\overline{t,t}, \cdots ,n-1,n-1, n,n,n,n-1,n-1, \cdots ,\overline{t,t}, \cdots,1,1, 
 $$
 where the bar above $(t,t)$ indicates that those terms are skipped. If $n=2t+1$, then the 
multiplicities are
 $$
 1,1,\cdots , \overline{t,t+1},  \cdots n-1,n-1,n,n,n,n-1,n-1, \cdots  \overline{t+1,t},  \cdots,1,1,
 $$
 where again the bar has the same interpretation as before. 

\bigskip

\item[-] $G= \SO_{2n+1}, \Gth =   \Sp(2n) \times \mathbb{G}_m/(-\mathrm{Id}_n, -1).$

Here $\varpi$ is given   by 
$$\varpi(x) =    (  \Sym^{2n-1} \begin{bmatrix}  \sqrt{x} & \\ &  1/\sqrt{x} \end{bmatrix},\sqrt{x}),$$
and we define the  standard representation of $\Gth$ to be the tensor product of the standard representation of $\Sp(2n)$
and the character $x \mapsto x^{2n-1}$ on $\mathbb{G}_m$. 
 This extends to $\GC$, by extending trivially on $\Gamma_F$. 
 
The Hodge numbers of $M$ are 
$$(2n-1,0), (2n-2,1), \ldots, (1,2n-2), (0,2n-1),$$
each with multiplicity one. 
The Hodge numbers of $\Ad(M)$ (which is of rank
$n(2n+1)$) range from $(2n-1,-(2n-1))$ to $(-(2n-1),(2n-1))$ and have multiplicities
 $$
 1,1,2,2,\cdots,n-1,n-1, n,n,n,n-1,n-1, \cdots , 2,2,1,1.
 $$

\end{itemize}

 \bibliography{PVBib}
   \bibliographystyle{plain}
  \end{document}